%% file: manuscript.tex
\documentclass[a4paper,11pt]{article}
\usepackage[export]{adjustbox}
\usepackage{xcolor}
\definecolor{lightblue}{HTML}{044E9E}
\definecolor{NWOblue}{RGB}{24, 101, 124}
\usepackage[hyperfootnotes = true,colorlinks,citecolor=NWOblue,urlcolor=NWOblue, linkcolor=NWOblue]{hyperref}

\usepackage{amsmath}
\usepackage{graphicx}
\usepackage{wrapfig}
\usepackage{psfrag}
\usepackage{epsf}
\usepackage{enumitem}
\usepackage[round,longnamesfirst,comma]{natbib}
\usepackage{url} 
\usepackage{amsthm}
\usepackage{xpatch}
\usepackage{subcaption,bm}
\usepackage{csquotes}

\usepackage{longtable}
\usepackage{mathrsfs}
\usepackage{tikz}
\usepackage{color}
\usepackage{amssymb}
\usepackage{latexsym}
\usepackage{xr}
\usepackage{epstopdf}
\usepackage{float}
\usepackage{multirow}
\usepackage{dsfont}
\usepackage{mathabx}
\usepackage{verbatim}

\numberwithin{equation}{section}

\usepackage{sectsty}
\subsectionfont{\normalfont\itshape}
\subsubsectionfont{\normalfont\itshape}

\newcommand{\ignore}[1]{}

\def\RR{\mathbb R}
\def\ZZ{\mathbb Z}

\def\NN{\mathbb N}

\newcommand{\PP}{\operatorname{P}} 
 
\newcommand{\E}{\operatorname{E}} 
\newcommand{\Var}{\operatorname{Var}} 
\newcommand{\Cov}[0]{\operatorname{Cov}}

\newcommand \abs[1]{\left|#1\right|}

\newcommand{\pr}[1]{\left(#1\right)}

\newcommand{\1}{\mathds{1}}
\newcommand{\e}{{\operatorname{e}}}

\renewcommand{\leq}{\leqslant}
\renewcommand{\geq}{\geqslant}

\makeatletter
\newcommand*{\defeq}{\mathrel{\rlap{%
           \raisebox{0.3ex}{$\m@th\cdot$}}%
           \raisebox{-0.3ex}{$\m@th\cdot$}}%
           =}
\makeatother

\makeatletter
\newcommand*{\eqdef}{=\mathrel{\rlap{%
           \raisebox{0.3ex}{$\m@th\cdot$}}%
           \raisebox{-0.3ex}{$\m@th\cdot$}}%
           }
\makeatother

\newcommand{\s}{\operatorname{s}}

\newtheorem{lemma}{Lemma}[section]

\newtheorem{theorem}[lemma]{Theorem}
\newtheorem{model}[lemma]{Model}

\newtheorem{definition}[lemma]{Definition}

\newtheorem{example}[lemma]{Example}

\newtheorem{remark}[lemma]{Remark}

\makeatletter
\xpatchcmd{\proof}{\@addpunct{.}}{\@addpunct{:}}{}{}
\makeatother

\DeclareFontFamily{U}{mathx}{\hyphenchar\font45}
\DeclareFontShape{U}{mathx}{m}{n}{<-> mathx10}{}
\DeclareSymbolFont{mathx}{U}{mathx}{m}{n}
\DeclareMathAccent{\widebar}{0}{mathx}{"73}

\usepackage{color}
\usepackage{marginnote}

\usepackage{listings}
\usepackage{geometry}
\begin{document}

\def\spacingset#1{\renewcommand{\baselinestretch}%
{#1}\small\normalsize} \spacingset{1.5}

\allowdisplaybreaks

\title{{\bf Higher-order approximation for uncertainty quantification in time series analysis}}

\author{
Annika Betken\hyperlink{myth}{\parbox{3pt}{\footnotemark[1]}} \\ \textit{University of Twente}
\and
Marie-Christine D\"uker\hyperlink{myth}{\parbox{3pt}{\footnotemark[2]}} \\ \textit{FAU Erlangen-N\"urnberg} }
\maketitle

\let\oldthefootnote\thefootnote
\renewcommand{\thefootnote}{\fnsymbol{footnote}}
\footnotetext[1]{\phantomsection\hypertarget{myth}{University of Twente, Faculty of Electrical Engineering, Mathematics and Computer Science, Enschede, Netherlands, 
\protect \url{a.betken@utwente.nl}}}
\footnotetext[2]{\phantomsection\hypertarget{myth}{FAU Erlangen-N\"urnberg, Department of Data Science, Erlangen, Germany,
\protect \url{marie.dueker@fau.de}}}
\let\thefootnote\oldthefootnote

\maketitle

\begin{abstract}
\noindent
For time series with high temporal correlation, the empirical process converges rather slowly to its limiting distribution. Many statistics in change-point analysis, goodness-of-fit testing and uncertainty quantification admit a representation as functionals of the empirical process and therefore inherit its slow convergence. As a result, inference based on the asymptotic distribution of those quantities is significantly affected by relatively small sample sizes. 
We assess the quality of higher-order approximations of the empirical process by deriving the asymptotic distribution of the corresponding error terms. Based on the limiting distribution of the higher-order terms, 
we propose a novel approach to calculate confidence intervals for statistical quantities such as the median.
In a simulation study, we compare coverage rates and lengths of these confidence intervals with those based on the asymptotic distribution of the empirical process and highlight some benefits of higher-order approximations of the empirical process.

\end{abstract}

\section{Introduction} \label{se:intro}

Let $X_n$, $n=1, \ldots, N$, be a time series stemming from a stationary stochastic process $X_n$, $n\in \NN,$ with marginal distribution function $F$, such that $F(x) = \PP(X_n \leq x)$ for all $n \in \NN$.
We study the empirical distribution function $F_N(x):=\frac{1}{N}\sum_{n=1}^N\1_{\left\{X_n\leq x\right\}}$. 
The rate of convergence, i.e., the increase of the sequence 
$a_N$, $N\in \NN$, which ensures weak convergence 
of the empirical process
\begin{align} \label{eq1}
a^{-1}_N N(F_N(x)-F(x))
\end{align}
to a non-degenerate limit, crucially depends on the 
behavior of the process' autocorrelation function
$\gamma(k):=\Cov(X_1, X_{k+1})$.
For short-range dependent time series, i.e., for stochastic processes with summable autocorrelations, $a_N=\sqrt{N}$. In contrast, for long-range dependent time series, i.e., for $\gamma(k)=k^{-(2-2H)}L(k)$ with $L$ a slowly varying function and $H \in (1/2, 1)$ the so-called Hurst parameter, we have $a_N=N^{H}L^{\frac{1}{2}}(N) $. In fact, under long-range dependence, the distribution of the empirical process converges much slower to its limit than under short-range dependence.
\par
To illustrate the slow convergence of the empirical process under strong temporal correlation, we would like to draw the reader's attention to Figure \ref{fig:empdistr_histograms}. The figure depicts the asymptotic behavior of the centered and standardized empirical distribution $F_N(x)$ evaluated at $x=0$ for different sample sizes. The underlying process is assumed to be fractional Gaussian noise with
Hurst parameter $H$. In this case, the sequence $a_N$, $N\in \NN$, in \eqref{eq1} can be explicitly calculated as $a_N = \varphi(0) N^H$ with $\varphi$ denoting the standard Gaussian density.
The quantity $\frac{1}{\varphi(0)}N^{1-H}(F_N(0)-F(0))$ is computed independently 10000 times for different sample sizes $N$ and the resulting values are summarized in the histograms in 
Figure \ref{fig:empdistr_histograms}. Due to Theorem 1.1 in \cite{dehling1989}, the quantity is expected to converge to a standard Gaussian random variable. Therefore, the histograms in Figure \ref{fig:empdistr_histograms} are expected to approach the standard Gaussian density function (depicted in red).
The convergence rate depends on the value of the Hurst parameter in that a small value ($H = 0.55$) results in relatively fast convergence while a large value ($H = 0.95$), and implied stronger temporal correlation, results in much slower convergence. This phenomenon is specific to long-range dependent time series and the focus of this work.

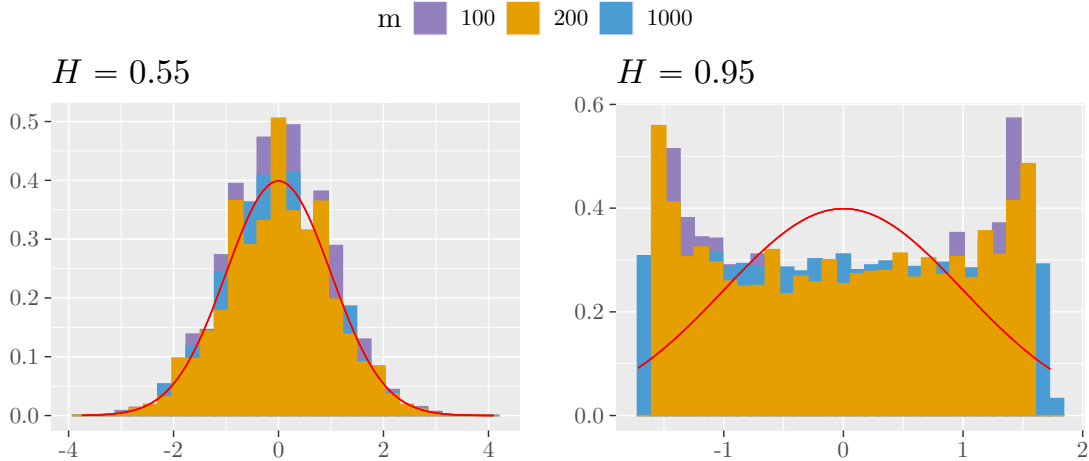
\begin{figure} 
\center
\scalebox{0.9}{\input{plots/empdistr_histograms.tex}}
\vspace{-0.5cm}
\caption{The empirical distribution of the centered and standardized empirical distribution $F_N(x)$ evaluated at zero ($x=0$) under the assumption of Gaussian long-range dependent data with different Hurst parameters and for sample sizes $m = 100, 200, 1000$. The red line depicts the standard Gaussian density function.}
\label{fig:empdistr_histograms}
\end{figure}

\par
The empirical process serves as a powerful tool for characterizing the asymptotic behavior of a variety of test statistics used in change-point analysis and goodness-of-fit testing (Wilcoxon, Kolmogorov-Smirnov and Cram\`{e}r-von Mises statistics); see \cite{Beran1992Stat, Dehling2013,betken2016testing,betken2017change,tewes2018block}. 
When testing the hypothesis of stationarity against the alternative hypothesis of a structural change in a time series, the
phenomenon illustrated in Figure \ref{fig:empdistr_histograms} results in a high number of false positives; see \cite{Dehling2013}.
\par
Against this background, the contribution of this paper is twofold: On the one hand, we address a statistical issue arising in the calculation of confidence intervals under strong temporal correlation. On the other hand, we push forward the theoretical investigation of the empirical process by proving a novel limit theorem.
More precisely:
\begin{itemize}
\item
We study the construction of confidence intervals for the marginal distribution of stationary time series data and confidence intervals for its quantiles in long-range dependent time series. 
We propose a novel approach to calculate confidence intervals based on a higher-order approximation of the empirical distribution function.
Under long-range dependence (LRD), an asymptotic expansion of the empirical distribution that is similar in spirit to a Taylor expansion can
be derived. This expansion can be used to obtain higher-order approximations of certain
statistical functionals of the empirical process.
\item
We establish the theoretical validity of our method for statistics that can be considered as functionals of the empirical process.
For statistical applications beyond the construction of confidence intervals, e.g., change-point and goodness-of-fit tests, uniform convergence of the one-parameter empirical process
\eqref{eq1} does not suffice in order to derive limit distributions of corresponding statistics. These typically require consideration of the two-parameter (or sequential) empirical process
\begin{align*} \label{eq:two_peram_process}
a_N^{-1}\lfloor Nt\rfloor(F_{\lfloor Nt\rfloor}(x)-F(x)), \ t\in [0, 1], \ x\in \RR.
\end{align*}
We derive the asymptotic distribution of higher-order approximations of the sequential empirical process by proposing a new chaining technique. 
\end{itemize}

Constructing confidence intervals for unknown quantities in time series is a problem of substantial interest in statistics. 
In the statistical literature, the main focus has been on approximating the limiting distribution through finite sample procedures like subsampling and bootstrapping; see \cite{buhlmann2002bootstraps, shao2010self, NordmanBunzelLahiri2013, kim2015nonparametric, HuangShao2016}. From an entirely theoretical perspective, \cite{YOUNDJE2006109} investigate consistency properties of kernel-type estimators of quantiles under long-range dependence.
The interest in confidence intervals is also due to their relevance for uncertainty quantification in other sciences where they are used in a variety of fields including climate science, economics, finance, industrial engineering and machine learning; see \cite{massah2016confidence}, \cite{fang2018online}, \cite{hoga2019confidence}, \cite{purwanto2021partial}.
\par
Empirical process theory became one of the major themes in the historical progress of non-parametric statistics; see \cite{Donsker1952justification}, \cite{dudley1978central}, \cite{doukhan1998functional}, \cite{shorack2009empirical}, \cite{wellner2013weak}.
The applications are manifold, especially since many statistics have a representation as functionals of the empirical process, such that statistical inference can be based on the properties of the empirical process itself. 
In the empirical sciences, confidence intervals for unknown parameters or critical values for hypothesis tests
are derived from the distributional properties of the empirical process.
\par
For stationary Gaussian processes \cite{koul:surgailis:2002}, derived the asymptotic distribution of higher-order terms of the empirical process.
We extend their results substantially by considering the sequential empirical process and by allowing the underlying time series to be driven by subordinated Gaussian processes. Subordination extends the model's flexibility by allowing for a large class of marginal distributions. Furthermore, we are the first to propose a utilization of higher-order approximations of the empirical process for the calculation of confidence intervals which are robust to high temporal correlation in time series data. 
\par
Although long-range dependent processes are a popular modeling tool in a variety of domains (\cite{rust2011confidence}, \cite{WERON2002285}), the construction of confidence intervals under long-range dependence has not gotten much attention. 
We provide an empirical study comparing confidence intervals derived from the asymptotic distribution of the empirical process to confidence intervals based on higher-order approximations of the empirical process.
\par
For the population mean, \cite{hall1998sampling} propose a sampling window method to set confidence intervals under long-range dependence.
\cite{nordman2007empirical} consider the empirical likelihood for confidence intervals. For mean functions, \cite{Bagchietal2016} study a monotone function plus noise model with potential long-range dependence in the noise term and derive confidence intervals for the monotone functions. In contrast, we deal with a different, rank-based class of statistics.
\par
The literature review, as well as our motivation illustrated in Figure \ref{fig:empdistr_histograms}, show the strong influence of high temporal correlation on the performance of statistics derived from the empirical process. In this paper, we aim to address this issue by introducing a procedure based on higher-order approximations of the empirical process to construct confidence intervals for statistics of long-range dependent time series robust to high temporal correlation. Our theoretical contribution is of independent interest and potentially has further applications in change-point analysis and goodness-of-fit testing.
Furthermore, a reduction to a limit theorem in the short-range dependent regime allows an application of established resampling procedures such as the moving block bootstrap which has been proved to be invalid under long-range dependence; see
\cite{lahiri1993moving}.
\par
The rest of the paper is organized as follows: In Section \ref{se:prelim}, we introduce the considered setting in all details. Section \ref{se:HOA} motivates the consideration of higher-order approximations of the empirical process. Section \ref{se:mainresults} focuses on theoretical contributions which manifest the formal validity of the proposed method. In Section \ref{Confidenceintervals}, we discuss how to calculate confidence intervals based on the asymptotic distribution of the empirical process and propose an alternative approach based on higher-order approximations. The numerical study in Section \ref{se:numericalstudy} provides a comparison between the two methods. We conclude with Section \ref{se:conclusion}. Proofs of the theoretical results can be found in Appendices \ref{app:proofmain}, \ref{se:appendixB}, \ref{se:appendixc1}, and \ref{se:appendixC}.

\section{Preliminaries} \label{se:prelim}
While Section \ref{se:intro} provides insight into the motivation for considering higher-order approximations of statistics, we introduce here model assumptions which allow for this type of approximations (Section \ref{se:setting}) and give some technical details necessary for our analysis (Section \ref{se:FoGRV}).

\subsection{Setting} \label{se:setting}
For future reference, we subsume assumptions on the data-generating process under the following model specification:
\begin{model} \label{model}
Let $X_n$, $n \in \NN$, be a subordinated Gaussian process, i.e., $X_n=G(\xi_n)$ for some measurable function $G:\RR\to \RR$ and with $\xi_n$, $n\in\NN$, denoting a (standardized) long-range dependent Gaussian process, i.e., $\E( \xi_n) =0$, $\Var(\xi_n)=1$, and 
\begin{equation} \label{eq:ACF}
\gamma(k)=\Cov(\xi_1, \xi_{k+1})=\E(\xi_1 \xi_{k+1})=k^{-D}L(k),
\end{equation}
where $D \in (0,1)$ (the so-called long-range dependence (LRD) parameter) and $L$ a slowly varying function.
\end{model}

Relation \eqref{eq:ACF} corresponds to one of multiple different ways to define long-range dependence. A more general definition characterizes long-range dependent time series by the non-summability of the absolute values of its autocovariance function; see (2.1.6) in \cite{PipirasTaqqu}. In fact, \eqref{eq:ACF} implies that the series of the autocovariances diverges. We refer to Chapter 2.1 in \cite{PipirasTaqqu} for a detailed representation of different ways to define long-range dependence and their relations to each other. 

For any particular distribution function $F$, 
an appropriate choice of the transformation $G$
yields subordinated Gaussian processes with marginal distribution $F$.
Moreover, there exist algorithms for generating Gaussian processes that,
after suitable transformation, yield subordinated Gaussian processes with marginal distribution $F$ and a predefined covariance structure; see \cite{PipirasTaqqu}. 

The following example presents a process which satisfies Model \ref{model}.
\begin{example}[Definition 2.8.3 in \cite{PipirasTaqqu}] \label{ex:fGn}
Let $B_{H}(t), t \in \RR$, be a fractional Brownian motion. Then, the process $\xi_{H}(k), k \in \ZZ$, defined by
$$\xi_{H}(k) := B_{H}(k + 1) - B_{H}(k)$$ is called fractional Gaussian noise with Hurst parameter $H$.
\end{example}

\subsection{Gaussian subordination} \label{se:FoGRV}

In the study of functionals of Gaussian processes, Hermite polynomials play a fundamental role. In particular, they form a basis for the space of finite-variance functions of Gaussian random variables. Since they are an inevitable tool in our analysis, we provide a short review.

Let $L^2(\RR, \varphi(x)dx)$ be the space of functions which are square-integrable with respect to the Gaussian measure (here denoted by $\varphi(x)dx$). For $g\in L^2(\RR, \varphi(x)dx)$ and $\xi_n$, $n\in\NN$, we call the
sequence $X_n=g(\xi_n)$, $n\in\NN$, a
subordinated Gaussian sequence.

A collection of orthogonal elements in $L^2(\RR, \varphi(x)dx)$ is given by the sequence of Hermite polynomials; see Proposition 5.1.3 in \cite{PipirasTaqqu}.

\begin{definition}
For $n\geq 0$, the Hermite polynomial of order $n$ is defined by
\begin{align*}
H_n(x)=(-1)^{n}\e^{\frac{1}{2}x^2}\frac{d^n}{d x^n}\e^{-\frac{1}{2}x^2}, \ x\in \RR.
\end{align*}
\end{definition}
The Hermite polynomials form an orthogonal basis of 
$L^2(\RR, \varphi(x)dx)$. As a result, every $g \in L^2(\RR, \varphi(x)dx)$
has an expansion in Hermite polynomials, i.e., for $g \in L^2(\RR, \varphi(x)dx)$ and $\xi$ standard normally distributed, we have
\begin{align}\label{eq:Hermite_expansion}
g(\xi)=\sum\limits_{r=0}^{\infty}\frac{J_r(g)}{r!}H_r(\xi),
\hspace{0.2cm}
J_r(g)=\E g(\xi)H_r(\xi),
\end{align}
where $J_r(g), r \geq 0$, are the so-called {\em Hermite coefficients}. 

Given the Hermite expansion
\eqref{eq:Hermite_expansion}, it is possible to characterize the dependence structure of subordinated Gaussian time series $g(\xi_n)$, $n\in \NN$.
In fact, it holds that
\begin{align}\label{eq:cov_sub_Gaussian}
\Cov(g(\xi_1), g(\xi_{k+1}))=\sum\limits_{r=1}^{\infty}\frac{J^{2}_r(g)}{r!}\gamma^r(k),
\end{align}
where $\gamma$ denotes the autocovariance function of $\xi_n$, $n\in \NN$; see Proposition 5.1.4 in \cite{PipirasTaqqu}.
 Under the assumption that, as $k$ tends to $\infty$, $\gamma(k)$ converges to $0$ with a certain rate, the asymptotically dominating term in the series \eqref{eq:cov_sub_Gaussian} is the summand corresponding to the smallest integer $r$ for which the Hermite coefficient $J_r(g)$ is non-zero. This index, which decisively depends on $g$, is called {\em Hermite rank}. 

\begin{definition}[Definition 5.2.1 in \cite{PipirasTaqqu}]
Let $g \in L^2(\RR, \varphi(x)dx)$ with $\E g(\xi)=0$ for standard normally distributed $X$ and let $J_r(g)$, $r\geq 0$, be the Hermite coefficients in the Hermite expansion of $g$. The smallest index $k\geq 1$ for which $J_k(g)\neq 0$ is called the Hermite rank of $g$, i.e.,
\begin{align*}
r\defeq \min\left\{k\geq 1: J_k(g)\neq 0\right\}.
\end{align*}
\end{definition}

\section{Higher-order approximation} \label{se:HOA}

We utilize our model assumptions and 
give details on a characterization of the empirical process as a sum of first- and higher-order terms.

Given time series data $X_1, \ldots, X_N$ stemming from a subordinated Gaussian process $X_n$, $n\in \NN$, according to Model \ref{model} and with marginal distribution function $F$, we are interested in characterizing higher-order approximations of the sequential empirical process
\begin{align} \label{eq:empiricalprocess}
e_N(t, x)\defeq\sum_{n=1}^{\lfloor Nt\rfloor}\left(\1_{\{X_n\leq x\}}-F(x)\right), \ t\in [0, 1], \ x\in \RR.
\end{align}
Higher-order approximations can be derived through the Hermite expansion
\begin{align*} \label{eq:Hermiteexpansion}
\1_{\{X_n\leq x\}}-F(x)=\sum\limits_{l=r}^{\infty}\frac{\widetilde{c}_l(x)}{l!}H_l(\xi_n),
\end{align*}
where $\widetilde{c}_l(x)=\E\pr{\1_{\{G(\xi_0)\leq x\}}H_l(\xi_0)}$ and where $r$ denotes the corresponding Hermite rank
\begin{equation*} \label{eq:Hermiterank}
r:=\min_{x \in \RR} r(x) 
\hspace{0.2cm}
\text{ with }
\hspace{0.2cm}
r(x) := \min\{ q \geq 1 ~|~ \widetilde{c}_{q}(x) \neq 0 \}.
\end{equation*}
\cite{dehling1989} show that the first summand of this expansion determines the asymptotic distribution of the empirical process through the reduction principle
\begin{align} \label{eq:reductionP}
\frac{1}{d_{N, r}}\sum\limits_{n=1}^N\left(\1_{\{X_n\leq x\}}-F(x)\right)=\frac{\widetilde{c}_r(x)}{r!}\frac{1}{d_{N, r}}\sum\limits_{n=1}^N H_r(\xi_n)+o_P(1),
\end{align}
where $d^2_{N, r} = \Var \left( \sum_{n=1}^NH_r(\xi_n) \right) $.

In order to study higher-order terms, we utilize the following observation:
\begin{equation} \label{eq:seriesACF}
\sum_{n \in \NN} | \Cov(H_{l}(\xi_{1}), H_{l}(\xi_{n+1}) | = l! \sum_{n \in \NN} |\gamma(n)|^{l} 
\begin{cases}
= \infty, \hspace{0.2cm} lD<1, \\
< \infty, \hspace{0.2cm} lD>1;
\end{cases}
\end{equation}
see equation (5.1.1) in \cite{PipirasTaqqu} for the first equality in \eqref{eq:seriesACF}. Then, distinguishing the two cases in \eqref{eq:seriesACF}, the last relation is a consequence of \eqref{eq:ACF} and the assumption that $\frac{1}{D}\notin \NN$. 

The convergence behavior of the partial sums of autocovariances
provides another way of distinguishing short- and long-range dependence. 
While convergence is associated with short-range dependence, 
divergence indicates long-range dependence.

As a result, the sequence $H_l(\xi_n)$, $n\in \NN$, can be considered as long-range dependent when $lD<1$, while short-range dependent when $lD>1$. Moreover, the following holds
\begin{align} \label{eq:ghghghghgh}
\frac{\widetilde{c}_l(x)}{l!}\sum\limits_{n=1}^NH_l(\xi_n)=\mathcal{O}_P(N^{-\frac{Dl}{2}+1}L^{\frac{l}{2}}(N))
\ \text{for $l<\frac{1}{D}$,
while } \ 
\frac{\widetilde{c}_l(x)}{l!}\sum\limits_{n=1}^N H_l(\xi_n)=\mathcal{O}_P(\sqrt{N}) \ \text{for $l>\frac{1}{D}$,}
\end{align}
where we refer to equation (4.20) in \cite{beran2016long} for the first relation in \eqref{eq:ghghghghgh}. Note also that the memory parameter $D$ corresponds to the Hurst parameter $H$ through the relation $H = 1-\frac{D}{2}$.

Motivated by the behavior of the series over the autocovariances in \eqref{eq:seriesACF} and the different convergence rates in \eqref{eq:ghghghghgh}, we consider the separation
\begin{align*}
\sum\limits_{l=r}^{\infty}\frac{\widetilde{c}_l(x)}{l!}H_l(\xi_n)=L_n(x)+S_n(x)
\end{align*}
with
\begin{align} \label{eq:LandS}
L_n(x)=\sum\limits_{l=r}^{\lfloor\frac{1}{D}\rfloor}\frac{\widetilde{c}_l(x)}{l!}H_l(\xi_n)
\hspace{0.2cm}
\text{ and }
\hspace{0.2cm}
S_n(x)=\sum\limits_{l=\lceil \frac{1}{D}\rceil}^{\infty}\frac{\widetilde{c}_l(x)}{l!}H_l(\xi_n),
\end{align}
where for some $x \in \RR$, $\lfloor x \rfloor$ and $\lceil x \rceil$ map $x$ to the greatest integer less than or equal and the smallest integer greater than or equal to $x$.
Based on \eqref{eq:seriesACF}, the series over the autocovariances of $L_n(x)$ diverges, while $S_n(x)$ has an absolutely summable autocovariance function.
We refer to $L_n(x)$ in \eqref{eq:LandS} as ``lower-order term" and to $S_n(x)$ as ``higher-order term".

For the empirical process \eqref{eq:empiricalprocess}, higher-order approximations result from 
\begin{align*}
\frac{1}{N}e_N(t, x)=\frac{1}{N}\sum\limits_{n=1}^{\lfloor Nt\rfloor}L_n(x)+\frac{1}{N}\sum_{n=1}^{\lfloor Nt\rfloor}S_n(x).
\end{align*}
Based on the previous considerations, the two summands are expected to converge at different rates.
For our purpose, 
we aim at proving the convergence of $\frac{1}{\sqrt{N}}\sum_{n=1}^{\lfloor Nt\rfloor}S_n(x)$ parameterized in $t$ and $x$.

\begin{wrapfigure}{l}{0pt}
\raisebox{0pt}[\dimexpr\height-1.2\baselineskip\relax]{\scalebox{0.8}{\input{plots/numbersummands.tex}}}
\caption{Number of summands in the ``lower-order term" given that the Hermite rank $r=1$.}
\label{fig:HurstvsSummands}
\vspace{-1.5cm}
\end{wrapfigure}
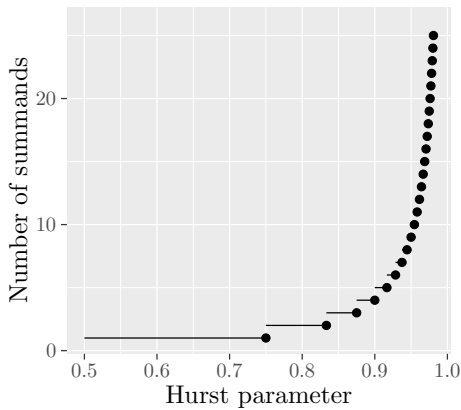

To illustrate the observations made in this section, Figure \ref{fig:HurstvsSummands} depicts the Hurst parameter $H \in (\frac{1}{2},1)$ and the corresponding number of summands $\lfloor\frac{1}{D}\rfloor$, $D = 2-2H$, which contribute to the long-range dependent part $L_{n}(x)$ in \eqref{eq:LandS}. Note that the number of summands contributing to the lower-order term increases exponentially with the value of the Hurst parameter, while the interval length, i.e., the length of the subintervals of $(\frac{1}{2},1)$ which correspond to a certain number of summands, decreases.

\section{Main result} \label{se:mainresults}

In this section, we present our main technical contributions. Our main result is stated in Section \ref{se:statemnt}, followed by a layout of the proof ideas in Section \ref{se:roadmap}.

\subsection{Statement} \label{se:statemnt}

We establish a limit theorem for the higher-order term in the decomposition of the sequential empirical process in two parameters. For this, recall that 
\begin{equation*}
\frac{\lfloor Nt\rfloor}{\sqrt{N}} \left(F_{\lfloor Nt\rfloor}(x) - F(x)\right) = 
\sum\limits_{l=r}^{\lfloor\frac{1}{D}\rfloor} N^{-\frac{lD}{2}} L^{\frac{1}{2}}(N) Z^{(l)}_{N}(t, x)
+
\frac{1}{\sqrt{N}}\sum_{n=1}^{\lfloor Nt\rfloor}S_n(x)
\end{equation*}
with $Z^{(l)}_{N}(t, x) = N^{\frac{lD}{2}-1} L^{-\frac{l}{2}}(N) \sum_{n=1}^{\lfloor Nt\rfloor} \frac{\widetilde{c}_l(x)}{l!}H_l(\xi_n) $. 
According to Theorem 5.3.1 in \cite{PipirasTaqqu}, if suitably standardized, each of the first $\lfloor\frac{1}{D}\rfloor -r +1$ summands converges to a Hermite process of order $l$. More precisely, it holds that
\begin{align*}
Z^{(l)}_{N}(t, x)
=
N^{\frac{lD}{2}-1} L^{-\frac{1}{2}}(N) \sum_{n=1}^{\lfloor Nt\rfloor} \frac{\widetilde{c}_l(x)}{l!}H_l(\xi_n)
\overset{\mathcal{D}}{\to} 
\frac{\widetilde{c}_l(x)}{l!}\beta_{l, H}Z_H^{(l)}(t)
\end{align*}
in $D\left([-\infty, \infty]\times [0, 1]\right)$, where $\beta_{l, H}$ is a constant and $Z_H^{(l)}(t)$, $t\in [0, 1]$, an Hermite process of order $l$ with self-similarity parameter $H=1-\frac{lD}{2}$.
 The limit of the higher-order term
$\frac{1}{\sqrt{N}}\sum_{n=1}^{\lfloor Nt\rfloor}S_n(x)$
 is characterized by the following theorem.

\begin{theorem} \label{th:weakconvSRD}
Suppose $X_{n}, n \in \NN,$ satisfies Model \ref{model} and $X_{n}$ has a strictly monotone, continuous distribution function $F$ \ignore{with density $f\in L^p$ for some $p>1$}and $\frac{1}{D} \notin \NN$. Then, as $N \to \infty$,
\begin{equation*}
\frac{1}{\sqrt{N}}\sum_{n=1}^{\lfloor Nt\rfloor}S_n(x) \overset{\mathcal{D}}{\to}S(x,t)
\end{equation*}
in $D([-\infty, \infty] \times [0,1])$, where $S(x,t)$ is a mean zero Gaussian process with cross-covariances
\begin{equation} \label{eq:covS}
\Cov(S(x,t),S(y,u)) = \min \{ t,u \} \sum_{n \in \ZZ} \Cov(S_{0}(x),S_{n}(y)).
\end{equation}
\end{theorem}
The proof of Theorem \ref{th:weakconvSRD} can be found in Appendix \ref{app:proofmain}.
\par

\begin{remark}
Note that we exclude the case $\frac{1}{D} \in \NN$. That excludes in particular the case $D = 1$, that is when the underlying time series is short-range dependent. Therefore, our result as it is stated cannot recover existing results for short-range dependent time series. Under short-range dependence, the empirical process is known to converge to the so-called Kiefer-M\"uller process; see \cite{muller1970glivenko,kiefer1972skorohod}.
\end{remark}

\subsection{Proof} \label{se:roadmap}

While the detailed proof of Theorem \ref{th:weakconvSRD} is given in Appendices \ref{app:proofmain}, \ref{se:appendixB}, and \ref{se:appendixc1}, we aim here to provide a roadmap of our proofs and to emphasize some of the main technical challenges.

For a proof of convergence in distribution as stated in Theorem \ref{th:weakconvSRD}, convergence of the finite-dimensional distributions and tightness are being established; see Sections \ref{subse:fdd} and \ref{subse:tightness}, respectively. While proving convergence of the finite-dimensional distributions can be considered straightforward, the main technical challenges arise in the proof of tightness. These challenges are the subject of this section.

The partial sum of the higher-order terms in \eqref{eq:LandS} 
\begin{equation} \label{eq:mtilde_mainbody}
\widetilde{m}_N\pr{x, t }\defeq 
\frac{1}{\sqrt{N}}\sum_{n=1}^{\lfloor Nt\rfloor}S_n(x)
=
\frac{1}{\sqrt{N}}\sum_{n=1}^{\lfloor Nt\rfloor}\sum\limits_{l=\lceil \frac{1}{D}\rceil}^{\infty}\frac{\widetilde{c}_l(x)}{l!}H_l(\xi_n)
\end{equation}
is a stochastic process in two parameters, which is one reason why proving tightness becomes particularly challenging. Another challenge results from the structure of the higher-order terms. In contrast to the empirical process, the higher-order terms are no longer bounded. While the transformed variables $H_l(\xi_n), n \in \NN$, for $l \geq \lceil \frac{1}{D}\rceil$ are short-range dependent, the underlying process $\xi_n, n \in \NN,$ is still long-range dependent with non-summable autocovariance function. The dependence on the memory parameter $D$ appears in the summation determining the number of summands going into the higher-order terms. 

The articles \cite{dehling1989}, \cite{koul:surgailis:2002} and \cite{elktaibi2016} are closest to our work. In the following layout of our proof, we emphasize how our results differ from these works.

\begin{enumerate}
\item It is necessary to prove tightness in two parameters, more precisely, in the space $D([-\infty,\infty] \times [0,1])$. Furthermore, we allow the underlying process to be subordinated Gaussian.
This makes our proofs decisively different from the proofs established in \cite{koul:surgailis:2002}, who only consider \eqref{eq:mtilde_mainbody} for fixed $t$ and did not allow for subordinated transformations of the underlying Gaussian process.
\item The first step of our proof is to reduce tightness in $D([-\infty,\infty] \times [0,1])$ to proving tightness in $D([0,1] \times [0,1])$. 
The corresponding object in $D([0,1] \times [0,1])$ can be written as
\begin{equation} \label{eq:mN:mainbody}
m_N\pr{x, t }\defeq \frac{1}{\sqrt{N}}\sum_{n=1}^{\lfloor Nt\rfloor} \sum\limits_{l=\lceil \frac{1}{D}\rceil}^{\infty}\frac{c_l(x)}{l!}H_l(\xi_n)
\hspace{0.2cm}
\text{ with }
\hspace{0.2cm}
c_l(x)=\E\pr{\1_{\{F(G(\xi_0))\leq x\}}H_l(\xi_0)}.
\end{equation}
\item We use a tightness criterion introduced in \cite{ivanoff:thefuncD1980} and later utilized in \cite{elktaibi2016} to prove tightness of the sequential empirical process under short-range dependence.
\cite{elktaibi2016} take advantage of the boundedness of the empirical process. Those techniques fail for \eqref{eq:mtilde_mainbody} since the higher-order terms of the empirical process can no longer be represented as an indicator function.
\item
In the main part of our proof, we reduce the tightness criterion in \cite{elktaibi2016} to bounding the probability
\begin{equation*} \label{eq:mbar:mainbody}
\PP\left(\widebar{m}_{N,b}(y,x) >\lambda \right)
\hspace{0,2cm}
\text{ with }
\hspace{0,2cm}
\widebar{m}_{N,b}(x, y)\defeq\sup_{t\in [0, b]}\left|m_N(y, t)-m_N(x, t)\right|
\end{equation*}
for some $b>0$ and $x,y \in [0,1]$ with $m_N(y, t)$ as in \eqref{eq:mN:mainbody}.
Typically, such bounds are derived through chaining techniques.
\cite{dehling1989} establish a corresponding argument for proving tightness of the empirical process of long-range dependent observations.
For this, they take advantage of the reduction principle as stated in \eqref{eq:reductionP}.
 The reduction principle reduces the problem to proving convergence of the partial sums of the dominating Hermite polynomial. 
Since none of the summands of the infinite series \eqref{eq:mtilde_mainbody} is asymptotically negligible, the chaining technique of \cite{dehling1989} does not apply to the considered situation.
\cite{betken2020change} establish a chaining technique 
for proving tightness of the tail empirical process of Long Memory Stochastic Volatility (LMSV) time series.
The major difference to
our argument results from a martingale structure of the tail empirical process of LMSV time series. This allows to apply
Freedman's inequality, i.e., a Bernstein-type inequality for martingale difference sequences which, as well, does not apply to the situation in this paper.
\item
A crucial part of the proof and second main technical contribution is to find a bound of the form
\begin{equation*}
\PP\left(\widebar{m}_{N,b}(y,x) >\lambda \right)
\leq 
C_{1,\gamma} \frac{1}{\lambda^4} b^{2-\theta} \frac{1}{N^{\theta}} \left(y-x\right)
+
C_{2,\gamma} \frac{1}{\lambda^4} b^2 \left(y-x\right)^{\frac{3}{2}}
\end{equation*}
for some $\theta>0$, any $b>0$ and all $x,y \in [0,1]$. The result is formally stated in Lemma \ref{le:bill}. Our proof consists of two major parts. The first one is to extend Theorem 12.2 in \cite{billingsley:1968} which provides a probabilistic bound for maxima over partial sums. In Lemma \ref{le:billmod12.2} we provide a similar result, allowing the bound to take a more general form.
The second main part of the proof is to verify the assumptions of Lemma \ref{le:billmod12.2}. Both, Lemma \ref{le:billmod12.2} and Theorem 12.2 in \cite{billingsley:1968} apply under very general assumptions in that both do not impose any assumptions on the dependence structure of the underlying process. However, both are based on a probabilistic bound on the distances between partial sums. Given strong temporal dependence, as in our setting, verifying this condition becomes particularly challenging.
\end{enumerate}

\section{Confidence intervals} \label{Confidenceintervals}
In this section, 
we focus on how to utilize higher-order approximations of the empirical process for the construction of confidence intervals.
To begin with, we determine confidence intervals for values of the marginal distribution $F$ of a time series $X_1, \ldots, X_N$
following Model \ref{model}.
The confidence intervals are based on the empirical analogue
$F_N(x):=\frac{1}{N}\sum_{n=1}^N \1_{\left\{X_n\leq x\right\}}$ of $F(x)$; see Section \ref{se:CImarginal}.
Following this, we derive confidence intervals for quantiles of the marginal distribution; see Section \ref{se:CIquantiles}. Section \ref{se:CI_Discussion} provides a comparative discussion.
\par
First and foremost, we are interested in how well these confidence intervals approximate optimal confidence
intervals.
For this, note that
the goodness of confidence intervals can be assessed on the basis of the following two criteria:
\begin{enumerate}[start=1,label={C\arabic*:}]
\item A high coverage probability, i.e., the probability that the true value of the estimated quantity 
lies in the considered confidence interval should be high.
\item A short length of the confidence interval.
\end{enumerate}
Based on these criteria we aim to compare the confidence intervals derived from higher-order approximations of the empirical process to confidence intervals that result from the asymptotic distribution of $F_N(x)$. Therefore, we will first rephrase how to compute the asymptotic confidence intervals and then move on to introducing our approach to derive confidence intervals.
\par
For ease of computations, we base all analysis on the assumption that we are given a subordinated Gaussian time series $X_n=G(\xi_n)$, $n=1, \ldots, N$, resulting from a strictly monotone function $G$.
In this case, the Hermite rank $r$ equals $1$ and the Hermite coefficients $\widetilde{c}_l(x)$ can be determined analytically. In particular, it holds that
\begin{align} \label{eq:ctilde}
\widetilde{c}_l(x)=
\begin{cases}
-H_{l-1}(G^{-1}(x))\varphi\left(G^{-1}(x)\right) \ &\text{if $G$ is increasing},\\
H_{l-1}(G^{-1}(x))\varphi\left(G^{-1}(x)\right) \ &\text{if $G$ is decreasing};
\end{cases}
\end{align}
see Lemma \ref{le:hermite_coefficient}.

\subsection{Confidence intervals for the marginal distribution} \label{se:CImarginal}

\noindent
\textit{Asymptotic confidence intervals:}
For a construction of confidence intervals based on the asymptotic distribution of the empirical process, note that 
\begin{align} \label{eq:empprocessconvergence}
\frac{N}{d_N}\left(F_N(x)-F(x)\right)\overset{\mathcal{D}}{\to}\widetilde{c}_1(x)Z, 
\end{align}
 where $Z$ is a standard normally distributed random variable, $\widetilde{c}_1(x)=\E\left( \1_{\left\{G(\xi_0)\leq x\right\}}\xi_0\right)$ and $d^2_{N}:= \Var\left(\sum_{i=1}^{N}\xi_i\right)\sim N^{2H}L(N)$; see \cite{dehling1989}.
Due to the fact that convergence in \eqref{eq:empprocessconvergence} holds in $D[-\infty, \infty]$, we have
 \begin{align*}
 1-\alpha = \Pr \left( |\widetilde{c}_1(x)|^{-1}\frac{N}{d_N}\left(F_N(x)-F(x)\right)\in \left(z_{\frac{\alpha}{2}}, z_{1-\frac{\alpha}{2}}\right) \right) + o(1),
 \end{align*}
where $z_{\alpha}:=\Phi^{-1}(\alpha)$ and $\Phi$ denotes the standard normal distribution function.
Therefore, an approximate $1-\alpha$ confidence interval for $F(x)$ based on the asymptotic distribution of the empirical process is given by
 \begin{align} \label{eq:CIasympemp}
 \left(F_N(x)-\frac{d_N}{N}|\widetilde{c}_1(x)|z_{1-\frac{\alpha}{2}}, F_N(x)-\frac{d_N}{N}|\widetilde{c}_1(x)|z_{\frac{\alpha}{2}}\right).
 \end{align}

Referring back to Example \ref{ex:fGn}, the following example establishes these confidence intervals \eqref{eq:CIasympemp} for fractional Gaussian noise.
 \begin{example}\label{ex:fGn_marginal}
For fractional Gaussian noise time series with Hurst parameter $H$, $d_N\sim N^{H}$ and $\left|\widetilde{c}_1(x)\right|=\varphi(x)$, such that the interval in \eqref{eq:CIasympemp} equals 
\begin{align*}
\left(F_N(x)-N^{H-1}\varphi(x)z_{1-\frac{\alpha}{2}}, F_N(x)-N^{H-1}\varphi(x)z_{\frac{\alpha}{2}}\right).
\end{align*}
\end{example}
\noindent
\textit{Confidence intervals based on higher-order approximations:}
For a construction of confidence intervals based on the higher-order approximation, note that according to Theorem \ref{th:weakconvSRD}
\begin{align*}
\sqrt{N}\left(F_N(x)-F(x)\right)-\frac{1}{\sqrt{N}}\sum\limits_{n=1}^NL_n(x)\overset{\mathcal{D}}{\to}Z(x),
\end{align*}
 where $Z(x)$ is normally distributed with mean zero and variance $\sigma^2(x):=\sum_{n\in \ZZ}\Cov(S_0(x), S_n(x))$ and convergence holds in $D[-\infty, \infty]$.
 As a result, we have
 \begin{align*}
 1-\alpha = \Pr \left((\sigma(x))^{-1}\left(\sqrt{N}\left(F_N(x)-F(x)\right)-\frac{1}{\sqrt{N}}\sum\limits_{n=1}^NL_n(x)\right)\in \left(z_{\frac{\alpha}{2}}, z_{1-\frac{\alpha}{2}}\right) \right) + o(1)
 \end{align*}
with $L_n$ as in \eqref{eq:LandS}.
Therefore, an approximate $1-\alpha$ confidence interval for $F(x)$ based on higher-order approximations of the empirical process is given by 
 \begin{align} \label{eq:CIsecemp}
 \left(F_N(x)-\frac{1}{N}\sum\limits_{n=1}^NL_n(x)-\frac{\sigma(x)}{\sqrt{N}}z_{1-\frac{\alpha}{2}}, F_N(x)-\frac{1}{N}\sum\limits_{n=1}^NL_n(x)-\frac{\sigma(x)}{\sqrt{N}}z_{\frac{\alpha}{2}}\right).
 \end{align}

\subsection{Confidence intervals for quantiles} \label{se:CIquantiles}

In this section, we establish confidence intervals for quantiles of the marginal distribution of long-range dependent time series. Initially, we 
describe the construction of confidence intervals for quantiles 
 based on the convergence of the empirical process. Subsequently, we discuss
 the construction of confidence intervals for quantiles
 based on higher-order approximations of the empirical process.

\textit{Asymptotic confidence intervals:}
The asymptotic distribution of empirical quantiles can be derived from the asymptotic behavior of the empirical process \eqref{eq:empprocessconvergence} and an application of the delta method. 
In fact, \cite{hossjer1995delta} showed that for a functional $ \phi : (D[-\infty, \infty], \| \cdot \|_{\infty}) \to \RR$, Hadamard-differentiable at $F$,
\begin{align} \label{eq:hossier}
\frac{N}{d_N}\left( \phi(F_N)-\phi(F) \right)\overset{\mathcal{D}}{\to} Z \phi'(F)(\phi(F)) \ \widetilde{c}_{1}(\phi(F)),
\end{align}
where $Z$ is a standard normally distributed random variable, $\widetilde{c}_1(x)=\E\left( \1_{\left\{G(\xi_0)\leq x\right\}}\xi_0\right)$, $\phi'(F)$ the derivative in $F$ and $d^2_{N}:= \Var\left(\sum_{i=1}^{N}\xi_i\right)\sim N^{2H}L(N)$; see Theorem 1 in \cite{hossjer1995delta}.
Since our goal is to establish confidence intervals for quantiles $ q_{p} = \inf\{ x ~|~ F(x) \geq p \}$, we consider $ \phi : (D[-\infty, \infty], \| \cdot \|_{\infty}) \to \RR$, $\phi(F) = F^{-1}(p)$. 
Given that $r=1$, \eqref{eq:hossier} corresponds to
\begin{align*}
\frac{N}{d_N}\left( F^{-1}_N(p)-F^{-1}(p) \right)
\overset{\mathcal{D}}{\to} 
-Z \frac{1}{F'(F^{-1}(p))} \widetilde{c}_{1}(F^{-1}(p)).
\end{align*}
As a result, we have
 \begin{align*}
 1-\alpha = \Pr \left( - F'(F^{-1}(p)) | \widetilde{c}_{1}(F^{-1}(p))|^{-1}\frac{N}{d_N}\left( F^{-1}_N(p)-F^{-1}(p) \right)\in \left(z_{\frac{\alpha}{2}}, z_{1-\frac{\alpha}{2}}\right)\right) + o(1).
 \end{align*}
Therefore, an approximate $1-\alpha$ confidence interval for $F^{-1}(p)$ based on the asymptotic distribution of the empirical process is given by
 \begin{align} \label{eq:asympCIquantiles}
 \left( F^{-1}_N(p) - \frac{d_N}{N} \frac{1}{F'(F^{-1}(p))} | \widetilde{c}_{1}(F^{-1}(p))|z_{\frac{\alpha}{2}}, F^{-1}_N(p) - \frac{d_N}{N} \frac{1}{F'(F^{-1}(p))} | \widetilde{c}_{1}(F^{-1}(p))|z_{1-\frac{\alpha}{2}}\right).
 \end{align}

\begin{example}\label{ex:fGn_quantile}
For fractional Gaussian noise time series with Hurst parameter $H$, $d_N \sim N^H$ and $\left|\widetilde{c}_1(x)\right|=\varphi(x)$, such that the interval in \eqref{eq:asympCIquantiles} equals 
\begin{align*}
\left( F^{-1}_N(p) + N^{H-1} z_{\frac{\alpha}{2}}, F^{-1}_N(p) + N^{H-1} z_{1-\frac{\alpha}{2}}\right).
\end{align*}
\end{example}

\textit{Confidence intervals based on higher-order approximations:}
We propose an alternative way to derive confidence intervals for the quantiles of the marginal distribution of long-range dependent time series based on higher-order approximations of the empirical process. 
Recall that quantiles can be written as a functional of the distribution $F$ as well as their estimated counterparts. Based on Taylor approximation of the functional $\phi$, we can then write
\begin{equation} \label{eq:Taylorapprox}
\frac{N}{d_N}\left(\phi(F_N)-\phi(F)\right) = \phi'_{F}\left(\frac{N}{d_N}(F_N - F)\right) + o_{P}(1);
\end{equation}
see \cite{hossjer1995delta} and also Theorem 20.8 in \cite{van2000asymptotic}.
The right-hand side can be further simplified by 
\begin{equation} \label{eq:derivativefunctional}
\phi'_{F}\left(\frac{N}{d_N}(F_N - F)\right)
= \frac{N}{d_N} \frac{p-F_{N}(F^{-1}(p))}{F'(F^{-1}(p))}
= - \frac{N}{d_N} \frac{(F_{N} - F)(F^{-1}(p))}{F'(F^{-1}(p))};
\end{equation}
see p.\ 294 in \cite{van2000asymptotic}.
Under the assumption that the underlying time series has Gaussian marginals and for $p=\frac{1}{2}$ (such that $F^{-1}(p)$ corresponds to the median) we get $\phi'_{F}(\frac{N}{d_N}(F_N - F)) = -\frac{N}{d_N}\frac{(F_{N} - F)(0)}{\varphi(0)}$.
Then, an approximate $1-\alpha$ confidence interval of $F^{-1}(p)$ can be written as
\begin{equation} \label{eq:secCIquantiles}
 \begin{aligned}
&
\Bigg(
\phi(F_N) + \frac{1}{ \varphi(F^{-1}(p)) } \Big( \frac{1}{N}\sum\limits_{n=1}^{N}L_n(F^{-1}(p)) + \frac{1}{\sqrt{N}} \sigma(F^{-1}(p)) z_{1-\frac{\alpha}{2}} \Big),
\\&\hspace{1cm}
\phi(F_N) + \frac{1}{ \varphi(F^{-1}(p)) } \Big( \frac{1}{N}\sum\limits_{n=1}^{N}L_n(F^{-1}(p)) + \frac{1}{\sqrt{N}} \sigma(F^{-1}(p)) z_{\frac{\alpha}{2}} \Big)
\Bigg);
\end{aligned}
\end{equation}
see Lemma \ref{le:CIquantiles} and its proof for more details on the calculations.

\subsection{Discussion} \label{se:CI_Discussion}

The confidence intervals established in Sections \ref{se:CImarginal} and \ref{se:CIquantiles}
all depend on the subordinating function $G$ as well as the Hurst parameter $H$, quantities that are unknown in practice. 
Additionally, 
confidence intervals based on the asymptotic distribution of the empirical process (such as \eqref{eq:CIasympemp} and \eqref{eq:asympCIquantiles}) depend on the slowly varying function $L$ through $d_N$.
By definition $d_N^2$ corresponds to the long-run variance of a long-range dependent Gaussian process. Due to the fact that the data is assumed to be subordinated to this process, $d_N$ cannot be estimated straightforwardly, i.e., by a long-run variance estimator applied to the observed data.
On the other hand, an estimation can be based on the asymptotic relation $d_N\sim N^HL^{\frac{1}{2}}(N)$. For this, it has to be taken into account that
$H$ characterizes the autocovariances of the Gaussian process (not the observed subordinated process). Only for a Hermite rank of the subordinating function $G$ that equals $1$, the Hurst parameter 
of the Gaussian process and that of the subordinated Gaussian process coincide, such that $H$ can be estimated by established methods (such as $R/S$-estimation or local Whittle estimation); see also Section \ref{se:estimatedHurst}.
Nonetheless, estimation of $d_N$ also requires an approximation of the slowly varying function $L$. 
Unfortunately, we are not aware of any estimation procedure meeting this task. In particular cases, e.g., when the data stems from fractional Gaussian noise, $L$ corresponds to a multiplicative constant depending on the parameter $H$ only; see Examples \ref{ex:fGn_marginal} and \ref{ex:fGn_quantile}. In these cases, the estimation can solely be based on estimation of $H$, but presupposes knowledge of the subordinating function $G$.
In contrast to confidence intervals based on the asymptotic distribution of the empirical process, confidence intervals based on higher-order approximations of the empirical process (such as \eqref{eq:CIsecemp} and \eqref{eq:secCIquantiles}) do not depend on $L$.
For an empirical comparison of the two procedures for confidence interval construction in Section \ref{se:numericalstudy} we assume knowledge of the slowly varying function $L$. This knowledge can be exploited for confidence interval construction based on the asymptotic distribution of the empirical process, but is not needed for the proposed method of confidence interval construction based on higher-order approximations of the empirical process. 
Moreover, we would like to point out 
that \eqref{eq:secCIquantiles} only depends on the Hurst parameter $H$ through $L_{n}$; see \eqref{eq:CIsecemp}. In particular, the ceiling function applied to $H$ determines the number of summands included in the construction of confidence intervals based on higher-order approximations. Accordingly, these are less sensitive to small errors in the estimation of $H$ than confidence intervals based on the asymptotic distribution.
Given that inference on long-range dependent time series relies on how well the corresponding Hurst parameter is estimated, we expect confidence intervals based on higher-order approximations to be more robust to misspecification of $H$.

Section \ref{se:CIquantiles} focuses on deriving confidence intervals for quantiles of the marginal distribution based on a higher-order approximation of the empirical process. The proposed procedure takes advantage of the Taylor expansion \eqref{eq:Taylorapprox} of a general functional $\phi$. Due to the generality of the results, we believe that similar results can be achieved for other estimators that have a representation as functionals of the empirical process such as Huber's estimator and M-estimators.

\section{Numerical Studies} \label{se:numericalstudy}

For our numerical studies, we consider the procedures proposed in Section \ref{Confidenceintervals}. We compare the coverage rate as well as the length of asymptotic confidence intervals with those based on higher-order approximations.
To assess the performance of the proposed procedures, we assume that the underlying time series follows Model \ref{model} with $G=id$, i.e., the time series is assumed to be long-range dependent with Gaussian marginals. 
In particular, we assume that $G=id$ is known although
in practice $G$ needs to be estimated. 
Estimation of $G$ can, for example, be based on the relation $X\overset{\mathcal{D}}{=}F^{-1}(\Phi(\xi))$ (resulting from $X=G(\xi)$ for a standard normally distributed random variable $\xi$), where $\Phi$ denotes the standard normal distribution function and $F$ the marginal distribution of $X$.
Accordingly, $G$ could be estimated by $\hat{F}^{-1}\circ \Phi$, where $\hat{F}^{-1}$ corresponds to the generalized inverse of the empirical distribution function of the observed data.
Note that, nonetheless, estimation of $G$ will add uncertainty to both procedures, such that for the purpose of comparison we refrain from estimation of $G$.

In the following, we focus on confidence intervals for the marginal distribution and confidence intervals for the median (Sections \ref{se:CImarginalnumerical} and \ref{se:CImediannumerical}). Section \ref{se:estimation} discusses the estimation of the long-run variance and the Hurst parameter.

\subsection{Estimation of long-run variance and Hurst parameter} \label{se:estimation}

In order to compute the confidence intervals discussed in Section \ref{Confidenceintervals}, we need to estimate the long-run variance. Furthermore, we provide simulation results under the assumption that the Hurst parameter $H$ is known and under the assumption that $H$ is unknown. 

The long-run variance $\sigma^2(x):=\sum_{n\in \ZZ}\Cov(S_0(x), S_n(x))$ cannot be computed analytically. 
In order to make our results applicable, we therefore need to estimate $\sigma^2(x)$. We use the kernel smoothing long-run variance estimator
\begin{equation*}
\widehat{\sigma}^{2}(x) = \sum_{j = -(N-1)}^{N-1} K\left( \frac{j}{b_{N}} \right) \widehat{\gamma}_{N}(j),
\end{equation*}
where $K(x) = (1 -|x|) \1_{\{|x| \leq 1\}}$ is the Bartlett kernel function, $b_{N}$ denotes a bandwidth parameter and $\widehat{\gamma}_{N}(j)$ is the sample autocovariance at lag $j$.
For our simulation study, we use the command \verb$hurstexp$ in the \verb$R$ package \verb$cointReg$.
To determine the bandwidth, we use the command \verb$getBandwidth$.
For an estimation of the Hurst parameter $H$ we used the 
$R/S$ procedure following the description in Section 2.1 in \cite{WERON2002285}. 
The estimator is implemented by \verb$getLongRunVar$ in the \verb$R$ package \verb$pracma$.

\subsection{Confidence intervals for the marginal distribution} \label{se:CImarginalnumerical}

We construct confidence intervals for the marginal distribution $F$ 
based on the asymptotic distribution and based on higher-order approximations of the empirical process of long-range dependent time series. For a visual comparison of the two different methods see Figures \ref{fig:marginaldistr_H=0.55_m=200}--\ref{fig:marginaldistr_bands}.
To numerically assess the quality of the computed intervals, we report their coverage rate and width evaluated at different $x$.
In our simulation study, we consider different scenarios ranging from small to large sample sizes ($N=200$ and $N=1000$) as well as from small to large Hurst parameters ($H=0.55$ and $H= 0.95$).
Pointing towards Figures \ref{fig:marginaldistr_H=0.55_m=200} and \ref{fig:marginaldistr_H=0.95_m=1000}, which are based on sample sizes $N=200$ and $N=1000$, we see only a slight improvement of the interval length for the asymptotic confidence intervals. The mild improvement emphasizes how the asymptotic confidence intervals are impacted by the slow convergence rate of the empirical process under long-range dependence.
That said, we fix the sample size to $N=200$ and compare Figures \ref{fig:marginaldistr_H=0.55_m=200} and \ref{fig:marginaldistr_H=0.95_m=200}. Focusing on $x=0$, one can observe that for larger Hurst parameters, the width increases significantly for the asymptotic method. Naturally, the increase in width results in a higher coverage rate. 
Confidence intervals based on the proposed 
 higher-order approximation method (HOA), however, 
 are robust with respect to the value of the Hurst parameter
and outperform the traditional construction of confidence intervals with respect to the coverage rate; see Figures \ref{fig:marginaldistr_H=0.55_m=200} and \ref{fig:marginaldistr_H=0.95_m=200}.

Most notably, Figures \ref{fig:marginaldistr_H=0.55_m=200} and \ref{fig:marginaldistr_H=0.95_m=200} reveal that asymptotic confidence intervals may have lower coverage rates than those based on higher-order approximations, while confidence intervals based on higher-order approximations are shorter.
This phenomenon results from the fact that the centers of the confidence intervals differ,
i.e., the smaller confidence interval is not necessarily contained in the larger one. In particular, it therefore happens that the asymptotic confidence interval is larger, but nonetheless does not cover $F(x)$.

\begin{figure}
\center
\scalebox{0.9}{\input{plots/normal_H=0.55_m=200_and_m=1000.tex}}
\vspace{-0.5cm}
\caption{The coverage rate and length of confidence intervals for the marginal distribution $F(x)$ evaluated at different $x$. 
The two displayed methods to calculate the confidence intervals are based on the asymptotic distribution (asymp) and our higher-order approximation (HOA).
The simulations are based on 2000 repetitions for Gaussian time series of length $N=200$ (first row) and $N=1000$ (second row) with Hurst parameter $H=0.55$. The dashed gray line depicts the significance level of $95\%$.}
\label{fig:marginaldistr_H=0.55_m=200}
\label{fig:marginaldistr_H=0.55_m=1000}
\end{figure}

\begin{figure}
\center
\scalebox{0.9}{\input{plots/normal_H=0.95_m=200_and_m=1000.tex}}
\vspace{-0.5cm}
\caption{The coverage rate and length of confidence intervals for the marginal distribution $F(x)$ evaluated at different $x$. 
The two displayed methods to calculate the confidence intervals are based on the asymptotic distribution (asymp) and our higher-order approximation (HOA).
The simulations are based on 2000 repetitions for Gaussian time series of length $N=200$ (first row) and $N=1000$ (second row) with Hurst parameter $H=0.95$. The dashed gray line depicts the significance level of $95\%$.}
\label{fig:marginaldistr_H=0.95_m=200}
\label{fig:marginaldistr_H=0.95_m=1000}
\end{figure}

 \begin{figure}
 \center
 \scalebox{0.9}{\input{plots/plot.tex}}
 \vspace{-0.5cm}
 \caption{Confidence intervals for the marginal distribution $F$. 
 The two displayed methods to calculate the confidence intervals are based on the asymptotic distribution (asymp) and our higher-order approximation (HOA).
 The simulations are based on 1000 repetitions for Gaussian time series of length $N=100$ with Hurst parameters $H=0.6, H=0.75$ and $H=0.9$. }
 \label{fig:marginaldistr_bands}
 \end{figure}

\subsection{Confidence intervals for the median} \label{se:CImediannumerical}
In this section, we consider confidence intervals for the median based on long-range dependent time series characterized by different Hurst parameters.
Again, we consider different scenarios ranging from small to large sample sizes ($N=200$ and $N=1000$) as well as from small to large Hurst parameters (from $H=0.55$ to $H= 0.95)$, and we assess the quality of the confidence intervals through interval length and coverage rate. 
Pointing towards Figure \ref{fig:median_m=200}, which is based on sample sizes $N=200$ and $N=1000$, we see only a slight improvement of the interval length for the asymptotic confidence intervals. This emphasizes how the asymptotic confidence intervals are impacted by the slow convergence rate of the empirical process under long-range dependence.
Therefore, instead of considering the impact of the sample size, we focus on how varying the Hurst parameter influences the coverage rate and interval length.
In this regard, Figure \ref{fig:median_m=200} clearly demonstrates that 
the length of a confidence interval constructed on the basis of the asymptotic distribution of the empirical process increases almost exponentially with increasing value of $H$. This may be attributed to the exponential increase of the number of summands needed to calculate the lower-order terms of the empirical process; see Figure \ref{fig:HurstvsSummands}.

In contrast to basing confidence intervals on the asymptotic distribution of the empirical process, Figure \ref{fig:median_m=200} illustrates robustness of the confidence interval lengths to different values of the Hurst parameter if the construction of confidence intervals is based on higher-order approximations of the empirical process. For Hurst parameters bigger than $H=0.9$ a significant drop of the coverage rate can be observed. 
 We attribute this observation to the fact that the stronger the dependence in a time series the higher the number of observations needed to reflect this dependence.
When adjusting confidence intervals by the true number of summands in the lower-order term this finite-sample phenomenon is not accounted for resulting in lower coverage rates.

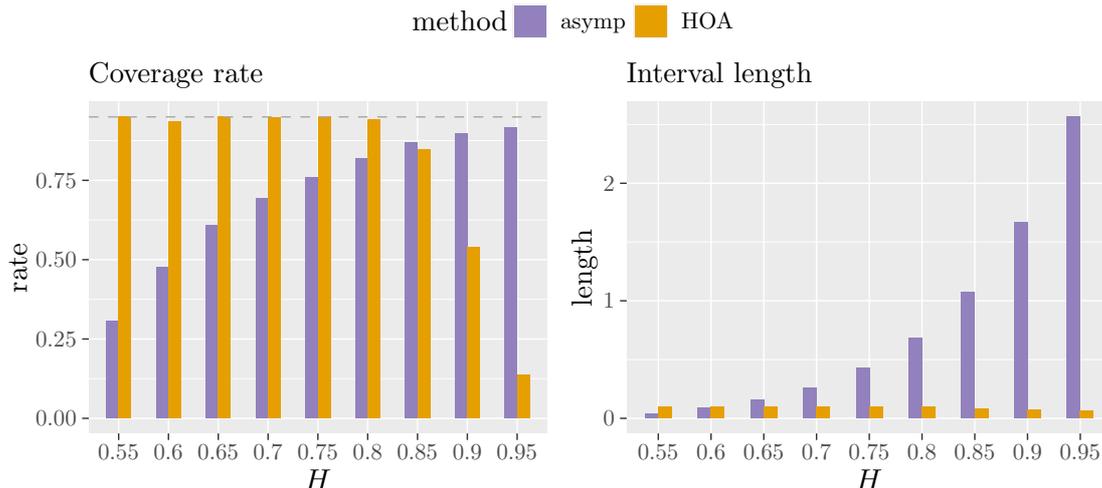
\begin{figure}
\center
\scalebox{0.9}{\input{plots/median_normal_m=200_and_m=1000.tex}}
\vspace{-0.5cm}
\caption{Coverage rate and interval length of confidence intervals for the median $F^{-1}(1/2)$ based on long-range dependent time series characterized by different Hurst parameters. 
For this, the distribution of the median is approximated by the asymptotic distribution (asymp) and our higher-order approximation (HOA) of the empirical process.
Simulations are based on 2000 repetitions for Gaussian time series of length $N=200$ (first row) and $N=1000$ (second row). The dashed gray line depicts the significance level of $95\%$.}
\label{fig:median_m=200}
\end{figure}

\subsection{Confidence intervals based on an estimated Hurst parameter} \label{se:estimatedHurst}
To make the construction of confidence intervals based on higher-order approximations of the empirical process feasible for practical purposes, we need to consider the case where the Hurst parameter is unknown. As discussed in Section \ref{se:estimation}, we base estimation of the Hurst parameter on the so-called $R/S$-method.
In this section, we focus on studying confidence intervals for the median. As done in Figure \ref{fig:median_m=200}, the median was considered for a range of different Hurst parameters. We therefore use the median to illustrate how an estimated Hurst parameter changes the empirical coverage rates and lengths; see Figure \ref{fig:median_estimated}.

Next, we compare the numerical results based on estimation of the Hurst parameter (Figure \ref{fig:median_estimated}) with the numerical results that are based on the assumption that the Hurst parameter is known (Figure \ref{fig:median_m=200}). 
It is notable that the lengths of confidence intervals resulting from approximation of the empirical process by its asymptotic distribution tend to be shorter when the Hurst parameter is estimated while their coverage rate is lower.
Although the coverage rate of confidence intervals that are based on higher-order approximations of the empirical process and estimation of the Hurst parameter declines for Hurst parameters larger than $0.9$, this effect is not as pronounced as for the simulations that assumed knowledge of the Hurst parameter.
We attribute this phenomenon to the fact that $R/S$ estimation tends to underestimate the Hurst parameter and that the higher the value of the Hurst parameter, the bigger the estimation bias; see \cite{taqqu1995estimators}.
An underestimation of the Hurst parameter results in a smaller number of lower-order terms entering the approximation of confidence intervals. 
We conjecture that basing the approximation of confidence intervals on a lower number of summands than suggested by our theory for long-range dependent time series compensates for very strong dependence in time series not being reflected in the finite-sample behavior of time series with relatively low sample size.
We call this phenomenon \enquote{benign underestimation of long-range dependence}. Our conjecture is supported by the fact that an increasing number of observations results in a drop of coverage rate based on the estimated Hurst parameter; see Figure 7. 

Note that similar to the confidence intervals for the marginal distribution, we can observe that the asymptotic confidence intervals result in a smaller coverage rate than the ones based on higher-order approximations while the latter are shorter. 
In particular, one can see that the larger the Hurst parameter, the smaller the coverage rate of the asymptotic confidence intervals while the intervals based on higher-order approximations maintain a constant coverage rate and interval length.

\begin{figure}
\center
\scalebox{0.9}{\input{plots/median_normal_Hestimated_m=200_and_m=1000.tex}}
\vspace{-0.5cm}
\caption{
Coverage rate and interval length of confidence intervals for the median $F^{-1}(1/2)$ based on long-range dependent time series characterized by different Hurst parameters. 
For this, the distribution of the median is approximated by the asymptotic distribution (asymp) and our higher-order approximation (HOA) of the empirical process. The Hurst parameter is replaced by its $R/S$-estimator.
Simulations are based on 2000 repetitions for Gaussian time series of length $N=200$ (first row) and $N=1000$ (second row). The dashed gray line depicts the significance level of $95\%$.}
\label{fig:median_estimated}
\end{figure}
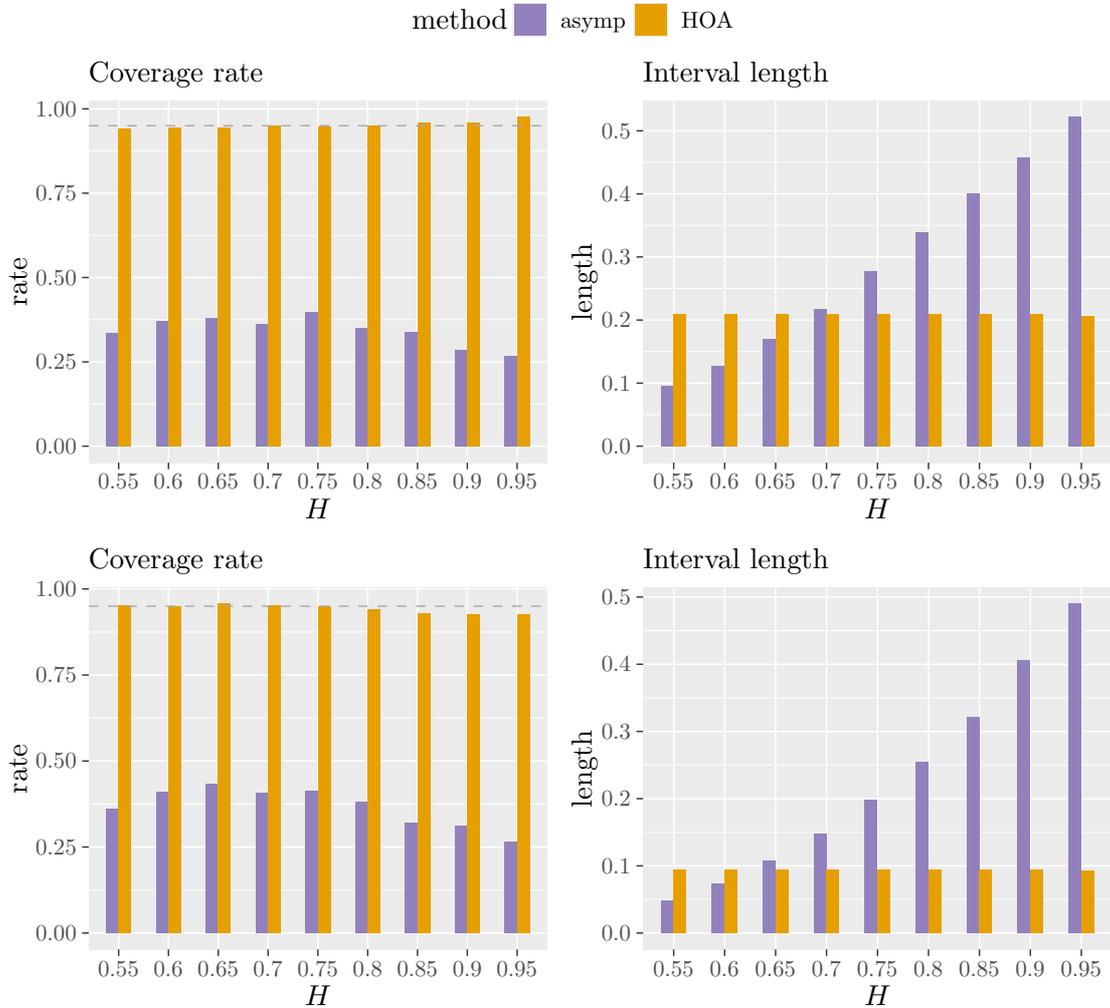

\section{Conclusion and Discussion} \label{se:conclusion}
In this work, we study higher-order approximations of the empirical process as an approach to improving statistical inference for long-range dependent time series.
More precisely, we study confidence intervals for values of the empirical process and for quantiles of the marginal distribution of stationary time series that are based on an approximation of the empirical process through higher-order terms in its Hermite expansion.
For statistics that can be expressed as partial sums of a subordinated Gaussian process, the Hermite expansion corresponds to an $L^2$-expansion of the subordinating function in orthonormal polynomials. The inclusion of higher-order terms in this expansion for the construction of confidence intervals results in narrower and more accurate confidence bands, especially when compared to those derived from first-order asymptotic theory.
Most notably, this approach differs from Gram–Charlier and Edgeworth expansions, which 
aim at improving approximations of the cumulative distribution function through incorporation of higher-order information. The latter techniques
incorporate correction terms relating to skewness, kurtosis, higher-order cumulants or moments and, to the best of our knowledge, have so far been analyzed against the background of short-range dependent time series and under the assumption of existing higher-order moments. Nonetheless, just as in the Edgeworth expansion framework, the number of terms included in the Hermite expansion of the considered statistic improves upon approximation quality from a theoretical perspective.
In practice, however, the theoretical improvement does not show due to the finite-sample behavior of statistics not adequately reflecting long-range dependence in time series.
Interestingly, our numerical results suggest that this mismatch between theory and practice may be mitigated by a phenomenon  we term {\em benign underestimation of long-range dependence}; see Section~\ref{se:estimatedHurst}. This effect appears to stabilize inference procedures in practice and presents another intriguing direction for future research.

The main theoretical contribution of this article is a proof for the convergence of higher-order terms in the Hermite expansion of the sequential empirical process for long-range dependent time series. This result is of general interest for empirical process theory and paves the way for novel approaches with respect to statistical inference for long-range dependent time series. 
First numerical approaches using the established theory illustrate an alternative way of constructing confidence intervals based on long-range dependent observations.

In comparison to the construction of confidence intervals based on the asymptotic distribution of the empirical process, the proposed procedure improves the quality of confidence intervals for the empirical process and quantiles of the marginal distribution. Generally speaking, our results provide  sufficient theoretical groundwork for the use of higher-order approximations for statistical inference on long-range dependent time series.
We conjecture that analogous theory would establish higher-order approximations for sequential partial sum processes of subordinated Gaussian sequences.
Such results would lay the foundation for improving upon statistical inference in change-point analysis for long-range dependent time series.
When testing stationarity against the alternative hypothesis of a structural change in a time series by means of the Wilcoxon test, the
phenomenon illustrated in Figure \ref{fig:empdistr_histograms} results in a high number of false positives; see \cite{Dehling2013}. While our proposed second-order approximation is expected to resolve an inflated test size, a  potential drawback could be a lower size being accompanied by a loss in test power. This is an expected trade-off since critical values derived from higher-order approximations would rely heavily on the lower-order term that also drives most of the behavior of the corresponding test statistic.
In addition to applications in change-point analysis, we envision
that the established theory for the two-parameter empirical process applies to goodness-of-fit testing in the presence of long-range dependence based on Kolmogorov-Smirnov and Cram\'{e}r-von Mises statistics.
We leave both change-point analysis and goodness-of-fit testing based on higher-order approximation of the empirical process as challenges for future research.

\appendix

\section{Proof of Theorem \ref{th:weakconvSRD}} \label{app:proofmain}

In order to prove Theorem \ref{th:weakconvSRD}, we first investigate the convergence of the finite-dimensional distributions and then tightness in $D([-\infty, \infty] \times [0,1])$; see Sections \ref{subse:fdd} and \ref{subse:tightness}, respectively. 

For the proof we will make use of the following notation
\begin{equation} \label{eq:tildem}
\widetilde{m}_N\pr{x, t }\defeq \frac{1}{\sqrt{N}}\sum_{n=1}^{\lfloor Nt\rfloor}S_n(x)
\end{equation}
with $S_n(x)$ as in \eqref{eq:LandS}.

\subsection{Convergence of the finite-dimensional distributions} \label{subse:fdd}

We need to show convergence of the finite-dimensional distributions, i.e., 
\begin{equation*}
\widetilde{m}_N\pr{x, t} \overset{f.d.d.}{\to} S(x,t),
\end{equation*}
where $\{S(x,t)\}$ is the limiting process with cross-covariances given in \eqref{eq:covS}. For this, it suffices to show
that for all $q_1,q_2 \in \NN$, and $(x_i, t_j)\in [-\infty, \infty] \times [0, 1]$, $i=1, \ldots, q_1; j=1,\dots,q_2$,
\begin{align*}
\left(\widetilde{m}_N\pr{x_{i}, t_j}\right)_{i=1,\dots,q_1;j=1,\dots,q_2}
\overset{\mathcal{D}}{\longrightarrow} 
\left(S(x_i, t_j)\right)_{i=1,\dots,q_1;j=1,\dots,q_2}.
\end{align*}
Recall from \eqref{eq:LandS} and \eqref{eq:tildem} that
\begin{align} \label{eq:fdd2}
\widetilde{m}_N\pr{x_{i}, t_j} 
= 
\frac{1}{\sqrt{N}}\sum_{n=1}^{\lfloor N t_j\rfloor} \sum\limits_{l=\lceil \frac{1}{D}\rceil}^{\infty}\frac{\widetilde{c}_l(x_{i})}{l!}H_l(\xi_n)
\end{align}
and set 
\begin{align} \label{eq:fdd3}
\left(\widetilde{m}_N\pr{x_{i}, t_j}\right)_{\substack{i=1,\dots,q_1 \\ j=1,\dots,q_2}}
=
\left(\frac{1}{\sqrt{N}}\sum_{n=1}^{\lfloor N t_j \rfloor} G_{i} (\xi_n)
\right)_{\substack{i=1,\dots,q_1 \\ j=1,\dots,q_2}} \
\text{ with } 
G_{i}(\cdot) =
\sum\limits_{l=\lceil \frac{1}{D}\rceil}^{\infty}\frac{\widetilde{c}_l(x_{i})}{l!}H_l(\cdot),
\end{align}
such that we have a $q_1\times q_2$-dimensional matrix of normalized partial sums of subordinated Gaussian sequences. In particular, different indices $i$ correspond to different functions $G_i$. 

Given \eqref{eq:ACF} and since the summation in \eqref{eq:fdd3} starts with $l=\lceil \frac{1}{D}\rceil$, all $G_i(\xi_n)$ are short-range dependent in the sense that their autocovariances are absolutely summable as shown in \eqref{eq:seriesACF}. 
Furthermore, let $r_i$ denote the Hermite rank of $G_i$. 
Due to \eqref{eq:fdd3}, we have $\lceil \frac{1}{D}\rceil \leq r_i $
such that $\frac{1}{D} < \lceil \frac{1}{D}\rceil \leq r_i$ since $\frac{1}{D} \notin \NN$.

Then, by Theorem 3 in \cite{bai2013multivariate}, we have 
\begin{align*}
\left(\frac{1}{\sqrt{N}}\sum_{n=1}^{\lfloor N t_j \rfloor} G_{i} (\xi_n)\right)_{\substack{i=1,\dots,q_1 \\ j=1,\dots,q_2}}
\overset{\mathcal{D}}{\longrightarrow}
(\mathcal{G}_1, \dots, \mathcal{G}_{q_1})',
\end{align*}
where $\mathcal{G}_i$, $i=1,\dots,q_1$, are $q_2$-dimensional Gaussian vectors $\mathcal{G}_i=(\mathcal{G}_i(t_1), \dots, \mathcal{G}_i(t_{q_2}))'$ with 
\begin{align}
\Cov ( \mathcal{G}_{i_1}(t_{j_1}), \mathcal{G}_{i_2}(t_{j_2}))
=&
\lim_{N \to \infty}
\frac{1}{N} \sum\limits_{n_1=1}^{\lfloor Nt_{j_1}\rfloor}
\sum\limits_{n_2=1}^{\lfloor Nt_{j_2}\rfloor}\sum\limits_{l_1,l_2=\lceil \frac{1}{D}\rceil}^{\infty}\frac{\widetilde{c}_{l_1}(x_{i_1})\widetilde{c}_{l_2}(x_{i_2})}{l_1!l_2!}\E(H_{l_1}(\xi_{n_1})H_{l_2}(\xi_{n_2})) \nonumber
\\=&
\min(t_{j_1}, t_{j_2})\sum\limits_{l=\lceil \frac{1}{D}\rceil}^{\infty}\frac{\widetilde{c}_l(x_{i_1})\widetilde{c}_l(x_{i_2})}{l!}\sum\limits_{n=-\infty}^{\infty}\gamma^{l}(n),
\label{eq:fdd4}
\end{align}
where \eqref{eq:fdd4} follows by equation (11) in \cite{bai2013multivariate} and since, for $l\geq \lceil \frac{1}{D}\rceil$ and due to \eqref{eq:seriesACF},
\begin{equation*}
\sum_{n \in \ZZ} |\gamma(n)|^l < \infty.
\end{equation*}

\subsection{Tightness} \label{subse:tightness}

Since the object of interest $\widetilde{m}_N\pr{x, t }$ in \eqref{eq:tildem} is a process in two parameters, proving tightness becomes particularly challenging.
We will first give a tightness criterion in $D([-\infty, \infty] \times [0,1])$ and then argue that it suffices to prove tightness in $D([0,1] \times [0,1])$.

In order to prove tightness of $\widetilde{m}_N\pr{x,t}$ in $D([-\infty, \infty] \times [0,1])$, we validate the following tightness criterion: for all $\varepsilon>0$
\begin{equation*}
\lim_{\delta\to 0}\limsup_{N\to \infty}
\PP\pr{\sup_{\substack{\abs{x_2-x_1}<\delta\\x_1,x_2 \in \RR} }
 \sup_{\substack{\abs{t_2-t_1}<\delta\\ 0\leq t_1,t_2\leq 1 }}\abs{
\widetilde{m}_N\pr{x_2,t_2}-\widetilde{m}_N\pr{x_1,t_1}
 }>\varepsilon  }=0;
\end{equation*}
see formula (26) in \cite{elktaibi2016}. In a more general setting, the criterion was introduced in \cite{ivanoff:thefuncD1980}.
We further write
\begin{equation*}
\widetilde{m}_N\pr{x_2,t_2}-\widetilde{m}_N\pr{x_1,t_1}=\widetilde{m}_N\pr{x_2,t_2}-\widetilde{m}_N\pr{x_1,t_2}+\widetilde{m}_N\pr{x_1,t_2}
-\widetilde{m}_N\pr{x_1,t_1}.
\end{equation*}
Then, it suffices to show
\begin{equation} \label{eq:tightness_Mn_part01}
\lim_{\delta\to 0}\limsup_{N\to \infty}
\PP\pr{\sup_{\substack{\abs{x_2-x_1}<\delta\\ x_1,x_2 \in \RR } }
 \sup_{t\in [0,1]}\abs{
\widetilde{m}_N\pr{x_2,t}-\widetilde{m}_N\pr{x_1,t}
 }>\varepsilon  }=0,
\end{equation}
\begin{equation} \label{eq:tightness_Mn_part02}
\lim_{\delta\to 0}\limsup_{N\to \infty}
\PP\pr{\sup_{ x\in \RR }
 \sup_{\substack{\abs{t_2-t_1}<\delta\\ 0\leq t_1,t_2\leq 1 }}\abs{
\widetilde{m}_N\pr{x,t_2}-\widetilde{m}_N\pr{x,t_1}
 }>\varepsilon  }=0.
\end{equation}

For \eqref{eq:tightness_Mn_part01}, note that due to continuity of $F^{-1}$ (following from strict monotonicity and continuity of $F$) for every $\delta >0$, there exists a $\widetilde{\delta}>0$ such that $|x_2-x_1|<\widetilde{\delta}$
implies $|F^{-1}(x_1)-F^{-1}(x_2)|<\delta$.
It then follows that 
\begin{align*}
 & \sup_{\substack{\abs{x_2-x_1}<\widetilde{\delta}\\ x_1,x_2 \in [-\infty, \infty] } }
 \sup_{t\in [0,1]}\abs{
\widetilde{m}_N\pr{x_2,t}-\widetilde{m}_N\pr{x_1,t}
 }\\
&=\sup_{\substack{\abs{x_2-x_1}<\widetilde{\delta}\\ x_1,x_2 \in [-\infty, \infty] } }
 \sup_{t\in [0,1]}\abs{
\widetilde{m}_N\pr{F(F^{-1}(x_2)),t}-\widetilde{m}_N\pr{F(F^{-1}(x_1)),t}
 }\\
  &\leq\sup_{\substack{\abs{x_2-x_1}<\delta\\ x_1,x_2 \in [0, 1] } }
 \sup_{t\in [0,1]}\abs{
\widetilde{m}_N\pr{F(x_2),t}-\widetilde{m}_N\pr{F(x_1),t}
  }\\
  &=\sup_{\substack{\abs{x_2-x_1}<\delta\\ x_1,x_2 \in [0, 1] } }
 \sup_{t\in [0,1]}\abs{
m_N\pr{x_2,t}-m_N\pr{x_1,t}
  }
\end{align*}
and accordingly
\begin{align*}
&
  \PP\pr{\sup_{\substack{\abs{x_2-x_1}<\widetilde{\delta}\\ x_1,x_2 \in [-\infty, \infty] } }
  \sup_{t\in [0,1]}\abs{
  \widetilde{m}_N\pr{x_2,t}-\widetilde{m}_N\pr{x_1,t}
  }>\varepsilon   }
\\&\leq
  \PP\pr{\sup_{\substack{\abs{x_2-x_1}<\delta\\ x_1,x_2 \in [0, 1] } }
  \sup_{t\in [0,1]}\abs{
  m_N\pr{x_2,t}-m_N\pr{x_1,t}
  }>\varepsilon   }.
\end{align*}

For \eqref{eq:tightness_Mn_part02}, note that due to $F:[-\infty, \infty]\longrightarrow [0,1 ]$ being a bijective function and due to $m_N(x,t)=\widetilde{m}_N(F^{-1}(x), t)$ with $m_N(x,t)$ as in \eqref{eq:mN:mainbody},
\begin{align*}
 & \lim_{\delta\to 0}\limsup_{N\to \infty}
\PP\pr{\sup_{ x\in [-\infty, \infty]}
 \sup_{\substack{\abs{t_2-t_1}<\delta\\ 0\leq t_1,t_2\leq 1 }}\abs{
\widetilde{m}_N\pr{x,t_2}-\widetilde{m}_N\pr{x,t_1}
  }>\varepsilon   }\\
  &=\lim_{\delta\to 0}\limsup_{N\to \infty}
\PP\pr{\sup_{ x\in [0, 1] }
 \sup_{\substack{\abs{t_2-t_1}<\delta\\ 0\leq t_1,t_2\leq 1 }}\abs{
\widetilde{m}_N\pr{F^{-1}(x),t_2}-\widetilde{m}_N\pr{F^{-1}(x),t_1}
  }>\varepsilon   }\\
   &=\lim_{\delta\to 0}\limsup_{N\to \infty}
\PP\pr{\sup_{ x\in [0, 1] }
 \sup_{\substack{\abs{t_2-t_1}<\delta\\ 0\leq t_1,t_2\leq 1 }}\abs{
m_N\pr{x,t_2}-m_N\pr{x,t_1}
  }>\varepsilon   }.
\end{align*}

It follows that the criteria \eqref{eq:tightness_Mn_part01} and \eqref{eq:tightness_Mn_part02} can be reformulated as
\begin{align}
\lim_{\delta\to 0}\limsup_{N\to \infty}
\PP\pr{\sup_{\substack{\abs{x_2-x_1}<\delta\\ 0\leq x_1,x_2\leq 1} }
 \sup_{t\in [0,1]}\abs{
m_N\pr{x_2,t}-m_N\pr{x_1,t}
  }>\varepsilon   } = 0 \label{eq:tightness_Mn_part1},
\end{align}
\begin{align}
\lim_{\delta\to 0}\limsup_{N\to \infty}
\PP\pr{\sup_{ x\in [0, 1] }
 \sup_{\substack{\abs{t_2-t_1}<\delta\\ 0\leq t_1,t_2\leq 1 }}\abs{
m_N\pr{x,t_2}-m_N\pr{x,t_1}
  }>\varepsilon   } = 0 \label{eq:tightness_Mn_part2}.
\end{align}

We consider \eqref{eq:tightness_Mn_part1} and \eqref{eq:tightness_Mn_part2} separately. Both proofs are based on chaining techniques following the ideas in
\citet[p.\ 1778]{dehling1989} and \citet[Section 5.1.4]{betken2020change}.

\subsubsection{Proof of \eqref{eq:tightness_Mn_part1}.}
In order to prove \eqref{eq:tightness_Mn_part1}, we apply a chaining technique. For this, we define the intervals
\begin{align*}
I_{1,p}\defeq[2p\delta, 2(p+1)\delta] \ \text{ and } \ I_{2, p}\defeq[(2p+1)\delta, (2(p+1)+1)\delta]
\end{align*}
for $p=0, \ldots, L_{\delta}\defeq\lceil \frac{1}{2\delta}-\frac{3}{2} \rceil$.
Then, the expression inside $\PP$ in \eqref{eq:tightness_Mn_part1} can be bounded as
\begin{align}
&\sup_{\substack{\abs{x_2-x_1}<\delta\\ 0 \leq x_1,x_2 \leq 1 } }
 \sup_{t\in [0,1]}\abs{
m_N\pr{x_2,t}-m_N\pr{x_1,t}
  } \nonumber
\\&\leq
\max\limits_{0\leq p\leq L_{\delta}}\sup_{\substack{x_1, x_2\in I_{1, p}\\ } }
 \sup_{t\in [0,1]}\abs{
m_N\pr{x_2,t}-m_N\pr{x_1,t}
  } \nonumber
\\& \hspace{1cm} +
\max\limits_{0\leq p\leq L_{\delta}}\sup_{\substack{x_1, x_2\in I_{2, p}\\ } }
 \sup_{t\in [0,1]}\abs{
m_N\pr{x_2,t}-m_N\pr{x_1,t}
  }.\label{eq:decomposition-1}
\end{align}
In the following, we consider only the first summand in \eqref{eq:decomposition-1}, since for the second summand analogous considerations hold. For this reason, it remains to show that
\begin{equation*}
\lim_{\delta\to 0}\limsup_{N\to \infty}
\PP\pr{\max\limits_{0\leq p\leq L_{\delta}}\sup_{\substack{x_1, x_2\in I_{1, p}\\ } }
 \sup_{t\in [0,1]}\abs{
m_N\pr{x_2,t}-m_N\pr{x_1,t}
  }>\varepsilon   }=0.
\end{equation*}
For this, it suffices to show that
\begin{align*}
\lim_{\delta\to 0}\limsup_{N\to \infty}\frac{1}{\delta}
\max_{0\leq p\leq L_\delta}
\PP\pr{\sup_{\substack{x_1, x_2\in I_{1, p}\\ } }
 \sup_{t\in [0,1]}\abs{
m_N\pr{x_2,t}-m_N\pr{x_1,t}
  }>\varepsilon   }=0.
\end{align*}
We write $I_{1, p}=[a_p, a_{p+1}]$, i.e., $a_p\defeq 2p\delta$ and $a_{p+1}\defeq 2(p+1)\delta$.
Note that
\begin{align} \label{eq:2supsup1}
\sup_{\substack{x_1, x_2\in I_{1, p}\\ } }
 \sup_{t\in [0,1]}\abs{
m_N\pr{x_2,t}-m_N\pr{x_1,t}
  }\leq 2\sup\limits_{x\in [0, 2\delta]} \sup_{t\in [0,1]}\abs{
m_N\pr{a_p,t}-m_N\pr{a_p+x,t}
  }.
\end{align}
Define refining partitions $x_i(k)$ for $k=0, \ldots, K_N$ with $K_N\to\infty$, for $N\to\infty$, and 
\begin{align} \label{eq:refiningpartitions00}
x_i(k)\defeq a_{p}+\frac{i}{2^k}2\delta, \hspace{0.2cm} i=0, \ldots, 2^k,
\end{align}
and choose $i_k(x)$ such that
\begin{align*} 
a_{p}+x\in \left( x_{i_k(x)}(k), x_{i_k(x)+1}(k)\right].
\end{align*}
We write 
\begin{equation} \label{eq:mbar}
\widebar{m}_{N,b}(x, y)\defeq\sup_{t\in [0, b]}\left|m_N(y, t)-m_N(x, t)\right|,
\hspace{0.2cm}
\widebar{m}_{N}(x, y)\defeq \widebar{m}_{N,1}(x, y).
\end{equation}
Then, with help of the introduced partition \eqref{eq:refiningpartitions00}, \eqref{eq:2supsup1} can be bounded as
\begin{align}
&\sup\limits_{t\in [0,1 ]}\abs{m_N(a_p, t)-m_N(a_p+x, t)} \nonumber\\
&\leq 
\sum\limits_{k=1}^{K_N} \widebar{m}_N(x_{i_{k}(x)}(k), x_{i_{k-1}(x)}(k-1)) + \widebar{m}_N(x_{i_{K_N}(x)}(K_N), a_p+x).
\label{eq:summandspart}
\end{align}
Consequently, \eqref{eq:summandspart} can be used to infer \eqref{eq:twoprobs-1} below
\begin{align}
&\PP\pr{\sup\limits_{x\in [0, 2\delta]}\sup\limits_{t\in [0,1 ]}\abs{m_N(a_p, t)-m_N(a_p+x, t)}>\varepsilon}\notag
\\&\leq
\sum\limits_{k=1}^{K_N} \PP\left(\sup\limits_{x\in [0, 2\delta]}\widebar{m}_N(x_{i_{k}(x)}(k), x_{i_{k-1}(x)}(k-1)) >\frac{\varepsilon}{(k+3)^2}\right)\notag\\
&\hspace{1cm}+ \PP\left(\sup\limits_{x\in [0, 2\delta]}\widebar{m}_N(x_{i_{K_N}(x)}(K_N), a_p+x)>\varepsilon-\sum\limits_{k=0}^{\infty}\frac{\varepsilon}{(k+3)^2}\right) \label{eq:twoprobs-1}
\\&\leq
\sum\limits_{k=1}^{K_N}\sum\limits_{i=0}^{2^k-1} \PP\left(\widebar{m}_N(x_{i+1}(k), x_{i}(k)) >\frac{\varepsilon}{(k+3)^2}
\right)\notag\\
&\hspace{1cm}+ \PP\left(\sup\limits_{x\in [0, 2\delta]}\widebar{m}_N(x_{i_{K_N}(x)}(K_N), a_p+x)>\frac{\varepsilon}{2}\right) \label{eq:twoprobs},
\end{align}
since $\sum\limits_{k=0}^{\infty}\frac{\varepsilon}{(k+3)^2}\leq \frac{\varepsilon}{2}$.

Throughout all following arguments, $C$ is a generic constant that can change upon each appearance. 
We consider the two probabilities in \eqref{eq:twoprobs} separately. 
The first one can be dealt with as follows:
\begin{align} 
&
 \sum\limits_{k=1}^{K_N}\sum\limits_{i=0}^{2^k-1} \PP\left(\widebar{m}_N(x_{i+1}(k), x_{i}(k)) >\frac{\varepsilon}{(k+3)^2}
\right) \nonumber \\
&\leq C \sum\limits_{k=1}^{K_N}\sum\limits_{i=0}^{2^k-1} \frac{(k+3)^8}{\varepsilon^4} 
\left( \frac{1}{N^{\theta}} \left(x_{i+1}(k)-x_i(k)\right) + \left(x_{i+1}(k)-x_i(k)\right)^{\frac{3}{2}} \right)
\label{eq:111} \\
&= C \sum\limits_{k=1}^{K_N}\sum\limits_{i=0}^{2^k-1} \frac{(k+3)^8}{\varepsilon^4}
\left( \frac{1}{N^{\theta}} \frac{2\delta}{2^k} + \left(\frac{2\delta}{2^k}\right)^{\frac{3}{2}} \right)
\label{eq:112}\\
&\leq 
C \delta \sum\limits_{k=1}^{K_N} \frac{(k+3)^8}{\varepsilon^4} \frac{1}{N^{\theta}}
+
C \delta^{\frac{3}{2}}\sum\limits_{k=1}^{K_N} \frac{(k+3)^8}{\varepsilon^4}\left(\frac{1}{2^k}\right)^{\frac{1}{2}} 
\notag\\
&\leq C \delta^{\frac{3}{2}} \label{eq:113}
\end{align}
for sufficiently large $N$, where
\eqref{eq:111} follows from Lemma \ref{le:bill} with $b=1$ and \eqref{eq:112} is a consequence of the choice of our partition in \eqref{eq:refiningpartitions00}. 
The last inequality \eqref{eq:113} is then satisfied for large enough $N$ since $\sum_{k=1}^{\infty} \frac{(k+3)^8}{\varepsilon^4}\left(\frac{1}{2^k}\right)^{\frac{1}{2}} < \infty$ by the ratio test for the convergence of series and
$\sum_{k=1}^{K_N} \frac{(k+3)^8}{\varepsilon^4} \frac{1}{N^\theta} \sim K_N^9 \frac{1}{N^\theta} \rightarrow 0$ choosing $K_N$ such that $K_N^9=o\left(N^\theta\right)$.

Now, we consider the second summand in \eqref{eq:twoprobs}. Choosing $K_{N}$ such that $K_N^9=o\left(N^{\theta}\right)$, but $\frac{K_{N}}{ \log_{2}(N) 
} \to\infty$, we get
\begin{equation*}
\lim_{\delta\to 0}\limsup_{N \to \infty}\PP\left(\sup\limits_{x\in [0, 2\delta]}\widebar{m}_N(x_{i_{K_N}(x)}(K_N), a_p+x)>\frac{\varepsilon}{2}\right)
\leq \lim_{\delta\to 0}\limsup_{N \to \infty}
\frac{C}{\varepsilon^4} 
\max\left\{ \frac{1}{N^{\frac{1}{2}}} , 
\frac{N}{2^{K_N \frac{1}{2}}} \right\}
=
0
\end{equation*}
for all $N \geq N_{\varepsilon}$ by applying Lemma \ref{le:randterm} below with $a=2\delta$, $b=1$ and $c=a_{p}$.

\begin{proof}[Proof of \eqref{eq:tightness_Mn_part2}]
In order to prove \eqref{eq:tightness_Mn_part2}, we first split the interval over $t_1,t_2$ in \eqref{eq:tightness_Mn_part2} into subintervals. This allows to bound the quantity of interest in terms of a supremum over a single parameter $t$ in a specific interval.
We then apply a similar chaining technique as in the proof of \eqref{eq:tightness_Mn_part1}. Note that here the chaining is applied to $x \in [0,1]$.

To deal with the supremum over $t_1,t_2$ in \eqref{eq:tightness_Mn_part2}, define
\begin{align*}
I_{1,p}\defeq[2p\delta, 2(p+1)\delta] \ \text{ and } \ I_{2, p}\defeq[(2p+1)\delta, (2(p+1)+1)\delta]
\end{align*}
for $p=0, \ldots, L_{\delta}\defeq\lceil \frac{1}{2\delta}-\frac{3}{2} \rceil$.
We first note that the expression in $\PP$ in \eqref{eq:tightness_Mn_part2} can be bounded through
\begin{align}
  \sup_{ 0\leq x \leq 1 }
  \sup_{\substack{\abs{t_2-t_1}<\delta\\ 0\leq t_1,t_2\leq 1 }}\abs{ m_N\pr{x,t_2}-m_N\pr{x,t_1}}
&\leq
  \sup_{ 0\leq x \leq 1 }\max_{0\leq p \leq L_{\delta}}
  \sup_{t_1, t_2\in I_{1, p} }\abs{m_N\pr{x,t_2}-m_N\pr{x,t_1}}
\notag\\
&\hspace{1cm}+
  \sup_{ 0\leq x \leq 1 }\max_{0\leq p \leq L_{\delta}}
  \sup_{t_1, t_2\in I_{2, p} }\abs{m_N\pr{x,t_2}-m_N\pr{x,t_1}}.\label{eq:decomposition-2}
\end{align}
In the following, we consider only the first summand in \eqref{eq:decomposition-2}, since for the second summand analogous considerations hold. For this reason, it remains to show that
\begin{align*}
\lim_{\delta\to 0}\limsup_{N\to \infty}
\PP\pr{ \sup_{ 0\leq x \leq 1 }\max_{0\leq p \leq L_{\delta}}
 \sup_{t_1, t_2\in I_{1, p} }\abs{
m_N\pr{x,t_2}-m_N\pr{x,t_1}}>\varepsilon   }=0.
\end{align*}
We write $I_{1, p}=[a_p, a_{p+1}]$, i.e., $a_p\defeq 2p\delta$ and $a_{p+1}\defeq 2(p+1)\delta$.
Note that
\begin{align}
\sup_{t_1, t_2\in I_{1, p} }\abs{
m_N\pr{x,t_2}-m_N\pr{x,t_1}}
\nonumber
&\leq \sup_{ t_2\in I_{1, p} }\abs{
m_N\pr{x,t_2}-m_N\pr{x,a_p}}
\\&\hspace{1cm}+
\sup_{t_1\in I_{1, p} }\abs{
m_N\pr{x,a_p}-m_N\pr{x,t_1}}
\nonumber \\
&\leq 2\sup_{t\in [0, 2\delta] }\abs{
m_N\pr{x,a_p}-m_N\pr{x,a_p+t}}.
\label{al:aaassspppooo}
\end{align}

For the supremum over $x\in [0,1]$, we apply a similar chaining technique as in the proof of \eqref{eq:tightness_Mn_part1}. 
Define refining partitions $x_i(k)$ for $k=0, \ldots, K_N$ with $K_N\to\infty$, for $N\to\infty$, and 
\begin{align} \label{eq:refiningpartitions11}
x_i(k)=\frac{i}{2^k}, \ i=0, \ldots, 2^k,
\end{align}
and choose $i_k(x)$ such that
\begin{align*}
x\in \left( x_{i_k(x)}(k), x_{i_k(x)+1}(k)\right].
\end{align*}

Moreover, define 
\begin{equation} \label{eq:notation_mns}
  \begin{aligned}
    m_{N}(x, t, p)\defeq m_N\pr{x,a_p}-m_N\pr{x,a_p+t}
    \\
    m_{N}(x, y, t, p)\defeq m_{N}(y, t, p)-m_{N}(x, t, p).
  \end{aligned}
\end{equation}
Then, continuing with \eqref{al:aaassspppooo}, it follows that
\begin{align*}
&\max_{0\leq p \leq L_{\delta}} 
\sup_{t\in [0, 2\delta] }\abs{
m_N\pr{x,a_p}-m_N\pr{x,a_p+t}}\\
&=
\max_{0\leq p \leq L_{\delta}} 
\sup_{t\in [0, 2\delta] }\abs{m_{N}(x, t, p)}
\\
& = \max_{0\leq p \leq L_{\delta}} 
\sup_{t\in [0, 2\delta] }\abs{m_{N}(0, t, p)-m_{N}(x, t, p)},
\end{align*}
since $c_l(0)=\E\pr{\1_{\{F(X_0)\leq 0\}}H_l(\xi_0)}=0$. We have
\begin{align}
&\max_{0\leq p \leq L_{\delta}} 
\sup_{t\in [0, 2\delta] }\abs{m_{N}(0, t, p)-m_{N}(x, t, p)}
\notag
\\&\leq 
\sum\limits_{k=1}^{K_N} \max_{0\leq p \leq L_{\delta}} 
\sup_{t\in [0, 2\delta] }\abs{m_{N}
(x_{i_{k}(x)}(k), x_{i_{k-1}(x)}(k-1), t, p)} \notag
\\&\hspace{1cm}+ 
\max_{0\leq p \leq L_{\delta}} 
\sup_{t\in [0, 2\delta] }\abs{m_{N}(x_{i_{K_N}(x)}(K_N), x, t, p)}
\notag \\
&=:
\sum\limits_{k=1}^{K_N} \widebar{m}_N
(x_{i_{k}(x)}(k), x_{i_{k-1}(x)}(k-1)) + 
\widebar{m}_N(x_{i_{K_N}(x)}(K_N), x),
\label{eq:summandspartXX}
\end{align}
where $\widebar{m}_N(x, y):=\max_{0\leq p \leq L_{\delta}} 
\sup_{t\in [0, 2\delta] }\abs{m_{N}
(x, y, t, p)}$. 
Consequently, \eqref{eq:summandspartXX} can be used to infer \eqref{al:border_terms-1} below
\begin{align}
&\PP\pr{\sup\limits_{x\in [0, 1]}\max_{0\leq p \leq L_{\delta}} 
\sup_{t\in [0, 2\delta] }\abs{m_{N}(0, t, p)-m_{N}(x, t, p)}>\varepsilon}\notag
\\&\leq
\sum\limits_{k=1}^{K_N} \PP\left(\sup\limits_{x\in [0, 1]}\widebar{m}_N(x_{i_{k}(x)}(k), x_{i_{k-1}(x)}(k-1)) >\frac{\varepsilon}{(k+3)^2}\right)\notag
\\& \hspace{1cm} + \PP\left(\sup\limits_{x\in [0, 1]}\widebar{m}_N(x_{i_{K_N}(x)}(K_N), x)>\varepsilon-\sum\limits_{k=0}^{\infty}\frac{\varepsilon}{(k+3)^2}\right) \label{al:border_terms-1}
\\&\leq
\sum\limits_{k=1}^{K_N}\sum\limits_{i=0}^{2^k} \PP\left(\widebar{m}_N(x_{i+1}(k), x_{i}(k)) >\frac{\varepsilon}{(k+3)^2} \right) \notag
\\& \hspace{1cm} + \PP\left(\sup\limits_{x\in [0, 1]}\widebar{m}_N(x_{i_{K_N}(x)}(K_N), x)>\frac{\varepsilon}{2}\right),\label{al:border_terms}
\end{align}
since $\sum_{k=0}^{\infty}\frac{\varepsilon}{(k+3)^2}\leq \frac{\varepsilon}{2}$. We consider the two summands in \eqref{al:border_terms} separately. For the first summand in \eqref{al:border_terms} we need some preliminary results. Note that for any $\eta >0$, 
\begin{align}
 \PP\left(\widebar{m}_N(x_{i+1}(k), x_{i}(k)) >\eta \right)
 &=\PP\left( \max_{0\leq p \leq L_{\delta}} \sup_{t\in [0, 2\delta] }\left|m_{N}
(x_{i+1}(k), x_{i}(k), t, p)\right|>\eta\right)
\nonumber
\\
&\leq\sum\limits_{p=0}^{L_{\delta}}\PP\left( \sup_{t\in [0, 2\delta] }\left|m_{N}
(x_{i+1}(k), x_{i}(k), t, p)\right|>\eta\right).
\label{eq:opopopopopopopopo0001}
\end{align}
Due to stationarity it follows that
\begin{align}
\PP\left( \sup_{t\in [0, 2\delta] }\left|m_{N}
(x_{i+1}(k), x_{i}(k), t, p)\right|>\eta\right)
=\PP\left( \sup_{t\in [0, 2\delta] }\left|m_{N}
(x_{i+1}(k), x_{i}(k), t, 0)\right|>\eta\right).
\label{eq:opopopopopopopopo0002}
\end{align}
Combining \eqref{eq:opopopopopopopopo0001} and \eqref{eq:opopopopopopopopo0002}, we get
\begin{align} \label{eq:opopopopopopopopo0003}
 \PP\left(\widebar{m}_N(x_{i+1}(k), x_{i}(k)) >\eta \right)
 \leq \frac{1}{\delta}\PP\left( \sup_{t\in [0, 2\delta] }\left| m_{N}
(x_{i+1}(k), x_{i}(k), t, 0)\right|>\eta\right).
\end{align}
Due to the notation in \eqref{eq:notation_mns}, we have
\begin{align*}
&m_{N}
(x_{i+1}(k), x_{i}(k), t, 0)\\
&=
m_{N}(x_{i}(k), t, 0)-m_{N}(x_{i+1}(k), t, 0)\\
&=
m_N\pr{x_{i}(k),0}-m_N\pr{x_{i}(k),t}
-(m_{N}(x_{i+1}(k), 0)-m_{N}(x_{i+1}(k), t))
\end{align*}
such that
\begin{align*}
 \sup_{t\in [0, 2\delta] }\abs{m_{N}
(x_{i+1}(k), x_{i}(k), t, 0)}
\leq 2\sup_{t\in [0, 2\delta] }\abs{m_{N}(x_{i+1}(k), t)-m_N\pr{x_{i}(k),t}}.
\end{align*}
We can then bound the first summand in \eqref{al:border_terms}, with further explanations given below, as follows
\begin{align}
&\sum\limits_{k=1}^{K_N}\sum\limits_{i=0}^{2^k} \PP\left(\widebar{m}_{N}(x_{i+1}(k), x_{i}(k)) >\frac{\varepsilon}{(k+3)^2} \right)
\notag \\ 
&\leq
\sum\limits_{k=1}^{K_N}\sum\limits_{i=0}^{2^k}\frac{1}{\delta}
P\left( 2\sup_{t\in [0, 2\delta] }\abs{m_{N}(x_{i+1}(k), t)-m_N\pr{x_{i}(k),t}}>\frac{\varepsilon}{(k+3)^2}\right)
\label{eq:21100}
\\
&\leq C \frac{1}{\delta}\sum\limits_{k=1}^{K_N}\sum\limits_{i=0}^{2^k} \frac{(k+3)^8}{\varepsilon^4}
\left( \delta^{2-\theta} \frac{1}{N^{\theta}} \left(x_{i+1}(k)-x_i(k)\right) + \delta^{2} \left(x_{i+1}(k)-x_i(k)\right)^{\frac{3}{2}} \right)
\label{eq:211}\\
&\leq \frac{1}{\delta}C \sum\limits_{k=1}^{K_N}\sum\limits_{i=0}^{2^k} \frac{(k+3)^8}{\varepsilon^4}
\left( \delta^{2-\theta} \frac{1}{N^{\theta}} \frac{1}{2^k} + \delta^{2} \left(\frac{1}{2^k}\right)^{\frac{3}{2}} \right)
\label{eq:212}\\
&\leq 
C \frac{1}{\delta}\delta^{2-\theta} \sum\limits_{k=1}^{K_N} \frac{(k+3)^8}{\varepsilon^4} \frac{1}{N^{\theta}}
+
C \frac{1}{\delta}\delta^2\sum\limits_{k=1}^{K_N} \frac{(k+3)^8}{\varepsilon^4}\left(\frac{1}{2^k}\right)^{\frac{1}{2}} 
\notag \\
&\leq C \max\{ \delta^{1-\theta}, \delta \} \label{eq:213}
\end{align}
for sufficiently large $N$, where \eqref{eq:21100} is due to 
\eqref{eq:opopopopopopopopo0003}
and \eqref{eq:211} follows from Lemma \ref{le:bill} with $b = \delta$, \eqref{eq:212} is a consequence of the choice of our partition in \eqref{eq:refiningpartitions11}. 
The last inequality \eqref{eq:213} is then satisfied for large enough $N$ since $\sum_{k=1}^{\infty} \frac{(k+3)^8}{\varepsilon^4}\left(\frac{1}{2^k}\right)^{\frac{1}{2}} < \infty$ by the ratio test for the convergence of series and
$\sum_{k=1}^{K_N} \frac{(k+3)^8}{\varepsilon^4} \frac{1}{N^{\theta}} \sim K_N^9 \frac{1}{N^{\theta}} \rightarrow 0$, as $N\rightarrow \infty$, choosing $K_N$ such that $K_N^9=o\left(N^{\theta}\right)$.

Now, we consider the second summand in \eqref{al:border_terms}. 
\begin{align*}
 &
 \PP\left(\sup\limits_{x\in [0, 1]}\text{$\widebar{m}_{N}$}(x_{i_{K_N}(x)}(K_N), x)>\frac{\varepsilon}{2}\right)\\
 &=
 \PP\left(\sup\limits_{x\in [0, 1]}\max_{0\leq p \leq L_{\delta}} 
\sup_{t\in [0, 2\delta] }\abs{m_{N}(x_{i_{K_N}(x)}(K_N), x, t, p)}>\frac{\varepsilon}{2}\right)\\
&\leq
\frac{1}{\delta}
 \PP\left(\sup\limits_{x\in [0, 1]}
\sup_{t\in [0, 2\delta] }\abs{m_{N}(x_{i_{K_N}(x)}(K_N), x, t, 0)}>\frac{\varepsilon}{2}\right)\\
&\leq 
\frac{1}{\delta}
 \PP\left(2\sup\limits_{x\in [0, 1]}
\sup\limits_{t\in [0, 2\delta]}\left|m_{N}(x_{i_{K_N}(x)}(K_N), t)-m_{N}(x, t)\right|>\frac{\varepsilon}{2}\right).
\end{align*}
Choosing $K_{N}$ such that $K_N^9=o\left(N^{\theta}\right)$, but $\frac{K_{N}}{ \log_{2}(N)} \to\infty$, we get
\begin{equation*}
\frac{1}{\delta}\PP\left(2\sup\limits_{x\in [0, 1]}\widebar{m}_{N}(x_{i_{K_N}(x)}(K_N), x)>\frac{\varepsilon}{2}\right)
\leq 
\frac{C}{\varepsilon^4}\delta^{\frac{1}{2}} 
\max\left\{ \frac{1}{N^{\frac{1}{2}}} , 
\frac{N}{2^{K_N \frac{1}{2}}} \right\}
\to 0 
\hspace{0.2cm} 
\text{ as }
N \to \infty,
\end{equation*}
for all $N \geq N_{\varepsilon}$ by applying Lemma \ref{le:randterm} below with $a=1$, $b=\delta$ and $c=0$.
\end{proof}

\section{Technical results and their proofs}
\label{se:appendixB}

In this section, we provide some technical results and their proofs.

\begin{lemma} \label{le:bill}
Let $\widebar{m}_{N,b}$ be as in 
\eqref{eq:mbar}. Then, there are constants $C_{1}, C_{2}>0$ and a $\theta \in (0,\frac{1}{2}]$ such that for any $\lambda>0$,
\begin{equation} \label{eq:le:bound}
\PP\left(\widebar{m}_{N,b}(x,y) >\lambda \right)
\leq 
C_{1} \frac{1}{\lambda^4} b^{2-\theta} \frac{1}{N^{\theta}} \left(y-x\right)
+
C_{2} \frac{1}{\lambda^4} b^2 \left(y-x\right)^{\frac{3}{2}}
\end{equation}
for any $b>0$ and all $x,y \in [0,1]$ with $y>x$.
\end{lemma}

\begin{proof}
In order to bound the probability in \eqref{eq:le:bound}, we use arguments from \cite{billingsley:1968}. 
For this, we express $\widebar{m}_{N,b}$ in \eqref{eq:mbar} as 
\begin{equation*}
\begin{aligned}
\widebar{m}_{N,b}(x, y) 
&= 
\sup_{t\in [0, b]} \frac{1}{\sqrt{N}} \left| \sum\limits_{n=1}^{\lfloor Nt \rfloor}\sum\limits_{l=\lceil \frac{1}{D}\rceil}^{\infty}\frac{c_l(y)-c_l(x)}{l!}H_l(\xi_n) \right|
\\&=
\max_{1 \leq k \leq\lfloor Nb \rfloor} \frac{1}{\sqrt{N}} \left| \sum\limits_{n=1}^{k}\sum\limits_{l=\lceil \frac{1}{D}\rceil}^{\infty}\frac{c_l(y)-c_l(x)}{l!}H_l(\xi_n) \right|
=: \max_{1 \leq k \leq \lfloor Nb \rfloor} \left| \s_{k} \right|.
\end{aligned}
\end{equation*}
Note that
\begin{align*}
\s_j-\s_i=\frac{1}{\sqrt{N}}\sum\limits_{n=i+1}^{j}\sum\limits_{l=\lceil \frac{1}{D}\rceil}^{\infty}\frac{c_l(y)-c_l(x)}{l!}H_l(\xi_n)
\end{align*}
and define
\begin{align*}
h_{x, y}(\xi_n)\defeq \sum\limits_{l=\lceil \frac{1}{D}\rceil}^{\infty}\frac{c_l(y)-c_l(x)}{l!}H_l(\xi_n)
=\1_{\{x<F(G(\xi_n))\leq y\}}- (y-x) -\sum\limits_{l=r}^{\lfloor \frac{1}{D}\rfloor}\frac{c_l(y)-c_l(x)}{l!}H_l(\xi_n).
\end{align*}
Then,
\begin{equation}\label{eq:sum}
\begin{aligned}
\E \left|\s_j-\s_i\right|^4
&=\E\left|\frac{1}{\sqrt{N}}\sum\limits_{n=i+1}^{j}h_{x, y}(\xi_n)\right|^4\\
&=\frac{1}{N^2} \E \left(\Sigma_1+4\Sigma_{21}+3\Sigma_{22}+6\Sigma_3+\Sigma_4\right)
\end{aligned}
\end{equation}
with
\begin{align*}
\Sigma_1	& \defeq \sum\limits_{n=i+1}^j h_{x, y}^4(\xi_n), \\ 
\Sigma_{21} & \defeq \sum{}^{'} h_{x, y}^3(\xi_{n_1}) h_{x, y}(\xi_{n_2}), \ \Sigma_{22} \defeq \sum{}^{'} h_{x, y}^2(\xi_{n_1}) h_{x, y}^2(\xi_{n_2}), \\
\Sigma_{3} & \defeq \sum{}^{'} h_{x, y}^2(\xi_{n_1}) h_{x, y}(\xi_{n_2})h_{x, y}(\xi_{n_3}),\\
\Sigma_{4} & \defeq \sum{}^{'} h_{x, y}(\xi_{n_1}) h_{x, y}(\xi_{n_2})h_{x, y}(\xi_{n_3})h_{x, y}(\xi_{n_4}),
\end{align*}
where $\sum\limits{}^{'}$ extends over all different indices $i+1\leq n_1, \ldots, n_p\leq j$, $n_r \neq n_s$, $r \neq s$, $p=1,\dots,4$.

Note that for any even integer $p\geq 2$, there is a constant $C > 0$ such that
\begin{equation} \label{eq:hp}
\E\left(h_{x, y}^p(\xi_0)\right) \leq C\left(y-x\right)
\end{equation}
since
\begin{align}
\E\left(h_{x, y}^p(\xi_0)\right)
&\leq 
C\left(\E \1_{\left\{x<F(G(\xi_0))\leq y\right\}}+ (y-x)^p +
\E\pr{\left| \sum\limits_{l=r}^{\lfloor \frac{1}{D}\rfloor}\left(\frac{c_l(y)-c_l(x)}{l!}\right)H_l(\xi_0)\right|^p}\right)
 \label{eq:hp:eq2-1} \\
&\leq 
C\left(\E \1_{\left\{x< F(G(\xi_0)) \leq y\right\}}+ (y-x)^p+
\sum\limits_{l=r}^{\lfloor \frac{1}{D}\rfloor}\left(\frac{c_l(y)-c_l(x)}{l!}\right)^p\E\left(\left|H_l(\xi_0)\right|^p\right)\right) \label{eq:hp:eq2} \\
&\leq C\left(\E \1_{\left\{x<F(G(\xi_0)) \leq y\right\}}+ (y-x)^p+\sum\limits_{l=r}^{\lfloor \frac{1}{D}\rfloor}\frac{\left|c_l(y)-c_l(x)\right|^p}{(l!)^{\frac{p}{2}}}(p-1)^{\frac{lp}{2}}\right) \label{eq:hp:eq3} \\
&\leq C\left(y-x\right) \label{eq:hp:eq4},
\end{align}
where $C$ is a generic constant that can change upon each appearance. Inequalities \eqref{eq:hp:eq2-1} and  
 \eqref{eq:hp:eq2} follow from
\begin{align*}
\left(\sum\limits_{k=1}^n|x_k|\right)^p\leq n^{p-1}\sum\limits_{k=1}^n|x_k|^p
\end{align*}
which holds for any $p\geq1$ and is a direct consequence of Hölder's inequality.
Inequality \eqref{eq:hp:eq3} follows by Nelson's inequality; see \cite{nourdin:rosinski:2014} Lemma 2.1. Finally, \eqref{eq:hp:eq4} is a consequence of applying the Cauchy-Schwarz inequality
\begin{equation} \label{eq:ccCS}
\pr{c_l(y)-c_l(x)}^2
=\E^2\left(\1_{\left\{x<F\pr{G\pr{\xi_0}}\leq y\right\}} H_l(\xi_0) \right)
\leq (y-x) \E\pr{H^2_l(\xi_0)}=(y-x)l!
\end{equation}
and by noticing that $(y-x)^{\frac{p}{2}} \leq y-x$ for $p \geq 2$ and $x, y \in (0,1)$.

We now consider the summands on the right-hand side of formula \eqref{eq:sum} separately. 
Starting with $\Sigma_1$, note that \eqref{eq:hp} gives
\begin{align*}
\E\left|\Sigma_1\right|\leq C (j-i) (y-x).
\end{align*}
In order to estimate the remaining quantities, we make use of Lemma 4.5 in \cite{taqqu1977law}.
This together with \eqref{eq:hp} above, immediately yields 
\begin{align*}
\E\left|\Sigma_{21}\right|&\leq C(j-i)^{\frac{3}{2}} \left(\E \left(h^2_{x, y}(\xi_0)\right)\right)^{\frac{1}{2}}\left(\E \left(h^6_{x, y}(\xi_0)\right)\right)^{\frac{1}{2}}\leq C(j-i)^{\frac{3}{2}} (y-x),\\
\E\left|\Sigma_3\right|&\leq C(j-i)^2 \E \left(h^2_{x, y}(\xi_0)\right)\left(\E \left(h^4_{x, y}(\xi_0)\right)\right)^{\frac{1}{2}}\leq C(j-i)^2 (y-x)^{\frac{3}{2}},\\
\E\left|\Sigma_4\right|&\leq C (j-i)^2 \E^2 \left(h^2_{x, y}(\xi_0)\right)\leq C(j-i)^2 (y-x)^2.
\end{align*}
It remains to find an upper bound for $\E\Sigma_{22}$. For this, define
\begin{equation*} \label{eq:LandLtilde}
L_{x,y}(\xi_{n}):=\sum\limits_{l=r}^{\lfloor \frac{1}{D}\rfloor}\frac{c_l(y)-c_l(x)}{l!}H_l(\xi_n)
\
\text{ and }
\
\widetilde{h}_{x, y}(z)=\1_{\{x< F(G(z)) \leq y\}}-(y-x).
\end{equation*} 
It then holds that
\begin{align}
\E \left(h^2_{x, y}(\xi_{n_1})h^2_{x, y}(\xi_{n_2})\right)
&=\E \left( (\widetilde{h}_{x, y}(\xi_{n_1})-L_{x,y}(\xi_{n_{1}}))^2(\widetilde{h}_{x, y}(\xi_{n_2})-L_{x,y}(\xi_{n_{2}}))^2\right)\notag\\
&\leq
4 \E \left( (\widetilde{h}^2_{x, y}(\xi_{n_1}) + L^2_{x,y}(\xi_{n_{1}}))(\widetilde{h}^2_{x, y}(\xi_{n_2})+L^2_{x,y}(\xi_{n_{2}})) \right)\notag\\
&=
4 \Big( \E (\widetilde{h}^2_{x, y}(\xi_{n_1})\widetilde{h}^2_{x, y}(\xi_{n_2})) + \E( \widetilde{h}^2_{x, y}(\xi_{n_1})L^2_{x,y}(\xi_{n_{2}})) \notag
\\ & \hspace{1cm}
+ \E(L^2_{x,y}(\xi_{n_{1}}) \widetilde{h}^2_{x, y}(\xi_{n_2}))+ 
\E (L^2_{x,y}(\xi_{n_{1}}) L^2_{x,y}(\xi_{n_{2}})) \Big). \label{eq:4_summands}
\end{align}

Before we consider the four summands in \eqref{eq:4_summands} separately, we make the following observation. With arguments given below,
\begin{align}
\E (L^4_{x,y}(\xi_{n}))
&= \E \left( \sum\limits_{l=r}^{\lfloor \frac{1}{D}\rfloor}\frac{c_l(y)-c_l(x)}{l!}H_l(\xi_n) \right)^4 \notag \\
&\leq C \sum\limits_{l=r}^{\lfloor \frac{1}{D}\rfloor} \left( \frac{c_l(y)-c_l(x)}{l!} \right)^4 \E H_l^4(\xi_n) \label{al:ggggggGGG1} \\
&\leq C \sum\limits_{l=r}^{\lfloor \frac{1}{D}\rfloor} \frac{\left(c_l(y)-c_l(x)\right)^4}{l!^2} 3^{2l}, \label{al:ggggggGGG2}
\end{align}
where \eqref{al:ggggggGGG1} follows by H\"older's inequality and \eqref{al:ggggggGGG2} by Nelson's inequality; see \cite{nourdin:rosinski:2014} Lemma 2.1. Then, combining \eqref{al:ggggggGGG2} with \eqref{eq:ccCS}, 
\begin{align} \label{eq:L4}
\E (L^4_{x,y}(\xi_{n}))\leq C(y-x)^2.
\end{align}
In the following, we consider the summands in \eqref{eq:4_summands} separately.
For the last one, the Cauchy-Schwarz inequality and \eqref{eq:L4} yield
\begin{align} \label{eq:h-firstsummand}
\E (L^2_{x,y}(\xi_{n_{1}}) L^2_{x,y}(\xi_{n_{2}})) 
\leq
(\E (L^4_{x,y}(\xi_{n_{1}})))^{\frac{1}{2}} (\E(L^4_{x,y}(\xi_{n_{2}})))^{\frac{1}{2}}
\leq C (y-x)^2.
\end{align}
Since $\E(\widetilde{h}^2_{x, y}(\xi_{n_1}) L^2_{x,y}(\xi_{n_{2}}))$ and 
$\E(L^2_{x,y}(\xi_{n_{1}}) \widetilde{h}^2_{x, y}(\xi_{n_2}))$ in \eqref{eq:4_summands} can be treated analogously, we only consider $\E(\widetilde{h}^2_{x, y}(\xi_{n_1}) L^2_{x,y}(\xi_{n_{2}}))$.
Given the definition of $\widetilde{h}^2_{x, y}$ in \eqref{eq:LandLtilde} and with further explanations provided below, we get
\begin{align}
\E( \widetilde{h}^2_{x, y}(\xi_{n_1})L^2_{x,y}(\xi_{n_{2}}))
&=
\E( (\1_{\{x< F(G(\xi_{n_{1}})) \leq y\}}-(y-x))^2L^2_{x,y}(\xi_{n_{2}})) \notag
\\&=
\E(\1_{\{x< F(G(\xi_{n_{1}})) \leq y\}}L^2_{x,y}(\xi_{n_{2}}))
+(y-x)^2\E(L^2_{x,y}(\xi_{n_{2}})) \notag
\\&\hspace{3cm}+2(y-x)\E(\1_{\{x< F(G(\xi_{n_{1}})) \leq y\}}L^2_{x,y}(\xi_{n_{2}})) \label{al:werty1}
\\&\leq 
C(y-x)^{\frac{3}{2}}+C(y-x)^3+C(y-x)^{\frac{5}{2}} \notag
\\&\leq C(y-x)^{\frac{3}{2}}. \label{eq:h-secondsummand}
\end{align}
For the first and third summand in \eqref{al:werty1}, the Cauchy-Schwarz inequality and \eqref{eq:L4} yield
\begin{align*}
\E(\1_{\{x< F(G(\xi_{n_{1}})) \leq y\}}L^2_{x,y}(\xi_{n_{2}}))
\leq \left(\E\pr{\1_{\{x< F(G(\xi_{n_{1}})) \leq y\}}}\right)^{\frac{1}{2}}\left( \E ( L^4_{x,y}(\xi_{n_{2}}) )\right)^{\frac{1}{2}}\leq C(y-x)^{\frac{3}{2}}.
\end{align*}
Moreover, we have, by orthogonality of the Hermite polynomials, that the second summand in \eqref{al:werty1} can be bounded as
\begin{align*}
\E(L^2_{x,y}(\xi_{n_{2}}))=\E\pr{\sum\limits_{l=r}^{\lfloor \frac{1}{D}\rfloor}\frac{c_l(y)-c_l(x)}{l!}H_l(\xi_{n_{2}})}^2
=\sum\limits_{l=r}^{\lfloor \frac{1}{D}\rfloor}\frac{\pr{c_l(y)-c_l(x)}^2}{(l!)^2}\E H^2_l(\xi_{n_{2}})\leq C(y-x).
\end{align*}
The first summand in \eqref{eq:4_summands} requires some more calculations, leading to
\begin{align*}
&\E \left(\widetilde{h}^2_{x, y}(\xi_{n_1})\widetilde{h}^2_{x, y}(\xi_{n_2})\right)\\
&=\E\left[\left(\1_{\left\{x< F(X_{n_1}) \leq y\right\}}-(y-x)\right)^2\left(\1_{\left\{x< F(X_{n_2}) \leq y\right\}}-(y-x)\right)^2\right]\\
&=\E\Bigg[\left(\1_{\left\{x< F(X_{n_1}) \leq y\right\}}+(y-x)^2-2(y-x)\1_{\left\{x<F(X_{n_1})\leq y\right\}}\right)
\\ &\hspace{1cm} \times 
\left(\1_{\left\{x< F(X_{n_2}) \leq y\right\}}+(y-x)^2-2(y-x) \1_{\left\{x<F(X_{n_2})\leq y\right\}}\right)\Bigg]\\
&=\E\left(\1_{\left\{x< F(X_{n_1}) \leq y\right\}}\1_{\left\{x< F(X_{n_2}) \leq y\right\}}\right)\left(1-2(y-x)\right)^2+2(y-x)^3-3(y-x)^4 \\
&=\E \left(\widetilde{h}_{x, y}(\xi_{n_1})\widetilde{h}_{x, y}(\xi_{n_2})\right)\left(1-2(y-x)\right)^2+(y-x)^2(1-(y-x))^{2}\\
&\leq \left|\E \left(\widetilde{h}_{x, y}(\xi_{n_1})\widetilde{h}_{x, y}(\xi_{n_2})\right)\right| + (y-x)^2.
\end{align*}

By orthogonality of the Hermite expansion
\begin{align*}
\E \left(\widetilde{h}_{x, y}(\xi_{n_1})\widetilde{h}_{x, y}(\xi_{n_2})\right)
&= \E (h_{x, y}(\xi_{n_1})h_{x, y}(\xi_{n_2}) ) +
\sum\limits_{l=r}^{\lfloor \frac{1}{D}\rfloor} \left( \frac{c_l(y)-c_l(x)}{l!} \right)^2 \E \left( H_l(\xi_{n_1}) H_l(\xi_{n_2}) \right)
\\&= \E (h_{x, y}(\xi_{n_1})h_{x, y}(\xi_{n_2}) ) +
\sum\limits_{l=r}^{\lfloor \frac{1}{D}\rfloor} \frac{\left(c_l(y)-c_l(x)\right)^2}{l!} \gamma^{l}(n_1 -n_2)
\\&\leq \E (h_{x, y}(\xi_{n_1})h_{x, y}(\xi_{n_2}) ) +
(y-x) \sum\limits_{l=r}^{\lfloor \frac{1}{D}\rfloor} \gamma^{l}(n_1 -n_2)
\\&\leq \E (h_{x, y}(\xi_{n_1})h_{x, y}(\xi_{n_2}) ) +
(y-x) C \gamma^{r}(n_1 -n_2).
\end{align*}
Then,
\begin{align*}
\left|\E \left(\widetilde{h}_{x, y}(\xi_{n_1})\widetilde{h}_{x, y}(\xi_{n_2})\right)\right|
\leq \left|\E (h_{x, y}(\xi_{n_1})h_{x, y}(\xi_{n_2}) ) \right|+ C (y-x)\gamma^{r}(n_1 -n_2),
\end{align*}
such that
\begin{align} \label{eq:h-thirdsummand}
&
\E \left(\widetilde{h}^2_{x, y}(\xi_{n_1})\widetilde{h}^2_{x, y}(\xi_{n_2})\right)
\leq
\left|\E ( h_{x, y}(\xi_{n_1})h_{x, y}(\xi_{n_2}) ) \right|+
C (y-x)\gamma^{r}(n_1 -n_2).
\end{align}
Combining \eqref{eq:4_summands}, \eqref{eq:h-firstsummand}, \eqref{eq:h-secondsummand} and \eqref{eq:h-thirdsummand} finally gives
\begin{align*}
\E \left(h^2_{x, y}(\xi_{n_1})h^2_{x, y}(\xi_{n_2})\right)
&\leq
C\left|\E h_{x, y}(\xi_{n_1})h_{x, y}(\xi_{n_2})\right|
\\& \hspace{1cm}+C(y-x)^2+C(y-x)^{\frac{3}{2}} + C (y-x)\gamma^{r}(n_1 -n_2).
\end{align*}
Lemma 4.5 in \cite{taqqu1977law} yields
\begin{align*}
\sum\limits{}^{'} \left|\E \left(h_{x, y}(\xi_{n_1})h_{x, y}(\xi_{n_2})\right)\right|\leq C(j-i)(y-x).
\end{align*}
Since $\gamma(k)=k^{-D}L(k)$, 
\begin{align*}
  \sum\limits_{i+1\leq n_1\neq n_2 \leq j}
  \gamma^r(n_1-n_2)
  &=  \sum\limits_{n_1=1}^{j-i}\sum\limits_{n_2=1}^{j-i}
  \gamma^r(n_1-n_2)-(j-i)\gamma^r(0)\\
  &=2\sum\limits_{k=1}^{j-i-1}
 (j-i-k) \gamma^r(k)\\
 &\sim \frac{2}{1-rD} (j-i)^{2-rD}L^{r}(j-i); 
\end{align*}
see Proposition 2.2.1 in \cite{PipirasTaqqu}.
Therefore, for any $\epsilon>0$ there exists an $N_0\in \NN$ such that 
\begin{align*}
   \sum\limits_{i+1\leq n_1\neq n_2 \leq j}
  \gamma^r(n_1-n_2)\leq (1+\epsilon) \frac{2}{1-rD} (j-i)^{2-rD}L^{r}(j-i)
\end{align*}
for $j-i\geq N_0$.
As a consequence thereof, there exists a constant $C>0$ such that 
\begin{align*}
    \sum\limits_{i+1\leq n_1\neq n_2 \leq j}
  \gamma^r(n_1-n_2)\leq C (j-i)^{2-rD}L^{r}(j-i)
\end{align*}
for all $i, j\in \NN$.

Due to slow variation of $L$ for any $\eta>0$, there exists a $C>0$, such that
\begin{align*}
\E\left|\Sigma_{22}\right|
&\leq C\sum\limits{}^{'} \left|\E \left(h_{x, y}(\xi_{n_1})h_{x, y}(\xi_{n_2})\right)\right| + C(j-i)^2(y-x)^2
\\&\hspace{1cm}+C(j-i)^2(y-x)^\frac{3}{2}+ C(y-x) (j-i)^{2-rD+\eta}\\
&\leq C (j-i)^2(y-x)^{\frac{3}{2}}+C(j-i)(y-x) + C(j-i)^{2-rD+\eta}(y-x).
\end{align*}
Finally, we can use the bounds on $\E \Sigma_{1}$ to $\E \Sigma_{4}$ to continue bounding \eqref{eq:sum} as follows
\begin{align*}
\E \left|\s_j-\s_i\right|^4
&=\frac{1}{N^2} \E \left(\Sigma_1+4\Sigma_{21}+3\Sigma_{22}+6\Sigma_3+\Sigma_4\right) \\
&\leq
C \frac{1}{N^2} \Big( (j-i) (y-x) + (j-i)^{\frac{3}{2}} (y-x) 
\\&\hspace{1cm}+ (j-i)(y-x) + 
(j-i)^{2-rD+\eta}(y-x)
+(j-i)^2 (y-x)^{\frac{3}{2}} + (j-i)^2 (y-x)^2 \Big) \\
&\leq
C \frac{1}{N^2} ( (j-i)^{\frac{3}{2}} (y-x) + 
(j-i)^{2-rD+\eta}(y-x)
+
(j-i)^2 (y-x)^{\frac{3}{2}} )
\\&\leq
C \frac{1}{N^2} ( (j-i)^{2 - \theta } (y-x) 
+
(j-i)^2 (y-x)^{\frac{3}{2}} )
\end{align*}
with $\theta = \min\{ \frac{1}{2}, rD - \eta \}$. 
We obtain,
\begin{equation} \label{eq:bound-fourth-moment}
\begin{aligned}
\E \left|\s_j-\s_i\right|^4
&\leq C \frac{1}{N^{\theta}} \left( \frac{j-i}{N} \right)^{2-\theta} (y-x) + C \left( \frac{j-i}{N} \right)^2 (y-x)^{\frac{3}{2}}
\\&\leq 
C\left( \left(y-x\right)^{\frac{1}{2-\theta}} \sum_{q=i+1}^j \frac{1}{N^{\frac{2}{2-\theta}} } \right)^{2 - \theta }
+
C\left( \left(y-x\right)^{\frac{3}{4}} \sum_{q=i+1}^j \frac{1}{N} \right)^{2}.
\end{aligned}
\end{equation}
Applying Markov's inequality to the probability $\PP( |\s_j-\s_i | > \lambda) = \PP( |\s_j-\s_i |^4 > \lambda^4) $ and using the bound \eqref{eq:bound-fourth-moment}, the conditions of Lemma \ref{le:billmod12.2} are satisfied with $\gamma=4$, $\alpha_{1} = 2-\theta, \alpha_{2}=2$, $v_1 = (y-x)^{\frac{1}{2-\theta}} \frac{1}{N^{\frac{2}{2-\theta}}}$, and $v_2 = (y-x)^{\frac{3}{4}}\frac{1}{N}$.
Then, applying Lemma \ref{le:billmod12.2}  yields
\begin{equation*}
\begin{aligned}
\PP\left(\max_{1 \leq k \leq \lfloor Nb\rfloor} \left| \s_{k} \right| >\lambda \right)
&\leq 
C_{\alpha_{1},\gamma} \frac{1}{\lambda^4} \left( \left(y-x\right)^{\frac{1}{2-\theta}} \sum_{n=1}^{\lfloor Nb\rfloor} \frac{1}{N^{\frac{2}{2-\theta}}} \right)^{2-\theta}
+
C_{\alpha_{2},\gamma} \frac{1}{\lambda^4} \left( \left(y-x\right)^{\frac{3}{4}} \sum_{n=1}^{\lfloor Nb\rfloor} \frac{1}{N} \right)^{2}
\\&\leq
C_{\alpha_{1},\gamma} \frac{1}{N^{\theta}} \frac{1}{\lambda^4} b^{2-\theta} \left(y-x\right)
+
C_{\alpha_{2},\gamma} \frac{1}{\lambda^4} b^2 \left(y-x\right)^{\frac{3}{2}}.
\end{aligned}
\end{equation*}
\end{proof}

The subsequent lemmas are used in Appendix \ref{app:proofmain} and are all formulated in terms of a generic sequence of refining partitions which covers the two sequences of partitions \eqref{eq:refiningpartitions00} and \eqref{eq:refiningpartitions11} in Appendix \ref{app:proofmain}. 
For $k=0, \ldots, K_N$
define refining partitions 
\begin{align} \label{eq:refiningpartitions1}
x_i(k)\defeq \widetilde{a}_{p}+\frac{i}{2^k}2\delta, \hspace{0.2cm} i=0, \ldots, 2^k,
\end{align}
of the interval $[\widetilde{a}_p, \widetilde{a}_p+a]$
and for $x\in [0, a]$ choose $i_k(x)$ such that
\begin{align*} 
\widetilde{a}_{p}+x\in \left( x_{i_k(x)}(k), x_{i_k(x)+1}(k)\right].
\end{align*}
Note that for the partitions defined in \eqref{eq:refiningpartitions00} and \eqref{eq:refiningpartitions11}, we consider 
$\widetilde{a}_p=a_p$, $a=2\delta$ and $\widetilde{a}_p=0$, $a=1$, respectively.
All following lemmas in this section refer to these partitions.

\begin{lemma} \label{le:randterm}
Let $m_{N}$ be defined as in
\eqref{eq:mN:mainbody}
 and let $\sqrt{N}/2^{K_N}\to 0$. Then, for all $\lambda, b \in (0, 1]$ there is an $N_{\lambda}$ and a constant $C>0$ such that
\begin{equation*}
\begin{aligned}
&\PP \left( \sup_{x\in [0, a]}\sup_{t\in [0, b]} \left| m_N(x_{i_{K_N}(x)}(K_N), t)-m_N(x+c, t)\right| > \lambda \right)
\leq
\frac{C}{\lambda^4}b^{\frac{3}{2}} 
\max\left\{ \frac{1}{N^{\frac{1}{2}}} , 
\frac{N}{2^{K_N \frac{1}{2}}} \right\}
\end{aligned}
\end{equation*}
for all $N \geq N_{\lambda}$ and for all $c \geq 0$ such that $x+c\leq x_{i_{K_N}(x)+1}(K_N)$.
\end{lemma}

\begin{proof}
Note first that for all $x,y$,
\begin{align}
\left| m_N(y, t)-m_N(x, t)\right|
&=\left|\sum\limits_{l=\lceil \frac{1}{D}\rceil}^{\infty}\frac{c_l(y)-c_l(x)}{l!}\frac{1}{\sqrt{N}}\sum\limits_{n=1}^{\lfloor Nt\rfloor}H_l(\xi_n)\right|
\nonumber \\
&=\left|\frac{1}{\sqrt{N}}\sum\limits_{n=1}^{\lfloor Nt\rfloor} \1_{\{x<F(G(\xi_n))\leq y\}}-\sum\limits_{l=r}^{\lfloor \frac{1}{D}\rfloor}\frac{c_l(y)-c_l(x)}{l!}\frac{1}{\sqrt{N}}\sum\limits_{n=1}^{\lfloor Nt\rfloor}H_l(\xi_n)\right|
\nonumber \\
&\leq \frac{1}{\sqrt{N}}\sum\limits_{n=1}^{\lfloor Nt\rfloor}\1_{\{x<F(G(\xi_n))\leq y\}}+\left|\sum\limits_{l=r}^{\lfloor \frac{1}{D}\rfloor}\frac{c_l(y)-c_l(x)}{l!}\frac{1}{\sqrt{N}}\sum\limits_{n=1}^{\lfloor Nt\rfloor}H_l(\xi_n)\right|.
\label{eq:le:rt1}
\end{align}
Then,
\begin{align}
&\PP \left( \sup_{x\in [0, a]}\sup_{t\in [0, b]} \left| m_N(x+c, t)-m_N(x_{i_{K_N}(x)}(K_N), t)\right| > \lambda \right) \nonumber
\\&\leq
\PP \left( \sup_{x\in [0, a]}\sup_{t\in [0, b]} \left|\frac{1}{\sqrt{N}}\sum\limits_{n=1}^{\lfloor Nt\rfloor} \1_{\{x_{i_{K_N}(x)}(K_N)<F(G(\xi_n))\leq x+c\}} \right| > \frac{\lambda}{2} \right) \nonumber
\\&\hspace{1cm}+
\PP \left( \sup_{x\in [0, a]}\sup_{t\in [0, b]} \left|\sum\limits_{l=r}^{\lfloor \frac{1}{D}\rfloor}\frac{c_l(x+c)-c_l(x_{i_{K_N}(x)}(K_N))}{l!}\frac{1}{\sqrt{N}}\sum\limits_{n=1}^{\lfloor Nt\rfloor}H_l(\xi_n)\right| > \frac{\lambda}{2} \right) \label{eq:le:rt21}
\\&\leq
\PP \left( \sup_{x\in [0, a]}\sup_{t\in [0, b]} \left|\frac{1}{\sqrt{N}}\sum\limits_{n=1}^{\lfloor Nt\rfloor} \1_{\{x_{i_{K_N}(x)}(K_N)<F(G(\xi_n))\leq x+c\}} \right| > \frac{\lambda}{2} \right)
+
\frac{4}{D^2} \frac{Nb^2}{2^{K_N}\lambda^2}, \label{eq:le:rt22}
\end{align}
where \eqref{eq:le:rt21} follows by \eqref{eq:le:rt1} and \eqref{eq:le:rt22} is a consequence of applying Lemma \ref{le:randterm1} below. It remains to bound the probability in \eqref{eq:le:rt22}. Therefore, we write
\begin{align}
&\PP\pr{
\sup_{x\in [0, a]}\sup_{t\in [0, b]}\left|\frac{1}{\sqrt{N}}\sum\limits_{n=1}^{\lfloor Nt\rfloor} \1_{\left\{x_{i_{K_N}(x)}(K_N) < F(G(\xi_n))\leq x + c \right\}}\right|> \frac{ \lambda}{2}
} \nonumber
\\
&\leq \PP\pr{
\sup_{x\in [0, a]}\left|\frac{1}{\sqrt{N}}\sum\limits_{n=1}^{\lfloor Nb \rfloor} \1_{\left\{x_{i_{K_N}(x)}(K_N) < F(G(\xi_n))\leq x + c \right\}}\right|> \frac{ \lambda}{2}
}\nonumber\\
&\leq \PP\pr{
\sup_{x\in [0, a]}\left|\frac{1}{\sqrt{N}}\sum\limits_{n=1}^{\lfloor Nb \rfloor} \1_{\left\{ x_{i_{K_N}(x)}(K_N) < F(G(\xi_{n})) \leq x_{i_{K_N}(x)+1}(K_N) \right\}}\right|> \frac{ \lambda}{2}
}\nonumber\\
&\leq 
\PP \Bigg(
\sup_{x \in [0, a]}
\Bigg| \frac{1}{\sqrt{N}} \sum\limits_{n=1}^{\lfloor Nb \rfloor} \Bigg( \1_{ \left\{ x_{i_{K_N}(x)}(K_N) < F(G(\xi_n)) \leq x_{i_{K_N}(x)+1}(K_N) \right\} }
\nonumber
\\&\hspace{1cm} -
\left( x_{i_{K_N}(x)+1}(K_N)-x_{i_{K_N}(x)}(K_N) \right) \Bigg)\Bigg|
\nonumber
\\&\hspace{2cm} +
\sup_{x\in [0, a]}\left|\frac{1}{\sqrt{N}}\sum\limits_{n=1}^{\lfloor Nb \rfloor}\left(x_{i_{K_N}(x)+1}(K_N)-x_{i_{K_N}(x)}(K_N)\right)\right|
>\frac{ \lambda}{2}
\Bigg).
\label{eq:le:rt3}
\end{align}
The second summand in \eqref{eq:le:rt3} is deterministic and can be bounded by
\begin{align*}
\sup_{x\in [0, a]}\left|\frac{1}{\sqrt{N}}\sum\limits_{n=1}^{\lfloor Nb \rfloor}\left(x_{i_{K_N}(x)+1}(K_N)-x_{i_{K_N}(x)}(K_N)\right)\right|
\leq \frac{\sqrt{N}}{2^{K_N}}b\leq \frac{\sqrt{N}}{2^{K_N}}.
\end{align*}
Then, choose $N_{\lambda}\geq 1$ such that $\frac{\sqrt{N}}{2^{K_N}}
< \frac{\lambda}{4}$ for all $N \geq N_{\lambda}$.
As a result, we get
\begin{align}
&\PP\pr{
\sup_{x\in [0, a]}\sup_{t\in [0, b]}\left|\frac{1}{\sqrt{N}}\sum\limits_{n=1}^{\lfloor Nt\rfloor} \1_{\left\{x_{i_{K_N}(x)}(K_N) < F(G(\xi_n))\leq x + c \right\}}\right|> \frac{ \lambda}{2}
} \nonumber\\
\leq
&\PP \Bigg(
\sup_{x \in [0, a]}
\Bigg| \frac{1}{\sqrt{N}} \sum\limits_{n=1}^{\lfloor Nb \rfloor} \Bigg( \1_{ \left\{ x_{i_{K_N}(x)}(K_N) < F(G(\xi_n)) \leq x_{i_{K_N}(x)+1}(K_N) \right\} }
\nonumber
\\&\hspace{1cm} -
\left( x_{i_{K_N}(x)+1}(K_N)-x_{i_{K_N}(x)}(K_N) \right) \Bigg) \Bigg|>\frac{ \lambda}{2} - \frac{\sqrt{N}}{2^{K_N}}
\Bigg)
\notag
\\
&\leq 
\PP \Bigg(
\sup_{x \in [0, a]}
\Bigg| \frac{1}{\sqrt{N}} \sum\limits_{n=1}^{\lfloor Nb \rfloor} \Bigg( \1_{ \left\{ x_{i_{K_N}(x)}(K_N) < F(G(\xi_n)) \leq x_{i_{K_N}(x)+1}(K_N) \right\} }
\nonumber
\\&\hspace{1cm} -
\left( x_{i_{K_N}(x)+1}(K_N)-x_{i_{K_N}(x)}(K_N) \right) \Bigg) \Bigg|>\frac{ \lambda}{4} 
\Bigg)
\label{eq:le:rt3000}
\end{align}
for all $N \geq N_{\lambda}$.
Lemma \ref{le:indicatorfct} gives an upper bound on the probability in \eqref{eq:le:rt3000}:
\begin{align*}
\PP\pr{
\sup_{x\in [0, a]}\sup_{t\in [0, b]}\left|\frac{1}{\sqrt{N}}\sum\limits_{n=1}^{\lfloor Nt\rfloor} \1_{\left\{x_{i_{K_N}(x)}(K_N) < F(G(\xi_n))\leq x + c \right\}}\right|> \frac{ \lambda}{2}
} \leq \frac{C}{\lambda^4}b^{\frac{3}{2}} 
\max\left\{ \frac{1}{N^{\frac{1}{2}}} , 
\frac{N}{2^{K_N \frac{1}{2}}} \right\}.
\end{align*}
\end{proof}

\begin{lemma} \label{le:randterm1}
Let $c_{l}(\cdot)$ be defined as in 
\eqref{eq:mN:mainbody}. For all $\lambda>0$, it holds that 
\begin{equation*}
\PP\left(\sup\limits_{x\in [0, a]}\sup_{t\in [0, b]} \left| \sum\limits_{l=r}^{\lfloor \frac{1}{D}\rfloor}\frac{c_l(x+c)-c_l(x_{i_{K_N}(x)}(K_N))}{l!} \frac{1}{\sqrt{N}} \sum\limits_{n=1}^{\lfloor Nt\rfloor}H_l(\xi_n) \right|> \lambda \right)
\leq
\frac{1}{D^2} \frac{Nb^2}{2^{K_N}\lambda^2}
\end{equation*}
for $a,b>0$ and $c \geq 0$, such that $x+c\leq x_{i_{K_N}(x)+1}(K_N)$.
\end{lemma}

\begin{proof}
In order to bound the probability of interest, we use \eqref{eq:ccCS} in \eqref{eq:L31} below
\begin{align}
&\PP\left(\sup\limits_{x\in [0, a]}\sup_{t\in [0, b]} \left| \sum\limits_{l=r}^{\lfloor \frac{1}{D}\rfloor}\frac{c_l(x+c)-c_l(x_{i_{K_N}(x)}(K_N))}{l!} \frac{1}{\sqrt{N}} \sum\limits_{n=1}^{\lfloor Nt\rfloor}H_l(\xi_n) \right|> \lambda \right) \notag
\\&\leq
\PP\left(\sup\limits_{x\in [0, a]}\sup_{t\in [0, b]} \sum\limits_{l=r}^{\lfloor \frac{1}{D}\rfloor}
\frac{(x+c-x_{i_{K_N}(x)}(K_N))^{\frac{1}{2}}}{l!} \sqrt{l!} \left| \frac{1}{\sqrt{N}} \sum\limits_{n=1}^{\lfloor Nt\rfloor}H_l(\xi_n) \right|>\lambda \right) \label{eq:L31}
\\&\leq
\PP\left(\sup\limits_{x\in [0, a]}\sup_{t\in [0, b]} \sum\limits_{l=r}^{\lfloor \frac{1}{D}\rfloor}
\frac{(x_{i_{K_N}(x)+1}(K_N)-x_{i_{K_N}(x)}(K_N))^{\frac{1}{2}}}{\sqrt{l!}} \frac{1}{\sqrt{N}} \sum\limits_{n=1}^{\lfloor Nt\rfloor} \left| H_l(\xi_n) \right|>\lambda \right) \notag
\\&\leq
\PP\left( \sum\limits_{l=r}^{\lfloor \frac{1}{D}\rfloor}
\frac{1}{\sqrt{2^{K_N} l!}} \frac{1}{\sqrt{N}} \sum\limits_{n=1}^{\lfloor Nb\rfloor} \left| H_l(\xi_n) \right|>\lambda \right) 
\notag
\\&\leq
\frac{1}{2^{K_N}} \E \left( \sum\limits_{l=r}^{\lfloor \frac{1}{D}\rfloor} \frac{1}{\sqrt{l!}} \frac{1}{\sqrt{N}} \sum\limits_{n=1}^{\lfloor Nb\rfloor} | H_l(\xi_n) | \right)^2 \Big(\frac{1}{\lambda}\Big)^2 
\label{eq:L33}
\\&\leq
\frac{1}{2^{K_N}}  \sum\limits_{l_{1},l_2=r}^{\lfloor \frac{1}{D}\rfloor}
\frac{1}{\sqrt{l_{1}!l_{2}!}} \frac{1}{N} \sum\limits_{n_{1},n_{2}=1}^{\lfloor Nb \rfloor} 
\E \left( |H_{l_{1}}(\xi_{n_{1}}) | | H_{l_{2}}(\xi_{n_{2}}) | \right) \Big(\frac{1}{\lambda}\Big)^2
 \notag
\\&
\leq
\frac{1}{2^{K_N}}  \sum\limits_{l_1, l_2=r}^{\lfloor \frac{1}{D}\rfloor}
\frac{1}{N} \sum\limits_{n_{1},n_{2}=1}^{\lfloor Nb \rfloor} \Big(\frac{1}{\lambda}\Big)^2 
\label{eq:L3300} 
\\&\leq
\frac{1}{D^2} \frac{N b^2}{2^{K_N} \lambda^2},\notag
\end{align}
where \eqref{eq:L33} follows by Markov's inequality. We then used Cauchy-Schwarz inequality to get \eqref{eq:L3300}.
\end{proof}

\begin{lemma} \label{le:indicatorfct}
Let $F$ denote the marginal distribution function of $X_{n}, n \in \NN$. 
Then, there is a constant $C>0$ such that 
\begin{align*}
&\PP\Bigg(
\sup_{x\in [0, a]}
\Bigg|
\frac{1}{\sqrt{N}}\sum\limits_{n=1}^{\lfloor Nb \rfloor}
\Bigg( \1_{\left\{x_{i_{K_N}(x)}(K_N) < F(X_n)\leq x_{i_{K_N}(x)+1}(K_N) \right\}}
\\&\hspace{3cm}
-\left(x_{i_{K_N}(x)+1}(K_N)-x_{i_{K_N}(x)}(K_N) \right)\Bigg)\Bigg|>\lambda
\Bigg)
\leq 
\frac{C}{\lambda^4}b^{\frac{3}{2}} 
\max\left\{ \frac{1}{N^{\frac{1}{2}}} , 
\frac{N}{2^{K_N \frac{1}{2}}} \right\}
\end{align*}
for $a>0$ and $b, \lambda \in (0, 1]$.
\end{lemma}

\begin{proof}
With further explanations given below, we can infer the following bounds
\begin{align}
&\PP\Bigg(
\sup_{x\in [0, a]}
\Bigg|
\frac{1}{\sqrt{N}}\sum\limits_{n=1}^{\lfloor Nb \rfloor}
\Bigg(\1_{\left\{ x_{i_{K_N}(x)}(K_N) < F(X_n)\leq x_{i_{K_N}(x)+1}(K_N) \right\}} \notag
\\&\hspace{5cm}
-\left( x_{i_{K_N}(x)+1}(K_N)-x_{i_{K_N}(x)}(K_N) \right)\Bigg)
\Bigg|>\lambda
\Bigg) \notag
\\&\leq 
\PP\pr{
\sup_{x\in [0, a]}\left| \sum\limits_{l=r}^{\infty} \frac{c_l(x_{i_{K_N}(x)+1}(K_N))-c_l(x_{i_{K_N}(x)}(K_N))}{l!}\frac{1}{\sqrt{N}} \sum\limits_{n=1}^{\lfloor Nb\rfloor} H_l(\xi_n) \right|>\lambda
} \notag \\
&\leq \sum\limits_{i=0}^{2^{K_N}}
\PP\pr{
\left |m_N\pr{x_{i+1}(K_N), b}-m_N\pr{x_{i}(K_N),b}\right|>\frac{\lambda}{2}
} \notag
\\& \hspace{1cm}+
\PP\pr{
\sup_{x\in [0, a]}\left| \sum\limits_{l=r}^{\lfloor\frac{1}{D}\rfloor} \frac{c_l(x_{i_{K_N}(x)+1}(K_N))-c_l(x_{i_{K_N}(x)}(K_N))}{l!}\frac{1}{\sqrt{N}} \sum\limits_{n=1}^{\lfloor Nb\rfloor} H_l(\xi_n) \right|>\frac{\lambda}{2}
} \label{eq:App31}
\\&\leq
  \frac{16 C_1}{\lambda^4}b^{\frac{3}{2}} \sum\limits_{i=0}^{2^{K_N}} \frac{1}{N^{\frac{1}{2}}} \left(x_{i+1}(K_N)- x_i(K_N)\right)
 +
 \frac{16 C_2}{\lambda^4} b^2\sum\limits_{i=0}^{2^{K_N}}\left(x_{i+1}(K_N)- x_i(K_N)\right)^{\frac{3}{2}}
+
\frac{4}{D^2} \frac{N b^2}{2^{K_{N}}\lambda^{2}} \label{eq:App32}
\\&\leq
  \frac{16 C_1}{\lambda^4} b^{\frac{3}{2}}\sum\limits_{i=0}^{2^{K_N}}\left(\frac{1}{2^{K_N}}\right) \frac{1}{N^{\frac{1}{2}}} 
 +
 \frac{16 C_2}{\lambda^4} b^2\sum\limits_{i=0}^{2^{K_N}}\left(\frac{1}{2^{K_N}}\right)^{\frac{3}{2}}
+
\frac{4}{D^2} \frac{N b^2}{2^{K_{N}}\lambda^{2}}
 \notag
\\&\leq
\frac{C}{\lambda^4}b^{\frac{3}{2}} 
\max\left\{ \frac{1}{N^{\frac{1}{2}}} , 
\frac{N}{2^{K_N \frac{1}{2}}} \right\}
 \notag
 ,
\end{align}
where we used the representation \eqref{eq:mN:mainbody} in the first summand of \eqref{eq:App31}. The first probability in \eqref{eq:App32} can be bounded by Lemma \ref{le:bill} and the second one by Lemma \ref{le:randterm1}. We deduce the last inequality by using that $b, \lambda \in (0, 1]$.
\end{proof}

\section{A complementary result and its proof}\label{se:appendixc1}

In order to prove Lemma \ref{le:bill}, we use a slightly modified version of Theorem 12.2 in \cite{billingsley:1968}. We recall some notation from Chapter 12 in \cite{billingsley:1968}. Let $\xi_{1}, \dots, \xi_{N}$ be independent or identically distributed random variables and $\s_{k} = \sum_{j=1}^{k} \xi_{j} $ with $\s_{0} = 0$ and set
\begin{equation*}
M_{N} = \max_{0 \leq k \leq N} |\s_{k}|.
\end{equation*}

\begin{lemma} \label{le:billmod12.2}
Suppose there are $\gamma > 0$, $\alpha_{1},\alpha_{2} > 1$, $v_{1}, v_{2}> 0$ and a positive sequence $(u_{\ell})_{1 \leq \ell \leq N}$, such that for all $\lambda > 0$,
\begin{equation} \label{eq:billmod12.2}
\PP\left( \left| \s_{j} - \s_{i} \right| >\lambda \right)
\leq 
\frac{1}{\lambda^\gamma} \left(
\left( v_{1} \sum_{\ell = i+1}^{j} u_{\ell} \right)^{\alpha_{1}}
+
\left( v_{2} \sum_{\ell = i+1}^{j} u_{\ell} \right)^{\alpha_{2}}
\right),
\hspace{0.2cm} 
0 \leq i \leq j \leq N.
\end{equation}
Then, there are constants $C_{\alpha_{1},\gamma}, C_{\alpha_{2},\gamma} > 0$ only depending on $\alpha_{i}$, $i=1,2$ and $\gamma$, such that
\begin{equation*}
\PP\left( M_{N} > \lambda \right)
\leq 
\frac{C_{\alpha_{1},\gamma}}{\lambda^\gamma} 
\left( v_{1} \sum_{\ell = 1}^{N} u_{\ell} \right)^{\alpha_{1}}
+
\frac{C_{\alpha_{2},\gamma}}{\lambda^\gamma} 
\left( v_{2} \sum_{\ell = 1}^{N} u_{\ell} \right)^{\alpha_{2}}.
\end{equation*}
\end{lemma}

\begin{remark}
Note that Lemma \ref{le:billmod12.2} reduces to Theorem 12.2 in \cite{billingsley:1968} by setting $v_{2} = 0$. In particular, our statement remains true if either $v_{1}$ or $v_{2}$ are zero. Our proof below reveals that the generalization only works when the two summands in \eqref{eq:billmod12.2} depend on the same sequence $(u_{\ell})_{1 \leq \ell \leq N}$.
\end{remark}

\begin{proof}
The proof follows the proofs of Theorems 12.1 and 12.2 in \cite{billingsley:1968} and requires only slight modifications of the arguments. First, note that
\begin{equation} \label{eq:billmodeq1}
\PP(M_{N} > \lambda) \leq \PP(M'_{N} > \lambda/2) + \PP(\s_{N} > \lambda/2)
\end{equation}
with $M'_{N} = \max_{0 \leq i \leq N} \min \{ | \s_{i} |, | \s_{N} - \s_{i} | \}$.
We consider the two probabilities in \eqref{eq:billmodeq1} separately. Using assumption \eqref{eq:billmod12.2} with $j=N$ and $i=0$, the second probability can be bounded as
\begin{equation*}
\PP(\s_{N} > \lambda/2)
\leq
\frac{2^{\gamma}}{\lambda^\gamma} \left(
\left( v_{1} \sum_{\ell = 1}^{N} u_{\ell} \right)^{\alpha_{1}}
+
\left( v_{2} \sum_{\ell = 1}^{N} u_{\ell} \right)^{\alpha_{2}}
\right).
\end{equation*}
We prove a bound for the first probability in \eqref{eq:billmodeq1} via induction over $N$. Our induction hypothesis is, for $\mu = \lambda/2$, 
\begin{equation} \label{eq:inductivehypothesis}
\PP(M'_{N} > \mu)
\leq 
\frac{C_{\alpha_{1},\gamma}}{\mu^\gamma} 
\left( v_{1} \sum_{\ell = 1}^{N} u_{\ell} \right)^{\alpha_{1}}
+
\frac{C_{\alpha_{2},\gamma}}{\mu^\gamma} 
\left( v_{2} \sum_{\ell = 1}^{N} u_{\ell} \right)^{\alpha_{2}}.
\end{equation}
Beginning the induction with the base case $N=2$, we get
\begin{align}
\PP(M'_{2} > \mu) 
=
\PP(\min \{ | \s_{1} |, | \s_{2} - \s_{1} | \} > \mu) 
&\leq
\frac{1}{\mu^\gamma}
\left(
\left( v_{1} \sum_{\ell = 1}^{2} u_{\ell} \right)^{\alpha_{1}}
+
\left( v_{2} \sum_{\ell = 1}^{2} u_{\ell} \right)^{\alpha_{2}}
\right)
\label{eq:bprove1111}
\end{align}
by applying Lemma \ref{le:LemmaD2} with $i=0, j =1, k=2$ in \eqref{eq:bprove1111}.

For the inductive step, we assume that the induction hypothesis \eqref{eq:inductivehypothesis} is satisfied for all integers smaller and equal to $N-1$ and move towards $N$ during the inductive step.
Note that by equation (12.29) in \cite{billingsley:1968}, there is an $h$ such that
\begin{equation} \label{eq:definitionofhinequality}
 \sum_{\ell = 1}^{h-1} u_{\ell} \leq \frac{1}{2} \sum_{\ell = 1}^{N} u_{\ell} \leq \sum_{\ell = 1}^{h} u_{\ell},
\end{equation}
where the sum on the left is zero if $h=1$. By algebraic computations one can infer
\begin{equation} \label{eq:definitionofhinequality2}
 \sum_{\ell = h+1}^{N} u_{\ell} \leq \frac{1}{2} \sum_{\ell = 1}^{N} u_{\ell} \leq \sum_{\ell = h}^{N} u_{\ell}.
\end{equation}
To see this, note that
\begin{align}
  \frac{1}{2} \sum_{\ell = 1}^{N} u_{\ell} \leq \sum_{\ell = 1}^{h} u_{\ell}
&\Rightarrow
  \frac{1}{2} \sum_{\ell = 1}^{N} u_{\ell} + \sum_{\ell = h+1}^{N} u_{\ell} 
  \leq \sum_{\ell = 1}^{h} u_{\ell} + \sum_{\ell = h+1}^{N} u_{\ell}
\nonumber
\\&\Rightarrow
  \sum_{\ell = h+1}^{N} u_{\ell} 
  \leq \sum_{\ell = 1}^{N} u_{\ell} - \frac{1}{2} \sum_{\ell = 1}^{N} u_{\ell}
\nonumber
\\&\Rightarrow
  \sum_{\ell = h+1}^{N} u_{\ell} \leq \frac{1}{2} \sum_{\ell = 1}^{N} u_{\ell}.
\label{eq:algebra1}
\end{align}
Using the first inequality of \eqref{eq:definitionofhinequality}, we get
\begin{align}
  \sum_{\ell = 1}^{h-1} u_{\ell} \leq \frac{1}{2} \sum_{\ell = 1}^{N} u_{\ell}
&\Rightarrow
  \sum_{\ell = 1}^{h-1} u_{\ell} + \sum_{\ell = h}^{N} u_{\ell} \leq \frac{1}{2} \sum_{\ell = 1}^{N} u_{\ell} + \sum_{\ell = h}^{N} u_{\ell}
\nonumber
\\&\Rightarrow
  \sum_{\ell = 1}^{N} u_{\ell} - \frac{1}{2} \sum_{\ell = 1}^{N} u_{\ell} \leq \sum_{\ell = h}^{N} u_{\ell}
\nonumber
\\&\Rightarrow
  \frac{1}{2} \sum_{\ell = 1}^{N} u_{\ell} \leq \sum_{\ell = h}^{N} u_{\ell}.
\label{eq:algebra2}
\end{align}
Combining \eqref{eq:algebra1} and \eqref{eq:algebra2} we get the desired result.

By (12.36) in \cite{billingsley:1968},
\begin{equation*}
M'_{N} \leq \max\{ U_{1} + D_{1}, U_{2} + D_{2} \}
\end{equation*}
and therefore
\begin{equation} \label{eq:ywoprobssum}
\PP( M'_{N} > \mu ) \leq \PP( U_{1} + D_{1} > \mu) + \PP( U_{2} + D_{2} > \mu )
\end{equation}
with
\begin{align}
U_{1} &= \max_{0 \leq i \leq h-1} \min \{ | \s_{i} |, | \s_{h-1} - \s_{i} | \},
\hspace{0.2cm}
U_{2} = \max_{h \leq i \leq N} \min \{ | \s_{j} - \s_{h} |, | \s_{N} - \s_{j} | \},
\label{eq:U1andU2}
\\
D_{1} &= \min \{ | \s_{h-1} |, | \s_{N} - \s_{h-1} | \},
\hspace{0.2cm}
D_{2} = \min \{ | \s_{h} |, | \s_{N} - \s_{h} | \}.
\label{eq:D1andD2}
\end{align}
The tail probabilities of the random variables \eqref{eq:U1andU2} and \eqref{eq:D1andD2} can be bounded by using the inductive hypothesis \eqref{eq:inductivehypothesis} and Lemma \ref{le:LemmaD2}, respectively.
Exemplarily, we consider $U_{1}$ and $D_{1}$. For $U_{1}$, we get the following bounds
\begin{align}
\PP(U_{1} > \mu) 
&\leq 
\frac{C_{\alpha_{1},\gamma}}{\mu^\gamma} 
\left( v_{1} \sum_{\ell = 1}^{h-1} u_{\ell} \right)^{\alpha_{1}}
+
\frac{C_{\alpha_{2},\gamma}}{\mu^\gamma} 
\left( v_{2} \sum_{\ell = 1}^{h-1} u_{\ell} \right)^{\alpha_{2}}
\label{eq:popopopopo1}
\\&\leq 
\frac{C_{\alpha_{1},\gamma}}{\mu^\gamma} \frac{1}{2^{\alpha_{1}}}
\left( v_{1} \sum_{\ell = 1}^{N} u_{\ell} \right)^{\alpha_{1}}
+
\frac{C_{\alpha_{2},\gamma}}{\mu^\gamma} \frac{1}{2^{\alpha_{2}}}
\left( v_{2} \sum_{\ell = 1}^{N} u_{\ell} \right)^{\alpha_{2}}
\label{eq:popopopopo2}
\end{align}
by applying the inductive hypothesis \eqref{eq:inductivehypothesis} in \eqref{eq:popopopopo1} and the inequality
\eqref{eq:definitionofhinequality} in \eqref{eq:popopopopo2}. The tail probability of $U_{2}$ can be dealt with analogously by applying \eqref{eq:definitionofhinequality2}.
For $D_{1}$, we get
\begin{align}
\PP(D_{1} > \mu) &\leq 
\frac{1}{\mu^\gamma} \left(
\left( v_{1} \sum_{\ell = 1}^{N} u_{\ell} \right)^{\alpha_{1}}
+
\left( v_{2} \sum_{\ell = 1}^{N} u_{\ell} \right)^{\alpha_{2}}
\right)
\label{eq:popopopopo3}
\end{align}
by Lemma \ref{le:LemmaD2} with $i=0, j=h-1$ and $k = N$. The tail probability of $D_{2}$ can be handled analogously by applying Lemma \ref{le:LemmaD2} with $i=0, j=h$ and $k = N$.

We now continue with bounding \eqref{eq:ywoprobssum} with focus on the first summand since the second summand can be bounded by analogous arguments. With explanations given below, for some positive $\mu_{0},\mu_{1}$ with $\mu_{0}+\mu_{1}=\mu$,
\begin{align}
&
\PP( U_{1} + D_{1} > \mu)
\nonumber
\\&\leq 
\PP( U_{1} > \mu_{0} ) + \PP( D_{1} > \mu_{1} ) 
\nonumber
\\&\leq
\frac{C_{\alpha_{1},\gamma}}{\mu_{0}^\gamma} \frac{1}{2^{\alpha_{1}}}
\left( v_{1} \sum_{\ell = 1}^{N} u_{\ell} \right)^{\alpha_{1}}
+
\frac{C_{\alpha_{2},\gamma}}{\mu_{0}^\gamma} \frac{1}{2^{\alpha_{2}}}
\left( v_{2} \sum_{\ell = 1}^{N} u_{\ell} \right)^{\alpha_{2}}
+
\frac{1}{\mu_{1}^\gamma} \left(
\left( v_{1} \sum_{\ell = 1}^{N} u_{\ell} \right)^{\alpha_{1}}
+
\left( v_{2} \sum_{\ell = 1}^{N} u_{\ell} \right)^{\alpha_{2}}
\right)
\label{al:Bproveineq1}
\\&=
\frac{1}{\mu^\gamma} \left(
\left(
\frac{C_{\alpha_{1},\gamma}}{ 2^{\alpha_{1}}}
\left( v_{1} \sum_{\ell = 1}^{N} u_{\ell} \right)^{\alpha_{1}}
+
\frac{C_{\alpha_{2},\gamma}}{2^{\alpha_{2}}}
\left( v_{2} \sum_{\ell = 1}^{N} u_{\ell} \right)^{\alpha_{2}}
\right)^{\delta}
+
\left(
\left( v_{1} \sum_{\ell = 1}^{N} u_{\ell} \right)^{\alpha_{1}}
+
\left( v_{2} \sum_{\ell = 1}^{N} u_{\ell} \right)^{\alpha_{2}}
\right)^{\delta}
\right)^{\frac{1}{\delta}}
\label{al:Bproveineq2}
\\&\leq
\frac{1}{\mu^\gamma} \left(
\left(
C_{\alpha_{1},\gamma}
\left( v_{1} \sum_{\ell = 1}^{N} u_{\ell} \right)^{\alpha_{1}}
+
C_{\alpha_{2},\gamma}
\left( v_{2} \sum_{\ell = 1}^{N} u_{\ell} \right)^{\alpha_{2}}
\right)^{\delta}
\left[
\left(
\frac{1}{ 2^{\alpha_{1}}} + \frac{1}{ 2^{\alpha_{2}}}
\right)^{\delta}
+
\left(
\frac{1}{ C_{\alpha_{1},\gamma}} + \frac{1}{ C_{\alpha_{2},\gamma}}
\right)^{\delta}
\right]
\right)^{\frac{1}{\delta}}
\label{al:Bproveineq3}
\\&\leq
\frac{1}{\mu^\gamma} 
\left(
C_{\alpha_{1},\gamma}
\left( v_{1} \sum_{\ell = 1}^{N} u_{\ell} \right)^{\alpha_{1}}
+
C_{\alpha_{2},\gamma}
\left( v_{2} \sum_{\ell = 1}^{N} u_{\ell} \right)^{\alpha_{2}}
\right),
\label{al:Bproveineq4}
\end{align}
where \eqref{al:Bproveineq1} follows by \eqref{eq:popopopopo2} and \eqref{eq:popopopopo3}. For \eqref{al:Bproveineq2} we recall (12.39) in \cite{billingsley:1968}. It states that for positive numbers $A,B, \lambda$, 
\begin{equation*}
\min_{ \substack{ \lambda_{0},\lambda_{1} > 0 \\ \lambda_{0} + \lambda_{1} = \lambda } } \left( \frac{A}{\lambda_{0}^{\gamma}} + \frac{B}{\lambda_{1}^{\gamma}} \right)
=
\frac{1}{\lambda^{\gamma}} ( A^{\delta} + B^{\delta} )^{\frac{1}{\delta}}
\end{equation*}
with $\delta = \frac{1}{\gamma + 1}$.
The inequality \eqref{al:Bproveineq3} follows since
\begin{equation*}
\left( v_{1} \sum_{\ell = 1}^{N} u_{\ell} \right)^{\alpha_{1}}
+
\left( v_{2} \sum_{\ell = 1}^{N} u_{\ell} \right)^{\alpha_{2}}
\leq
\left(
C_{\alpha_{1},\gamma}
\left( v_{1} \sum_{\ell = 1}^{N} u_{\ell} \right)^{\alpha_{1}}
+
C_{\alpha_{2},\gamma}
\left( v_{2} \sum_{\ell = 1}^{N} u_{\ell} \right)^{\alpha_{2}}
\right)
\left(
\frac{1}{ C_{\alpha_{1},\gamma}} + \frac{1}{ C_{\alpha_{2},\gamma}}
\right).
\end{equation*}
Finally, we get \eqref{al:Bproveineq4} by choosing the constants $C_{\alpha_{1},\gamma}, C_{\alpha_{2},\gamma}$ large enough to get 
\begin{equation} \label{eq:iuiuiuiu}
\left[
\left(
\frac{1}{ 2^{\alpha_{1}}} + \frac{1}{ 2^{\alpha_{2}}}
\right)^{\delta}
+
\left(
\frac{1}{ C_{\alpha_{1},\gamma}} + \frac{1}{ C_{\alpha_{2},\gamma}}
\right)^{\delta}
\right]^{\frac{1}{\delta}} \leq 1
\end{equation}
which is possible since \eqref{eq:iuiuiuiu} is equivalent to
\begin{equation*}
\left(
\frac{1}{ 2^{\alpha_{1}}} + \frac{1}{ 2^{\alpha_{2}}}
\right)^{\delta}
+
\left(
\frac{1}{ C_{\alpha_{1},\gamma}} + \frac{1}{ C_{\alpha_{2},\gamma}}
\right)^{\delta}
\leq 1
\end{equation*}
and due to our assumption that $\alpha_{1},\alpha_{2}>1$.
\end{proof}

\begin{lemma} \label{le:LemmaD2}
Suppose there are $\gamma > 0$, $\alpha_{1},\alpha_{2} > 1$, \eqref{eq:billmod12.2} is satisfied with $v_{1}, v_{2}> 0$ and a positive sequence $(u_{\ell})_{1 \leq \ell \leq N}$. Then, for all $\lambda > 0$,
\begin{align*}
\PP\left( \left| \s_{j} - \s_{i} \right| >\lambda, \left| \s_{k} - \s_{j} \right| >\lambda \right)
\leq
\frac{1}{\lambda^\gamma}
\left(
\left( v_{1} \sum_{\ell = i+1}^{k} u_{\ell} \right)^{\alpha_{1}}
+
\left( v_{2} \sum_{\ell = i+1}^{k} u_{\ell} \right)^{\alpha_{2}}
\right),
\hspace{0.2cm} 
0 \leq i \leq j \leq k \leq N.
\end{align*}
\end{lemma}

\begin{proof}
We follow the arguments in the proof of Theorem 12.1 in \cite{billingsley:1968}. That is, 
\begin{align}
&
\PP\left( \left| \s_{j} - \s_{i} \right| >\lambda, \left| \s_{k} - \s_{j} \right| >\lambda \right) \nonumber
\\&\leq
\PP^{\frac{1}{2}}\left( \left| \s_{j} - \s_{i} \right| >\lambda \right) \PP^{\frac{1}{2}}\left( \left| \s_{k} - \s_{j} \right| >\lambda \right) \nonumber
\\&\leq
\frac{1}{\lambda^\gamma}
\left(
\left( v_{1} \sum_{\ell = i+1}^{j} u_{\ell} \right)^{\alpha_{1}}
+
\left( v_{2} \sum_{\ell = i+1}^{j} u_{\ell} \right)^{\alpha_{2}}
\right)^{\frac{1}{2}}
\left(
\left( v_{1} \sum_{\ell = j+1}^{k} u_{\ell} \right)^{\alpha_{1}}
+
\left( v_{2} \sum_{\ell = j+1}^{k} u_{\ell} \right)^{\alpha_{2}}
\right)^{\frac{1}{2}}
\label{eq:billmodeq000}
\\&=
\frac{1}{\lambda^\gamma}
\Bigg(
\left( v_{1} \sum_{\ell = i+1}^{j} u_{\ell} \right)^{\alpha_{1}}
\left( v_{1} \sum_{\ell = j+1}^{k} u_{\ell} \right)^{\alpha_{1}}
+
\left( v_{2} \sum_{\ell = i+1}^{j} u_{\ell} \right)^{\alpha_{2}}
\left( v_{1} \sum_{\ell = j+1}^{k} u_{\ell} \right)^{\alpha_{1}}
\nonumber
\\&\hspace{1cm}+
\left( v_{1} \sum_{\ell = i+1}^{j} u_{\ell} \right)^{\alpha_{1}}
\left( v_{2} \sum_{\ell = j+1}^{k} u_{\ell} \right)^{\alpha_{2}}
+
\left( v_{2} \sum_{\ell = i+1}^{j} u_{\ell} \right)^{\alpha_{2}}
\left( v_{2} \sum_{\ell = j+1}^{k} u_{\ell} \right)^{\alpha_{2}}
\Bigg)^{\frac{1}{2}}
\nonumber
\\&\leq
\frac{1}{\lambda^\gamma}
\Bigg(
\left( v_{1} \sum_{\ell = i+1}^{k} u_{\ell} \right)^{ 2 \alpha_{1}}
+
2
\left( v_{1} \sum_{\ell = i+1}^{k} u_{\ell} \right)^{\alpha_{1}}
\left( v_{2} \sum_{\ell = i+1}^{k} u_{\ell} \right)^{\alpha_{2}}
+
\left( v_{2} \sum_{\ell = i+1}^{k} u_{\ell} \right)^{ 2 \alpha_{2}}
\Bigg)^{\frac{1}{2}}
\label{eq:billmodeq2}
\\&\leq
\frac{1}{\lambda^\gamma}
\left(
\left( v_{1} \sum_{\ell = i+1}^{k} u_{\ell} \right)^{\alpha_{1}}
+
\left( v_{2} \sum_{\ell = i+1}^{k} u_{\ell} \right)^{\alpha_{2}}
\right),
\nonumber
\end{align}
where \eqref{eq:billmodeq000} is due to \eqref{eq:billmod12.2} and \eqref{eq:billmodeq2} follows since $xy \leq (x+ y)^2$ for $x,y >0$ and $\sum_{\ell = j+1}^{k} u_{\ell} \leq \sum_{\ell = i+1}^{k} u_{\ell}$.
\end{proof}

\section{Additional results and their proofs} \label{se:appendixC}
For shortness' sake we set $\widebar{L}_{N}(x) = \frac{1}{N}\sum\limits_{n=1}^{N}L_n(x)$ in this section.
\begin{lemma} \label{le:CIquantiles}
An approximate $1-\alpha$ confidence interval of $F^{-1}(p)$ can be written as
\begin{align*}
&
\Bigg(
\phi(F_N) + \frac{1}{ \varphi(F^{-1}(p)) } \Big( \widebar{L}_{N}(F^{-1}(p)) + \frac{1}{\sqrt{N}} \sigma(F^{-1}(p)) z_{1-\frac{\alpha}{2}} \Big)
\\&\hspace{1cm}, 
\phi(F_N) + \frac{1}{ \varphi(F^{-1}(p)) } \Big( \widebar{L}_{N}(F^{-1}(p)) + \frac{1}{\sqrt{N}} \sigma(F^{-1}(p)) z_{\frac{\alpha}{2}} \Big)
\Bigg).
\end{align*}
\end{lemma}

\begin{proof}
Note that
\begin{align}
\PP \left( c_{1} \leq \frac{N}{d_N}\left( \phi(F_N)-\phi(F) \right) \leq c_{2} \right)
&=
\PP \left( \frac{d_N}{N} c_{1} \leq \phi(F_N)-\phi(F) \leq \frac{d_N}{N} c_{2} \right) \notag
\\&=
\PP \left( \phi(F_N) - \frac{d_N}{N} c_{2} \leq \phi(F) \leq \phi(F_N) - \frac{d_N}{N} c_{1} \right).
\label{al:appyuo}
\end{align}
In order to find an approximate $1-\alpha$ confidence interval, one has to determine the critical values $c_{1}, c_{2}$ in \eqref{al:appyuo}. Instead of utilizing the asymptotic distribution of the empirical process, we consider the asymptotic behavior of the higher-order approximation of the empirical process.

Then, with explanations given below,
\begin{align}
&\PP \left( c_{1} \leq \frac{N}{d_N}\left( \phi(F_N)-\phi(F) \right) \leq c_{2} \right) \notag
\\&=
\PP \left( c_{1} \leq \phi'_{F}\left(\frac{N}{d_N} (F_N - F) \right) \leq c_{2} \right)
 + o(1)
 \label{al:opopopopop1}
\\&=
\PP \left( \frac{d_N}{N} c_{1} \leq \frac{-1}{\varphi(F^{-1}(p))} (F_{N}(F^{-1}(p)) - F(F^{-1}(p))) \leq \frac{d_N}{N} c_{2} \right)
 + o(1)
\label{al:opopopopop2}
\\&=
\PP \left( -\varphi(F^{-1}(p)) \frac{d_N}{N} c_{2} \leq F_{N}(F^{-1}(p)) - F(F^{-1}(p)) \leq -\varphi(F^{-1}(p)) \frac{d_N}{N} c_{1} \right)
 + o(1)\notag
\\&=
\PP \Big( -\varphi(F^{-1}(p)) \frac{d_N}{\sqrt{N}} c_{2} - \sqrt{N}\widebar{L}_{N}(F^{-1}(p))
\notag
\\& \hspace{2cm} 
\leq \sqrt{N}\widebar{S}_{N}(F^{-1}(p)) \leq 
-\varphi(F^{-1}(p)) \frac{d_N}{\sqrt{N}} c_{1} - \sqrt{N}\widebar{L}_{N}(F^{-1}(p)) \Big)
 + o(1) \notag
\\&=
\PP \left( \sigma(F^{-1}(p)) z_{1-\frac{\alpha}{2}}
\leq S(F^{-1}(p),1) \leq 
\sigma(F^{-1}(p)) z_{\frac{\alpha}{2}} \right) + o(1), \label{al:opopopopop4}
\end{align}
where \eqref{al:opopopopop1} is due to the Taylor approximation 
\eqref{eq:Taylorapprox}, \eqref{al:opopopopop2} follows by the relation \eqref{eq:derivativefunctional} and \eqref{al:opopopopop4} is due to the asymptotic result in Theorem \ref{th:weakconvSRD} where $S(x,1)$ is a mean zero Gaussian process with cross-covariances 
$\sigma^2(x) = \sum_{n \in \ZZ} \Cov(S_{0}(x),S_{n}(x))$ given in \eqref{eq:covS}.

Based on the last approximation \eqref{al:opopopopop4}, we can infer the following relation between $c_{1}, c_{2}$ and the quantiles of the normal distribution
\begin{equation} \label{eq:relationcz}
\begin{gathered}
\sigma(F^{-1}(p)) z_{1-\frac{\alpha}{2}} = -\varphi(F^{-1}(p)) \frac{d_N}{\sqrt{N}} c_{1} - \sqrt{N}\widebar{L}_{N}(F^{-1}(p)),
\\
\sigma(F^{-1}(p)) z_{\frac{\alpha}{2}} = -\varphi(F^{-1}(p)) \frac{d_N}{\sqrt{N}} c_{2} - \sqrt{N}\widebar{L}_{N}(F^{-1}(p)).
\end{gathered}
\end{equation}
Combining \eqref{al:appyuo} and \eqref{eq:relationcz}, we can then infer the statement of the lemma since for example
\begin{align*}
&\phi(F_N) - \frac{d_N}{N} c_{1}
\\&=
\phi(F_N) - \frac{1}{ \varphi(F^{-1}(p)) } \frac{d_N}{N} \frac{\sqrt{N}}{d_N} \Big( -\sqrt{N}\widebar{L}_{N}(F^{-1}(p)) - \sigma(F^{-1}(p)) z_{1-\frac{\alpha}{2}} \Big)
\\&=
\phi(F_N) + \frac{1}{ \varphi(F^{-1}(p)) } \Big( \widebar{L}_{N}(F^{-1}(p)) + \frac{1}{\sqrt{N}} \sigma(F^{-1}(p)) z_{1-\frac{\alpha}{2}} \Big).
\end{align*}
\end{proof}

For the following lemma, recall from Section \ref{se:HOA} that $\widetilde{c}_l(x)=\E\pr{\1_{\{G(\xi_0)\leq x\}}H_l(\xi_0)}$.
\begin{lemma} \label{le:hermite_coefficient}
Suppose $G: \RR \to \RR$ is a monotonically increasing (decreasing), bijective function. Then,
\begin{align*}
\widetilde{c}_l(x)=
\begin{cases}
-H_{l-1}(G^{-1}(x))\varphi\left(G^{-1}(x)\right) \ &\text{if $G$ is increasing,}\\
H_{l-1}(G^{-1}(x))\varphi\left(G^{-1}(x)\right) \ &\text{if $G$ is decreasing.}
\end{cases}
\end{align*}
\end{lemma}

\begin{proof}
For a strictly monotonically increasing, bijective function $G$, it holds that
\begin{align*}
\widetilde{c}_l(x)=&\E\pr{\1_{\{G(\xi_0)\leq x\}}H_l(\xi_0)}\\
=&\int_{\RR} \1_{\{G(y)\leq x\}}H_l(y)\varphi(y)dy\\
=&\int_{\RR} \1_{\{y\leq G^{-1}(x)\}}H_l(y)\varphi(y)dy\\
=&\int_{-\infty}^{ G^{-1}(x)}H_l(y)\varphi(y)dy
=-H_{l-1}(G^{-1}(x))\varphi(G^{-1}(x)), 
\end{align*}
where the last equality follows from 
the definition of the Hermite polynomial $H_l$
as
\begin{align*}
H_l(x)=(-1)^l \frac{1}{\varphi(x)} \frac{\partial^l}{\partial x ^l}\varphi(x);
\end{align*}
see formula (4.1.1) in \cite{PipirasTaqqu}.
Analogously, it follows that 
for a strictly monotonically decreasing, bijective function $G$, it holds that
\begin{align*}
\widetilde{c}_l(x)=&\E\pr{\1_{\{G(\xi_0)\leq x\}}H_l(\xi_0)}\\
=&\int_{\RR} \1_{\{G(y)\leq x\}}H_l(y)\varphi(y)dy\\
=&\int_{\RR} \1_{\{y\geq G^{-1}(x)\}}H_l(y)\varphi(y)dy\\
=&\int_{ G^{-1}(x)}^{\infty}H_l(y)\varphi(y)dy
=H_{l-1}(G^{-1}(x))\varphi(G^{-1}(x)).
\end{align*}
\end{proof}

\bigskip
\noindent 

\section*{Acknowledgments}
We thank the editor, co-editor, and referees for their careful reading of the manuscript and their thoughtful comments, which led to a significant improvement of the article.
Annika Betken gratefully acknowledges financial support from the Dutch Research Council (NWO) through VENI grant 212.164.
Marie-Christine D\"uker gratefully acknowledges financial support from the National Science Foundation under grants 1934985, 1940124, 1940276, and 2114143. 
This research was conducted with support from the Cornell University Center for Advanced Computing, which receives funding from Cornell University, the National Science Foundation, and members of its Partner Program.

\bigskip

\bibliographystyle{apalike}
\bibliography{empprocess}

\end{document}

%% file: plots/empdistr_histograms.tex
\begin{tikzpicture}[x=1pt,y=1pt]
\definecolor{fillColor}{RGB}{255,255,255}
\path[use as bounding box,fill=fillColor,fill opacity=0.00] (0,0) rectangle (469.75,216.81);
\begin{scope}
\path[clip] (  0.00,  0.00) rectangle (469.75,216.81);
\definecolor{fillColor}{RGB}{255,255,255}

\path[fill=fillColor] (164.62,191.36) rectangle (305.14,216.81);
\end{scope}
\begin{scope}
\path[clip] (  0.00,  0.00) rectangle (469.75,216.81);
\definecolor{drawColor}{RGB}{0,0,0}

\node[text=drawColor,anchor=base west,inner sep=0pt, outer sep=0pt, scale=  1.10] at (170.12,200.30) {$m$};
\end{scope}
\begin{scope}
\path[clip] (  0.00,  0.00) rectangle (469.75,216.81);
\definecolor{fillColor}{gray}{0.95}

\path[fill=fillColor] (184.78,196.86) rectangle (199.24,211.31);
\end{scope}
\begin{scope}
\path[clip] (  0.00,  0.00) rectangle (469.75,216.81);
\definecolor{drawColor}{RGB}{146,129,188}
\definecolor{fillColor}{RGB}{146,129,188}

\path[draw=drawColor,line width= 0.6pt,line cap=rect,fill=fillColor,fill opacity=0.50] (185.50,197.57) rectangle (198.53,210.60);
\end{scope}
\begin{scope}
\path[clip] (  0.00,  0.00) rectangle (469.75,216.81);
\definecolor{fillColor}{gray}{0.95}

\path[fill=fillColor] (223.43,196.86) rectangle (237.89,211.31);
\end{scope}
\begin{scope}
\path[clip] (  0.00,  0.00) rectangle (469.75,216.81);
\definecolor{drawColor}{RGB}{230,159,0}
\definecolor{fillColor}{RGB}{230,159,0}

\path[draw=drawColor,line width= 0.6pt,line cap=rect,fill=fillColor,fill opacity=0.50] (224.15,197.57) rectangle (237.18,210.60);
\end{scope}
\begin{scope}
\path[clip] (  0.00,  0.00) rectangle (469.75,216.81);
\definecolor{fillColor}{gray}{0.95}

\path[fill=fillColor] (262.09,196.86) rectangle (276.54,211.31);
\end{scope}
\begin{scope}
\path[clip] (  0.00,  0.00) rectangle (469.75,216.81);
\definecolor{drawColor}{RGB}{75,156,211}
\definecolor{fillColor}{RGB}{75,156,211}

\path[draw=drawColor,line width= 0.6pt,line cap=rect,fill=fillColor,fill opacity=0.50] (262.80,197.57) rectangle (275.83,210.60);
\end{scope}
\begin{scope}
\path[clip] (  0.00,  0.00) rectangle (469.75,216.81);
\definecolor{drawColor}{RGB}{0,0,0}

\node[text=drawColor,anchor=base west,inner sep=0pt, outer sep=0pt, scale=  0.88] at (204.74,201.05) {100};
\end{scope}
\begin{scope}
\path[clip] (  0.00,  0.00) rectangle (469.75,216.81);
\definecolor{drawColor}{RGB}{0,0,0}

\node[text=drawColor,anchor=base west,inner sep=0pt, outer sep=0pt, scale=  0.88] at (243.39,201.05) {200};
\end{scope}
\begin{scope}
\path[clip] (  0.00,  0.00) rectangle (469.75,216.81);
\definecolor{drawColor}{RGB}{0,0,0}

\node[text=drawColor,anchor=base west,inner sep=0pt, outer sep=0pt, scale=  0.88] at (282.04,201.05) {1000};
\end{scope}
\begin{scope}
\path[clip] (  0.00,  0.00) rectangle (234.88,191.36);
\definecolor{drawColor}{RGB}{255,255,255}
\definecolor{fillColor}{RGB}{255,255,255}

\path[draw=drawColor,line width= 0.6pt,line join=round,line cap=round,fill=fillColor] (  0.00,  0.00) rectangle (234.88,191.36);
\end{scope}
\begin{scope}
\path[clip] ( 34.45, 31.25) rectangle (229.38,168.39);
\definecolor{fillColor}{gray}{0.92}

\path[fill=fillColor] ( 34.45, 31.25) rectangle (229.38,168.39);
\definecolor{drawColor}{RGB}{255,255,255}

\path[draw=drawColor,line width= 0.3pt,line join=round] ( 34.45, 49.81) --
	(229.38, 49.81);

\path[draw=drawColor,line width= 0.3pt,line join=round] ( 34.45, 74.46) --
	(229.38, 74.46);

\path[draw=drawColor,line width= 0.3pt,line join=round] ( 34.45, 99.12) --
	(229.38, 99.12);

\path[draw=drawColor,line width= 0.3pt,line join=round] ( 34.45,123.77) --
	(229.38,123.77);

\path[draw=drawColor,line width= 0.3pt,line join=round] ( 34.45,148.42) --
	(229.38,148.42);

\path[draw=drawColor,line width= 0.3pt,line join=round] ( 63.49, 31.25) --
	( 63.49,168.39);

\path[draw=drawColor,line width= 0.3pt,line join=round] (107.14, 31.25) --
	(107.14,168.39);

\path[draw=drawColor,line width= 0.3pt,line join=round] (150.78, 31.25) --
	(150.78,168.39);

\path[draw=drawColor,line width= 0.3pt,line join=round] (194.42, 31.25) --
	(194.42,168.39);

\path[draw=drawColor,line width= 0.6pt,line join=round] ( 34.45, 37.49) --
	(229.38, 37.49);

\path[draw=drawColor,line width= 0.6pt,line join=round] ( 34.45, 62.14) --
	(229.38, 62.14);

\path[draw=drawColor,line width= 0.6pt,line join=round] ( 34.45, 86.79) --
	(229.38, 86.79);

\path[draw=drawColor,line width= 0.6pt,line join=round] ( 34.45,111.44) --
	(229.38,111.44);

\path[draw=drawColor,line width= 0.6pt,line join=round] ( 34.45,136.09) --
	(229.38,136.09);

\path[draw=drawColor,line width= 0.6pt,line join=round] ( 34.45,160.75) --
	(229.38,160.75);

\path[draw=drawColor,line width= 0.6pt,line join=round] ( 41.67, 31.25) --
	( 41.67,168.39);

\path[draw=drawColor,line width= 0.6pt,line join=round] ( 85.32, 31.25) --
	( 85.32,168.39);

\path[draw=drawColor,line width= 0.6pt,line join=round] (128.96, 31.25) --
	(128.96,168.39);

\path[draw=drawColor,line width= 0.6pt,line join=round] (172.60, 31.25) --
	(172.60,168.39);

\path[draw=drawColor,line width= 0.6pt,line join=round] (216.25, 31.25) --
	(216.25,168.39);
\definecolor{drawColor}{RGB}{146,129,188}
\definecolor{fillColor}{RGB}{146,129,188}

\path[draw=drawColor,line width= 0.6pt,line cap=rect,fill=fillColor,fill opacity=0.50] ( 43.31, 37.49) rectangle ( 49.21, 37.49);

\path[draw=drawColor,line width= 0.6pt,line cap=rect,fill=fillColor,fill opacity=0.50] ( 49.21, 37.49) rectangle ( 55.12, 37.49);

\path[draw=drawColor,line width= 0.6pt,line cap=rect,fill=fillColor,fill opacity=0.50] ( 55.12, 37.49) rectangle ( 61.03, 37.58);

\path[draw=drawColor,line width= 0.6pt,line cap=rect,fill=fillColor,fill opacity=0.50] ( 61.03, 37.49) rectangle ( 66.94, 39.58);

\path[draw=drawColor,line width= 0.6pt,line cap=rect,fill=fillColor,fill opacity=0.50] ( 66.94, 37.49) rectangle ( 72.84, 40.95);

\path[draw=drawColor,line width= 0.6pt,line cap=rect,fill=fillColor,fill opacity=0.50] ( 72.84, 37.49) rectangle ( 78.75, 40.86);

\path[draw=drawColor,line width= 0.6pt,line cap=rect,fill=fillColor,fill opacity=0.50] ( 78.75, 37.49) rectangle ( 84.66, 50.51);

\path[draw=drawColor,line width= 0.6pt,line cap=rect,fill=fillColor,fill opacity=0.50] ( 84.66, 37.49) rectangle ( 90.56, 47.69);

\path[draw=drawColor,line width= 0.6pt,line cap=rect,fill=fillColor,fill opacity=0.50] ( 90.56, 37.49) rectangle ( 96.47, 71.64);

\path[draw=drawColor,line width= 0.6pt,line cap=rect,fill=fillColor,fill opacity=0.50] ( 96.47, 37.49) rectangle (102.38, 61.26);

\path[draw=drawColor,line width= 0.6pt,line cap=rect,fill=fillColor,fill opacity=0.50] (102.38, 37.49) rectangle (108.28,104.88);

\path[draw=drawColor,line width= 0.6pt,line cap=rect,fill=fillColor,fill opacity=0.50] (108.28, 37.49) rectangle (114.19,134.75);

\path[draw=drawColor,line width= 0.6pt,line cap=rect,fill=fillColor,fill opacity=0.50] (114.19, 37.49) rectangle (120.10, 88.85);

\path[draw=drawColor,line width= 0.6pt,line cap=rect,fill=fillColor,fill opacity=0.50] (120.10, 37.49) rectangle (126.01,154.15);

\path[draw=drawColor,line width= 0.6pt,line cap=rect,fill=fillColor,fill opacity=0.50] (126.01, 37.49) rectangle (131.91, 98.59);

\path[draw=drawColor,line width= 0.6pt,line cap=rect,fill=fillColor,fill opacity=0.50] (131.91, 37.49) rectangle (137.82,159.34);

\path[draw=drawColor,line width= 0.6pt,line cap=rect,fill=fillColor,fill opacity=0.50] (137.82, 37.49) rectangle (143.73, 90.03);

\path[draw=drawColor,line width= 0.6pt,line cap=rect,fill=fillColor,fill opacity=0.50] (143.73, 37.49) rectangle (149.63,131.56);

\path[draw=drawColor,line width= 0.6pt,line cap=rect,fill=fillColor,fill opacity=0.50] (149.63, 37.49) rectangle (155.54,108.79);

\path[draw=drawColor,line width= 0.6pt,line cap=rect,fill=fillColor,fill opacity=0.50] (155.54, 37.49) rectangle (161.45, 61.98);

\path[draw=drawColor,line width= 0.6pt,line cap=rect,fill=fillColor,fill opacity=0.50] (161.45, 37.49) rectangle (167.35, 69.54);

\path[draw=drawColor,line width= 0.6pt,line cap=rect,fill=fillColor,fill opacity=0.50] (167.35, 37.49) rectangle (173.26, 48.23);

\path[draw=drawColor,line width= 0.6pt,line cap=rect,fill=fillColor,fill opacity=0.50] (173.26, 37.49) rectangle (179.17, 48.42);

\path[draw=drawColor,line width= 0.6pt,line cap=rect,fill=fillColor,fill opacity=0.50] (179.17, 37.49) rectangle (185.08, 40.22);

\path[draw=drawColor,line width= 0.6pt,line cap=rect,fill=fillColor,fill opacity=0.50] (185.08, 37.49) rectangle (190.98, 41.22);

\path[draw=drawColor,line width= 0.6pt,line cap=rect,fill=fillColor,fill opacity=0.50] (190.98, 37.49) rectangle (196.89, 39.22);

\path[draw=drawColor,line width= 0.6pt,line cap=rect,fill=fillColor,fill opacity=0.50] (196.89, 37.49) rectangle (202.80, 37.85);

\path[draw=drawColor,line width= 0.6pt,line cap=rect,fill=fillColor,fill opacity=0.50] (202.80, 37.49) rectangle (208.70, 37.58);

\path[draw=drawColor,line width= 0.6pt,line cap=rect,fill=fillColor,fill opacity=0.50] (208.70, 37.49) rectangle (214.61, 37.58);

\path[draw=drawColor,line width= 0.6pt,line cap=rect,fill=fillColor,fill opacity=0.50] (214.61, 37.49) rectangle (220.52, 37.49);
\definecolor{drawColor}{RGB}{75,156,211}
\definecolor{fillColor}{RGB}{75,156,211}

\path[draw=drawColor,line width= 0.6pt,line cap=rect,fill=fillColor,fill opacity=0.50] ( 43.31, 37.49) rectangle ( 49.21, 37.67);

\path[draw=drawColor,line width= 0.6pt,line cap=rect,fill=fillColor,fill opacity=0.50] ( 49.21, 37.49) rectangle ( 55.12, 37.58);

\path[draw=drawColor,line width= 0.6pt,line cap=rect,fill=fillColor,fill opacity=0.50] ( 55.12, 37.49) rectangle ( 61.03, 38.03);

\path[draw=drawColor,line width= 0.6pt,line cap=rect,fill=fillColor,fill opacity=0.50] ( 61.03, 37.49) rectangle ( 66.94, 38.49);

\path[draw=drawColor,line width= 0.6pt,line cap=rect,fill=fillColor,fill opacity=0.50] ( 66.94, 37.49) rectangle ( 72.84, 39.95);

\path[draw=drawColor,line width= 0.6pt,line cap=rect,fill=fillColor,fill opacity=0.50] ( 72.84, 37.49) rectangle ( 78.75, 40.67);

\path[draw=drawColor,line width= 0.6pt,line cap=rect,fill=fillColor,fill opacity=0.50] ( 78.75, 37.49) rectangle ( 84.66, 50.78);

\path[draw=drawColor,line width= 0.6pt,line cap=rect,fill=fillColor,fill opacity=0.50] ( 84.66, 37.49) rectangle ( 90.56, 51.60);

\path[draw=drawColor,line width= 0.6pt,line cap=rect,fill=fillColor,fill opacity=0.50] ( 90.56, 37.49) rectangle ( 96.47, 66.72);

\path[draw=drawColor,line width= 0.6pt,line cap=rect,fill=fillColor,fill opacity=0.50] ( 96.47, 37.49) rectangle (102.38, 73.55);

\path[draw=drawColor,line width= 0.6pt,line cap=rect,fill=fillColor,fill opacity=0.50] (102.38, 37.49) rectangle (108.28, 97.77);

\path[draw=drawColor,line width= 0.6pt,line cap=rect,fill=fillColor,fill opacity=0.50] (108.28, 37.49) rectangle (114.19, 97.41);

\path[draw=drawColor,line width= 0.6pt,line cap=rect,fill=fillColor,fill opacity=0.50] (114.19, 37.49) rectangle (120.10,127.01);

\path[draw=drawColor,line width= 0.6pt,line cap=rect,fill=fillColor,fill opacity=0.50] (120.10, 37.49) rectangle (126.01,138.21);

\path[draw=drawColor,line width= 0.6pt,line cap=rect,fill=fillColor,fill opacity=0.50] (126.01, 37.49) rectangle (131.91,127.83);

\path[draw=drawColor,line width= 0.6pt,line cap=rect,fill=fillColor,fill opacity=0.50] (131.91, 37.49) rectangle (137.82,139.48);

\path[draw=drawColor,line width= 0.6pt,line cap=rect,fill=fillColor,fill opacity=0.50] (137.82, 37.49) rectangle (143.73,112.53);

\path[draw=drawColor,line width= 0.6pt,line cap=rect,fill=fillColor,fill opacity=0.50] (143.73, 37.49) rectangle (149.63,115.72);

\path[draw=drawColor,line width= 0.6pt,line cap=rect,fill=fillColor,fill opacity=0.50] (149.63, 37.49) rectangle (155.54, 85.39);

\path[draw=drawColor,line width= 0.6pt,line cap=rect,fill=fillColor,fill opacity=0.50] (155.54, 37.49) rectangle (161.45, 83.39);

\path[draw=drawColor,line width= 0.6pt,line cap=rect,fill=fillColor,fill opacity=0.50] (161.45, 37.49) rectangle (167.35, 60.71);

\path[draw=drawColor,line width= 0.6pt,line cap=rect,fill=fillColor,fill opacity=0.50] (167.35, 37.49) rectangle (173.26, 56.52);

\path[draw=drawColor,line width= 0.6pt,line cap=rect,fill=fillColor,fill opacity=0.50] (173.26, 37.49) rectangle (179.17, 47.05);

\path[draw=drawColor,line width= 0.6pt,line cap=rect,fill=fillColor,fill opacity=0.50] (179.17, 37.49) rectangle (185.08, 41.68);

\path[draw=drawColor,line width= 0.6pt,line cap=rect,fill=fillColor,fill opacity=0.50] (185.08, 37.49) rectangle (190.98, 40.58);

\path[draw=drawColor,line width= 0.6pt,line cap=rect,fill=fillColor,fill opacity=0.50] (190.98, 37.49) rectangle (196.89, 38.58);

\path[draw=drawColor,line width= 0.6pt,line cap=rect,fill=fillColor,fill opacity=0.50] (196.89, 37.49) rectangle (202.80, 37.67);

\path[draw=drawColor,line width= 0.6pt,line cap=rect,fill=fillColor,fill opacity=0.50] (202.80, 37.49) rectangle (208.70, 37.67);

\path[draw=drawColor,line width= 0.6pt,line cap=rect,fill=fillColor,fill opacity=0.50] (208.70, 37.49) rectangle (214.61, 37.49);

\path[draw=drawColor,line width= 0.6pt,line cap=rect,fill=fillColor,fill opacity=0.50] (214.61, 37.49) rectangle (220.52, 37.58);
\definecolor{drawColor}{RGB}{230,159,0}
\definecolor{fillColor}{RGB}{230,159,0}

\path[draw=drawColor,line width= 0.6pt,line cap=rect,fill=fillColor,fill opacity=0.50] ( 43.31, 37.49) rectangle ( 49.21, 37.58);

\path[draw=drawColor,line width= 0.6pt,line cap=rect,fill=fillColor,fill opacity=0.50] ( 49.21, 37.49) rectangle ( 55.12, 37.67);

\path[draw=drawColor,line width= 0.6pt,line cap=rect,fill=fillColor,fill opacity=0.50] ( 55.12, 37.49) rectangle ( 61.03, 37.94);

\path[draw=drawColor,line width= 0.6pt,line cap=rect,fill=fillColor,fill opacity=0.50] ( 61.03, 37.49) rectangle ( 66.94, 38.31);

\path[draw=drawColor,line width= 0.6pt,line cap=rect,fill=fillColor,fill opacity=0.50] ( 66.94, 37.49) rectangle ( 72.84, 40.67);

\path[draw=drawColor,line width= 0.6pt,line cap=rect,fill=fillColor,fill opacity=0.50] ( 72.84, 37.49) rectangle ( 78.75, 42.22);

\path[draw=drawColor,line width= 0.6pt,line cap=rect,fill=fillColor,fill opacity=0.50] ( 78.75, 37.49) rectangle ( 84.66, 45.32);

\path[draw=drawColor,line width= 0.6pt,line cap=rect,fill=fillColor,fill opacity=0.50] ( 84.66, 37.49) rectangle ( 90.56, 61.62);

\path[draw=drawColor,line width= 0.6pt,line cap=rect,fill=fillColor,fill opacity=0.50] ( 90.56, 37.49) rectangle ( 96.47, 61.35);

\path[draw=drawColor,line width= 0.6pt,line cap=rect,fill=fillColor,fill opacity=0.50] ( 96.47, 37.49) rectangle (102.38, 72.91);

\path[draw=drawColor,line width= 0.6pt,line cap=rect,fill=fillColor,fill opacity=0.50] (102.38, 37.49) rectangle (108.28, 81.47);

\path[draw=drawColor,line width= 0.6pt,line cap=rect,fill=fillColor,fill opacity=0.50] (108.28, 37.49) rectangle (114.19,127.55);

\path[draw=drawColor,line width= 0.6pt,line cap=rect,fill=fillColor,fill opacity=0.50] (114.19, 37.49) rectangle (120.10,109.07);

\path[draw=drawColor,line width= 0.6pt,line cap=rect,fill=fillColor,fill opacity=0.50] (120.10, 37.49) rectangle (126.01,119.18);

\path[draw=drawColor,line width= 0.6pt,line cap=rect,fill=fillColor,fill opacity=0.50] (126.01, 37.49) rectangle (131.91,162.16);

\path[draw=drawColor,line width= 0.6pt,line cap=rect,fill=fillColor,fill opacity=0.50] (131.91, 37.49) rectangle (137.82,123.27);

\path[draw=drawColor,line width= 0.6pt,line cap=rect,fill=fillColor,fill opacity=0.50] (137.82, 37.49) rectangle (143.73,115.26);

\path[draw=drawColor,line width= 0.6pt,line cap=rect,fill=fillColor,fill opacity=0.50] (143.73, 37.49) rectangle (149.63,127.37);

\path[draw=drawColor,line width= 0.6pt,line cap=rect,fill=fillColor,fill opacity=0.50] (149.63, 37.49) rectangle (155.54, 86.30);

\path[draw=drawColor,line width= 0.6pt,line cap=rect,fill=fillColor,fill opacity=0.50] (155.54, 37.49) rectangle (161.45, 71.64);

\path[draw=drawColor,line width= 0.6pt,line cap=rect,fill=fillColor,fill opacity=0.50] (161.45, 37.49) rectangle (167.35, 59.98);

\path[draw=drawColor,line width= 0.6pt,line cap=rect,fill=fillColor,fill opacity=0.50] (167.35, 37.49) rectangle (173.26, 58.34);

\path[draw=drawColor,line width= 0.6pt,line cap=rect,fill=fillColor,fill opacity=0.50] (173.26, 37.49) rectangle (179.17, 46.59);

\path[draw=drawColor,line width= 0.6pt,line cap=rect,fill=fillColor,fill opacity=0.50] (179.17, 37.49) rectangle (185.08, 42.22);

\path[draw=drawColor,line width= 0.6pt,line cap=rect,fill=fillColor,fill opacity=0.50] (185.08, 37.49) rectangle (190.98, 40.22);

\path[draw=drawColor,line width= 0.6pt,line cap=rect,fill=fillColor,fill opacity=0.50] (190.98, 37.49) rectangle (196.89, 38.31);

\path[draw=drawColor,line width= 0.6pt,line cap=rect,fill=fillColor,fill opacity=0.50] (196.89, 37.49) rectangle (202.80, 37.94);

\path[draw=drawColor,line width= 0.6pt,line cap=rect,fill=fillColor,fill opacity=0.50] (202.80, 37.49) rectangle (208.70, 37.67);

\path[draw=drawColor,line width= 0.6pt,line cap=rect,fill=fillColor,fill opacity=0.50] (208.70, 37.49) rectangle (214.61, 37.67);

\path[draw=drawColor,line width= 0.6pt,line cap=rect,fill=fillColor,fill opacity=0.50] (214.61, 37.49) rectangle (220.52, 37.49);
\definecolor{drawColor}{RGB}{255,0,0}

\path[draw=drawColor,line width= 0.6pt,line join=round] ( 47.17, 37.57) --
	( 48.88, 37.60) --
	( 50.60, 37.64) --
	( 52.31, 37.69) --
	( 54.02, 37.76) --
	( 55.73, 37.84) --
	( 57.45, 37.94) --
	( 59.16, 38.08) --
	( 60.87, 38.24) --
	( 62.59, 38.45) --
	( 64.30, 38.71) --
	( 66.01, 39.02) --
	( 67.73, 39.41) --
	( 69.44, 39.87) --
	( 71.15, 40.43) --
	( 72.87, 41.10) --
	( 74.58, 41.89) --
	( 76.29, 42.83) --
	( 78.00, 43.93) --
	( 79.72, 45.20) --
	( 81.43, 46.66) --
	( 83.14, 48.34) --
	( 84.86, 50.25) --
	( 86.57, 52.39) --
	( 88.28, 54.80) --
	( 90.00, 57.46) --
	( 91.71, 60.40) --
	( 93.42, 63.60) --
	( 95.13, 67.07) --
	( 96.85, 70.80) --
	( 98.56, 74.76) --
	(100.27, 78.94) --
	(101.99, 83.30) --
	(103.70, 87.82) --
	(105.41, 92.43) --
	(107.13, 97.11) --
	(108.84,101.78) --
	(110.55,106.39) --
	(112.26,110.88) --
	(113.98,115.19) --
	(115.69,119.24) --
	(117.40,122.97) --
	(119.12,126.32) --
	(120.83,129.24) --
	(122.54,131.67) --
	(124.26,133.58) --
	(125.97,134.92) --
	(127.68,135.67) --
	(129.40,135.81) --
	(131.11,135.36) --
	(132.82,134.31) --
	(134.53,132.68) --
	(136.25,130.50) --
	(137.96,127.81) --
	(139.67,124.67) --
	(141.39,121.11) --
	(143.10,117.21) --
	(144.81,113.02) --
	(146.53,108.61) --
	(148.24,104.05) --
	(149.95, 99.40) --
	(151.66, 94.72) --
	(153.38, 90.07) --
	(155.09, 85.50) --
	(156.80, 81.06) --
	(158.52, 76.78) --
	(160.23, 72.71) --
	(161.94, 68.87) --
	(163.66, 65.27) --
	(165.37, 61.93) --
	(167.08, 58.87) --
	(168.79, 56.07) --
	(170.51, 53.54) --
	(172.22, 51.27) --
	(173.93, 49.25) --
	(175.65, 47.46) --
	(177.36, 45.89) --
	(179.07, 44.53) --
	(180.79, 43.35) --
	(182.50, 42.33) --
	(184.21, 41.47) --
	(185.93, 40.74) --
	(187.64, 40.13) --
	(189.35, 39.62) --
	(191.06, 39.20) --
	(192.78, 38.85) --
	(194.49, 38.57) --
	(196.20, 38.34) --
	(197.92, 38.15) --
	(199.63, 38.01) --
	(201.34, 37.89) --
	(203.06, 37.80) --
	(204.77, 37.72) --
	(206.48, 37.67) --
	(208.19, 37.62) --
	(209.91, 37.59) --
	(211.62, 37.56) --
	(213.33, 37.54) --
	(215.05, 37.53) --
	(216.76, 37.52) --
	(218.47, 37.51);

\path[draw=drawColor,line width= 0.6pt,line join=round] ( 47.17, 37.57) --
	( 48.88, 37.60) --
	( 50.60, 37.64) --
	( 52.31, 37.69) --
	( 54.02, 37.76) --
	( 55.73, 37.84) --
	( 57.45, 37.94) --
	( 59.16, 38.08) --
	( 60.87, 38.24) --
	( 62.59, 38.45) --
	( 64.30, 38.71) --
	( 66.01, 39.02) --
	( 67.73, 39.41) --
	( 69.44, 39.87) --
	( 71.15, 40.43) --
	( 72.87, 41.10) --
	( 74.58, 41.89) --
	( 76.29, 42.83) --
	( 78.00, 43.93) --
	( 79.72, 45.20) --
	( 81.43, 46.66) --
	( 83.14, 48.34) --
	( 84.86, 50.25) --
	( 86.57, 52.39) --
	( 88.28, 54.80) --
	( 90.00, 57.46) --
	( 91.71, 60.40) --
	( 93.42, 63.60) --
	( 95.13, 67.07) --
	( 96.85, 70.80) --
	( 98.56, 74.76) --
	(100.27, 78.94) --
	(101.99, 83.30) --
	(103.70, 87.82) --
	(105.41, 92.43) --
	(107.13, 97.11) --
	(108.84,101.78) --
	(110.55,106.39) --
	(112.26,110.88) --
	(113.98,115.19) --
	(115.69,119.24) --
	(117.40,122.97) --
	(119.12,126.32) --
	(120.83,129.24) --
	(122.54,131.67) --
	(124.26,133.58) --
	(125.97,134.92) --
	(127.68,135.67) --
	(129.40,135.81) --
	(131.11,135.36) --
	(132.82,134.31) --
	(134.53,132.68) --
	(136.25,130.50) --
	(137.96,127.81) --
	(139.67,124.67) --
	(141.39,121.11) --
	(143.10,117.21) --
	(144.81,113.02) --
	(146.53,108.61) --
	(148.24,104.05) --
	(149.95, 99.40) --
	(151.66, 94.72) --
	(153.38, 90.07) --
	(155.09, 85.50) --
	(156.80, 81.06) --
	(158.52, 76.78) --
	(160.23, 72.71) --
	(161.94, 68.87) --
	(163.66, 65.27) --
	(165.37, 61.93) --
	(167.08, 58.87) --
	(168.79, 56.07) --
	(170.51, 53.54) --
	(172.22, 51.27) --
	(173.93, 49.25) --
	(175.65, 47.46) --
	(177.36, 45.89) --
	(179.07, 44.53) --
	(180.79, 43.35) --
	(182.50, 42.33) --
	(184.21, 41.47) --
	(185.93, 40.74) --
	(187.64, 40.13) --
	(189.35, 39.62) --
	(191.06, 39.20) --
	(192.78, 38.85) --
	(194.49, 38.57) --
	(196.20, 38.34) --
	(197.92, 38.15) --
	(199.63, 38.01) --
	(201.34, 37.89) --
	(203.06, 37.80) --
	(204.77, 37.72) --
	(206.48, 37.67) --
	(208.19, 37.62) --
	(209.91, 37.59) --
	(211.62, 37.56) --
	(213.33, 37.54) --
	(215.05, 37.53) --
	(216.76, 37.52) --
	(218.47, 37.51);

\path[draw=drawColor,line width= 0.6pt,line join=round] ( 47.17, 37.57) --
	( 48.88, 37.60) --
	( 50.60, 37.64) --
	( 52.31, 37.69) --
	( 54.02, 37.76) --
	( 55.73, 37.84) --
	( 57.45, 37.94) --
	( 59.16, 38.08) --
	( 60.87, 38.24) --
	( 62.59, 38.45) --
	( 64.30, 38.71) --
	( 66.01, 39.02) --
	( 67.73, 39.41) --
	( 69.44, 39.87) --
	( 71.15, 40.43) --
	( 72.87, 41.10) --
	( 74.58, 41.89) --
	( 76.29, 42.83) --
	( 78.00, 43.93) --
	( 79.72, 45.20) --
	( 81.43, 46.66) --
	( 83.14, 48.34) --
	( 84.86, 50.25) --
	( 86.57, 52.39) --
	( 88.28, 54.80) --
	( 90.00, 57.46) --
	( 91.71, 60.40) --
	( 93.42, 63.60) --
	( 95.13, 67.07) --
	( 96.85, 70.80) --
	( 98.56, 74.76) --
	(100.27, 78.94) --
	(101.99, 83.30) --
	(103.70, 87.82) --
	(105.41, 92.43) --
	(107.13, 97.11) --
	(108.84,101.78) --
	(110.55,106.39) --
	(112.26,110.88) --
	(113.98,115.19) --
	(115.69,119.24) --
	(117.40,122.97) --
	(119.12,126.32) --
	(120.83,129.24) --
	(122.54,131.67) --
	(124.26,133.58) --
	(125.97,134.92) --
	(127.68,135.67) --
	(129.40,135.81) --
	(131.11,135.36) --
	(132.82,134.31) --
	(134.53,132.68) --
	(136.25,130.50) --
	(137.96,127.81) --
	(139.67,124.67) --
	(141.39,121.11) --
	(143.10,117.21) --
	(144.81,113.02) --
	(146.53,108.61) --
	(148.24,104.05) --
	(149.95, 99.40) --
	(151.66, 94.72) --
	(153.38, 90.07) --
	(155.09, 85.50) --
	(156.80, 81.06) --
	(158.52, 76.78) --
	(160.23, 72.71) --
	(161.94, 68.87) --
	(163.66, 65.27) --
	(165.37, 61.93) --
	(167.08, 58.87) --
	(168.79, 56.07) --
	(170.51, 53.54) --
	(172.22, 51.27) --
	(173.93, 49.25) --
	(175.65, 47.46) --
	(177.36, 45.89) --
	(179.07, 44.53) --
	(180.79, 43.35) --
	(182.50, 42.33) --
	(184.21, 41.47) --
	(185.93, 40.74) --
	(187.64, 40.13) --
	(189.35, 39.62) --
	(191.06, 39.20) --
	(192.78, 38.85) --
	(194.49, 38.57) --
	(196.20, 38.34) --
	(197.92, 38.15) --
	(199.63, 38.01) --
	(201.34, 37.89) --
	(203.06, 37.80) --
	(204.77, 37.72) --
	(206.48, 37.67) --
	(208.19, 37.62) --
	(209.91, 37.59) --
	(211.62, 37.56) --
	(213.33, 37.54) --
	(215.05, 37.53) --
	(216.76, 37.52) --
	(218.47, 37.51);
\end{scope}
\begin{scope}
\path[clip] (  0.00,  0.00) rectangle (469.75,216.81);
\definecolor{drawColor}{gray}{0.30}

\node[text=drawColor,anchor=base east,inner sep=0pt, outer sep=0pt, scale=  0.88] at ( 29.50, 34.46) {0.0};

\node[text=drawColor,anchor=base east,inner sep=0pt, outer sep=0pt, scale=  0.88] at ( 29.50, 59.11) {0.1};

\node[text=drawColor,anchor=base east,inner sep=0pt, outer sep=0pt, scale=  0.88] at ( 29.50, 83.76) {0.2};

\node[text=drawColor,anchor=base east,inner sep=0pt, outer sep=0pt, scale=  0.88] at ( 29.50,108.41) {0.3};

\node[text=drawColor,anchor=base east,inner sep=0pt, outer sep=0pt, scale=  0.88] at ( 29.50,133.06) {0.4};

\node[text=drawColor,anchor=base east,inner sep=0pt, outer sep=0pt, scale=  0.88] at ( 29.50,157.72) {0.5};
\end{scope}
\begin{scope}
\path[clip] (  0.00,  0.00) rectangle (469.75,216.81);
\definecolor{drawColor}{gray}{0.20}

\path[draw=drawColor,line width= 0.6pt,line join=round] ( 31.70, 37.49) --
	( 34.45, 37.49);

\path[draw=drawColor,line width= 0.6pt,line join=round] ( 31.70, 62.14) --
	( 34.45, 62.14);

\path[draw=drawColor,line width= 0.6pt,line join=round] ( 31.70, 86.79) --
	( 34.45, 86.79);

\path[draw=drawColor,line width= 0.6pt,line join=round] ( 31.70,111.44) --
	( 34.45,111.44);

\path[draw=drawColor,line width= 0.6pt,line join=round] ( 31.70,136.09) --
	( 34.45,136.09);

\path[draw=drawColor,line width= 0.6pt,line join=round] ( 31.70,160.75) --
	( 34.45,160.75);
\end{scope}
\begin{scope}
\path[clip] (  0.00,  0.00) rectangle (469.75,216.81);
\definecolor{drawColor}{gray}{0.20}

\path[draw=drawColor,line width= 0.6pt,line join=round] ( 41.67, 28.50) --
	( 41.67, 31.25);

\path[draw=drawColor,line width= 0.6pt,line join=round] ( 85.32, 28.50) --
	( 85.32, 31.25);

\path[draw=drawColor,line width= 0.6pt,line join=round] (128.96, 28.50) --
	(128.96, 31.25);

\path[draw=drawColor,line width= 0.6pt,line join=round] (172.60, 28.50) --
	(172.60, 31.25);

\path[draw=drawColor,line width= 0.6pt,line join=round] (216.25, 28.50) --
	(216.25, 31.25);
\end{scope}
\begin{scope}
\path[clip] (  0.00,  0.00) rectangle (469.75,216.81);
\definecolor{drawColor}{gray}{0.30}

\node[text=drawColor,anchor=base,inner sep=0pt, outer sep=0pt, scale=  0.88] at ( 41.67, 20.24) {-4};

\node[text=drawColor,anchor=base,inner sep=0pt, outer sep=0pt, scale=  0.88] at ( 85.32, 20.24) {-2};

\node[text=drawColor,anchor=base,inner sep=0pt, outer sep=0pt, scale=  0.88] at (128.96, 20.24) {0};

\node[text=drawColor,anchor=base,inner sep=0pt, outer sep=0pt, scale=  0.88] at (172.60, 20.24) {2};

\node[text=drawColor,anchor=base,inner sep=0pt, outer sep=0pt, scale=  0.88] at (216.25, 20.24) {4};
\end{scope}
\begin{scope}
\path[clip] (  0.00,  0.00) rectangle (469.75,216.81);
\definecolor{drawColor}{RGB}{0,0,0}

\node[text=drawColor,rotate= 90.00,anchor=base,inner sep=0pt, outer sep=0pt, scale=  1.10] at ( 13.08, 99.82) { };
\end{scope}
\begin{scope}
\path[clip] (  0.00,  0.00) rectangle (469.75,216.81);
\definecolor{drawColor}{RGB}{0,0,0}

\node[text=drawColor,anchor=base west,inner sep=0pt, outer sep=0pt, scale=  1.32] at ( 34.45,176.76) {$H = 0.55$};
\end{scope}
\begin{scope}
\path[clip] (234.88,  0.00) rectangle (469.75,191.36);
\definecolor{drawColor}{RGB}{255,255,255}
\definecolor{fillColor}{RGB}{255,255,255}

\path[draw=drawColor,line width= 0.6pt,line join=round,line cap=round,fill=fillColor] (234.88,  0.00) rectangle (469.76,191.36);
\end{scope}
\begin{scope}
\path[clip] (269.32, 31.25) rectangle (464.25,168.39);
\definecolor{fillColor}{gray}{0.92}

\path[fill=fillColor] (269.32, 31.25) rectangle (464.25,168.39);
\definecolor{drawColor}{RGB}{255,255,255}

\path[draw=drawColor,line width= 0.3pt,line join=round] (269.32, 59.23) --
	(464.25, 59.23);

\path[draw=drawColor,line width= 0.3pt,line join=round] (269.32,102.70) --
	(464.25,102.70);

\path[draw=drawColor,line width= 0.3pt,line join=round] (269.32,146.18) --
	(464.25,146.18);

\path[draw=drawColor,line width= 0.3pt,line join=round] (289.22, 31.25) --
	(289.22,168.39);

\path[draw=drawColor,line width= 0.3pt,line join=round] (338.96, 31.25) --
	(338.96,168.39);

\path[draw=drawColor,line width= 0.3pt,line join=round] (388.71, 31.25) --
	(388.71,168.39);

\path[draw=drawColor,line width= 0.3pt,line join=round] (438.46, 31.25) --
	(438.46,168.39);

\path[draw=drawColor,line width= 0.6pt,line join=round] (269.32, 37.49) --
	(464.25, 37.49);

\path[draw=drawColor,line width= 0.6pt,line join=round] (269.32, 80.96) --
	(464.25, 80.96);

\path[draw=drawColor,line width= 0.6pt,line join=round] (269.32,124.44) --
	(464.25,124.44);

\path[draw=drawColor,line width= 0.6pt,line join=round] (269.32,167.92) --
	(464.25,167.92);

\path[draw=drawColor,line width= 0.6pt,line join=round] (314.09, 31.25) --
	(314.09,168.39);

\path[draw=drawColor,line width= 0.6pt,line join=round] (363.84, 31.25) --
	(363.84,168.39);

\path[draw=drawColor,line width= 0.6pt,line join=round] (413.58, 31.25) --
	(413.58,168.39);

\path[draw=drawColor,line width= 0.6pt,line join=round] (463.33, 31.25) --
	(463.33,168.39);
\definecolor{drawColor}{RGB}{146,129,188}
\definecolor{fillColor}{RGB}{146,129,188}

\path[draw=drawColor,line width= 0.6pt,line cap=rect,fill=fillColor,fill opacity=0.50] (278.18, 37.49) rectangle (284.09, 37.49);

\path[draw=drawColor,line width= 0.6pt,line cap=rect,fill=fillColor,fill opacity=0.50] (284.09, 37.49) rectangle (290.00,111.08);

\path[draw=drawColor,line width= 0.6pt,line cap=rect,fill=fillColor,fill opacity=0.50] (290.00, 37.49) rectangle (295.91,149.35);

\path[draw=drawColor,line width= 0.6pt,line cap=rect,fill=fillColor,fill opacity=0.50] (295.91, 37.49) rectangle (301.81,120.42);

\path[draw=drawColor,line width= 0.6pt,line cap=rect,fill=fillColor,fill opacity=0.50] (301.81, 37.49) rectangle (307.72,112.36);

\path[draw=drawColor,line width= 0.6pt,line cap=rect,fill=fillColor,fill opacity=0.50] (307.72, 37.49) rectangle (313.63,111.82);

\path[draw=drawColor,line width= 0.6pt,line cap=rect,fill=fillColor,fill opacity=0.50] (313.63, 37.49) rectangle (319.53,100.65);

\path[draw=drawColor,line width= 0.6pt,line cap=rect,fill=fillColor,fill opacity=0.50] (319.53, 37.49) rectangle (325.44, 99.55);

\path[draw=drawColor,line width= 0.6pt,line cap=rect,fill=fillColor,fill opacity=0.50] (325.44, 37.49) rectangle (331.35,105.22);

\path[draw=drawColor,line width= 0.6pt,line cap=rect,fill=fillColor,fill opacity=0.50] (331.35, 37.49) rectangle (337.25, 96.80);

\path[draw=drawColor,line width= 0.6pt,line cap=rect,fill=fillColor,fill opacity=0.50] (337.25, 37.49) rectangle (343.16, 97.54);

\path[draw=drawColor,line width= 0.6pt,line cap=rect,fill=fillColor,fill opacity=0.50] (343.16, 37.49) rectangle (349.07, 92.59);

\path[draw=drawColor,line width= 0.6pt,line cap=rect,fill=fillColor,fill opacity=0.50] (349.07, 37.49) rectangle (354.98, 92.41);

\path[draw=drawColor,line width= 0.6pt,line cap=rect,fill=fillColor,fill opacity=0.50] (354.98, 37.49) rectangle (360.88, 89.85);

\path[draw=drawColor,line width= 0.6pt,line cap=rect,fill=fillColor,fill opacity=0.50] (360.88, 37.49) rectangle (366.79, 92.04);

\path[draw=drawColor,line width= 0.6pt,line cap=rect,fill=fillColor,fill opacity=0.50] (366.79, 37.49) rectangle (372.70, 94.42);

\path[draw=drawColor,line width= 0.6pt,line cap=rect,fill=fillColor,fill opacity=0.50] (372.70, 37.49) rectangle (378.60, 85.09);

\path[draw=drawColor,line width= 0.6pt,line cap=rect,fill=fillColor,fill opacity=0.50] (378.60, 37.49) rectangle (384.51, 92.23);

\path[draw=drawColor,line width= 0.6pt,line cap=rect,fill=fillColor,fill opacity=0.50] (384.51, 37.49) rectangle (390.42, 91.13);

\path[draw=drawColor,line width= 0.6pt,line cap=rect,fill=fillColor,fill opacity=0.50] (390.42, 37.49) rectangle (396.32, 99.92);

\path[draw=drawColor,line width= 0.6pt,line cap=rect,fill=fillColor,fill opacity=0.50] (396.32, 37.49) rectangle (402.23, 93.87);

\path[draw=drawColor,line width= 0.6pt,line cap=rect,fill=fillColor,fill opacity=0.50] (402.23, 37.49) rectangle (408.14, 99.92);

\path[draw=drawColor,line width= 0.6pt,line cap=rect,fill=fillColor,fill opacity=0.50] (408.14, 37.49) rectangle (414.05,114.20);

\path[draw=drawColor,line width= 0.6pt,line cap=rect,fill=fillColor,fill opacity=0.50] (414.05, 37.49) rectangle (419.95, 96.62);

\path[draw=drawColor,line width= 0.6pt,line cap=rect,fill=fillColor,fill opacity=0.50] (419.95, 37.49) rectangle (425.86,111.45);

\path[draw=drawColor,line width= 0.6pt,line cap=rect,fill=fillColor,fill opacity=0.50] (425.86, 37.49) rectangle (431.77,118.22);

\path[draw=drawColor,line width= 0.6pt,line cap=rect,fill=fillColor,fill opacity=0.50] (431.77, 37.49) rectangle (437.67,162.16);

\path[draw=drawColor,line width= 0.6pt,line cap=rect,fill=fillColor,fill opacity=0.50] (437.67, 37.49) rectangle (443.58,112.00);

\path[draw=drawColor,line width= 0.6pt,line cap=rect,fill=fillColor,fill opacity=0.50] (443.58, 37.49) rectangle (449.49, 37.49);

\path[draw=drawColor,line width= 0.6pt,line cap=rect,fill=fillColor,fill opacity=0.50] (449.49, 37.49) rectangle (455.39, 37.49);
\definecolor{drawColor}{RGB}{75,156,211}
\definecolor{fillColor}{RGB}{75,156,211}

\path[draw=drawColor,line width= 0.6pt,line cap=rect,fill=fillColor,fill opacity=0.50] (278.18, 37.49) rectangle (284.09,104.49);

\path[draw=drawColor,line width= 0.6pt,line cap=rect,fill=fillColor,fill opacity=0.50] (284.09, 37.49) rectangle (290.00, 99.18);

\path[draw=drawColor,line width= 0.6pt,line cap=rect,fill=fillColor,fill opacity=0.50] (290.00, 37.49) rectangle (295.91,103.76);

\path[draw=drawColor,line width= 0.6pt,line cap=rect,fill=fillColor,fill opacity=0.50] (295.91, 37.49) rectangle (301.81, 94.42);

\path[draw=drawColor,line width= 0.6pt,line cap=rect,fill=fillColor,fill opacity=0.50] (301.81, 37.49) rectangle (307.72,100.28);

\path[draw=drawColor,line width= 0.6pt,line cap=rect,fill=fillColor,fill opacity=0.50] (307.72, 37.49) rectangle (313.63,106.14);

\path[draw=drawColor,line width= 0.6pt,line cap=rect,fill=fillColor,fill opacity=0.50] (313.63, 37.49) rectangle (319.53, 99.73);

\path[draw=drawColor,line width= 0.6pt,line cap=rect,fill=fillColor,fill opacity=0.50] (319.53, 37.49) rectangle (325.44,101.20);

\path[draw=drawColor,line width= 0.6pt,line cap=rect,fill=fillColor,fill opacity=0.50] (325.44, 37.49) rectangle (331.35,100.10);

\path[draw=drawColor,line width= 0.6pt,line cap=rect,fill=fillColor,fill opacity=0.50] (331.35, 37.49) rectangle (337.25,101.20);

\path[draw=drawColor,line width= 0.6pt,line cap=rect,fill=fillColor,fill opacity=0.50] (337.25, 37.49) rectangle (343.16, 99.73);

\path[draw=drawColor,line width= 0.6pt,line cap=rect,fill=fillColor,fill opacity=0.50] (343.16, 37.49) rectangle (349.07, 98.08);

\path[draw=drawColor,line width= 0.6pt,line cap=rect,fill=fillColor,fill opacity=0.50] (349.07, 37.49) rectangle (354.98,103.21);

\path[draw=drawColor,line width= 0.6pt,line cap=rect,fill=fillColor,fill opacity=0.50] (354.98, 37.49) rectangle (360.88, 95.34);

\path[draw=drawColor,line width= 0.6pt,line cap=rect,fill=fillColor,fill opacity=0.50] (360.88, 37.49) rectangle (366.79,105.22);

\path[draw=drawColor,line width= 0.6pt,line cap=rect,fill=fillColor,fill opacity=0.50] (366.79, 37.49) rectangle (372.70, 98.63);

\path[draw=drawColor,line width= 0.6pt,line cap=rect,fill=fillColor,fill opacity=0.50] (372.70, 37.49) rectangle (378.60,100.65);

\path[draw=drawColor,line width= 0.6pt,line cap=rect,fill=fillColor,fill opacity=0.50] (378.60, 37.49) rectangle (384.51,102.30);

\path[draw=drawColor,line width= 0.6pt,line cap=rect,fill=fillColor,fill opacity=0.50] (384.51, 37.49) rectangle (390.42, 97.17);

\path[draw=drawColor,line width= 0.6pt,line cap=rect,fill=fillColor,fill opacity=0.50] (390.42, 37.49) rectangle (396.32, 99.92);

\path[draw=drawColor,line width= 0.6pt,line cap=rect,fill=fillColor,fill opacity=0.50] (396.32, 37.49) rectangle (402.23, 98.08);

\path[draw=drawColor,line width= 0.6pt,line cap=rect,fill=fillColor,fill opacity=0.50] (402.23, 37.49) rectangle (408.14,101.75);

\path[draw=drawColor,line width= 0.6pt,line cap=rect,fill=fillColor,fill opacity=0.50] (408.14, 37.49) rectangle (414.05,100.65);

\path[draw=drawColor,line width= 0.6pt,line cap=rect,fill=fillColor,fill opacity=0.50] (414.05, 37.49) rectangle (419.95, 99.37);

\path[draw=drawColor,line width= 0.6pt,line cap=rect,fill=fillColor,fill opacity=0.50] (419.95, 37.49) rectangle (425.86, 98.08);

\path[draw=drawColor,line width= 0.6pt,line cap=rect,fill=fillColor,fill opacity=0.50] (425.86, 37.49) rectangle (431.77,101.38);

\path[draw=drawColor,line width= 0.6pt,line cap=rect,fill=fillColor,fill opacity=0.50] (431.77, 37.49) rectangle (437.67,101.75);

\path[draw=drawColor,line width= 0.6pt,line cap=rect,fill=fillColor,fill opacity=0.50] (437.67, 37.49) rectangle (443.58, 97.90);

\path[draw=drawColor,line width= 0.6pt,line cap=rect,fill=fillColor,fill opacity=0.50] (443.58, 37.49) rectangle (449.49,101.01);

\path[draw=drawColor,line width= 0.6pt,line cap=rect,fill=fillColor,fill opacity=0.50] (449.49, 37.49) rectangle (455.39, 44.63);
\definecolor{drawColor}{RGB}{230,159,0}
\definecolor{fillColor}{RGB}{230,159,0}

\path[draw=drawColor,line width= 0.6pt,line cap=rect,fill=fillColor,fill opacity=0.50] (278.18, 37.49) rectangle (284.09, 37.49);

\path[draw=drawColor,line width= 0.6pt,line cap=rect,fill=fillColor,fill opacity=0.50] (284.09, 37.49) rectangle (290.00,159.05);

\path[draw=drawColor,line width= 0.6pt,line cap=rect,fill=fillColor,fill opacity=0.50] (290.00, 37.49) rectangle (295.91,127.01);

\path[draw=drawColor,line width= 0.6pt,line cap=rect,fill=fillColor,fill opacity=0.50] (295.91, 37.49) rectangle (301.81,104.13);

\path[draw=drawColor,line width= 0.6pt,line cap=rect,fill=fillColor,fill opacity=0.50] (301.81, 37.49) rectangle (307.72,108.15);

\path[draw=drawColor,line width= 0.6pt,line cap=rect,fill=fillColor,fill opacity=0.50] (307.72, 37.49) rectangle (313.63,101.75);

\path[draw=drawColor,line width= 0.6pt,line cap=rect,fill=fillColor,fill opacity=0.50] (313.63, 37.49) rectangle (319.53, 93.87);

\path[draw=drawColor,line width= 0.6pt,line cap=rect,fill=fillColor,fill opacity=0.50] (319.53, 37.49) rectangle (325.44, 91.68);

\path[draw=drawColor,line width= 0.6pt,line cap=rect,fill=fillColor,fill opacity=0.50] (325.44, 37.49) rectangle (331.35, 92.04);

\path[draw=drawColor,line width= 0.6pt,line cap=rect,fill=fillColor,fill opacity=0.50] (331.35, 37.49) rectangle (337.25,107.06);

\path[draw=drawColor,line width= 0.6pt,line cap=rect,fill=fillColor,fill opacity=0.50] (337.25, 37.49) rectangle (343.16, 88.56);

\path[draw=drawColor,line width= 0.6pt,line cap=rect,fill=fillColor,fill opacity=0.50] (343.16, 37.49) rectangle (349.07, 95.89);

\path[draw=drawColor,line width= 0.6pt,line cap=rect,fill=fillColor,fill opacity=0.50] (349.07, 37.49) rectangle (354.98, 93.51);

\path[draw=drawColor,line width= 0.6pt,line cap=rect,fill=fillColor,fill opacity=0.50] (354.98, 37.49) rectangle (360.88,102.84);

\path[draw=drawColor,line width= 0.6pt,line cap=rect,fill=fillColor,fill opacity=0.50] (360.88, 37.49) rectangle (366.79, 92.78);

\path[draw=drawColor,line width= 0.6pt,line cap=rect,fill=fillColor,fill opacity=0.50] (366.79, 37.49) rectangle (372.70, 96.99);

\path[draw=drawColor,line width= 0.6pt,line cap=rect,fill=fillColor,fill opacity=0.50] (372.70, 37.49) rectangle (378.60, 97.90);

\path[draw=drawColor,line width= 0.6pt,line cap=rect,fill=fillColor,fill opacity=0.50] (378.60, 37.49) rectangle (384.51, 98.27);

\path[draw=drawColor,line width= 0.6pt,line cap=rect,fill=fillColor,fill opacity=0.50] (384.51, 37.49) rectangle (390.42,105.59);

\path[draw=drawColor,line width= 0.6pt,line cap=rect,fill=fillColor,fill opacity=0.50] (390.42, 37.49) rectangle (396.32, 95.52);

\path[draw=drawColor,line width= 0.6pt,line cap=rect,fill=fillColor,fill opacity=0.50] (396.32, 37.49) rectangle (402.23,103.58);

\path[draw=drawColor,line width= 0.6pt,line cap=rect,fill=fillColor,fill opacity=0.50] (402.23, 37.49) rectangle (408.14, 96.62);

\path[draw=drawColor,line width= 0.6pt,line cap=rect,fill=fillColor,fill opacity=0.50] (408.14, 37.49) rectangle (414.05,104.13);

\path[draw=drawColor,line width= 0.6pt,line cap=rect,fill=fillColor,fill opacity=0.50] (414.05, 37.49) rectangle (419.95, 95.34);

\path[draw=drawColor,line width= 0.6pt,line cap=rect,fill=fillColor,fill opacity=0.50] (419.95, 37.49) rectangle (425.86,114.93);

\path[draw=drawColor,line width= 0.6pt,line cap=rect,fill=fillColor,fill opacity=0.50] (425.86, 37.49) rectangle (431.77,105.04);

\path[draw=drawColor,line width= 0.6pt,line cap=rect,fill=fillColor,fill opacity=0.50] (431.77, 37.49) rectangle (437.67,127.56);

\path[draw=drawColor,line width= 0.6pt,line cap=rect,fill=fillColor,fill opacity=0.50] (437.67, 37.49) rectangle (443.58,143.12);

\path[draw=drawColor,line width= 0.6pt,line cap=rect,fill=fillColor,fill opacity=0.50] (443.58, 37.49) rectangle (449.49, 37.49);

\path[draw=drawColor,line width= 0.6pt,line cap=rect,fill=fillColor,fill opacity=0.50] (449.49, 37.49) rectangle (455.39, 37.49);
\definecolor{drawColor}{RGB}{255,0,0}

\path[draw=drawColor,line width= 0.6pt,line join=round] (278.54, 57.43) --
	(280.25, 58.63) --
	(281.97, 59.87) --
	(283.68, 61.17) --
	(285.39, 62.50) --
	(287.10, 63.88) --
	(288.82, 65.31) --
	(290.53, 66.77) --
	(292.24, 68.28) --
	(293.96, 69.82) --
	(295.67, 71.40) --
	(297.38, 73.02) --
	(299.10, 74.67) --
	(300.81, 76.35) --
	(302.52, 78.06) --
	(304.24, 79.80) --
	(305.95, 81.55) --
	(307.66, 83.33) --
	(309.37, 85.12) --
	(311.09, 86.92) --
	(312.80, 88.72) --
	(314.51, 90.54) --
	(316.23, 92.35) --
	(317.94, 94.15) --
	(319.65, 95.94) --
	(321.37, 97.72) --
	(323.08, 99.48) --
	(324.79,101.22) --
	(326.50,102.93) --
	(328.22,104.60) --
	(329.93,106.24) --
	(331.64,107.83) --
	(333.36,109.37) --
	(335.07,110.86) --
	(336.78,112.29) --
	(338.50,113.66) --
	(340.21,114.96) --
	(341.92,116.19) --
	(343.63,117.35) --
	(345.35,118.42) --
	(347.06,119.42) --
	(348.77,120.33) --
	(350.49,121.14) --
	(352.20,121.87) --
	(353.91,122.50) --
	(355.63,123.04) --
	(357.34,123.47) --
	(359.05,123.81) --
	(360.76,124.05) --
	(362.48,124.18) --
	(364.19,124.21) --
	(365.90,124.14) --
	(367.62,123.96) --
	(369.33,123.68) --
	(371.04,123.31) --
	(372.76,122.83) --
	(374.47,122.25) --
	(376.18,121.58) --
	(377.90,120.82) --
	(379.61,119.96) --
	(381.32,119.02) --
	(383.03,117.99) --
	(384.75,116.88) --
	(386.46,115.69) --
	(388.17,114.43) --
	(389.89,113.10) --
	(391.60,111.70) --
	(393.31,110.25) --
	(395.03,108.74) --
	(396.74,107.17) --
	(398.45,105.56) --
	(400.16,103.91) --
	(401.88,102.22) --
	(403.59,100.50) --
	(405.30, 98.76) --
	(407.02, 96.99) --
	(408.73, 95.20) --
	(410.44, 93.40) --
	(412.16, 91.60) --
	(413.87, 89.79) --
	(415.58, 87.98) --
	(417.29, 86.17) --
	(419.01, 84.37) --
	(420.72, 82.59) --
	(422.43, 80.82) --
	(424.15, 79.08) --
	(425.86, 77.35) --
	(427.57, 75.65) --
	(429.29, 73.98) --
	(431.00, 72.35) --
	(432.71, 70.74) --
	(434.43, 69.18) --
	(436.14, 67.65) --
	(437.85, 66.16) --
	(439.56, 64.71) --
	(441.28, 63.30) --
	(442.99, 61.94) --
	(444.70, 60.62) --
	(446.42, 59.35) --
	(448.13, 58.13) --
	(449.84, 56.94);

\path[draw=drawColor,line width= 0.6pt,line join=round] (278.54, 57.43) --
	(280.25, 58.63) --
	(281.97, 59.87) --
	(283.68, 61.17) --
	(285.39, 62.50) --
	(287.10, 63.88) --
	(288.82, 65.31) --
	(290.53, 66.77) --
	(292.24, 68.28) --
	(293.96, 69.82) --
	(295.67, 71.40) --
	(297.38, 73.02) --
	(299.10, 74.67) --
	(300.81, 76.35) --
	(302.52, 78.06) --
	(304.24, 79.80) --
	(305.95, 81.55) --
	(307.66, 83.33) --
	(309.37, 85.12) --
	(311.09, 86.92) --
	(312.80, 88.72) --
	(314.51, 90.54) --
	(316.23, 92.35) --
	(317.94, 94.15) --
	(319.65, 95.94) --
	(321.37, 97.72) --
	(323.08, 99.48) --
	(324.79,101.22) --
	(326.50,102.93) --
	(328.22,104.60) --
	(329.93,106.24) --
	(331.64,107.83) --
	(333.36,109.37) --
	(335.07,110.86) --
	(336.78,112.29) --
	(338.50,113.66) --
	(340.21,114.96) --
	(341.92,116.19) --
	(343.63,117.35) --
	(345.35,118.42) --
	(347.06,119.42) --
	(348.77,120.33) --
	(350.49,121.14) --
	(352.20,121.87) --
	(353.91,122.50) --
	(355.63,123.04) --
	(357.34,123.47) --
	(359.05,123.81) --
	(360.76,124.05) --
	(362.48,124.18) --
	(364.19,124.21) --
	(365.90,124.14) --
	(367.62,123.96) --
	(369.33,123.68) --
	(371.04,123.31) --
	(372.76,122.83) --
	(374.47,122.25) --
	(376.18,121.58) --
	(377.90,120.82) --
	(379.61,119.96) --
	(381.32,119.02) --
	(383.03,117.99) --
	(384.75,116.88) --
	(386.46,115.69) --
	(388.17,114.43) --
	(389.89,113.10) --
	(391.60,111.70) --
	(393.31,110.25) --
	(395.03,108.74) --
	(396.74,107.17) --
	(398.45,105.56) --
	(400.16,103.91) --
	(401.88,102.22) --
	(403.59,100.50) --
	(405.30, 98.76) --
	(407.02, 96.99) --
	(408.73, 95.20) --
	(410.44, 93.40) --
	(412.16, 91.60) --
	(413.87, 89.79) --
	(415.58, 87.98) --
	(417.29, 86.17) --
	(419.01, 84.37) --
	(420.72, 82.59) --
	(422.43, 80.82) --
	(424.15, 79.08) --
	(425.86, 77.35) --
	(427.57, 75.65) --
	(429.29, 73.98) --
	(431.00, 72.35) --
	(432.71, 70.74) --
	(434.43, 69.18) --
	(436.14, 67.65) --
	(437.85, 66.16) --
	(439.56, 64.71) --
	(441.28, 63.30) --
	(442.99, 61.94) --
	(444.70, 60.62) --
	(446.42, 59.35) --
	(448.13, 58.13) --
	(449.84, 56.94);

\path[draw=drawColor,line width= 0.6pt,line join=round] (278.54, 57.43) --
	(280.25, 58.63) --
	(281.97, 59.87) --
	(283.68, 61.17) --
	(285.39, 62.50) --
	(287.10, 63.88) --
	(288.82, 65.31) --
	(290.53, 66.77) --
	(292.24, 68.28) --
	(293.96, 69.82) --
	(295.67, 71.40) --
	(297.38, 73.02) --
	(299.10, 74.67) --
	(300.81, 76.35) --
	(302.52, 78.06) --
	(304.24, 79.80) --
	(305.95, 81.55) --
	(307.66, 83.33) --
	(309.37, 85.12) --
	(311.09, 86.92) --
	(312.80, 88.72) --
	(314.51, 90.54) --
	(316.23, 92.35) --
	(317.94, 94.15) --
	(319.65, 95.94) --
	(321.37, 97.72) --
	(323.08, 99.48) --
	(324.79,101.22) --
	(326.50,102.93) --
	(328.22,104.60) --
	(329.93,106.24) --
	(331.64,107.83) --
	(333.36,109.37) --
	(335.07,110.86) --
	(336.78,112.29) --
	(338.50,113.66) --
	(340.21,114.96) --
	(341.92,116.19) --
	(343.63,117.35) --
	(345.35,118.42) --
	(347.06,119.42) --
	(348.77,120.33) --
	(350.49,121.14) --
	(352.20,121.87) --
	(353.91,122.50) --
	(355.63,123.04) --
	(357.34,123.47) --
	(359.05,123.81) --
	(360.76,124.05) --
	(362.48,124.18) --
	(364.19,124.21) --
	(365.90,124.14) --
	(367.62,123.96) --
	(369.33,123.68) --
	(371.04,123.31) --
	(372.76,122.83) --
	(374.47,122.25) --
	(376.18,121.58) --
	(377.90,120.82) --
	(379.61,119.96) --
	(381.32,119.02) --
	(383.03,117.99) --
	(384.75,116.88) --
	(386.46,115.69) --
	(388.17,114.43) --
	(389.89,113.10) --
	(391.60,111.70) --
	(393.31,110.25) --
	(395.03,108.74) --
	(396.74,107.17) --
	(398.45,105.56) --
	(400.16,103.91) --
	(401.88,102.22) --
	(403.59,100.50) --
	(405.30, 98.76) --
	(407.02, 96.99) --
	(408.73, 95.20) --
	(410.44, 93.40) --
	(412.16, 91.60) --
	(413.87, 89.79) --
	(415.58, 87.98) --
	(417.29, 86.17) --
	(419.01, 84.37) --
	(420.72, 82.59) --
	(422.43, 80.82) --
	(424.15, 79.08) --
	(425.86, 77.35) --
	(427.57, 75.65) --
	(429.29, 73.98) --
	(431.00, 72.35) --
	(432.71, 70.74) --
	(434.43, 69.18) --
	(436.14, 67.65) --
	(437.85, 66.16) --
	(439.56, 64.71) --
	(441.28, 63.30) --
	(442.99, 61.94) --
	(444.70, 60.62) --
	(446.42, 59.35) --
	(448.13, 58.13) --
	(449.84, 56.94);
\end{scope}
\begin{scope}
\path[clip] (  0.00,  0.00) rectangle (469.75,216.81);
\definecolor{drawColor}{gray}{0.30}

\node[text=drawColor,anchor=base east,inner sep=0pt, outer sep=0pt, scale=  0.88] at (264.37, 34.46) {0.0};

\node[text=drawColor,anchor=base east,inner sep=0pt, outer sep=0pt, scale=  0.88] at (264.37, 77.93) {0.2};

\node[text=drawColor,anchor=base east,inner sep=0pt, outer sep=0pt, scale=  0.88] at (264.37,121.41) {0.4};

\node[text=drawColor,anchor=base east,inner sep=0pt, outer sep=0pt, scale=  0.88] at (264.37,164.89) {0.6};
\end{scope}
\begin{scope}
\path[clip] (  0.00,  0.00) rectangle (469.75,216.81);
\definecolor{drawColor}{gray}{0.20}

\path[draw=drawColor,line width= 0.6pt,line join=round] (266.57, 37.49) --
	(269.32, 37.49);

\path[draw=drawColor,line width= 0.6pt,line join=round] (266.57, 80.96) --
	(269.32, 80.96);

\path[draw=drawColor,line width= 0.6pt,line join=round] (266.57,124.44) --
	(269.32,124.44);

\path[draw=drawColor,line width= 0.6pt,line join=round] (266.57,167.92) --
	(269.32,167.92);
\end{scope}
\begin{scope}
\path[clip] (  0.00,  0.00) rectangle (469.75,216.81);
\definecolor{drawColor}{gray}{0.20}

\path[draw=drawColor,line width= 0.6pt,line join=round] (314.09, 28.50) --
	(314.09, 31.25);

\path[draw=drawColor,line width= 0.6pt,line join=round] (363.84, 28.50) --
	(363.84, 31.25);

\path[draw=drawColor,line width= 0.6pt,line join=round] (413.58, 28.50) --
	(413.58, 31.25);

\path[draw=drawColor,line width= 0.6pt,line join=round] (463.33, 28.50) --
	(463.33, 31.25);
\end{scope}
\begin{scope}
\path[clip] (  0.00,  0.00) rectangle (469.75,216.81);
\definecolor{drawColor}{gray}{0.30}

\node[text=drawColor,anchor=base,inner sep=0pt, outer sep=0pt, scale=  0.88] at (314.09, 20.24) {-1};

\node[text=drawColor,anchor=base,inner sep=0pt, outer sep=0pt, scale=  0.88] at (363.84, 20.24) {0};

\node[text=drawColor,anchor=base,inner sep=0pt, outer sep=0pt, scale=  0.88] at (413.58, 20.24) {1};

\node[text=drawColor,anchor=base,inner sep=0pt, outer sep=0pt, scale=  0.88] at (463.33, 20.24) {2};
\end{scope}
\begin{scope}
\path[clip] (  0.00,  0.00) rectangle (469.75,216.81);
\definecolor{drawColor}{RGB}{0,0,0}

\node[text=drawColor,rotate= 90.00,anchor=base,inner sep=0pt, outer sep=0pt, scale=  1.10] at (247.95, 99.82) { };
\end{scope}
\begin{scope}
\path[clip] (  0.00,  0.00) rectangle (469.75,216.81);
\definecolor{drawColor}{RGB}{0,0,0}

\node[text=drawColor,anchor=base west,inner sep=0pt, outer sep=0pt, scale=  1.32] at (269.32,176.76) {$H = 0.95$};
\end{scope}
\end{tikzpicture}

%% file: plots/numbersummands.tex
\begin{tikzpicture}[x=1pt,y=1pt]
\definecolor{fillColor}{RGB}{255,255,255}
\path[use as bounding box,fill=fillColor,fill opacity=0.00] (0,0) rectangle (216.81,216.81);
\begin{scope}
\path[clip] (  0.00,  0.00) rectangle (216.81,216.81);
\definecolor{drawColor}{RGB}{255,255,255}
\definecolor{fillColor}{RGB}{255,255,255}

\path[draw=drawColor,line width= 0.6pt,line join=round,line cap=round,fill=fillColor] (  0.00,  0.00) rectangle (216.81,216.81);
\end{scope}
\begin{scope}
\path[clip] ( 31.71, 30.69) rectangle (211.31,194.15);
\definecolor{fillColor}{gray}{0.92}

\path[fill=fillColor] ( 31.71, 30.69) rectangle (211.31,194.15);
\definecolor{drawColor}{RGB}{255,255,255}

\path[draw=drawColor,line width= 0.3pt,line join=round] ( 31.71, 61.89) --
	(211.31, 61.89);

\path[draw=drawColor,line width= 0.3pt,line join=round] ( 31.71,121.34) --
	(211.31,121.34);

\path[draw=drawColor,line width= 0.3pt,line join=round] ( 31.71,180.78) --
	(211.31,180.78);

\path[draw=drawColor,line width= 0.3pt,line join=round] ( 56.86, 30.69) --
	( 56.86,194.15);

\path[draw=drawColor,line width= 0.3pt,line join=round] ( 90.82, 30.69) --
	( 90.82,194.15);

\path[draw=drawColor,line width= 0.3pt,line join=round] (124.78, 30.69) --
	(124.78,194.15);

\path[draw=drawColor,line width= 0.3pt,line join=round] (158.74, 30.69) --
	(158.74,194.15);

\path[draw=drawColor,line width= 0.3pt,line join=round] (192.70, 30.69) --
	(192.70,194.15);

\path[draw=drawColor,line width= 0.6pt,line join=round] ( 31.71, 32.17) --
	(211.31, 32.17);

\path[draw=drawColor,line width= 0.6pt,line join=round] ( 31.71, 91.61) --
	(211.31, 91.61);

\path[draw=drawColor,line width= 0.6pt,line join=round] ( 31.71,151.06) --
	(211.31,151.06);

\path[draw=drawColor,line width= 0.6pt,line join=round] ( 39.88, 30.69) --
	( 39.88,194.15);

\path[draw=drawColor,line width= 0.6pt,line join=round] ( 73.84, 30.69) --
	( 73.84,194.15);

\path[draw=drawColor,line width= 0.6pt,line join=round] (107.80, 30.69) --
	(107.80,194.15);

\path[draw=drawColor,line width= 0.6pt,line join=round] (141.76, 30.69) --
	(141.76,194.15);

\path[draw=drawColor,line width= 0.6pt,line join=round] (175.72, 30.69) --
	(175.72,194.15);

\path[draw=drawColor,line width= 0.6pt,line join=round] (209.68, 30.69) --
	(209.68,194.15);
\definecolor{drawColor}{RGB}{0,0,0}
\definecolor{fillColor}{RGB}{0,0,0}

\path[draw=drawColor,line width= 0.4pt,line join=round,line cap=round,fill=fillColor] (124.78, 38.12) circle (  1.96);

\path[draw=drawColor,line width= 0.4pt,line join=round,line cap=round,fill=fillColor] (153.08, 44.06) circle (  1.96);

\path[draw=drawColor,line width= 0.4pt,line join=round,line cap=round,fill=fillColor] (167.23, 50.00) circle (  1.96);

\path[draw=drawColor,line width= 0.4pt,line join=round,line cap=round,fill=fillColor] (175.72, 55.95) circle (  1.96);

\path[draw=drawColor,line width= 0.4pt,line join=round,line cap=round,fill=fillColor] (181.38, 61.89) circle (  1.96);

\path[draw=drawColor,line width= 0.4pt,line join=round,line cap=round,fill=fillColor] (185.42, 67.84) circle (  1.96);

\path[draw=drawColor,line width= 0.4pt,line join=round,line cap=round,fill=fillColor] (188.45, 73.78) circle (  1.96);

\path[draw=drawColor,line width= 0.4pt,line join=round,line cap=round,fill=fillColor] (190.81, 79.73) circle (  1.96);

\path[draw=drawColor,line width= 0.4pt,line join=round,line cap=round,fill=fillColor] (192.70, 85.67) circle (  1.96);

\path[draw=drawColor,line width= 0.4pt,line join=round,line cap=round,fill=fillColor] (194.24, 91.61) circle (  1.96);

\path[draw=drawColor,line width= 0.4pt,line join=round,line cap=round,fill=fillColor] (195.53, 97.56) circle (  1.96);

\path[draw=drawColor,line width= 0.4pt,line join=round,line cap=round,fill=fillColor] (196.62,103.50) circle (  1.96);

\path[draw=drawColor,line width= 0.4pt,line join=round,line cap=round,fill=fillColor] (197.55,109.45) circle (  1.96);

\path[draw=drawColor,line width= 0.4pt,line join=round,line cap=round,fill=fillColor] (198.36,115.39) circle (  1.96);

\path[draw=drawColor,line width= 0.4pt,line join=round,line cap=round,fill=fillColor] (199.06,121.34) circle (  1.96);

\path[draw=drawColor,line width= 0.4pt,line join=round,line cap=round,fill=fillColor] (199.69,127.28) circle (  1.96);

\path[draw=drawColor,line width= 0.4pt,line join=round,line cap=round,fill=fillColor] (200.24,133.22) circle (  1.96);

\path[draw=drawColor,line width= 0.4pt,line join=round,line cap=round,fill=fillColor] (200.74,139.17) circle (  1.96);

\path[draw=drawColor,line width= 0.4pt,line join=round,line cap=round,fill=fillColor] (201.19,145.11) circle (  1.96);

\path[draw=drawColor,line width= 0.4pt,line join=round,line cap=round,fill=fillColor] (201.59,151.06) circle (  1.96);

\path[draw=drawColor,line width= 0.4pt,line join=round,line cap=round,fill=fillColor] (201.96,157.00) circle (  1.96);

\path[draw=drawColor,line width= 0.4pt,line join=round,line cap=round,fill=fillColor] (202.29,162.95) circle (  1.96);

\path[draw=drawColor,line width= 0.4pt,line join=round,line cap=round,fill=fillColor] (202.60,168.89) circle (  1.96);

\path[draw=drawColor,line width= 0.4pt,line join=round,line cap=round,fill=fillColor] (202.89,174.83) circle (  1.96);

\path[draw=drawColor,line width= 0.4pt,line join=round,line cap=round,fill=fillColor] (203.15,180.78) circle (  1.96);

\path[draw=drawColor,line width= 0.6pt,line join=round] ( 39.88, 38.12) -- (124.78, 38.12);

\path[draw=drawColor,line width= 0.6pt,line join=round] (124.78, 44.06) -- (153.08, 44.06);

\path[draw=drawColor,line width= 0.6pt,line join=round] (153.08, 50.00) -- (167.23, 50.00);

\path[draw=drawColor,line width= 0.6pt,line join=round] (167.23, 55.95) -- (175.72, 55.95);

\path[draw=drawColor,line width= 0.6pt,line join=round] (175.72, 61.89) -- (181.38, 61.89);

\path[draw=drawColor,line width= 0.6pt,line join=round] (181.38, 67.84) -- (185.42, 67.84);

\path[draw=drawColor,line width= 0.6pt,line join=round] (185.42, 73.78) -- (188.45, 73.78);

\path[draw=drawColor,line width= 0.6pt,line join=round] (188.45, 79.73) -- (190.81, 79.73);

\path[draw=drawColor,line width= 0.6pt,line join=round] (190.81, 85.67) -- (192.70, 85.67);

\path[draw=drawColor,line width= 0.6pt,line join=round] (192.70, 91.61) -- (194.24, 91.61);

\path[draw=drawColor,line width= 0.6pt,line join=round] (194.24, 97.56) -- (195.53, 97.56);

\path[draw=drawColor,line width= 0.6pt,line join=round] (195.53,103.50) -- (196.62,103.50);

\path[draw=drawColor,line width= 0.6pt,line join=round] (196.62,109.45) -- (197.55,109.45);

\path[draw=drawColor,line width= 0.6pt,line join=round] (197.55,115.39) -- (198.36,115.39);

\path[draw=drawColor,line width= 0.6pt,line join=round] (198.36,121.34) -- (199.06,121.34);

\path[draw=drawColor,line width= 0.6pt,line join=round] (199.06,127.28) -- (199.69,127.28);

\path[draw=drawColor,line width= 0.6pt,line join=round] (199.69,133.22) -- (200.24,133.22);

\path[draw=drawColor,line width= 0.6pt,line join=round] (200.24,139.17) -- (200.74,139.17);

\path[draw=drawColor,line width= 0.6pt,line join=round] (200.74,145.11) -- (201.19,145.11);

\path[draw=drawColor,line width= 0.6pt,line join=round] (201.19,151.06) -- (201.59,151.06);

\path[draw=drawColor,line width= 0.6pt,line join=round] (201.59,157.00) -- (201.96,157.00);

\path[draw=drawColor,line width= 0.6pt,line join=round] (201.96,162.95) -- (202.29,162.95);

\path[draw=drawColor,line width= 0.6pt,line join=round] (202.29,168.89) -- (202.60,168.89);

\path[draw=drawColor,line width= 0.6pt,line join=round] (202.60,174.83) -- (202.89,174.83);

\path[draw=drawColor,line width= 0.6pt,line join=round] (202.89,180.78) -- (203.15,180.78);
\end{scope}
\begin{scope}
\path[clip] (  0.00,  0.00) rectangle (216.81,216.81);
\definecolor{drawColor}{gray}{0.30}

\node[text=drawColor,anchor=base east,inner sep=0pt, outer sep=0pt, scale=  0.88] at ( 26.76, 29.14) {0};

\node[text=drawColor,anchor=base east,inner sep=0pt, outer sep=0pt, scale=  0.88] at ( 26.76, 88.58) {10};

\node[text=drawColor,anchor=base east,inner sep=0pt, outer sep=0pt, scale=  0.88] at ( 26.76,148.03) {20};
\end{scope}
\begin{scope}
\path[clip] (  0.00,  0.00) rectangle (216.81,216.81);
\definecolor{drawColor}{gray}{0.20}

\path[draw=drawColor,line width= 0.6pt,line join=round] ( 28.96, 32.17) --
	( 31.71, 32.17);

\path[draw=drawColor,line width= 0.6pt,line join=round] ( 28.96, 91.61) --
	( 31.71, 91.61);

\path[draw=drawColor,line width= 0.6pt,line join=round] ( 28.96,151.06) --
	( 31.71,151.06);
\end{scope}
\begin{scope}
\path[clip] (  0.00,  0.00) rectangle (216.81,216.81);
\definecolor{drawColor}{gray}{0.20}

\path[draw=drawColor,line width= 0.6pt,line join=round] ( 39.88, 27.94) --
	( 39.88, 30.69);

\path[draw=drawColor,line width= 0.6pt,line join=round] ( 73.84, 27.94) --
	( 73.84, 30.69);

\path[draw=drawColor,line width= 0.6pt,line join=round] (107.80, 27.94) --
	(107.80, 30.69);

\path[draw=drawColor,line width= 0.6pt,line join=round] (141.76, 27.94) --
	(141.76, 30.69);

\path[draw=drawColor,line width= 0.6pt,line join=round] (175.72, 27.94) --
	(175.72, 30.69);

\path[draw=drawColor,line width= 0.6pt,line join=round] (209.68, 27.94) --
	(209.68, 30.69);
\end{scope}
\begin{scope}
\path[clip] (  0.00,  0.00) rectangle (216.81,216.81);
\definecolor{drawColor}{gray}{0.30}

\node[text=drawColor,anchor=base,inner sep=0pt, outer sep=0pt, scale=  0.88] at ( 39.88, 19.68) {0.5};

\node[text=drawColor,anchor=base,inner sep=0pt, outer sep=0pt, scale=  0.88] at ( 73.84, 19.68) {0.6};

\node[text=drawColor,anchor=base,inner sep=0pt, outer sep=0pt, scale=  0.88] at (107.80, 19.68) {0.7};

\node[text=drawColor,anchor=base,inner sep=0pt, outer sep=0pt, scale=  0.88] at (141.76, 19.68) {0.8};

\node[text=drawColor,anchor=base,inner sep=0pt, outer sep=0pt, scale=  0.88] at (175.72, 19.68) {0.9};

\node[text=drawColor,anchor=base,inner sep=0pt, outer sep=0pt, scale=  0.88] at (209.68, 19.68) {1.0};
\end{scope}
\begin{scope}
\path[clip] (  0.00,  0.00) rectangle (216.81,216.81);
\definecolor{drawColor}{RGB}{0,0,0}

\node[text=drawColor,anchor=base,inner sep=0pt, outer sep=0pt, scale=  1.10] at (121.51,  7.64) {Hurst parameter};
\end{scope}
\begin{scope}
\path[clip] (  0.00,  0.00) rectangle (216.81,216.81);
\definecolor{drawColor}{RGB}{0,0,0}

\node[text=drawColor,rotate= 90.00,anchor=base,inner sep=0pt, outer sep=0pt, scale=  1.10] at ( 13.08,112.42) {Number of summands};
\end{scope}
\begin{scope}
\path[clip] (  0.00,  0.00) rectangle (216.81,216.81);
\definecolor{drawColor}{RGB}{0,0,0}

\node[text=drawColor,anchor=base west,inner sep=0pt, outer sep=0pt, scale=  1.32] at ( 31.71,202.22) { };
\end{scope}
\end{tikzpicture}

%% file: plots/normal_H=0.55_m=200_and_m=1000.tex
\begin{tikzpicture}[x=1pt,y=1pt]
\definecolor{fillColor}{RGB}{255,255,255}
\path[use as bounding box,fill=fillColor,fill opacity=0.00] (0,0) rectangle (469.75,433.62);
\begin{scope}
\path[clip] (  0.00,  0.00) rectangle (469.75,433.62);
\definecolor{fillColor}{RGB}{255,255,255}

\path[fill=fillColor] (167.74,408.17) rectangle (302.02,433.62);
\end{scope}
\begin{scope}
\path[clip] (  0.00,  0.00) rectangle (469.75,433.62);
\definecolor{drawColor}{RGB}{0,0,0}

\end{scope}
\begin{scope}
\path[clip] (  0.00,  0.00) rectangle (469.75,433.62);
\definecolor{fillColor}{gray}{0.95}

\path[fill=fillColor] (215.09,413.67) rectangle (229.54,428.12);
\end{scope}
\begin{scope}
\path[clip] (  0.00,  0.00) rectangle (469.75,433.62);
\definecolor{fillColor}{RGB}{146,129,188}

\path[fill=fillColor] (215.80,414.38) rectangle (228.83,427.41);
\end{scope}
\begin{scope}
\path[clip] (  0.00,  0.00) rectangle (469.75,433.62);
\definecolor{fillColor}{gray}{0.95}

\path[fill=fillColor] (265.27,413.67) rectangle (279.73,428.12);
\end{scope}
\begin{scope}
\path[clip] (  0.00,  0.00) rectangle (469.75,433.62);
\definecolor{fillColor}{RGB}{230,159,0}

\path[fill=fillColor] (265.99,414.38) rectangle (279.02,427.41);
\end{scope}
\begin{scope}
\path[clip] (  0.00,  0.00) rectangle (469.75,433.62);
\definecolor{drawColor}{RGB}{0,0,0}

\node[text=drawColor,anchor=base west,inner sep=0pt, outer sep=0pt, scale=  1.10] at (170,417.86) {method:};
\node[text=drawColor,anchor=base west,inner sep=0pt, outer sep=0pt, scale=  0.88] at (235.04,417.86) {asymp};
\end{scope}
\begin{scope}
\path[clip] (  0.00,  0.00) rectangle (469.75,433.62);
\definecolor{drawColor}{RGB}{0,0,0}

\node[text=drawColor,anchor=base west,inner sep=0pt, outer sep=0pt, scale=  0.88] at (285.23,417.86) {HOA};
\end{scope}
\begin{scope}
\path[clip] (  0.00,204.08) rectangle (234.88,408.17);
\definecolor{drawColor}{RGB}{255,255,255}
\definecolor{fillColor}{RGB}{255,255,255}

\path[draw=drawColor,line width= 0.6pt,line join=round,line cap=round,fill=fillColor] (  0.00,204.08) rectangle (234.88,408.17);
\end{scope}
\begin{scope}
\path[clip] ( 38.85,235.34) rectangle (229.38,387.16);
\definecolor{fillColor}{gray}{0.92}

\path[fill=fillColor] ( 38.85,235.34) rectangle (229.38,387.16);
\definecolor{drawColor}{RGB}{255,255,255}

\path[draw=drawColor,line width= 0.3pt,line join=round] ( 38.85,260.40) --
	(229.38,260.40);

\path[draw=drawColor,line width= 0.3pt,line join=round] ( 38.85,296.72) --
	(229.38,296.72);

\path[draw=drawColor,line width= 0.3pt,line join=round] ( 38.85,333.04) --
	(229.38,333.04);

\path[draw=drawColor,line width= 0.3pt,line join=round] ( 38.85,369.36) --
	(229.38,369.36);

\path[draw=drawColor,line width= 0.3pt,line join=round] ( 54.17,235.34) --
	( 54.17,387.16);

\path[draw=drawColor,line width= 0.3pt,line join=round] (107.46,235.34) --
	(107.46,387.16);

\path[draw=drawColor,line width= 0.3pt,line join=round] (160.76,235.34) --
	(160.76,387.16);

\path[draw=drawColor,line width= 0.3pt,line join=round] (214.06,235.34) --
	(214.06,387.16);

\path[draw=drawColor,line width= 0.6pt,line join=round] ( 38.85,242.24) --
	(229.38,242.24);

\path[draw=drawColor,line width= 0.6pt,line join=round] ( 38.85,278.56) --
	(229.38,278.56);

\path[draw=drawColor,line width= 0.6pt,line join=round] ( 38.85,314.88) --
	(229.38,314.88);

\path[draw=drawColor,line width= 0.6pt,line join=round] ( 38.85,351.20) --
	(229.38,351.20);

\path[draw=drawColor,line width= 0.6pt,line join=round] ( 80.82,235.34) --
	( 80.82,387.16);

\path[draw=drawColor,line width= 0.6pt,line join=round] (134.11,235.34) --
	(134.11,387.16);

\path[draw=drawColor,line width= 0.6pt,line join=round] (187.41,235.34) --
	(187.41,387.16);
\definecolor{fillColor}{RGB}{146,129,188}

\path[fill=fillColor] ( 47.51,242.24) rectangle ( 54.17,242.24);

\path[fill=fillColor] ( 74.15,242.24) rectangle ( 80.82,337.62);

\path[fill=fillColor] (100.80,242.24) rectangle (107.46,370.31);

\path[fill=fillColor] (127.45,242.24) rectangle (134.11,375.76);

\path[fill=fillColor] (154.10,242.24) rectangle (160.76,368.93);

\path[fill=fillColor] (180.75,242.24) rectangle (187.41,335.73);

\path[fill=fillColor] (207.39,242.24) rectangle (214.06,242.24);
\definecolor{fillColor}{RGB}{230,159,0}

\path[fill=fillColor] ( 54.17,242.24) rectangle ( 60.83,280.23);

\path[fill=fillColor] ( 80.82,242.24) rectangle ( 87.48,374.74);

\path[fill=fillColor] (107.46,242.24) rectangle (114.13,380.26);

\path[fill=fillColor] (134.11,242.24) rectangle (140.77,379.24);

\path[fill=fillColor] (160.76,242.24) rectangle (167.42,379.32);

\path[fill=fillColor] (187.41,242.24) rectangle (194.07,375.54);

\path[fill=fillColor] (214.06,242.24) rectangle (220.72,282.55);
\definecolor{drawColor}{RGB}{169,169,169}

\path[draw=drawColor,line width= 0.6pt,dash pattern=on 4pt off 4pt ,line join=round] ( 38.85,380.26) -- (229.38,380.26);
\end{scope}
\begin{scope}
\path[clip] (  0.00,  0.00) rectangle (469.75,433.62);
\definecolor{drawColor}{gray}{0.30}

\node[text=drawColor,anchor=base east,inner sep=0pt, outer sep=0pt, scale=  0.88] at ( 33.90,239.21) {0.00};

\node[text=drawColor,anchor=base east,inner sep=0pt, outer sep=0pt, scale=  0.88] at ( 33.90,275.53) {0.25};

\node[text=drawColor,anchor=base east,inner sep=0pt, outer sep=0pt, scale=  0.88] at ( 33.90,311.85) {0.50};

\node[text=drawColor,anchor=base east,inner sep=0pt, outer sep=0pt, scale=  0.88] at ( 33.90,348.17) {0.75};
\end{scope}
\begin{scope}
\path[clip] (  0.00,  0.00) rectangle (469.75,433.62);
\definecolor{drawColor}{gray}{0.20}

\path[draw=drawColor,line width= 0.6pt,line join=round] ( 36.10,242.24) --
	( 38.85,242.24);

\path[draw=drawColor,line width= 0.6pt,line join=round] ( 36.10,278.56) --
	( 38.85,278.56);

\path[draw=drawColor,line width= 0.6pt,line join=round] ( 36.10,314.88) --
	( 38.85,314.88);

\path[draw=drawColor,line width= 0.6pt,line join=round] ( 36.10,351.20) --
	( 38.85,351.20);
\end{scope}
\begin{scope}
\path[clip] (  0.00,  0.00) rectangle (469.75,433.62);
\definecolor{drawColor}{gray}{0.20}

\path[draw=drawColor,line width= 0.6pt,line join=round] ( 80.82,232.59) --
	( 80.82,235.34);

\path[draw=drawColor,line width= 0.6pt,line join=round] (134.11,232.59) --
	(134.11,235.34);

\path[draw=drawColor,line width= 0.6pt,line join=round] (187.41,232.59) --
	(187.41,235.34);
\end{scope}
\begin{scope}
\path[clip] (  0.00,  0.00) rectangle (469.75,433.62);
\definecolor{drawColor}{gray}{0.30}

\node[text=drawColor,anchor=base,inner sep=0pt, outer sep=0pt, scale=  0.88] at ( 80.82,224.33) {-2};

\node[text=drawColor,anchor=base,inner sep=0pt, outer sep=0pt, scale=  0.88] at (134.11,224.33) {0};

\node[text=drawColor,anchor=base,inner sep=0pt, outer sep=0pt, scale=  0.88] at (187.41,224.33) {2};
\end{scope}
\begin{scope}
\path[clip] (  0.00,  0.00) rectangle (469.75,433.62);
\definecolor{drawColor}{RGB}{0,0,0}

\node[text=drawColor,anchor=base,inner sep=0pt, outer sep=0pt, scale=  1.10] at (134.11,212.01) {$x$};
\end{scope}
\begin{scope}
\path[clip] (  0.00,  0.00) rectangle (469.75,433.62);
\definecolor{drawColor}{RGB}{0,0,0}

\node[text=drawColor,rotate= 90.00,anchor=base,inner sep=0pt, outer sep=0pt, scale=  1.10] at ( 13.08,311.25) {rate};
\end{scope}
\begin{scope}
\path[clip] (  0.00,  0.00) rectangle (469.75,433.62);
\definecolor{drawColor}{RGB}{0,0,0}

\node[text=drawColor,anchor=base west,inner sep=0pt, outer sep=0pt, scale=  1.10] at ( 38.85,395.09) {Coverage rate $H=0.55$, $N=200$};
\end{scope}
\begin{scope}
\path[clip] (234.88,204.08) rectangle (469.75,408.17);
\definecolor{drawColor}{RGB}{255,255,255}
\definecolor{fillColor}{RGB}{255,255,255}

\path[draw=drawColor,line width= 0.6pt,line join=round,line cap=round,fill=fillColor] (234.88,204.08) rectangle (469.76,408.17);
\end{scope}
\begin{scope}
\path[clip] (273.72,235.34) rectangle (464.25,387.16);
\definecolor{fillColor}{gray}{0.92}

\path[fill=fillColor] (273.72,235.34) rectangle (464.26,387.16);
\definecolor{drawColor}{RGB}{255,255,255}

\path[draw=drawColor,line width= 0.3pt,line join=round] (273.72,266.18) --
	(464.25,266.18);

\path[draw=drawColor,line width= 0.3pt,line join=round] (273.72,314.06) --
	(464.25,314.06);

\path[draw=drawColor,line width= 0.3pt,line join=round] (273.72,361.95) --
	(464.25,361.95);

\path[draw=drawColor,line width= 0.3pt,line join=round] (289.05,235.34) --
	(289.05,387.16);

\path[draw=drawColor,line width= 0.3pt,line join=round] (342.34,235.34) --
	(342.34,387.16);

\path[draw=drawColor,line width= 0.3pt,line join=round] (395.64,235.34) --
	(395.64,387.16);

\path[draw=drawColor,line width= 0.3pt,line join=round] (448.93,235.34) --
	(448.93,387.16);

\path[draw=drawColor,line width= 0.6pt,line join=round] (273.72,242.24) --
	(464.25,242.24);

\path[draw=drawColor,line width= 0.6pt,line join=round] (273.72,290.12) --
	(464.25,290.12);

\path[draw=drawColor,line width= 0.6pt,line join=round] (273.72,338.01) --
	(464.25,338.01);

\path[draw=drawColor,line width= 0.6pt,line join=round] (273.72,385.89) --
	(464.25,385.89);

\path[draw=drawColor,line width= 0.6pt,line join=round] (315.69,235.34) --
	(315.69,387.16);

\path[draw=drawColor,line width= 0.6pt,line join=round] (368.99,235.34) --
	(368.99,387.16);

\path[draw=drawColor,line width= 0.6pt,line join=round] (422.28,235.34) --
	(422.28,387.16);
\definecolor{fillColor}{RGB}{146,129,188}

\path[fill=fillColor] (282.38,242.24) rectangle (289.05,243.77);

\path[fill=fillColor] (309.03,242.24) rectangle (315.69,260.92);

\path[fill=fillColor] (335.68,242.24) rectangle (342.34,325.95);

\path[fill=fillColor] (362.33,242.24) rectangle (368.99,380.26);

\path[fill=fillColor] (388.97,242.24) rectangle (395.64,325.95);

\path[fill=fillColor] (415.62,242.24) rectangle (422.28,260.92);

\path[fill=fillColor] (442.27,242.24) rectangle (448.93,243.77);
\definecolor{fillColor}{RGB}{230,159,0}

\path[fill=fillColor] (289.05,242.24) rectangle (295.71,247.68);

\path[fill=fillColor] (315.69,242.24) rectangle (322.36,278.11);

\path[fill=fillColor] (342.34,242.24) rectangle (349.00,314.86);

\path[fill=fillColor] (368.99,242.24) rectangle (375.65,322.30);

\path[fill=fillColor] (395.64,242.24) rectangle (402.30,314.73);

\path[fill=fillColor] (422.28,242.24) rectangle (428.95,278.19);

\path[fill=fillColor] (448.93,242.24) rectangle (455.59,247.76);
\end{scope}
\begin{scope}
\path[clip] (  0.00,  0.00) rectangle (469.75,433.62);
\definecolor{drawColor}{gray}{0.30}

\node[text=drawColor,anchor=base east,inner sep=0pt, outer sep=0pt, scale=  0.88] at (268.77,239.21) {0.00};

\node[text=drawColor,anchor=base east,inner sep=0pt, outer sep=0pt, scale=  0.88] at (268.77,287.09) {0.05};

\node[text=drawColor,anchor=base east,inner sep=0pt, outer sep=0pt, scale=  0.88] at (268.77,334.98) {0.10};

\node[text=drawColor,anchor=base east,inner sep=0pt, outer sep=0pt, scale=  0.88] at (268.77,382.86) {0.15};
\end{scope}
\begin{scope}
\path[clip] (  0.00,  0.00) rectangle (469.75,433.62);
\definecolor{drawColor}{gray}{0.20}

\path[draw=drawColor,line width= 0.6pt,line join=round] (270.97,242.24) --
	(273.72,242.24);

\path[draw=drawColor,line width= 0.6pt,line join=round] (270.97,290.12) --
	(273.72,290.12);

\path[draw=drawColor,line width= 0.6pt,line join=round] (270.97,338.01) --
	(273.72,338.01);

\path[draw=drawColor,line width= 0.6pt,line join=round] (270.97,385.89) --
	(273.72,385.89);
\end{scope}
\begin{scope}
\path[clip] (  0.00,  0.00) rectangle (469.75,433.62);
\definecolor{drawColor}{gray}{0.20}

\path[draw=drawColor,line width= 0.6pt,line join=round] (315.69,232.59) --
	(315.69,235.34);

\path[draw=drawColor,line width= 0.6pt,line join=round] (368.99,232.59) --
	(368.99,235.34);

\path[draw=drawColor,line width= 0.6pt,line join=round] (422.28,232.59) --
	(422.28,235.34);
\end{scope}
\begin{scope}
\path[clip] (  0.00,  0.00) rectangle (469.75,433.62);
\definecolor{drawColor}{gray}{0.30}

\node[text=drawColor,anchor=base,inner sep=0pt, outer sep=0pt, scale=  0.88] at (315.69,224.33) {-2};

\node[text=drawColor,anchor=base,inner sep=0pt, outer sep=0pt, scale=  0.88] at (368.99,224.33) {0};

\node[text=drawColor,anchor=base,inner sep=0pt, outer sep=0pt, scale=  0.88] at (422.28,224.33) {2};
\end{scope}
\begin{scope}
\path[clip] (  0.00,  0.00) rectangle (469.75,433.62);
\definecolor{drawColor}{RGB}{0,0,0}

\node[text=drawColor,anchor=base,inner sep=0pt, outer sep=0pt, scale=  1.10] at (368.99,212.01) {$x$};
\end{scope}
\begin{scope}
\path[clip] (  0.00,  0.00) rectangle (469.75,433.62);
\definecolor{drawColor}{RGB}{0,0,0}

\node[text=drawColor,rotate= 90.00,anchor=base,inner sep=0pt, outer sep=0pt, scale=  1.10] at (247.95,311.25) {length};
\end{scope}
\begin{scope}
\path[clip] (  0.00,  0.00) rectangle (469.75,433.62);
\definecolor{drawColor}{RGB}{0,0,0}

\node[text=drawColor,anchor=base west,inner sep=0pt, outer sep=0pt, scale=  1.10] at (273.72,395.09) {Interval length $H=0.55$, $N=200$};
\end{scope}
\begin{scope}
\path[clip] (  0.00,  0.00) rectangle (234.88,204.08);
\definecolor{drawColor}{RGB}{255,255,255}
\definecolor{fillColor}{RGB}{255,255,255}

\path[draw=drawColor,line width= 0.6pt,line join=round,line cap=round,fill=fillColor] (  0.00,  0.00) rectangle (234.88,204.08);
\end{scope}
\begin{scope}
\path[clip] ( 38.85, 31.25) rectangle (229.38,183.08);
\definecolor{fillColor}{gray}{0.92}

\path[fill=fillColor] ( 38.85, 31.25) rectangle (229.38,183.08);
\definecolor{drawColor}{RGB}{255,255,255}

\path[draw=drawColor,line width= 0.3pt,line join=round] ( 38.85, 56.32) --
	(229.38, 56.32);

\path[draw=drawColor,line width= 0.3pt,line join=round] ( 38.85, 92.64) --
	(229.38, 92.64);

\path[draw=drawColor,line width= 0.3pt,line join=round] ( 38.85,128.96) --
	(229.38,128.96);

\path[draw=drawColor,line width= 0.3pt,line join=round] ( 38.85,165.28) --
	(229.38,165.28);

\path[draw=drawColor,line width= 0.3pt,line join=round] ( 54.17, 31.25) --
	( 54.17,183.08);

\path[draw=drawColor,line width= 0.3pt,line join=round] (107.46, 31.25) --
	(107.46,183.08);

\path[draw=drawColor,line width= 0.3pt,line join=round] (160.76, 31.25) --
	(160.76,183.08);

\path[draw=drawColor,line width= 0.3pt,line join=round] (214.06, 31.25) --
	(214.06,183.08);

\path[draw=drawColor,line width= 0.6pt,line join=round] ( 38.85, 38.15) --
	(229.38, 38.15);

\path[draw=drawColor,line width= 0.6pt,line join=round] ( 38.85, 74.48) --
	(229.38, 74.48);

\path[draw=drawColor,line width= 0.6pt,line join=round] ( 38.85,110.80) --
	(229.38,110.80);

\path[draw=drawColor,line width= 0.6pt,line join=round] ( 38.85,147.12) --
	(229.38,147.12);

\path[draw=drawColor,line width= 0.6pt,line join=round] ( 80.82, 31.25) --
	( 80.82,183.08);

\path[draw=drawColor,line width= 0.6pt,line join=round] (134.11, 31.25) --
	(134.11,183.08);

\path[draw=drawColor,line width= 0.6pt,line join=round] (187.41, 31.25) --
	(187.41,183.08);
\definecolor{fillColor}{RGB}{146,129,188}

\path[fill=fillColor] ( 47.51, 38.15) rectangle ( 54.17, 89.44);

\path[fill=fillColor] ( 74.15, 38.15) rectangle ( 80.82,129.39);

\path[fill=fillColor] (100.80, 38.15) rectangle (107.46,164.92);

\path[fill=fillColor] (127.45, 38.15) rectangle (134.11,171.24);

\path[fill=fillColor] (154.10, 38.15) rectangle (160.76,164.19);

\path[fill=fillColor] (180.75, 38.15) rectangle (187.41,125.33);

\path[fill=fillColor] (207.39, 38.15) rectangle (214.06, 87.99);
\definecolor{fillColor}{RGB}{230,159,0}

\path[fill=fillColor] ( 54.17, 38.15) rectangle ( 60.83,144.07);

\path[fill=fillColor] ( 80.82, 38.15) rectangle ( 87.48,176.03);

\path[fill=fillColor] (107.46, 38.15) rectangle (114.13,173.49);

\path[fill=fillColor] (134.11, 38.15) rectangle (140.77,175.09);

\path[fill=fillColor] (160.76, 38.15) rectangle (167.42,175.89);

\path[fill=fillColor] (187.41, 38.15) rectangle (194.07,174.51);

\path[fill=fillColor] (214.06, 38.15) rectangle (220.72,145.16);
\definecolor{drawColor}{RGB}{169,169,169}

\path[draw=drawColor,line width= 0.6pt,dash pattern=on 4pt off 4pt ,line join=round] ( 38.85,176.18) -- (229.38,176.18);
\end{scope}
\begin{scope}
\path[clip] (  0.00,  0.00) rectangle (469.75,433.62);
\definecolor{drawColor}{gray}{0.30}

\node[text=drawColor,anchor=base east,inner sep=0pt, outer sep=0pt, scale=  0.88] at ( 33.90, 35.12) {0.00};

\node[text=drawColor,anchor=base east,inner sep=0pt, outer sep=0pt, scale=  0.88] at ( 33.90, 71.45) {0.25};

\node[text=drawColor,anchor=base east,inner sep=0pt, outer sep=0pt, scale=  0.88] at ( 33.90,107.77) {0.50};

\node[text=drawColor,anchor=base east,inner sep=0pt, outer sep=0pt, scale=  0.88] at ( 33.90,144.09) {0.75};
\end{scope}
\begin{scope}
\path[clip] (  0.00,  0.00) rectangle (469.75,433.62);
\definecolor{drawColor}{gray}{0.20}

\path[draw=drawColor,line width= 0.6pt,line join=round] ( 36.10, 38.15) --
	( 38.85, 38.15);

\path[draw=drawColor,line width= 0.6pt,line join=round] ( 36.10, 74.48) --
	( 38.85, 74.48);

\path[draw=drawColor,line width= 0.6pt,line join=round] ( 36.10,110.80) --
	( 38.85,110.80);

\path[draw=drawColor,line width= 0.6pt,line join=round] ( 36.10,147.12) --
	( 38.85,147.12);
\end{scope}
\begin{scope}
\path[clip] (  0.00,  0.00) rectangle (469.75,433.62);
\definecolor{drawColor}{gray}{0.20}

\path[draw=drawColor,line width= 0.6pt,line join=round] ( 80.82, 28.50) --
	( 80.82, 31.25);

\path[draw=drawColor,line width= 0.6pt,line join=round] (134.11, 28.50) --
	(134.11, 31.25);

\path[draw=drawColor,line width= 0.6pt,line join=round] (187.41, 28.50) --
	(187.41, 31.25);
\end{scope}
\begin{scope}
\path[clip] (  0.00,  0.00) rectangle (469.75,433.62);
\definecolor{drawColor}{gray}{0.30}

\node[text=drawColor,anchor=base,inner sep=0pt, outer sep=0pt, scale=  0.88] at ( 80.82, 20.24) {-2};

\node[text=drawColor,anchor=base,inner sep=0pt, outer sep=0pt, scale=  0.88] at (134.11, 20.24) {0};

\node[text=drawColor,anchor=base,inner sep=0pt, outer sep=0pt, scale=  0.88] at (187.41, 20.24) {2};
\end{scope}
\begin{scope}
\path[clip] (  0.00,  0.00) rectangle (469.75,433.62);
\definecolor{drawColor}{RGB}{0,0,0}

\node[text=drawColor,anchor=base,inner sep=0pt, outer sep=0pt, scale=  1.10] at (134.11,  7.93) {$x$};
\end{scope}
\begin{scope}
\path[clip] (  0.00,  0.00) rectangle (469.75,433.62);
\definecolor{drawColor}{RGB}{0,0,0}

\node[text=drawColor,rotate= 90.00,anchor=base,inner sep=0pt, outer sep=0pt, scale=  1.10] at ( 13.08,107.17) {rate};
\end{scope}
\begin{scope}
\path[clip] (  0.00,  0.00) rectangle (469.75,433.62);
\definecolor{drawColor}{RGB}{0,0,0}

\node[text=drawColor,anchor=base west,inner sep=0pt, outer sep=0pt, scale=  1.10] at ( 38.85,191.01) {Coverage rate $H=0.55$, $N=1000$};
\end{scope}
\begin{scope}
\path[clip] (234.88,  0.00) rectangle (469.75,204.08);
\definecolor{drawColor}{RGB}{255,255,255}
\definecolor{fillColor}{RGB}{255,255,255}

\path[draw=drawColor,line width= 0.6pt,line join=round,line cap=round,fill=fillColor] (234.88,  0.00) rectangle (469.76,204.08);
\end{scope}
\begin{scope}
\path[clip] (273.72, 31.25) rectangle (464.25,183.08);
\definecolor{fillColor}{gray}{0.92}

\path[fill=fillColor] (273.72, 31.25) rectangle (464.26,183.08);
\definecolor{drawColor}{RGB}{255,255,255}

\path[draw=drawColor,line width= 0.3pt,line join=round] (273.72, 57.91) --
	(464.25, 57.91);

\path[draw=drawColor,line width= 0.3pt,line join=round] (273.72, 97.43) --
	(464.25, 97.43);

\path[draw=drawColor,line width= 0.3pt,line join=round] (273.72,136.95) --
	(464.25,136.95);

\path[draw=drawColor,line width= 0.3pt,line join=round] (273.72,176.47) --
	(464.25,176.47);

\path[draw=drawColor,line width= 0.3pt,line join=round] (289.05, 31.25) --
	(289.05,183.08);

\path[draw=drawColor,line width= 0.3pt,line join=round] (342.34, 31.25) --
	(342.34,183.08);

\path[draw=drawColor,line width= 0.3pt,line join=round] (395.64, 31.25) --
	(395.64,183.08);

\path[draw=drawColor,line width= 0.3pt,line join=round] (448.93, 31.25) --
	(448.93,183.08);

\path[draw=drawColor,line width= 0.6pt,line join=round] (273.72, 38.15) --
	(464.25, 38.15);

\path[draw=drawColor,line width= 0.6pt,line join=round] (273.72, 77.67) --
	(464.25, 77.67);

\path[draw=drawColor,line width= 0.6pt,line join=round] (273.72,117.19) --
	(464.25,117.19);

\path[draw=drawColor,line width= 0.6pt,line join=round] (273.72,156.71) --
	(464.25,156.71);

\path[draw=drawColor,line width= 0.6pt,line join=round] (315.69, 31.25) --
	(315.69,183.08);

\path[draw=drawColor,line width= 0.6pt,line join=round] (368.99, 31.25) --
	(368.99,183.08);

\path[draw=drawColor,line width= 0.6pt,line join=round] (422.28, 31.25) --
	(422.28,183.08);
\definecolor{fillColor}{RGB}{146,129,188}

\path[fill=fillColor] (282.38, 38.15) rectangle (289.05, 39.69);

\path[fill=fillColor] (309.03, 38.15) rectangle (315.69, 56.83);

\path[fill=fillColor] (335.68, 38.15) rectangle (342.34,121.87);

\path[fill=fillColor] (362.33, 38.15) rectangle (368.99,176.18);

\path[fill=fillColor] (388.97, 38.15) rectangle (395.64,121.87);

\path[fill=fillColor] (415.62, 38.15) rectangle (422.28, 56.83);

\path[fill=fillColor] (442.27, 38.15) rectangle (448.93, 39.69);
\definecolor{fillColor}{RGB}{230,159,0}

\path[fill=fillColor] (289.05, 38.15) rectangle (295.71, 45.78);

\path[fill=fillColor] (315.69, 38.15) rectangle (322.36, 72.11);

\path[fill=fillColor] (342.34, 38.15) rectangle (349.00,105.14);

\path[fill=fillColor] (368.99, 38.15) rectangle (375.65,111.96);

\path[fill=fillColor] (395.64, 38.15) rectangle (402.30,105.21);

\path[fill=fillColor] (422.28, 38.15) rectangle (428.95, 72.09);

\path[fill=fillColor] (448.93, 38.15) rectangle (455.59, 45.79);
\end{scope}
\begin{scope}
\path[clip] (  0.00,  0.00) rectangle (469.75,433.62);
\definecolor{drawColor}{gray}{0.30}

\node[text=drawColor,anchor=base east,inner sep=0pt, outer sep=0pt, scale=  0.88] at (268.77, 35.12) {0.00};

\node[text=drawColor,anchor=base east,inner sep=0pt, outer sep=0pt, scale=  0.88] at (268.77, 74.64) {0.02};

\node[text=drawColor,anchor=base east,inner sep=0pt, outer sep=0pt, scale=  0.88] at (268.77,114.16) {0.04};

\node[text=drawColor,anchor=base east,inner sep=0pt, outer sep=0pt, scale=  0.88] at (268.77,153.68) {0.06};
\end{scope}
\begin{scope}
\path[clip] (  0.00,  0.00) rectangle (469.75,433.62);
\definecolor{drawColor}{gray}{0.20}

\path[draw=drawColor,line width= 0.6pt,line join=round] (270.97, 38.15) --
	(273.72, 38.15);

\path[draw=drawColor,line width= 0.6pt,line join=round] (270.97, 77.67) --
	(273.72, 77.67);

\path[draw=drawColor,line width= 0.6pt,line join=round] (270.97,117.19) --
	(273.72,117.19);

\path[draw=drawColor,line width= 0.6pt,line join=round] (270.97,156.71) --
	(273.72,156.71);
\end{scope}
\begin{scope}
\path[clip] (  0.00,  0.00) rectangle (469.75,433.62);
\definecolor{drawColor}{gray}{0.20}

\path[draw=drawColor,line width= 0.6pt,line join=round] (315.69, 28.50) --
	(315.69, 31.25);

\path[draw=drawColor,line width= 0.6pt,line join=round] (368.99, 28.50) --
	(368.99, 31.25);

\path[draw=drawColor,line width= 0.6pt,line join=round] (422.28, 28.50) --
	(422.28, 31.25);
\end{scope}
\begin{scope}
\path[clip] (  0.00,  0.00) rectangle (469.75,433.62);
\definecolor{drawColor}{gray}{0.30}

\node[text=drawColor,anchor=base,inner sep=0pt, outer sep=0pt, scale=  0.88] at (315.69, 20.24) {-2};

\node[text=drawColor,anchor=base,inner sep=0pt, outer sep=0pt, scale=  0.88] at (368.99, 20.24) {0};

\node[text=drawColor,anchor=base,inner sep=0pt, outer sep=0pt, scale=  0.88] at (422.28, 20.24) {2};
\end{scope}
\begin{scope}
\path[clip] (  0.00,  0.00) rectangle (469.75,433.62);
\definecolor{drawColor}{RGB}{0,0,0}

\node[text=drawColor,anchor=base,inner sep=0pt, outer sep=0pt, scale=  1.10] at (368.99,  7.93) {$x$};
\end{scope}
\begin{scope}
\path[clip] (  0.00,  0.00) rectangle (469.75,433.62);
\definecolor{drawColor}{RGB}{0,0,0}

\node[text=drawColor,rotate= 90.00,anchor=base,inner sep=0pt, outer sep=0pt, scale=  1.10] at (247.95,107.17) {length};
\end{scope}
\begin{scope}
\path[clip] (  0.00,  0.00) rectangle (469.75,433.62);
\definecolor{drawColor}{RGB}{0,0,0}

\node[text=drawColor,anchor=base west,inner sep=0pt, outer sep=0pt, scale=  1.10] at (273.72,191.01) {Interval length $H=0.55$, $N=1000$};
\end{scope}
\end{tikzpicture}

%% file: plots/normal_H=0.95_m=200_and_m=1000.tex
\begin{tikzpicture}[x=1pt,y=1pt]
\definecolor{fillColor}{RGB}{255,255,255}
\path[use as bounding box,fill=fillColor,fill opacity=0.00] (0,0) rectangle (469.75,433.62);
\begin{scope}
\path[clip] (  0.00,  0.00) rectangle (469.75,433.62);
\definecolor{fillColor}{RGB}{255,255,255}

\path[fill=fillColor] (167.74,408.17) rectangle (302.02,433.62);
\end{scope}
\begin{scope}
\path[clip] (  0.00,  0.00) rectangle (469.75,433.62);
\definecolor{drawColor}{RGB}{0,0,0}

\node[text=drawColor,anchor=base west,inner sep=0pt, outer sep=0pt, scale=  1.10] at (170,200.30) {method:};
\end{scope}
\begin{scope}
\path[clip] (  0.00,  0.00) rectangle (469.75,433.62);
\definecolor{fillColor}{gray}{0.95}

\path[fill=fillColor] (215.09,413.67) rectangle (229.54,428.12);
\end{scope}
\begin{scope}
\path[clip] (  0.00,  0.00) rectangle (469.75,433.62);
\definecolor{fillColor}{RGB}{146,129,188}

\path[fill=fillColor] (215.80,414.38) rectangle (228.83,427.41);
\end{scope}
\begin{scope}
\path[clip] (  0.00,  0.00) rectangle (469.75,433.62);
\definecolor{fillColor}{gray}{0.95}

\path[fill=fillColor] (265.27,413.67) rectangle (279.73,428.12);
\end{scope}
\begin{scope}
\path[clip] (  0.00,  0.00) rectangle (469.75,433.62);
\definecolor{fillColor}{RGB}{230,159,0}

\path[fill=fillColor] (265.99,414.38) rectangle (279.02,427.41);
\end{scope}
\begin{scope}
\path[clip] (  0.00,  0.00) rectangle (469.75,433.62);
\definecolor{drawColor}{RGB}{0,0,0}

\node[text=drawColor,anchor=base west,inner sep=0pt, outer sep=0pt, scale=  1.10] at (170,417.86) {method:};
\node[text=drawColor,anchor=base west,inner sep=0pt, outer sep=0pt, scale=  0.88] at (235.04,417.86) {asymp};
\end{scope}
\begin{scope}
\path[clip] (  0.00,  0.00) rectangle (469.75,433.62);
\definecolor{drawColor}{RGB}{0,0,0}

\node[text=drawColor,anchor=base west,inner sep=0pt, outer sep=0pt, scale=  0.88] at (285.23,417.86) {HOA};
\end{scope}
\begin{scope}
\path[clip] (  0.00,204.08) rectangle (234.88,408.17);
\definecolor{drawColor}{RGB}{255,255,255}
\definecolor{fillColor}{RGB}{255,255,255}

\path[draw=drawColor,line width= 0.6pt,line join=round,line cap=round,fill=fillColor] (  0.00,204.08) rectangle (234.88,408.17);
\end{scope}
\begin{scope}
\path[clip] ( 38.85,235.34) rectangle (229.38,387.16);
\definecolor{fillColor}{gray}{0.92}

\path[fill=fillColor] ( 38.85,235.34) rectangle (229.38,387.16);
\definecolor{drawColor}{RGB}{255,255,255}

\path[draw=drawColor,line width= 0.3pt,line join=round] ( 38.85,259.49) --
	(229.38,259.49);

\path[draw=drawColor,line width= 0.3pt,line join=round] ( 38.85,294.00) --
	(229.38,294.00);

\path[draw=drawColor,line width= 0.3pt,line join=round] ( 38.85,328.50) --
	(229.38,328.50);

\path[draw=drawColor,line width= 0.3pt,line join=round] ( 38.85,363.01) --
	(229.38,363.01);

\path[draw=drawColor,line width= 0.3pt,line join=round] ( 54.17,235.34) --
	( 54.17,387.16);

\path[draw=drawColor,line width= 0.3pt,line join=round] (107.46,235.34) --
	(107.46,387.16);

\path[draw=drawColor,line width= 0.3pt,line join=round] (160.76,235.34) --
	(160.76,387.16);

\path[draw=drawColor,line width= 0.3pt,line join=round] (214.06,235.34) --
	(214.06,387.16);

\path[draw=drawColor,line width= 0.6pt,line join=round] ( 38.85,242.24) --
	(229.38,242.24);

\path[draw=drawColor,line width= 0.6pt,line join=round] ( 38.85,276.74) --
	(229.38,276.74);

\path[draw=drawColor,line width= 0.6pt,line join=round] ( 38.85,311.25) --
	(229.38,311.25);

\path[draw=drawColor,line width= 0.6pt,line join=round] ( 38.85,345.75) --
	(229.38,345.75);

\path[draw=drawColor,line width= 0.6pt,line join=round] ( 38.85,380.26) --
	(229.38,380.26);

\path[draw=drawColor,line width= 0.6pt,line join=round] ( 80.82,235.34) --
	( 80.82,387.16);

\path[draw=drawColor,line width= 0.6pt,line join=round] (134.11,235.34) --
	(134.11,387.16);

\path[draw=drawColor,line width= 0.6pt,line join=round] (187.41,235.34) --
	(187.41,387.16);
\definecolor{fillColor}{RGB}{146,129,188}

\path[fill=fillColor] ( 47.51,242.24) rectangle ( 54.17,376.05);

\path[fill=fillColor] ( 74.15,242.24) rectangle ( 80.82,371.36);

\path[fill=fillColor] (100.80,242.24) rectangle (107.46,366.39);

\path[fill=fillColor] (127.45,242.24) rectangle (134.11,380.26);

\path[fill=fillColor] (154.10,242.24) rectangle (160.76,367.98);

\path[fill=fillColor] (180.75,242.24) rectangle (187.41,372.53);

\path[fill=fillColor] (207.39,242.24) rectangle (214.06,376.26);
\definecolor{fillColor}{RGB}{230,159,0}

\path[fill=fillColor] ( 54.17,242.24) rectangle ( 60.83,350.79);

\path[fill=fillColor] ( 80.82,242.24) rectangle ( 87.48,358.04);

\path[fill=fillColor] (107.46,242.24) rectangle (114.13,366.25);

\path[fill=fillColor] (134.11,242.24) rectangle (140.77,368.80);

\path[fill=fillColor] (160.76,242.24) rectangle (167.42,365.97);

\path[fill=fillColor] (187.41,242.24) rectangle (194.07,360.66);

\path[fill=fillColor] (214.06,242.24) rectangle (220.72,349.14);
\definecolor{drawColor}{RGB}{169,169,169}

\path[draw=drawColor,line width= 0.6pt,dash pattern=on 4pt off 4pt ,line join=round] ( 38.85,373.36) -- (229.38,373.36);
\end{scope}
\begin{scope}
\path[clip] (  0.00,  0.00) rectangle (469.75,433.62);
\definecolor{drawColor}{gray}{0.30}

\node[text=drawColor,anchor=base east,inner sep=0pt, outer sep=0pt, scale=  0.88] at ( 33.90,239.21) {0.00};

\node[text=drawColor,anchor=base east,inner sep=0pt, outer sep=0pt, scale=  0.88] at ( 33.90,273.71) {0.25};

\node[text=drawColor,anchor=base east,inner sep=0pt, outer sep=0pt, scale=  0.88] at ( 33.90,308.22) {0.50};

\node[text=drawColor,anchor=base east,inner sep=0pt, outer sep=0pt, scale=  0.88] at ( 33.90,342.72) {0.75};

\node[text=drawColor,anchor=base east,inner sep=0pt, outer sep=0pt, scale=  0.88] at ( 33.90,377.23) {1.00};
\end{scope}
\begin{scope}
\path[clip] (  0.00,  0.00) rectangle (469.75,433.62);
\definecolor{drawColor}{gray}{0.20}

\path[draw=drawColor,line width= 0.6pt,line join=round] ( 36.10,242.24) --
	( 38.85,242.24);

\path[draw=drawColor,line width= 0.6pt,line join=round] ( 36.10,276.74) --
	( 38.85,276.74);

\path[draw=drawColor,line width= 0.6pt,line join=round] ( 36.10,311.25) --
	( 38.85,311.25);

\path[draw=drawColor,line width= 0.6pt,line join=round] ( 36.10,345.75) --
	( 38.85,345.75);

\path[draw=drawColor,line width= 0.6pt,line join=round] ( 36.10,380.26) --
	( 38.85,380.26);
\end{scope}
\begin{scope}
\path[clip] (  0.00,  0.00) rectangle (469.75,433.62);
\definecolor{drawColor}{gray}{0.20}

\path[draw=drawColor,line width= 0.6pt,line join=round] ( 80.82,232.59) --
	( 80.82,235.34);

\path[draw=drawColor,line width= 0.6pt,line join=round] (134.11,232.59) --
	(134.11,235.34);

\path[draw=drawColor,line width= 0.6pt,line join=round] (187.41,232.59) --
	(187.41,235.34);
\end{scope}
\begin{scope}
\path[clip] (  0.00,  0.00) rectangle (469.75,433.62);
\definecolor{drawColor}{gray}{0.30}

\node[text=drawColor,anchor=base,inner sep=0pt, outer sep=0pt, scale=  0.88] at ( 80.82,224.33) {-2};

\node[text=drawColor,anchor=base,inner sep=0pt, outer sep=0pt, scale=  0.88] at (134.11,224.33) {0};

\node[text=drawColor,anchor=base,inner sep=0pt, outer sep=0pt, scale=  0.88] at (187.41,224.33) {2};
\end{scope}
\begin{scope}
\path[clip] (  0.00,  0.00) rectangle (469.75,433.62);
\definecolor{drawColor}{RGB}{0,0,0}

\node[text=drawColor,anchor=base,inner sep=0pt, outer sep=0pt, scale=  1.10] at (134.11,212.01) {$x$};
\end{scope}
\begin{scope}
\path[clip] (  0.00,  0.00) rectangle (469.75,433.62);
\definecolor{drawColor}{RGB}{0,0,0}

\node[text=drawColor,rotate= 90.00,anchor=base,inner sep=0pt, outer sep=0pt, scale=  1.10] at ( 13.08,311.25) {rate};
\end{scope}
\begin{scope}
\path[clip] (  0.00,  0.00) rectangle (469.75,433.62);
\definecolor{drawColor}{RGB}{0,0,0}

\node[text=drawColor,anchor=base west,inner sep=0pt, outer sep=0pt, scale=  1.10] at ( 38.85,395.09) {Coverage rate $H=0.95$, $N=200$};
\end{scope}
\begin{scope}
\path[clip] (234.88,204.08) rectangle (469.75,408.17);
\definecolor{drawColor}{RGB}{255,255,255}
\definecolor{fillColor}{RGB}{255,255,255}

\path[draw=drawColor,line width= 0.6pt,line join=round,line cap=round,fill=fillColor] (234.88,204.08) rectangle (469.76,408.17);
\end{scope}
\begin{scope}
\path[clip] (273.72,235.34) rectangle (464.25,387.16);
\definecolor{fillColor}{gray}{0.92}

\path[fill=fillColor] (273.72,235.34) rectangle (464.26,387.16);
\definecolor{drawColor}{RGB}{255,255,255}

\path[draw=drawColor,line width= 0.3pt,line join=round] (273.72,256.62) --
	(464.25,256.62);

\path[draw=drawColor,line width= 0.3pt,line join=round] (273.72,285.37) --
	(464.25,285.37);

\path[draw=drawColor,line width= 0.3pt,line join=round] (273.72,314.13) --
	(464.25,314.13);

\path[draw=drawColor,line width= 0.3pt,line join=round] (273.72,342.89) --
	(464.25,342.89);

\path[draw=drawColor,line width= 0.3pt,line join=round] (273.72,371.65) --
	(464.25,371.65);

\path[draw=drawColor,line width= 0.3pt,line join=round] (289.05,235.34) --
	(289.05,387.16);

\path[draw=drawColor,line width= 0.3pt,line join=round] (342.34,235.34) --
	(342.34,387.16);

\path[draw=drawColor,line width= 0.3pt,line join=round] (395.64,235.34) --
	(395.64,387.16);

\path[draw=drawColor,line width= 0.3pt,line join=round] (448.93,235.34) --
	(448.93,387.16);

\path[draw=drawColor,line width= 0.6pt,line join=round] (273.72,242.24) --
	(464.25,242.24);

\path[draw=drawColor,line width= 0.6pt,line join=round] (273.72,270.99) --
	(464.25,270.99);

\path[draw=drawColor,line width= 0.6pt,line join=round] (273.72,299.75) --
	(464.25,299.75);

\path[draw=drawColor,line width= 0.6pt,line join=round] (273.72,328.51) --
	(464.25,328.51);

\path[draw=drawColor,line width= 0.6pt,line join=round] (273.72,357.27) --
	(464.25,357.27);

\path[draw=drawColor,line width= 0.6pt,line join=round] (273.72,386.03) --
	(464.25,386.03);

\path[draw=drawColor,line width= 0.6pt,line join=round] (315.69,235.34) --
	(315.69,387.16);

\path[draw=drawColor,line width= 0.6pt,line join=round] (368.99,235.34) --
	(368.99,387.16);

\path[draw=drawColor,line width= 0.6pt,line join=round] (422.28,235.34) --
	(422.28,387.16);
\definecolor{fillColor}{RGB}{146,129,188}

\path[fill=fillColor] (282.38,242.24) rectangle (289.05,243.77);

\path[fill=fillColor] (309.03,242.24) rectangle (315.69,260.92);

\path[fill=fillColor] (335.68,242.24) rectangle (342.34,325.95);

\path[fill=fillColor] (362.33,242.24) rectangle (368.99,380.26);

\path[fill=fillColor] (388.97,242.24) rectangle (395.64,325.95);

\path[fill=fillColor] (415.62,242.24) rectangle (422.28,260.92);

\path[fill=fillColor] (442.27,242.24) rectangle (448.93,243.77);
\definecolor{fillColor}{RGB}{230,159,0}

\path[fill=fillColor] (289.05,242.24) rectangle (295.71,242.49);

\path[fill=fillColor] (315.69,242.24) rectangle (322.36,243.84);

\path[fill=fillColor] (342.34,242.24) rectangle (349.00,246.64);

\path[fill=fillColor] (368.99,242.24) rectangle (375.65,248.48);

\path[fill=fillColor] (395.64,242.24) rectangle (402.30,246.67);

\path[fill=fillColor] (422.28,242.24) rectangle (428.95,243.82);

\path[fill=fillColor] (448.93,242.24) rectangle (455.59,242.51);
\end{scope}
\begin{scope}
\path[clip] (  0.00,  0.00) rectangle (469.75,433.62);
\definecolor{drawColor}{gray}{0.30}

\node[text=drawColor,anchor=base east,inner sep=0pt, outer sep=0pt, scale=  0.88] at (268.77,239.21) {0.00};

\node[text=drawColor,anchor=base east,inner sep=0pt, outer sep=0pt, scale=  0.88] at (268.77,267.96) {0.25};

\node[text=drawColor,anchor=base east,inner sep=0pt, outer sep=0pt, scale=  0.88] at (268.77,296.72) {0.50};

\node[text=drawColor,anchor=base east,inner sep=0pt, outer sep=0pt, scale=  0.88] at (268.77,325.48) {0.75};

\node[text=drawColor,anchor=base east,inner sep=0pt, outer sep=0pt, scale=  0.88] at (268.77,354.24) {1.00};

\node[text=drawColor,anchor=base east,inner sep=0pt, outer sep=0pt, scale=  0.88] at (268.77,383.00) {1.25};
\end{scope}
\begin{scope}
\path[clip] (  0.00,  0.00) rectangle (469.75,433.62);
\definecolor{drawColor}{gray}{0.20}

\path[draw=drawColor,line width= 0.6pt,line join=round] (270.97,242.24) --
	(273.72,242.24);

\path[draw=drawColor,line width= 0.6pt,line join=round] (270.97,270.99) --
	(273.72,270.99);

\path[draw=drawColor,line width= 0.6pt,line join=round] (270.97,299.75) --
	(273.72,299.75);

\path[draw=drawColor,line width= 0.6pt,line join=round] (270.97,328.51) --
	(273.72,328.51);

\path[draw=drawColor,line width= 0.6pt,line join=round] (270.97,357.27) --
	(273.72,357.27);

\path[draw=drawColor,line width= 0.6pt,line join=round] (270.97,386.03) --
	(273.72,386.03);
\end{scope}
\begin{scope}
\path[clip] (  0.00,  0.00) rectangle (469.75,433.62);
\definecolor{drawColor}{gray}{0.20}

\path[draw=drawColor,line width= 0.6pt,line join=round] (315.69,232.59) --
	(315.69,235.34);

\path[draw=drawColor,line width= 0.6pt,line join=round] (368.99,232.59) --
	(368.99,235.34);

\path[draw=drawColor,line width= 0.6pt,line join=round] (422.28,232.59) --
	(422.28,235.34);
\end{scope}
\begin{scope}
\path[clip] (  0.00,  0.00) rectangle (469.75,433.62);
\definecolor{drawColor}{gray}{0.30}

\node[text=drawColor,anchor=base,inner sep=0pt, outer sep=0pt, scale=  0.88] at (315.69,224.33) {-2};

\node[text=drawColor,anchor=base,inner sep=0pt, outer sep=0pt, scale=  0.88] at (368.99,224.33) {0};

\node[text=drawColor,anchor=base,inner sep=0pt, outer sep=0pt, scale=  0.88] at (422.28,224.33) {2};
\end{scope}
\begin{scope}
\path[clip] (  0.00,  0.00) rectangle (469.75,433.62);
\definecolor{drawColor}{RGB}{0,0,0}

\node[text=drawColor,anchor=base,inner sep=0pt, outer sep=0pt, scale=  1.10] at (368.99,212.01) {$x$};
\end{scope}
\begin{scope}
\path[clip] (  0.00,  0.00) rectangle (469.75,433.62);
\definecolor{drawColor}{RGB}{0,0,0}

\node[text=drawColor,rotate= 90.00,anchor=base,inner sep=0pt, outer sep=0pt, scale=  1.10] at (247.95,311.25) {length};
\end{scope}
\begin{scope}
\path[clip] (  0.00,  0.00) rectangle (469.75,433.62);
\definecolor{drawColor}{RGB}{0,0,0}

\node[text=drawColor,anchor=base west,inner sep=0pt, outer sep=0pt, scale=  1.10] at (273.72,395.09) {Interval length $H=0.95$, $N=200$};
\end{scope}
\begin{scope}
\path[clip] (  0.00,  0.00) rectangle (234.88,204.08);
\definecolor{drawColor}{RGB}{255,255,255}
\definecolor{fillColor}{RGB}{255,255,255}

\path[draw=drawColor,line width= 0.6pt,line join=round,line cap=round,fill=fillColor] (  0.00,  0.00) rectangle (234.88,204.08);
\end{scope}
\begin{scope}
\path[clip] ( 38.85, 31.25) rectangle (229.38,183.08);
\definecolor{fillColor}{gray}{0.92}

\path[fill=fillColor] ( 38.85, 31.25) rectangle (229.38,183.08);
\definecolor{drawColor}{RGB}{255,255,255}

\path[draw=drawColor,line width= 0.3pt,line join=round] ( 38.85, 55.41) --
	(229.38, 55.41);

\path[draw=drawColor,line width= 0.3pt,line join=round] ( 38.85, 89.91) --
	(229.38, 89.91);

\path[draw=drawColor,line width= 0.3pt,line join=round] ( 38.85,124.42) --
	(229.38,124.42);

\path[draw=drawColor,line width= 0.3pt,line join=round] ( 38.85,158.92) --
	(229.38,158.92);

\path[draw=drawColor,line width= 0.3pt,line join=round] ( 54.17, 31.25) --
	( 54.17,183.08);

\path[draw=drawColor,line width= 0.3pt,line join=round] (107.46, 31.25) --
	(107.46,183.08);

\path[draw=drawColor,line width= 0.3pt,line join=round] (160.76, 31.25) --
	(160.76,183.08);

\path[draw=drawColor,line width= 0.3pt,line join=round] (214.06, 31.25) --
	(214.06,183.08);

\path[draw=drawColor,line width= 0.6pt,line join=round] ( 38.85, 38.15) --
	(229.38, 38.15);

\path[draw=drawColor,line width= 0.6pt,line join=round] ( 38.85, 72.66) --
	(229.38, 72.66);

\path[draw=drawColor,line width= 0.6pt,line join=round] ( 38.85,107.17) --
	(229.38,107.17);

\path[draw=drawColor,line width= 0.6pt,line join=round] ( 38.85,141.67) --
	(229.38,141.67);

\path[draw=drawColor,line width= 0.6pt,line join=round] ( 38.85,176.18) --
	(229.38,176.18);

\path[draw=drawColor,line width= 0.6pt,line join=round] ( 80.82, 31.25) --
	( 80.82,183.08);

\path[draw=drawColor,line width= 0.6pt,line join=round] (134.11, 31.25) --
	(134.11,183.08);

\path[draw=drawColor,line width= 0.6pt,line join=round] (187.41, 31.25) --
	(187.41,183.08);
\definecolor{fillColor}{RGB}{146,129,188}

\path[fill=fillColor] ( 47.51, 38.15) rectangle ( 54.17,171.69);

\path[fill=fillColor] ( 74.15, 38.15) rectangle ( 80.82,166.86);

\path[fill=fillColor] (100.80, 38.15) rectangle (107.46,164.17);

\path[fill=fillColor] (127.45, 38.15) rectangle (134.11,176.18);

\path[fill=fillColor] (154.10, 38.15) rectangle (160.76,165.89);

\path[fill=fillColor] (180.75, 38.15) rectangle (187.41,165.69);

\path[fill=fillColor] (207.39, 38.15) rectangle (214.06,170.93);
\definecolor{fillColor}{RGB}{230,159,0}

\path[fill=fillColor] ( 54.17, 38.15) rectangle ( 60.83,148.02);

\path[fill=fillColor] ( 80.82, 38.15) rectangle ( 87.48,159.61);

\path[fill=fillColor] (107.46, 38.15) rectangle (114.13,163.41);

\path[fill=fillColor] (134.11, 38.15) rectangle (140.77,163.75);

\path[fill=fillColor] (160.76, 38.15) rectangle (167.42,161.68);

\path[fill=fillColor] (187.41, 38.15) rectangle (194.07,157.61);

\path[fill=fillColor] (214.06, 38.15) rectangle (220.72,143.60);
\definecolor{drawColor}{RGB}{169,169,169}

\path[draw=drawColor,line width= 0.6pt,dash pattern=on 4pt off 4pt ,line join=round] ( 38.85,169.28) -- (229.38,169.28);
\end{scope}
\begin{scope}
\path[clip] (  0.00,  0.00) rectangle (469.75,433.62);
\definecolor{drawColor}{gray}{0.30}

\node[text=drawColor,anchor=base east,inner sep=0pt, outer sep=0pt, scale=  0.88] at ( 33.90, 35.12) {0.00};

\node[text=drawColor,anchor=base east,inner sep=0pt, outer sep=0pt, scale=  0.88] at ( 33.90, 69.63) {0.25};

\node[text=drawColor,anchor=base east,inner sep=0pt, outer sep=0pt, scale=  0.88] at ( 33.90,104.14) {0.50};

\node[text=drawColor,anchor=base east,inner sep=0pt, outer sep=0pt, scale=  0.88] at ( 33.90,138.64) {0.75};

\node[text=drawColor,anchor=base east,inner sep=0pt, outer sep=0pt, scale=  0.88] at ( 33.90,173.15) {1.00};
\end{scope}
\begin{scope}
\path[clip] (  0.00,  0.00) rectangle (469.75,433.62);
\definecolor{drawColor}{gray}{0.20}

\path[draw=drawColor,line width= 0.6pt,line join=round] ( 36.10, 38.15) --
	( 38.85, 38.15);

\path[draw=drawColor,line width= 0.6pt,line join=round] ( 36.10, 72.66) --
	( 38.85, 72.66);

\path[draw=drawColor,line width= 0.6pt,line join=round] ( 36.10,107.17) --
	( 38.85,107.17);

\path[draw=drawColor,line width= 0.6pt,line join=round] ( 36.10,141.67) --
	( 38.85,141.67);

\path[draw=drawColor,line width= 0.6pt,line join=round] ( 36.10,176.18) --
	( 38.85,176.18);
\end{scope}
\begin{scope}
\path[clip] (  0.00,  0.00) rectangle (469.75,433.62);
\definecolor{drawColor}{gray}{0.20}

\path[draw=drawColor,line width= 0.6pt,line join=round] ( 80.82, 28.50) --
	( 80.82, 31.25);

\path[draw=drawColor,line width= 0.6pt,line join=round] (134.11, 28.50) --
	(134.11, 31.25);

\path[draw=drawColor,line width= 0.6pt,line join=round] (187.41, 28.50) --
	(187.41, 31.25);
\end{scope}
\begin{scope}
\path[clip] (  0.00,  0.00) rectangle (469.75,433.62);
\definecolor{drawColor}{gray}{0.30}

\node[text=drawColor,anchor=base,inner sep=0pt, outer sep=0pt, scale=  0.88] at ( 80.82, 20.24) {-2};

\node[text=drawColor,anchor=base,inner sep=0pt, outer sep=0pt, scale=  0.88] at (134.11, 20.24) {0};

\node[text=drawColor,anchor=base,inner sep=0pt, outer sep=0pt, scale=  0.88] at (187.41, 20.24) {2};
\end{scope}
\begin{scope}
\path[clip] (  0.00,  0.00) rectangle (469.75,433.62);
\definecolor{drawColor}{RGB}{0,0,0}

\node[text=drawColor,anchor=base,inner sep=0pt, outer sep=0pt, scale=  1.10] at (134.11,  7.93) {$x$};
\end{scope}
\begin{scope}
\path[clip] (  0.00,  0.00) rectangle (469.75,433.62);
\definecolor{drawColor}{RGB}{0,0,0}

\node[text=drawColor,rotate= 90.00,anchor=base,inner sep=0pt, outer sep=0pt, scale=  1.10] at ( 13.08,107.17) {rate};
\end{scope}
\begin{scope}
\path[clip] (  0.00,  0.00) rectangle (469.75,433.62);
\definecolor{drawColor}{RGB}{0,0,0}

\node[text=drawColor,anchor=base west,inner sep=0pt, outer sep=0pt, scale=  1.10] at ( 38.85,191.01) {Coverage rate $H=0.95$, $N=1000$};
\end{scope}
\begin{scope}
\path[clip] (234.88,  0.00) rectangle (469.75,204.08);
\definecolor{drawColor}{RGB}{255,255,255}
\definecolor{fillColor}{RGB}{255,255,255}

\path[draw=drawColor,line width= 0.6pt,line join=round,line cap=round,fill=fillColor] (234.88,  0.00) rectangle (469.76,204.08);
\end{scope}
\begin{scope}
\path[clip] (269.32, 31.25) rectangle (464.25,183.08);
\definecolor{fillColor}{gray}{0.92}

\path[fill=fillColor] (269.32, 31.25) rectangle (464.25,183.08);
\definecolor{drawColor}{RGB}{255,255,255}

\path[draw=drawColor,line width= 0.3pt,line join=round] (269.32, 56.85) --
	(464.25, 56.85);

\path[draw=drawColor,line width= 0.3pt,line join=round] (269.32, 94.26) --
	(464.25, 94.26);

\path[draw=drawColor,line width= 0.3pt,line join=round] (269.32,131.66) --
	(464.25,131.66);

\path[draw=drawColor,line width= 0.3pt,line join=round] (269.32,169.06) --
	(464.25,169.06);

\path[draw=drawColor,line width= 0.3pt,line join=round] (285.00, 31.25) --
	(285.00,183.08);

\path[draw=drawColor,line width= 0.3pt,line join=round] (339.53, 31.25) --
	(339.53,183.08);

\path[draw=drawColor,line width= 0.3pt,line join=round] (394.05, 31.25) --
	(394.05,183.08);

\path[draw=drawColor,line width= 0.3pt,line join=round] (448.58, 31.25) --
	(448.58,183.08);

\path[draw=drawColor,line width= 0.6pt,line join=round] (269.32, 38.15) --
	(464.25, 38.15);

\path[draw=drawColor,line width= 0.6pt,line join=round] (269.32, 75.56) --
	(464.25, 75.56);

\path[draw=drawColor,line width= 0.6pt,line join=round] (269.32,112.96) --
	(464.25,112.96);

\path[draw=drawColor,line width= 0.6pt,line join=round] (269.32,150.36) --
	(464.25,150.36);

\path[draw=drawColor,line width= 0.6pt,line join=round] (312.26, 31.25) --
	(312.26,183.08);

\path[draw=drawColor,line width= 0.6pt,line join=round] (366.79, 31.25) --
	(366.79,183.08);

\path[draw=drawColor,line width= 0.6pt,line join=round] (421.32, 31.25) --
	(421.32,183.08);
\definecolor{fillColor}{RGB}{146,129,188}

\path[fill=fillColor] (278.18, 38.15) rectangle (285.00, 39.69);

\path[fill=fillColor] (305.45, 38.15) rectangle (312.26, 56.83);

\path[fill=fillColor] (332.71, 38.15) rectangle (339.53,121.87);

\path[fill=fillColor] (359.97, 38.15) rectangle (366.79,176.18);

\path[fill=fillColor] (387.24, 38.15) rectangle (394.05,121.87);

\path[fill=fillColor] (414.50, 38.15) rectangle (421.32, 56.83);

\path[fill=fillColor] (441.76, 38.15) rectangle (448.58, 39.69);
\definecolor{fillColor}{RGB}{230,159,0}

\path[fill=fillColor] (285.00, 38.15) rectangle (291.82, 38.32);

\path[fill=fillColor] (312.26, 38.15) rectangle (319.08, 38.97);

\path[fill=fillColor] (339.53, 38.15) rectangle (346.34, 40.36);

\path[fill=fillColor] (366.79, 38.15) rectangle (373.61, 41.21);

\path[fill=fillColor] (394.05, 38.15) rectangle (400.87, 40.35);

\path[fill=fillColor] (421.32, 38.15) rectangle (428.13, 39.02);

\path[fill=fillColor] (448.58, 38.15) rectangle (455.39, 38.31);
\end{scope}
\begin{scope}
\path[clip] (  0.00,  0.00) rectangle (469.75,433.62);
\definecolor{drawColor}{gray}{0.30}

\node[text=drawColor,anchor=base east,inner sep=0pt, outer sep=0pt, scale=  0.88] at (264.37, 35.12) {0.0};

\node[text=drawColor,anchor=base east,inner sep=0pt, outer sep=0pt, scale=  0.88] at (264.37, 72.52) {0.3};

\node[text=drawColor,anchor=base east,inner sep=0pt, outer sep=0pt, scale=  0.88] at (264.37,109.93) {0.6};

\node[text=drawColor,anchor=base east,inner sep=0pt, outer sep=0pt, scale=  0.88] at (264.37,147.33) {0.9};
\end{scope}
\begin{scope}
\path[clip] (  0.00,  0.00) rectangle (469.75,433.62);
\definecolor{drawColor}{gray}{0.20}

\path[draw=drawColor,line width= 0.6pt,line join=round] (266.57, 38.15) --
	(269.32, 38.15);

\path[draw=drawColor,line width= 0.6pt,line join=round] (266.57, 75.56) --
	(269.32, 75.56);

\path[draw=drawColor,line width= 0.6pt,line join=round] (266.57,112.96) --
	(269.32,112.96);

\path[draw=drawColor,line width= 0.6pt,line join=round] (266.57,150.36) --
	(269.32,150.36);
\end{scope}
\begin{scope}
\path[clip] (  0.00,  0.00) rectangle (469.75,433.62);
\definecolor{drawColor}{gray}{0.20}

\path[draw=drawColor,line width= 0.6pt,line join=round] (312.26, 28.50) --
	(312.26, 31.25);

\path[draw=drawColor,line width= 0.6pt,line join=round] (366.79, 28.50) --
	(366.79, 31.25);

\path[draw=drawColor,line width= 0.6pt,line join=round] (421.32, 28.50) --
	(421.32, 31.25);
\end{scope}
\begin{scope}
\path[clip] (  0.00,  0.00) rectangle (469.75,433.62);
\definecolor{drawColor}{gray}{0.30}

\node[text=drawColor,anchor=base,inner sep=0pt, outer sep=0pt, scale=  0.88] at (312.26, 20.24) {-2};

\node[text=drawColor,anchor=base,inner sep=0pt, outer sep=0pt, scale=  0.88] at (366.79, 20.24) {0};

\node[text=drawColor,anchor=base,inner sep=0pt, outer sep=0pt, scale=  0.88] at (421.32, 20.24) {2};
\end{scope}
\begin{scope}
\path[clip] (  0.00,  0.00) rectangle (469.75,433.62);
\definecolor{drawColor}{RGB}{0,0,0}

\node[text=drawColor,anchor=base,inner sep=0pt, outer sep=0pt, scale=  1.10] at (366.79,  7.93) {$x$};
\end{scope}
\begin{scope}
\path[clip] (  0.00,  0.00) rectangle (469.75,433.62);
\definecolor{drawColor}{RGB}{0,0,0}

\node[text=drawColor,rotate= 90.00,anchor=base,inner sep=0pt, outer sep=0pt, scale=  1.10] at (247.95,107.17) {length};
\end{scope}
\begin{scope}
\path[clip] (  0.00,  0.00) rectangle (469.75,433.62);
\definecolor{drawColor}{RGB}{0,0,0}

\node[text=drawColor,anchor=base west,inner sep=0pt, outer sep=0pt, scale=  1.10] at (269.32,191.01) {Interval length $H=0.95$, $N=1000$};
\end{scope}
\end{tikzpicture}

%% file: plots/plot.tex
\begin{tikzpicture}[x=1pt,y=1pt]
\definecolor{fillColor}{RGB}{255,255,255}
\path[use as bounding box,fill=fillColor,fill opacity=0.00] (0,0) rectangle (469.75,216.81);
\begin{scope}
\path[clip] (  0.00,  0.00) rectangle (469.75,216.81);
\definecolor{fillColor}{RGB}{255,255,255}
\path[fill=fillColor] (162.74,192.78) rectangle (307.02,221.13);
\end{scope}
\begin{scope}
\path[clip] (  0.00,  0.00) rectangle (469.75,216.81);
\definecolor{drawColor}{RGB}{0,0,0}

\node[text=drawColor,anchor=base west,inner sep=0pt, outer sep=0pt, scale=  0.88] at (165.24,200.30) {};
\end{scope}
\begin{scope}
\path[clip] (  0.00,  0.00) rectangle (469.75,216.81);
\definecolor{fillColor}{gray}{0.95}

\path[fill=fillColor] (215.09,196.86) rectangle (229.54,211.31);
\end{scope}
\begin{scope}
\path[clip] (  0.00,  0.00) rectangle (469.75,216.81);
\definecolor{fillColor}{RGB}{146,129,188}

\path[fill=fillColor, fill opacity=0.40] (215.80,197.57) rectangle (228.83,210.60);
\end{scope}
\begin{scope}
\path[clip] (  0.00,  0.00) rectangle (469.75,216.81);
\definecolor{fillColor}{gray}{0.95}

\path[fill=fillColor] (265.27,196.86) rectangle (279.73,211.31);
\end{scope}
\begin{scope}
\path[clip] (  0.00,  0.00) rectangle (469.75,216.81);
\definecolor{fillColor}{RGB}{230,159,0}

\path[fill=fillColor, fill opacity=0.40] (265.99,197.57) rectangle (279.02,210.60);
\end{scope}
\begin{scope}
\path[clip] (  0.00,  0.00) rectangle (469.75,216.81);
\definecolor{drawColor}{RGB}{0,0,0}

\node[text=drawColor,anchor=base west,inner sep=0pt, outer sep=0pt, scale=  1.10] at (170,200.30) {method:};
\node[text=drawColor,anchor=base west,inner sep=0pt, outer sep=0pt, scale=  0.88] at (235.0,201.05) {asymp};
\end{scope}
\begin{scope}
\path[clip] (  0.00,  0.00) rectangle (469.75,216.81);
\definecolor{drawColor}{RGB}{0,0,0}
\node[text=drawColor,anchor=base west,inner sep=0pt, outer sep=0pt, scale=  0.88] at (285.23,201.05) {HOA};
\end{scope}
\begin{scope}
\path[clip] (  0.00,  0.00) rectangle (234.88,191.36);
\definecolor{drawColor}{RGB}{255,255,255}
\definecolor{fillColor}{RGB}{255,255,255}

\path[draw=drawColor,line width= 0.6pt,line join=round,line cap=round,fill=fillColor] (  0.00,  0.00) rectangle (156.59,197.10);
\end{scope}

\begin{scope}
\path[clip] (  0.00,  0.00) rectangle (469.75,216.81);
\definecolor{fillColor}{RGB}{230,159,0}

\definecolor{drawColor}{RGB}{51,102,255}

\end{scope}

\begin{scope}
\path[clip] ( 21.01, 21.01) rectangle (151.09,150.14);
\definecolor{fillColor}{gray}{0.92}

\path[fill=fillColor] ( 21.01, 21.01) rectangle (151.08,150.14);
\definecolor{drawColor}{RGB}{255,255,255}

\path[draw=drawColor,line width= 0.3pt,line join=round] ( 21.01, 41.55) --
	(151.09, 41.55);

\path[draw=drawColor,line width= 0.3pt,line join=round] ( 21.01, 70.90) --
	(151.09, 70.90);

\path[draw=drawColor,line width= 0.3pt,line join=round] ( 21.01,100.25) --
	(151.09,100.25);

\path[draw=drawColor,line width= 0.3pt,line join=round] ( 21.01,129.59) --
	(151.09,129.59);

\path[draw=drawColor,line width= 0.3pt,line join=round] ( 26.92, 21.01) --
	( 26.92,150.14);

\path[draw=drawColor,line width= 0.3pt,line join=round] ( 50.57, 21.01) --
	( 50.57,150.14);

\path[draw=drawColor,line width= 0.3pt,line join=round] ( 74.22, 21.01) --
	( 74.22,150.14);

\path[draw=drawColor,line width= 0.3pt,line join=round] ( 97.87, 21.01) --
	( 97.87,150.14);

\path[draw=drawColor,line width= 0.3pt,line join=round] (121.52, 21.01) --
	(121.52,150.14);

\path[draw=drawColor,line width= 0.3pt,line join=round] (145.17, 21.01) --
	(145.17,150.14);

\path[draw=drawColor,line width= 0.6pt,line join=round] ( 21.01, 26.87) --
	(151.09, 26.87);

\path[draw=drawColor,line width= 0.6pt,line join=round] ( 21.01, 56.22) --
	(151.09, 56.22);

\path[draw=drawColor,line width= 0.6pt,line join=round] ( 21.01, 85.57) --
	(151.09, 85.57);

\path[draw=drawColor,line width= 0.6pt,line join=round] ( 21.01,114.92) --
	(151.09,114.92);

\path[draw=drawColor,line width= 0.6pt,line join=round] ( 21.01,144.27) --
	(151.09,144.27);

\path[draw=drawColor,line width= 0.6pt,line join=round] ( 38.74, 21.01) --
	( 38.74,150.14);

\path[draw=drawColor,line width= 0.6pt,line join=round] ( 62.39, 21.01) --
	( 62.39,150.14);

\path[draw=drawColor,line width= 0.6pt,line join=round] ( 86.05, 21.01) --
	( 86.05,150.14);

\path[draw=drawColor,line width= 0.6pt,line join=round] (109.70, 21.01) --
	(109.70,150.14);

\path[draw=drawColor,line width= 0.6pt,line join=round] (133.35, 21.01) --
	(133.35,150.14);
\definecolor{fillColor}{RGB}{146,129,188}

\path[fill=fillColor,fill opacity=0.40] ( 26.92, 28.25) --
	( 29.28, 28.61) --
	( 31.65, 29.17) --
	( 34.01, 29.80) --
	( 36.38, 30.60) --
	( 38.74, 31.50) --
	( 41.11, 32.65) --
	( 43.47, 33.85) --
	( 45.84, 35.55) --
	( 48.20, 37.29) --
	( 50.57, 39.57) --
	( 52.93, 41.84) --
	( 55.30, 44.56) --
	( 57.66, 47.48) --
	( 60.03, 50.77) --
	( 62.39, 54.26) --
	( 64.76, 57.95) --
	( 67.12, 62.57) --
	( 69.49, 66.66) --
	( 71.85, 71.79) --
	( 74.22, 75.83) --
	( 76.58, 80.10) --
	( 78.95, 85.91) --
	( 81.31, 90.52) --
	( 83.68, 95.38) --
	( 86.05,100.20) --
	( 88.41,104.64) --
	( 90.78,108.99) --
	( 93.14,113.39) --
	( 95.51,117.25) --
	( 97.87,120.93) --
	(100.24,124.15) --
	(102.60,127.01) --
	(104.97,129.86) --
	(107.33,132.83) --
	(109.70,134.46) --
	(112.06,136.24) --
	(114.43,137.80) --
	(116.79,139.13) --
	(119.16,140.24) --
	(121.52,141.17) --
	(123.89,141.88) --
	(126.25,142.46) --
	(128.62,142.90) --
	(130.98,143.36) --
	(133.35,143.55) --
	(135.71,143.77) --
	(138.08,143.90) --
	(140.44,144.05) --
	(142.81,144.10) --
	(145.17,144.15) --
	(145.17,142.90) --
	(142.81,142.51) --
	(140.44,142.03) --
	(138.08,141.28) --
	(135.71,140.56) --
	(133.35,139.68) --
	(130.98,138.47) --
	(128.62,137.22) --
	(126.25,135.54) --
	(123.89,133.84) --
	(121.52,131.61) --
	(119.16,129.31) --
	(116.79,126.54) --
	(114.43,123.53) --
	(112.06,120.30) --
	(109.70,116.93) --
	(107.33,112.71) --
	(104.97,109.11) --
	(102.60,104.39) --
	(100.24, 99.80) --
	( 97.87, 95.01) --
	( 95.51, 90.47) --
	( 93.14, 85.40) --
	( 90.78, 80.55) --
	( 88.41, 74.97) --
	( 86.05, 71.15) --
	( 83.68, 66.40) --
	( 81.31, 61.79) --
	( 78.95, 58.03) --
	( 76.58, 53.67) --
	( 74.22, 50.22) --
	( 71.85, 46.86) --
	( 69.49, 43.81) --
	( 67.12, 41.20) --
	( 64.76, 38.93) --
	( 62.39, 36.60) --
	( 60.03, 34.87) --
	( 57.66, 33.36) --
	( 55.30, 32.12) --
	( 52.93, 30.71) --
	( 50.57, 30.06) --
	( 48.20, 29.10) --
	( 45.84, 28.66) --
	( 43.47, 28.18) --
	( 41.11, 27.84) --
	( 38.74, 27.52) --
	( 36.38, 27.38) --
	( 34.01, 27.28) --
	( 31.65, 27.20) --
	( 29.28, 27.04) --
	( 26.92, 26.92) --
	cycle;

\path[] ( 26.92, 28.25) --
	( 29.28, 28.61) --
	( 31.65, 29.17) --
	( 34.01, 29.80) --
	( 36.38, 30.60) --
	( 38.74, 31.50) --
	( 41.11, 32.65) --
	( 43.47, 33.85) --
	( 45.84, 35.55) --
	( 48.20, 37.29) --
	( 50.57, 39.57) --
	( 52.93, 41.84) --
	( 55.30, 44.56) --
	( 57.66, 47.48) --
	( 60.03, 50.77) --
	( 62.39, 54.26) --
	( 64.76, 57.95) --
	( 67.12, 62.57) --
	( 69.49, 66.66) --
	( 71.85, 71.79) --
	( 74.22, 75.83) --
	( 76.58, 80.10) --
	( 78.95, 85.91) --
	( 81.31, 90.52) --
	( 83.68, 95.38) --
	( 86.05,100.20) --
	( 88.41,104.64) --
	( 90.78,108.99) --
	( 93.14,113.39) --
	( 95.51,117.25) --
	( 97.87,120.93) --
	(100.24,124.15) --
	(102.60,127.01) --
	(104.97,129.86) --
	(107.33,132.83) --
	(109.70,134.46) --
	(112.06,136.24) --
	(114.43,137.80) --
	(116.79,139.13) --
	(119.16,140.24) --
	(121.52,141.17) --
	(123.89,141.88) --
	(126.25,142.46) --
	(128.62,142.90) --
	(130.98,143.36) --
	(133.35,143.55) --
	(135.71,143.77) --
	(138.08,143.90) --
	(140.44,144.05) --
	(142.81,144.10) --
	(145.17,144.15);

\path[] (145.17,142.90) --
	(142.81,142.51) --
	(140.44,142.03) --
	(138.08,141.28) --
	(135.71,140.56) --
	(133.35,139.68) --
	(130.98,138.47) --
	(128.62,137.22) --
	(126.25,135.54) --
	(123.89,133.84) --
	(121.52,131.61) --
	(119.16,129.31) --
	(116.79,126.54) --
	(114.43,123.53) --
	(112.06,120.30) --
	(109.70,116.93) --
	(107.33,112.71) --
	(104.97,109.11) --
	(102.60,104.39) --
	(100.24, 99.80) --
	( 97.87, 95.01) --
	( 95.51, 90.47) --
	( 93.14, 85.40) --
	( 90.78, 80.55) --
	( 88.41, 74.97) --
	( 86.05, 71.15) --
	( 83.68, 66.40) --
	( 81.31, 61.79) --
	( 78.95, 58.03) --
	( 76.58, 53.67) --
	( 74.22, 50.22) --
	( 71.85, 46.86) --
	( 69.49, 43.81) --
	( 67.12, 41.20) --
	( 64.76, 38.93) --
	( 62.39, 36.60) --
	( 60.03, 34.87) --
	( 57.66, 33.36) --
	( 55.30, 32.12) --
	( 52.93, 30.71) --
	( 50.57, 30.06) --
	( 48.20, 29.10) --
	( 45.84, 28.66) --
	( 43.47, 28.18) --
	( 41.11, 27.84) --
	( 38.74, 27.52) --
	( 36.38, 27.38) --
	( 34.01, 27.28) --
	( 31.65, 27.20) --
	( 29.28, 27.04) --
	( 26.92, 26.92);
\definecolor{drawColor}{RGB}{146,129,188}

\path[draw=drawColor,line width= 1.1pt,line join=round] ( 26.92, 27.60) --
	( 29.28, 27.84) --
	( 31.65, 28.13) --
	( 34.01, 28.51) --
	( 36.38, 28.97) --
	( 38.74, 29.55) --
	( 41.11, 30.25) --
	( 43.47, 31.09) --
	( 45.84, 32.11) --
	( 48.20, 33.31) --
	( 50.57, 34.72) --
	( 52.93, 36.35) --
	( 55.30, 38.24) --
	( 57.66, 40.38) --
	( 60.03, 42.80) --
	( 62.39, 45.50) --
	( 64.76, 48.48) --
	( 67.12, 51.74) --
	( 69.49, 55.28) --
	( 71.85, 59.07) --
	( 74.22, 63.09) --
	( 76.58, 67.33) --
	( 78.95, 71.73) --
	( 81.31, 76.27) --
	( 83.68, 80.90) --
	( 86.05, 85.57) --
	( 88.41, 90.25) --
	( 90.78, 94.88) --
	( 93.14, 99.41) --
	( 95.51,103.82) --
	( 97.87,108.05) --
	(100.24,112.07) --
	(102.60,115.86) --
	(104.97,119.40) --
	(107.33,122.66) --
	(109.70,125.64) --
	(112.06,128.34) --
	(114.43,130.76) --
	(116.79,132.90) --
	(119.16,134.79) --
	(121.52,136.42) --
	(123.89,137.83) --
	(126.25,139.04) --
	(128.62,140.05) --
	(130.98,140.90) --
	(133.35,141.60) --
	(135.71,142.17) --
	(138.08,142.64) --
	(140.44,143.01) --
	(142.81,143.31) --
	(145.17,143.54);
\definecolor{fillColor}{RGB}{230,159,0}

\path[fill=fillColor,fill opacity=0.40] ( 26.92, 28.89) --
	( 29.28, 29.47) --
	( 31.65, 30.01) --
	( 34.01, 30.86) --
	( 36.38, 31.53) --
	( 38.74, 32.39) --
	( 41.11, 33.73) --
	( 43.47, 34.89) --
	( 45.84, 36.30) --
	( 48.20, 37.71) --
	( 50.57, 39.53) --
	( 52.93, 41.67) --
	( 55.30, 43.76) --
	( 57.66, 46.13) --
	( 60.03, 48.75) --
	( 62.39, 51.92) --
	( 64.76, 54.96) --
	( 67.12, 58.19) --
	( 69.49, 61.90) --
	( 71.85, 65.90) --
	( 74.22, 70.00) --
	( 76.58, 74.30) --
	( 78.95, 78.53) --
	( 81.31, 83.25) --
	( 83.68, 87.99) --
	( 86.05, 92.56) --
	( 88.41, 97.05) --
	( 90.78,101.67) --
	( 93.14,106.45) --
	( 95.51,110.79) --
	( 97.87,114.79) --
	(100.24,118.82) --
	(102.60,122.51) --
	(104.97,125.87) --
	(107.33,129.02) --
	(109.70,132.02) --
	(112.06,134.36) --
	(114.43,136.54) --
	(116.79,138.45) --
	(119.16,139.89) --
	(121.52,141.26) --
	(123.89,142.38) --
	(126.25,143.13) --
	(128.62,143.74) --
	(130.98,144.25) --
	(133.35,144.27) --
	(135.71,144.27) --
	(138.08,144.27) --
	(140.44,144.27) --
	(142.81,144.27) --
	(145.17,144.27) --
	(145.17,142.16) --
	(142.81,141.64) --
	(140.44,141.12) --
	(138.08,140.19) --
	(135.71,139.61) --
	(133.35,138.61) --
	(130.98,137.39) --
	(128.62,136.19) --
	(126.25,134.99) --
	(123.89,133.38) --
	(121.52,131.78) --
	(119.16,129.54) --
	(116.79,127.40) --
	(114.43,124.92) --
	(112.06,122.29) --
	(109.70,119.60) --
	(107.33,116.26) --
	(104.97,112.85) --
	(102.60,109.14) --
	(100.24,105.10) --
	( 97.87,101.15) --
	( 95.51, 96.90) --
	( 93.14, 92.39) --
	( 90.78, 87.70) --
	( 88.41, 83.51) --
	( 86.05, 78.41) --
	( 83.68, 74.03) --
	( 81.31, 69.42) --
	( 78.95, 64.76) --
	( 76.58, 60.45) --
	( 74.22, 56.02) --
	( 71.85, 52.09) --
	( 69.49, 48.47) --
	( 67.12, 45.08) --
	( 64.76, 41.95) --
	( 62.39, 39.29) --
	( 60.03, 36.79) --
	( 57.66, 34.68) --
	( 55.30, 32.69) --
	( 52.93, 31.12) --
	( 50.57, 29.91) --
	( 48.20, 28.83) --
	( 45.84, 27.93) --
	( 43.47, 27.30) --
	( 41.11, 26.87) --
	( 38.74, 26.87) --
	( 36.38, 26.87) --
	( 34.01, 26.87) --
	( 31.65, 26.87) --
	( 29.28, 26.87) --
	( 26.92, 26.87) --
	cycle;

\path[] ( 26.92, 28.89) --
	( 29.28, 29.47) --
	( 31.65, 30.01) --
	( 34.01, 30.86) --
	( 36.38, 31.53) --
	( 38.74, 32.39) --
	( 41.11, 33.73) --
	( 43.47, 34.89) --
	( 45.84, 36.30) --
	( 48.20, 37.71) --
	( 50.57, 39.53) --
	( 52.93, 41.67) --
	( 55.30, 43.76) --
	( 57.66, 46.13) --
	( 60.03, 48.75) --
	( 62.39, 51.92) --
	( 64.76, 54.96) --
	( 67.12, 58.19) --
	( 69.49, 61.90) --
	( 71.85, 65.90) --
	( 74.22, 70.00) --
	( 76.58, 74.30) --
	( 78.95, 78.53) --
	( 81.31, 83.25) --
	( 83.68, 87.99) --
	( 86.05, 92.56) --
	( 88.41, 97.05) --
	( 90.78,101.67) --
	( 93.14,106.45) --
	( 95.51,110.79) --
	( 97.87,114.79) --
	(100.24,118.82) --
	(102.60,122.51) --
	(104.97,125.87) --
	(107.33,129.02) --
	(109.70,132.02) --
	(112.06,134.36) --
	(114.43,136.54) --
	(116.79,138.45) --
	(119.16,139.89) --
	(121.52,141.26) --
	(123.89,142.38) --
	(126.25,143.13) --
	(128.62,143.74) --
	(130.98,144.25) --
	(133.35,144.27) --
	(135.71,144.27) --
	(138.08,144.27) --
	(140.44,144.27) --
	(142.81,144.27) --
	(145.17,144.27);

\path[] (145.17,142.16) --
	(142.81,141.64) --
	(140.44,141.12) --
	(138.08,140.19) --
	(135.71,139.61) --
	(133.35,138.61) --
	(130.98,137.39) --
	(128.62,136.19) --
	(126.25,134.99) --
	(123.89,133.38) --
	(121.52,131.78) --
	(119.16,129.54) --
	(116.79,127.40) --
	(114.43,124.92) --
	(112.06,122.29) --
	(109.70,119.60) --
	(107.33,116.26) --
	(104.97,112.85) --
	(102.60,109.14) --
	(100.24,105.10) --
	( 97.87,101.15) --
	( 95.51, 96.90) --
	( 93.14, 92.39) --
	( 90.78, 87.70) --
	( 88.41, 83.51) --
	( 86.05, 78.41) --
	( 83.68, 74.03) --
	( 81.31, 69.42) --
	( 78.95, 64.76) --
	( 76.58, 60.45) --
	( 74.22, 56.02) --
	( 71.85, 52.09) --
	( 69.49, 48.47) --
	( 67.12, 45.08) --
	( 64.76, 41.95) --
	( 62.39, 39.29) --
	( 60.03, 36.79) --
	( 57.66, 34.68) --
	( 55.30, 32.69) --
	( 52.93, 31.12) --
	( 50.57, 29.91) --
	( 48.20, 28.83) --
	( 45.84, 27.93) --
	( 43.47, 27.30) --
	( 41.11, 26.87) --
	( 38.74, 26.87) --
	( 36.38, 26.87) --
	( 34.01, 26.87) --
	( 31.65, 26.87) --
	( 29.28, 26.87) --
	( 26.92, 26.87);
\definecolor{drawColor}{RGB}{230,159,0}

\path[draw=drawColor,line width= 1.1pt,line join=round] ( 26.92, 27.60) --
	( 29.28, 27.84) --
	( 31.65, 28.13) --
	( 34.01, 28.51) --
	( 36.38, 28.97) --
	( 38.74, 29.55) --
	( 41.11, 30.25) --
	( 43.47, 31.09) --
	( 45.84, 32.11) --
	( 48.20, 33.31) --
	( 50.57, 34.72) --
	( 52.93, 36.35) --
	( 55.30, 38.24) --
	( 57.66, 40.38) --
	( 60.03, 42.80) --
	( 62.39, 45.50) --
	( 64.76, 48.48) --
	( 67.12, 51.74) --
	( 69.49, 55.28) --
	( 71.85, 59.07) --
	( 74.22, 63.09) --
	( 76.58, 67.33) --
	( 78.95, 71.73) --
	( 81.31, 76.27) --
	( 83.68, 80.90) --
	( 86.05, 85.57) --
	( 88.41, 90.25) --
	( 90.78, 94.88) --
	( 93.14, 99.41) --
	( 95.51,103.82) --
	( 97.87,108.05) --
	(100.24,112.07) --
	(102.60,115.86) --
	(104.97,119.40) --
	(107.33,122.66) --
	(109.70,125.64) --
	(112.06,128.34) --
	(114.43,130.76) --
	(116.79,132.90) --
	(119.16,134.79) --
	(121.52,136.42) --
	(123.89,137.83) --
	(126.25,139.04) --
	(128.62,140.05) --
	(130.98,140.90) --
	(133.35,141.60) --
	(135.71,142.17) --
	(138.08,142.64) --
	(140.44,143.01) --
	(142.81,143.31) --
	(145.17,143.54);
\definecolor{drawColor}{RGB}{0,0,0}

\path[draw=drawColor,line width= 0.6pt,line join=round] ( 26.92, 27.60) --
	( 29.28, 27.84) --
	( 31.65, 28.13) --
	( 34.01, 28.51) --
	( 36.38, 28.97) --
	( 38.74, 29.55) --
	( 41.11, 30.25) --
	( 43.47, 31.09) --
	( 45.84, 32.11) --
	( 48.20, 33.31) --
	( 50.57, 34.72) --
	( 52.93, 36.35) --
	( 55.30, 38.24) --
	( 57.66, 40.38) --
	( 60.03, 42.80) --
	( 62.39, 45.50) --
	( 64.76, 48.48) --
	( 67.12, 51.74) --
	( 69.49, 55.28) --
	( 71.85, 59.07) --
	( 74.22, 63.09) --
	( 76.58, 67.33) --
	( 78.95, 71.73) --
	( 81.31, 76.27) --
	( 83.68, 80.90) --
	( 86.05, 85.57) --
	( 88.41, 90.25) --
	( 90.78, 94.88) --
	( 93.14, 99.41) --
	( 95.51,103.82) --
	( 97.87,108.05) --
	(100.24,112.07) --
	(102.60,115.86) --
	(104.97,119.40) --
	(107.33,122.66) --
	(109.70,125.64) --
	(112.06,128.34) --
	(114.43,130.76) --
	(116.79,132.90) --
	(119.16,134.79) --
	(121.52,136.42) --
	(123.89,137.83) --
	(126.25,139.04) --
	(128.62,140.05) --
	(130.98,140.90) --
	(133.35,141.60) --
	(135.71,142.17) --
	(138.08,142.64) --
	(140.44,143.01) --
	(142.81,143.31) --
	(145.17,143.54);
\end{scope}
\begin{scope}
\path[clip] (  0.00,  0.00) rectangle (469.75,216.81);
\definecolor{drawColor}{RGB}{0,0,0}

\node[text=drawColor,anchor=base west,inner sep=0pt, outer sep=0pt, scale=  1.1] at ( 21.01,162.26) {\bfseries $H=0.6$};
\end{scope}
\begin{scope}
\path[clip] (156.59,  0.00) rectangle (313.17,188.47);
\definecolor{drawColor}{RGB}{255,255,255}
\definecolor{fillColor}{RGB}{255,255,255}

\path[draw=drawColor,line width= 0.6pt,line join=round,line cap=round,fill=fillColor] (156.59,  0.00) rectangle (313.17,188.47);
\end{scope}
\begin{scope}
\path[clip] (177.59, 21.01) rectangle (307.67,150.14);
\definecolor{fillColor}{gray}{0.92}

\path[fill=fillColor] (177.59, 21.01) rectangle (307.67,150.14);
\definecolor{drawColor}{RGB}{255,255,255}

\path[draw=drawColor,line width= 0.3pt,line join=round] (177.59, 41.55) --
	(307.67, 41.55);

\path[draw=drawColor,line width= 0.3pt,line join=round] (177.59, 70.90) --
	(307.67, 70.90);

\path[draw=drawColor,line width= 0.3pt,line join=round] (177.59,100.25) --
	(307.67,100.25);

\path[draw=drawColor,line width= 0.3pt,line join=round] (177.59,129.59) --
	(307.67,129.59);

\path[draw=drawColor,line width= 0.3pt,line join=round] (183.50, 21.01) --
	(183.50,150.14);

\path[draw=drawColor,line width= 0.3pt,line join=round] (207.15, 21.01) --
	(207.15,150.14);

\path[draw=drawColor,line width= 0.3pt,line join=round] (230.80, 21.01) --
	(230.80,150.14);

\path[draw=drawColor,line width= 0.3pt,line join=round] (254.46, 21.01) --
	(254.46,150.14);

\path[draw=drawColor,line width= 0.3pt,line join=round] (278.11, 21.01) --
	(278.11,150.14);

\path[draw=drawColor,line width= 0.3pt,line join=round] (301.76, 21.01) --
	(301.76,150.14);

\path[draw=drawColor,line width= 0.6pt,line join=round] (177.59, 26.87) --
	(307.67, 26.87);

\path[draw=drawColor,line width= 0.6pt,line join=round] (177.59, 56.22) --
	(307.67, 56.22);

\path[draw=drawColor,line width= 0.6pt,line join=round] (177.59, 85.57) --
	(307.67, 85.57);

\path[draw=drawColor,line width= 0.6pt,line join=round] (177.59,114.92) --
	(307.67,114.92);

\path[draw=drawColor,line width= 0.6pt,line join=round] (177.59,144.27) --
	(307.67,144.27);

\path[draw=drawColor,line width= 0.6pt,line join=round] (195.33, 21.01) --
	(195.33,150.14);

\path[draw=drawColor,line width= 0.6pt,line join=round] (218.98, 21.01) --
	(218.98,150.14);

\path[draw=drawColor,line width= 0.6pt,line join=round] (242.63, 21.01) --
	(242.63,150.14);

\path[draw=drawColor,line width= 0.6pt,line join=round] (266.28, 21.01) --
	(266.28,150.14);

\path[draw=drawColor,line width= 0.6pt,line join=round] (289.93, 21.01) --
	(289.93,150.14);
\definecolor{fillColor}{RGB}{146,129,188}

\path[fill=fillColor,fill opacity=0.40] (183.50, 28.92) --
	(185.87, 29.49) --
	(188.23, 30.22) --
	(190.60, 31.07) --
	(192.96, 32.05) --
	(195.33, 33.35) --
	(197.69, 34.86) --
	(200.06, 36.87) --
	(202.42, 38.93) --
	(204.79, 41.35) --
	(207.15, 44.08) --
	(209.52, 47.11) --
	(211.88, 50.72) --
	(214.25, 54.53) --
	(216.61, 58.13) --
	(218.98, 63.11) --
	(221.34, 68.38) --
	(223.71, 72.73) --
	(226.07, 77.69) --
	(228.44, 83.23) --
	(230.80, 88.47) --
	(233.17, 94.32) --
	(235.53, 99.66) --
	(237.90,105.74) --
	(240.26,109.86) --
	(242.63,114.72) --
	(245.00,119.12) --
	(247.36,123.83) --
	(249.73,127.06) --
	(252.09,130.79) --
	(254.46,133.18) --
	(256.82,137.02) --
	(259.19,138.64) --
	(261.55,140.73) --
	(263.92,142.14) --
	(266.28,143.19) --
	(268.65,144.27) --
	(271.01,144.27) --
	(273.38,144.27) --
	(275.74,144.27) --
	(278.11,144.27) --
	(280.47,144.27) --
	(282.84,144.27) --
	(285.20,144.27) --
	(287.57,144.27) --
	(289.93,144.27) --
	(292.30,144.27) --
	(294.66,144.27) --
	(297.03,144.27) --
	(299.39,144.27) --
	(301.76,144.27) --
	(301.76,142.30) --
	(299.39,141.71) --
	(297.03,140.86) --
	(294.66,140.04) --
	(292.30,139.00) --
	(289.93,137.68) --
	(287.57,136.09) --
	(285.20,134.05) --
	(282.84,132.19) --
	(280.47,129.69) --
	(278.11,127.15) --
	(275.74,123.58) --
	(273.38,120.36) --
	(271.01,116.59) --
	(268.65,112.25) --
	(266.28,108.13) --
	(263.92,103.24) --
	(261.55, 98.27) --
	(259.19, 93.22) --
	(256.82, 87.89) --
	(254.46, 82.92) --
	(252.09, 76.76) --
	(249.73, 71.32) --
	(247.36, 67.50) --
	(245.00, 61.39) --
	(242.63, 56.49) --
	(240.26, 51.87) --
	(237.90, 47.65) --
	(235.53, 43.93) --
	(233.17, 41.01) --
	(230.80, 38.01) --
	(228.44, 34.78) --
	(226.07, 32.02) --
	(223.71, 30.79) --
	(221.34, 29.20) --
	(218.98, 27.56) --
	(216.61, 26.87) --
	(214.25, 26.87) --
	(211.88, 26.87) --
	(209.52, 26.87) --
	(207.15, 26.87) --
	(204.79, 26.87) --
	(202.42, 26.87) --
	(200.06, 26.87) --
	(197.69, 26.87) --
	(195.33, 26.87) --
	(192.96, 26.87) --
	(190.60, 26.87) --
	(188.23, 26.87) --
	(185.87, 26.87) --
	(183.50, 26.87) --
	cycle;

\path[] (183.50, 28.92) --
	(185.87, 29.49) --
	(188.23, 30.22) --
	(190.60, 31.07) --
	(192.96, 32.05) --
	(195.33, 33.35) --
	(197.69, 34.86) --
	(200.06, 36.87) --
	(202.42, 38.93) --
	(204.79, 41.35) --
	(207.15, 44.08) --
	(209.52, 47.11) --
	(211.88, 50.72) --
	(214.25, 54.53) --
	(216.61, 58.13) --
	(218.98, 63.11) --
	(221.34, 68.38) --
	(223.71, 72.73) --
	(226.07, 77.69) --
	(228.44, 83.23) --
	(230.80, 88.47) --
	(233.17, 94.32) --
	(235.53, 99.66) --
	(237.90,105.74) --
	(240.26,109.86) --
	(242.63,114.72) --
	(245.00,119.12) --
	(247.36,123.83) --
	(249.73,127.06) --
	(252.09,130.79) --
	(254.46,133.18) --
	(256.82,137.02) --
	(259.19,138.64) --
	(261.55,140.73) --
	(263.92,142.14) --
	(266.28,143.19) --
	(268.65,144.27) --
	(271.01,144.27) --
	(273.38,144.27) --
	(275.74,144.27) --
	(278.11,144.27) --
	(280.47,144.27) --
	(282.84,144.27) --
	(285.20,144.27) --
	(287.57,144.27) --
	(289.93,144.27) --
	(292.30,144.27) --
	(294.66,144.27) --
	(297.03,144.27) --
	(299.39,144.27) --
	(301.76,144.27);

\path[] (301.76,142.30) --
	(299.39,141.71) --
	(297.03,140.86) --
	(294.66,140.04) --
	(292.30,139.00) --
	(289.93,137.68) --
	(287.57,136.09) --
	(285.20,134.05) --
	(282.84,132.19) --
	(280.47,129.69) --
	(278.11,127.15) --
	(275.74,123.58) --
	(273.38,120.36) --
	(271.01,116.59) --
	(268.65,112.25) --
	(266.28,108.13) --
	(263.92,103.24) --
	(261.55, 98.27) --
	(259.19, 93.22) --
	(256.82, 87.89) --
	(254.46, 82.92) --
	(252.09, 76.76) --
	(249.73, 71.32) --
	(247.36, 67.50) --
	(245.00, 61.39) --
	(242.63, 56.49) --
	(240.26, 51.87) --
	(237.90, 47.65) --
	(235.53, 43.93) --
	(233.17, 41.01) --
	(230.80, 38.01) --
	(228.44, 34.78) --
	(226.07, 32.02) --
	(223.71, 30.79) --
	(221.34, 29.20) --
	(218.98, 27.56) --
	(216.61, 26.87) --
	(214.25, 26.87) --
	(211.88, 26.87) --
	(209.52, 26.87) --
	(207.15, 26.87) --
	(204.79, 26.87) --
	(202.42, 26.87) --
	(200.06, 26.87) --
	(197.69, 26.87) --
	(195.33, 26.87) --
	(192.96, 26.87) --
	(190.60, 26.87) --
	(188.23, 26.87) --
	(185.87, 26.87) --
	(183.50, 26.87);
\definecolor{drawColor}{RGB}{146,129,188}

\path[draw=drawColor,line width= 1.1pt,line join=round] (183.50, 27.60) --
	(185.87, 27.84) --
	(188.23, 28.13) --
	(190.60, 28.51) --
	(192.96, 28.97) --
	(195.33, 29.55) --
	(197.69, 30.25) --
	(200.06, 31.09) --
	(202.42, 32.11) --
	(204.79, 33.31) --
	(207.15, 34.72) --
	(209.52, 36.35) --
	(211.88, 38.24) --
	(214.25, 40.38) --
	(216.61, 42.80) --
	(218.98, 45.50) --
	(221.34, 48.48) --
	(223.71, 51.74) --
	(226.07, 55.28) --
	(228.44, 59.07) --
	(230.80, 63.09) --
	(233.17, 67.33) --
	(235.53, 71.73) --
	(237.90, 76.27) --
	(240.26, 80.90) --
	(242.63, 85.57) --
	(245.00, 90.25) --
	(247.36, 94.88) --
	(249.73, 99.41) --
	(252.09,103.82) --
	(254.46,108.05) --
	(256.82,112.07) --
	(259.19,115.86) --
	(261.55,119.40) --
	(263.92,122.66) --
	(266.28,125.64) --
	(268.65,128.34) --
	(271.01,130.76) --
	(273.38,132.90) --
	(275.74,134.79) --
	(278.11,136.42) --
	(280.47,137.83) --
	(282.84,139.04) --
	(285.20,140.05) --
	(287.57,140.90) --
	(289.93,141.60) --
	(292.30,142.17) --
	(294.66,142.64) --
	(297.03,143.01) --
	(299.39,143.31) --
	(301.76,143.54);
\definecolor{fillColor}{RGB}{230,159,0}

\path[fill=fillColor,fill opacity=0.40] (183.50, 28.87) --
	(185.87, 29.41) --
	(188.23, 29.96) --
	(190.60, 30.42) --
	(192.96, 31.33) --
	(195.33, 32.12) --
	(197.69, 33.04) --
	(200.06, 34.12) --
	(202.42, 35.33) --
	(204.79, 36.69) --
	(207.15, 38.36) --
	(209.52, 40.31) --
	(211.88, 42.32) --
	(214.25, 44.59) --
	(216.61, 47.31) --
	(218.98, 50.30) --
	(221.34, 53.56) --
	(223.71, 57.09) --
	(226.07, 61.09) --
	(228.44, 64.99) --
	(230.80, 69.42) --
	(233.17, 73.87) --
	(235.53, 78.42) --
	(237.90, 82.82) --
	(240.26, 87.66) --
	(242.63, 92.66) --
	(245.00, 97.25) --
	(247.36,101.60) --
	(249.73,106.10) --
	(252.09,110.07) --
	(254.46,114.23) --
	(256.82,117.93) --
	(259.19,121.47) --
	(261.55,124.65) --
	(263.92,127.88) --
	(266.28,130.39) --
	(268.65,132.84) --
	(271.01,135.20) --
	(273.38,136.86) --
	(275.74,138.66) --
	(278.11,140.16) --
	(280.47,141.36) --
	(282.84,142.14) --
	(285.20,143.04) --
	(287.57,143.78) --
	(289.93,144.08) --
	(292.30,144.27) --
	(294.66,144.27) --
	(297.03,144.27) --
	(299.39,144.27) --
	(301.76,144.27) --
	(301.76,142.27) --
	(299.39,141.69) --
	(297.03,141.22) --
	(294.66,140.62) --
	(292.30,139.88) --
	(289.93,139.02) --
	(287.57,138.20) --
	(285.20,137.14) --
	(282.84,135.67) --
	(280.47,134.44) --
	(278.11,132.73) --
	(275.74,131.02) --
	(273.38,128.88) --
	(271.01,126.41) --
	(268.65,123.66) --
	(266.28,120.60) --
	(263.92,117.55) --
	(261.55,114.15) --
	(259.19,110.21) --
	(256.82,106.20) --
	(254.46,101.72) --
	(252.09, 97.31) --
	(249.73, 92.77) --
	(247.36, 88.26) --
	(245.00, 83.52) --
	(242.63, 78.49) --
	(240.26, 73.93) --
	(237.90, 69.53) --
	(235.53, 64.89) --
	(233.17, 60.83) --
	(230.80, 57.03) --
	(228.44, 53.01) --
	(226.07, 49.68) --
	(223.71, 46.54) --
	(221.34, 43.28) --
	(218.98, 40.57) --
	(216.61, 38.26) --
	(214.25, 35.96) --
	(211.88, 34.15) --
	(209.52, 32.56) --
	(207.15, 31.00) --
	(204.79, 29.86) --
	(202.42, 28.82) --
	(200.06, 28.01) --
	(197.69, 27.41) --
	(195.33, 27.00) --
	(192.96, 26.87) --
	(190.60, 26.87) --
	(188.23, 26.87) --
	(185.87, 26.87) --
	(183.50, 26.87) --
	cycle;

\path[] (183.50, 28.87) --
	(185.87, 29.41) --
	(188.23, 29.96) --
	(190.60, 30.42) --
	(192.96, 31.33) --
	(195.33, 32.12) --
	(197.69, 33.04) --
	(200.06, 34.12) --
	(202.42, 35.33) --
	(204.79, 36.69) --
	(207.15, 38.36) --
	(209.52, 40.31) --
	(211.88, 42.32) --
	(214.25, 44.59) --
	(216.61, 47.31) --
	(218.98, 50.30) --
	(221.34, 53.56) --
	(223.71, 57.09) --
	(226.07, 61.09) --
	(228.44, 64.99) --
	(230.80, 69.42) --
	(233.17, 73.87) --
	(235.53, 78.42) --
	(237.90, 82.82) --
	(240.26, 87.66) --
	(242.63, 92.66) --
	(245.00, 97.25) --
	(247.36,101.60) --
	(249.73,106.10) --
	(252.09,110.07) --
	(254.46,114.23) --
	(256.82,117.93) --
	(259.19,121.47) --
	(261.55,124.65) --
	(263.92,127.88) --
	(266.28,130.39) --
	(268.65,132.84) --
	(271.01,135.20) --
	(273.38,136.86) --
	(275.74,138.66) --
	(278.11,140.16) --
	(280.47,141.36) --
	(282.84,142.14) --
	(285.20,143.04) --
	(287.57,143.78) --
	(289.93,144.08) --
	(292.30,144.27) --
	(294.66,144.27) --
	(297.03,144.27) --
	(299.39,144.27) --
	(301.76,144.27);

\path[] (301.76,142.27) --
	(299.39,141.69) --
	(297.03,141.22) --
	(294.66,140.62) --
	(292.30,139.88) --
	(289.93,139.02) --
	(287.57,138.20) --
	(285.20,137.14) --
	(282.84,135.67) --
	(280.47,134.44) --
	(278.11,132.73) --
	(275.74,131.02) --
	(273.38,128.88) --
	(271.01,126.41) --
	(268.65,123.66) --
	(266.28,120.60) --
	(263.92,117.55) --
	(261.55,114.15) --
	(259.19,110.21) --
	(256.82,106.20) --
	(254.46,101.72) --
	(252.09, 97.31) --
	(249.73, 92.77) --
	(247.36, 88.26) --
	(245.00, 83.52) --
	(242.63, 78.49) --
	(240.26, 73.93) --
	(237.90, 69.53) --
	(235.53, 64.89) --
	(233.17, 60.83) --
	(230.80, 57.03) --
	(228.44, 53.01) --
	(226.07, 49.68) --
	(223.71, 46.54) --
	(221.34, 43.28) --
	(218.98, 40.57) --
	(216.61, 38.26) --
	(214.25, 35.96) --
	(211.88, 34.15) --
	(209.52, 32.56) --
	(207.15, 31.00) --
	(204.79, 29.86) --
	(202.42, 28.82) --
	(200.06, 28.01) --
	(197.69, 27.41) --
	(195.33, 27.00) --
	(192.96, 26.87) --
	(190.60, 26.87) --
	(188.23, 26.87) --
	(185.87, 26.87) --
	(183.50, 26.87);
\definecolor{drawColor}{RGB}{230,159,0}

\path[draw=drawColor,line width= 1.1pt,line join=round] (183.50, 27.60) --
	(185.87, 27.84) --
	(188.23, 28.13) --
	(190.60, 28.51) --
	(192.96, 28.97) --
	(195.33, 29.55) --
	(197.69, 30.25) --
	(200.06, 31.09) --
	(202.42, 32.11) --
	(204.79, 33.31) --
	(207.15, 34.72) --
	(209.52, 36.35) --
	(211.88, 38.24) --
	(214.25, 40.38) --
	(216.61, 42.80) --
	(218.98, 45.50) --
	(221.34, 48.48) --
	(223.71, 51.74) --
	(226.07, 55.28) --
	(228.44, 59.07) --
	(230.80, 63.09) --
	(233.17, 67.33) --
	(235.53, 71.73) --
	(237.90, 76.27) --
	(240.26, 80.90) --
	(242.63, 85.57) --
	(245.00, 90.25) --
	(247.36, 94.88) --
	(249.73, 99.41) --
	(252.09,103.82) --
	(254.46,108.05) --
	(256.82,112.07) --
	(259.19,115.86) --
	(261.55,119.40) --
	(263.92,122.66) --
	(266.28,125.64) --
	(268.65,128.34) --
	(271.01,130.76) --
	(273.38,132.90) --
	(275.74,134.79) --
	(278.11,136.42) --
	(280.47,137.83) --
	(282.84,139.04) --
	(285.20,140.05) --
	(287.57,140.90) --
	(289.93,141.60) --
	(292.30,142.17) --
	(294.66,142.64) --
	(297.03,143.01) --
	(299.39,143.31) --
	(301.76,143.54);
\definecolor{drawColor}{RGB}{0,0,0}

\path[draw=drawColor,line width= 0.6pt,line join=round] (183.50, 27.60) --
	(185.87, 27.84) --
	(188.23, 28.13) --
	(190.60, 28.51) --
	(192.96, 28.97) --
	(195.33, 29.55) --
	(197.69, 30.25) --
	(200.06, 31.09) --
	(202.42, 32.11) --
	(204.79, 33.31) --
	(207.15, 34.72) --
	(209.52, 36.35) --
	(211.88, 38.24) --
	(214.25, 40.38) --
	(216.61, 42.80) --
	(218.98, 45.50) --
	(221.34, 48.48) --
	(223.71, 51.74) --
	(226.07, 55.28) --
	(228.44, 59.07) --
	(230.80, 63.09) --
	(233.17, 67.33) --
	(235.53, 71.73) --
	(237.90, 76.27) --
	(240.26, 80.90) --
	(242.63, 85.57) --
	(245.00, 90.25) --
	(247.36, 94.88) --
	(249.73, 99.41) --
	(252.09,103.82) --
	(254.46,108.05) --
	(256.82,112.07) --
	(259.19,115.86) --
	(261.55,119.40) --
	(263.92,122.66) --
	(266.28,125.64) --
	(268.65,128.34) --
	(271.01,130.76) --
	(273.38,132.90) --
	(275.74,134.79) --
	(278.11,136.42) --
	(280.47,137.83) --
	(282.84,139.04) --
	(285.20,140.05) --
	(287.57,140.90) --
	(289.93,141.60) --
	(292.30,142.17) --
	(294.66,142.64) --
	(297.03,143.01) --
	(299.39,143.31) --
	(301.76,143.54);
\end{scope}
\begin{scope}
\path[clip] (  0.00,  0.00) rectangle (469.75,216.81);
\definecolor{drawColor}{RGB}{0,0,0}

\node[text=drawColor,anchor=base west,inner sep=0pt, outer sep=0pt, scale=  1.1] at (177.59,162.26) {\bfseries $H=0.75$};
\end{scope}
\begin{scope}
\path[clip] (313.17,  0.00) rectangle (469.75,188.47);
\definecolor{drawColor}{RGB}{255,255,255}
\definecolor{fillColor}{RGB}{255,255,255}

\path[draw=drawColor,line width= 0.6pt,line join=round,line cap=round,fill=fillColor] (313.17,  0.00) rectangle (469.75,188.47);
\end{scope}
\begin{scope}
\path[clip] (334.18, 21.01) rectangle (464.25,150.14);
\definecolor{fillColor}{gray}{0.92}

\path[fill=fillColor] (334.18, 21.01) rectangle (464.25,150.14);
\definecolor{drawColor}{RGB}{255,255,255}

\path[draw=drawColor,line width= 0.3pt,line join=round] (334.18, 41.55) --
	(464.25, 41.55);

\path[draw=drawColor,line width= 0.3pt,line join=round] (334.18, 70.90) --
	(464.25, 70.90);

\path[draw=drawColor,line width= 0.3pt,line join=round] (334.18,100.25) --
	(464.25,100.25);

\path[draw=drawColor,line width= 0.3pt,line join=round] (334.18,129.59) --
	(464.25,129.59);

\path[draw=drawColor,line width= 0.3pt,line join=round] (340.09, 21.01) --
	(340.09,150.14);

\path[draw=drawColor,line width= 0.3pt,line join=round] (363.74, 21.01) --
	(363.74,150.14);

\path[draw=drawColor,line width= 0.3pt,line join=round] (387.39, 21.01) --
	(387.39,150.14);

\path[draw=drawColor,line width= 0.3pt,line join=round] (411.04, 21.01) --
	(411.04,150.14);

\path[draw=drawColor,line width= 0.3pt,line join=round] (434.69, 21.01) --
	(434.69,150.14);

\path[draw=drawColor,line width= 0.3pt,line join=round] (458.34, 21.01) --
	(458.34,150.14);

\path[draw=drawColor,line width= 0.6pt,line join=round] (334.18, 26.87) --
	(464.25, 26.87);

\path[draw=drawColor,line width= 0.6pt,line join=round] (334.18, 56.22) --
	(464.25, 56.22);

\path[draw=drawColor,line width= 0.6pt,line join=round] (334.18, 85.57) --
	(464.25, 85.57);

\path[draw=drawColor,line width= 0.6pt,line join=round] (334.18,114.92) --
	(464.25,114.92);

\path[draw=drawColor,line width= 0.6pt,line join=round] (334.18,144.27) --
	(464.25,144.27);

\path[draw=drawColor,line width= 0.6pt,line join=round] (351.91, 21.01) --
	(351.91,150.14);

\path[draw=drawColor,line width= 0.6pt,line join=round] (375.56, 21.01) --
	(375.56,150.14);

\path[draw=drawColor,line width= 0.6pt,line join=round] (399.22, 21.01) --
	(399.22,150.14);

\path[draw=drawColor,line width= 0.6pt,line join=round] (422.87, 21.01) --
	(422.87,150.14);

\path[draw=drawColor,line width= 0.6pt,line join=round] (446.52, 21.01) --
	(446.52,150.14);
\definecolor{fillColor}{RGB}{146,129,188}

\path[fill=fillColor,fill opacity=0.40] (340.09, 30.08) --
	(342.45, 31.04) --
	(344.82, 32.33) --
	(347.18, 33.72) --
	(349.55, 35.31) --
	(351.91, 37.41) --
	(354.28, 40.08) --
	(356.64, 42.84) --
	(359.01, 46.03) --
	(361.37, 48.94) --
	(363.74, 53.09) --
	(366.10, 57.52) --
	(368.47, 63.20) --
	(370.83, 69.22) --
	(373.20, 74.67) --
	(375.56, 80.77) --
	(377.93, 86.98) --
	(380.29, 91.98) --
	(382.66, 99.27) --
	(385.02,106.75) --
	(387.39,113.42) --
	(389.75,120.15) --
	(392.12,127.91) --
	(394.48,133.91) --
	(396.85,138.38) --
	(399.22,142.98) --
	(401.58,144.27) --
	(403.95,144.27) --
	(406.31,144.27) --
	(408.68,144.27) --
	(411.04,144.27) --
	(413.41,144.27) --
	(415.77,144.27) --
	(418.14,144.27) --
	(420.50,144.27) --
	(422.87,144.27) --
	(425.23,144.27) --
	(427.60,144.27) --
	(429.96,144.27) --
	(432.33,144.27) --
	(434.69,144.27) --
	(437.06,144.27) --
	(439.42,144.27) --
	(441.79,144.27) --
	(444.15,144.27) --
	(446.52,144.27) --
	(448.88,144.27) --
	(451.25,144.27) --
	(453.61,144.27) --
	(455.98,144.27) --
	(458.34,144.27) --
	(458.34,141.06) --
	(455.98,140.11) --
	(453.61,138.93) --
	(451.25,137.43) --
	(448.88,135.84) --
	(446.52,133.86) --
	(444.15,131.38) --
	(441.79,128.57) --
	(439.42,125.49) --
	(437.06,121.76) --
	(434.69,117.97) --
	(432.33,113.47) --
	(429.96,107.94) --
	(427.60,101.87) --
	(425.23, 97.63) --
	(422.87, 91.03) --
	(420.50, 84.40) --
	(418.14, 77.29) --
	(415.77, 69.09) --
	(413.41, 63.72) --
	(411.04, 56.80) --
	(408.68, 48.86) --
	(406.31, 45.39) --
	(403.95, 38.01) --
	(401.58, 33.58) --
	(399.22, 27.74) --
	(396.85, 26.87) --
	(394.48, 26.87) --
	(392.12, 26.87) --
	(389.75, 26.87) --
	(387.39, 26.87) --
	(385.02, 26.87) --
	(382.66, 26.87) --
	(380.29, 26.87) --
	(377.93, 26.87) --
	(375.56, 26.87) --
	(373.20, 26.87) --
	(370.83, 26.87) --
	(368.47, 26.87) --
	(366.10, 26.87) --
	(363.74, 26.87) --
	(361.37, 26.87) --
	(359.01, 26.87) --
	(356.64, 26.87) --
	(354.28, 26.87) --
	(351.91, 26.87) --
	(349.55, 26.87) --
	(347.18, 26.87) --
	(344.82, 26.87) --
	(342.45, 26.87) --
	(340.09, 26.87) --
	cycle;

\path[] (340.09, 30.08) --
	(342.45, 31.04) --
	(344.82, 32.33) --
	(347.18, 33.72) --
	(349.55, 35.31) --
	(351.91, 37.41) --
	(354.28, 40.08) --
	(356.64, 42.84) --
	(359.01, 46.03) --
	(361.37, 48.94) --
	(363.74, 53.09) --
	(366.10, 57.52) --
	(368.47, 63.20) --
	(370.83, 69.22) --
	(373.20, 74.67) --
	(375.56, 80.77) --
	(377.93, 86.98) --
	(380.29, 91.98) --
	(382.66, 99.27) --
	(385.02,106.75) --
	(387.39,113.42) --
	(389.75,120.15) --
	(392.12,127.91) --
	(394.48,133.91) --
	(396.85,138.38) --
	(399.22,142.98) --
	(401.58,144.27) --
	(403.95,144.27) --
	(406.31,144.27) --
	(408.68,144.27) --
	(411.04,144.27) --
	(413.41,144.27) --
	(415.77,144.27) --
	(418.14,144.27) --
	(420.50,144.27) --
	(422.87,144.27) --
	(425.23,144.27) --
	(427.60,144.27) --
	(429.96,144.27) --
	(432.33,144.27) --
	(434.69,144.27) --
	(437.06,144.27) --
	(439.42,144.27) --
	(441.79,144.27) --
	(444.15,144.27) --
	(446.52,144.27) --
	(448.88,144.27) --
	(451.25,144.27) --
	(453.61,144.27) --
	(455.98,144.27) --
	(458.34,144.27);

\path[] (458.34,141.06) --
	(455.98,140.11) --
	(453.61,138.93) --
	(451.25,137.43) --
	(448.88,135.84) --
	(446.52,133.86) --
	(444.15,131.38) --
	(441.79,128.57) --
	(439.42,125.49) --
	(437.06,121.76) --
	(434.69,117.97) --
	(432.33,113.47) --
	(429.96,107.94) --
	(427.60,101.87) --
	(425.23, 97.63) --
	(422.87, 91.03) --
	(420.50, 84.40) --
	(418.14, 77.29) --
	(415.77, 69.09) --
	(413.41, 63.72) --
	(411.04, 56.80) --
	(408.68, 48.86) --
	(406.31, 45.39) --
	(403.95, 38.01) --
	(401.58, 33.58) --
	(399.22, 27.74) --
	(396.85, 26.87) --
	(394.48, 26.87) --
	(392.12, 26.87) --
	(389.75, 26.87) --
	(387.39, 26.87) --
	(385.02, 26.87) --
	(382.66, 26.87) --
	(380.29, 26.87) --
	(377.93, 26.87) --
	(375.56, 26.87) --
	(373.20, 26.87) --
	(370.83, 26.87) --
	(368.47, 26.87) --
	(366.10, 26.87) --
	(363.74, 26.87) --
	(361.37, 26.87) --
	(359.01, 26.87) --
	(356.64, 26.87) --
	(354.28, 26.87) --
	(351.91, 26.87) --
	(349.55, 26.87) --
	(347.18, 26.87) --
	(344.82, 26.87) --
	(342.45, 26.87) --
	(340.09, 26.87);
\definecolor{drawColor}{RGB}{146,129,188}

\path[draw=drawColor,line width= 1.1pt,line join=round] (340.09, 27.60) --
	(342.45, 27.84) --
	(344.82, 28.13) --
	(347.18, 28.51) --
	(349.55, 28.97) --
	(351.91, 29.55) --
	(354.28, 30.25) --
	(356.64, 31.09) --
	(359.01, 32.11) --
	(361.37, 33.31) --
	(363.74, 34.72) --
	(366.10, 36.35) --
	(368.47, 38.24) --
	(370.83, 40.38) --
	(373.20, 42.80) --
	(375.56, 45.50) --
	(377.93, 48.48) --
	(380.29, 51.74) --
	(382.66, 55.28) --
	(385.02, 59.07) --
	(387.39, 63.09) --
	(389.75, 67.33) --
	(392.12, 71.73) --
	(394.48, 76.27) --
	(396.85, 80.90) --
	(399.22, 85.57) --
	(401.58, 90.25) --
	(403.95, 94.88) --
	(406.31, 99.41) --
	(408.68,103.82) --
	(411.04,108.05) --
	(413.41,112.07) --
	(415.77,115.86) --
	(418.14,119.40) --
	(420.50,122.66) --
	(422.87,125.64) --
	(425.23,128.34) --
	(427.60,130.76) --
	(429.96,132.90) --
	(432.33,134.79) --
	(434.69,136.42) --
	(437.06,137.83) --
	(439.42,139.04) --
	(441.79,140.05) --
	(444.15,140.90) --
	(446.52,141.60) --
	(448.88,142.17) --
	(451.25,142.64) --
	(453.61,143.01) --
	(455.98,143.31) --
	(458.34,143.54);
\definecolor{fillColor}{RGB}{230,159,0}

\path[fill=fillColor,fill opacity=0.40] (340.09, 28.30) --
	(342.45, 28.75) --
	(344.82, 29.18) --
	(347.18, 29.69) --
	(349.55, 30.49) --
	(351.91, 31.14) --
	(354.28, 32.12) --
	(356.64, 33.07) --
	(359.01, 34.38) --
	(361.37, 35.75) --
	(363.74, 37.34) --
	(366.10, 39.28) --
	(368.47, 41.33) --
	(370.83, 43.81) --
	(373.20, 46.48) --
	(375.56, 49.50) --
	(377.93, 52.87) --
	(380.29, 56.28) --
	(382.66, 60.05) --
	(385.02, 63.85) --
	(387.39, 68.03) --
	(389.75, 72.46) --
	(392.12, 76.92) --
	(394.48, 81.53) --
	(396.85, 85.99) --
	(399.22, 90.70) --
	(401.58, 95.51) --
	(403.95, 99.99) --
	(406.31,104.60) --
	(408.68,108.86) --
	(411.04,113.01) --
	(413.41,117.21) --
	(415.77,120.57) --
	(418.14,123.77) --
	(420.50,126.79) --
	(422.87,129.49) --
	(425.23,132.03) --
	(427.60,134.23) --
	(429.96,135.91) --
	(432.33,137.61) --
	(434.69,139.14) --
	(437.06,140.21) --
	(439.42,141.13) --
	(441.79,142.08) --
	(444.15,142.82) --
	(446.52,143.24) --
	(448.88,143.52) --
	(451.25,143.87) --
	(453.61,144.01) --
	(455.98,144.20) --
	(458.34,144.22) --
	(458.34,142.86) --
	(455.98,142.41) --
	(453.61,141.97) --
	(451.25,141.41) --
	(448.88,140.73) --
	(446.52,139.93) --
	(444.15,139.06) --
	(441.79,138.12) --
	(439.42,136.77) --
	(437.06,135.35) --
	(434.69,133.72) --
	(432.33,131.91) --
	(429.96,129.81) --
	(427.60,127.52) --
	(425.23,124.73) --
	(422.87,121.70) --
	(420.50,118.53) --
	(418.14,114.97) --
	(415.77,111.23) --
	(413.41,107.09) --
	(411.04,103.04) --
	(408.68, 98.73) --
	(406.31, 94.30) --
	(403.95, 89.68) --
	(401.58, 85.19) --
	(399.22, 80.43) --
	(396.85, 75.64) --
	(394.48, 71.12) --
	(392.12, 66.56) --
	(389.75, 62.24) --
	(387.39, 57.95) --
	(385.02, 53.97) --
	(382.66, 50.48) --
	(380.29, 47.28) --
	(377.93, 44.35) --
	(375.56, 41.61) --
	(373.20, 39.13) --
	(370.83, 37.07) --
	(368.47, 35.15) --
	(366.10, 33.51) --
	(363.74, 32.08) --
	(361.37, 30.86) --
	(359.01, 29.84) --
	(356.64, 29.08) --
	(354.28, 28.47) --
	(351.91, 27.96) --
	(349.55, 27.48) --
	(347.18, 27.35) --
	(344.82, 27.13) --
	(342.45, 26.94) --
	(340.09, 26.91) --
	cycle;

\path[] (340.09, 28.30) --
	(342.45, 28.75) --
	(344.82, 29.18) --
	(347.18, 29.69) --
	(349.55, 30.49) --
	(351.91, 31.14) --
	(354.28, 32.12) --
	(356.64, 33.07) --
	(359.01, 34.38) --
	(361.37, 35.75) --
	(363.74, 37.34) --
	(366.10, 39.28) --
	(368.47, 41.33) --
	(370.83, 43.81) --
	(373.20, 46.48) --
	(375.56, 49.50) --
	(377.93, 52.87) --
	(380.29, 56.28) --
	(382.66, 60.05) --
	(385.02, 63.85) --
	(387.39, 68.03) --
	(389.75, 72.46) --
	(392.12, 76.92) --
	(394.48, 81.53) --
	(396.85, 85.99) --
	(399.22, 90.70) --
	(401.58, 95.51) --
	(403.95, 99.99) --
	(406.31,104.60) --
	(408.68,108.86) --
	(411.04,113.01) --
	(413.41,117.21) --
	(415.77,120.57) --
	(418.14,123.77) --
	(420.50,126.79) --
	(422.87,129.49) --
	(425.23,132.03) --
	(427.60,134.23) --
	(429.96,135.91) --
	(432.33,137.61) --
	(434.69,139.14) --
	(437.06,140.21) --
	(439.42,141.13) --
	(441.79,142.08) --
	(444.15,142.82) --
	(446.52,143.24) --
	(448.88,143.52) --
	(451.25,143.87) --
	(453.61,144.01) --
	(455.98,144.20) --
	(458.34,144.22);

\path[] (458.34,142.86) --
	(455.98,142.41) --
	(453.61,141.97) --
	(451.25,141.41) --
	(448.88,140.73) --
	(446.52,139.93) --
	(444.15,139.06) --
	(441.79,138.12) --
	(439.42,136.77) --
	(437.06,135.35) --
	(434.69,133.72) --
	(432.33,131.91) --
	(429.96,129.81) --
	(427.60,127.52) --
	(425.23,124.73) --
	(422.87,121.70) --
	(420.50,118.53) --
	(418.14,114.97) --
	(415.77,111.23) --
	(413.41,107.09) --
	(411.04,103.04) --
	(408.68, 98.73) --
	(406.31, 94.30) --
	(403.95, 89.68) --
	(401.58, 85.19) --
	(399.22, 80.43) --
	(396.85, 75.64) --
	(394.48, 71.12) --
	(392.12, 66.56) --
	(389.75, 62.24) --
	(387.39, 57.95) --
	(385.02, 53.97) --
	(382.66, 50.48) --
	(380.29, 47.28) --
	(377.93, 44.35) --
	(375.56, 41.61) --
	(373.20, 39.13) --
	(370.83, 37.07) --
	(368.47, 35.15) --
	(366.10, 33.51) --
	(363.74, 32.08) --
	(361.37, 30.86) --
	(359.01, 29.84) --
	(356.64, 29.08) --
	(354.28, 28.47) --
	(351.91, 27.96) --
	(349.55, 27.48) --
	(347.18, 27.35) --
	(344.82, 27.13) --
	(342.45, 26.94) --
	(340.09, 26.91);
\definecolor{drawColor}{RGB}{230,159,0}

\path[draw=drawColor,line width= 1.1pt,line join=round] (340.09, 27.60) --
	(342.45, 27.84) --
	(344.82, 28.13) --
	(347.18, 28.51) --
	(349.55, 28.97) --
	(351.91, 29.55) --
	(354.28, 30.25) --
	(356.64, 31.09) --
	(359.01, 32.11) --
	(361.37, 33.31) --
	(363.74, 34.72) --
	(366.10, 36.35) --
	(368.47, 38.24) --
	(370.83, 40.38) --
	(373.20, 42.80) --
	(375.56, 45.50) --
	(377.93, 48.48) --
	(380.29, 51.74) --
	(382.66, 55.28) --
	(385.02, 59.07) --
	(387.39, 63.09) --
	(389.75, 67.33) --
	(392.12, 71.73) --
	(394.48, 76.27) --
	(396.85, 80.90) --
	(399.22, 85.57) --
	(401.58, 90.25) --
	(403.95, 94.88) --
	(406.31, 99.41) --
	(408.68,103.82) --
	(411.04,108.05) --
	(413.41,112.07) --
	(415.77,115.86) --
	(418.14,119.40) --
	(420.50,122.66) --
	(422.87,125.64) --
	(425.23,128.34) --
	(427.60,130.76) --
	(429.96,132.90) --
	(432.33,134.79) --
	(434.69,136.42) --
	(437.06,137.83) --
	(439.42,139.04) --
	(441.79,140.05) --
	(444.15,140.90) --
	(446.52,141.60) --
	(448.88,142.17) --
	(451.25,142.64) --
	(453.61,143.01) --
	(455.98,143.31) --
	(458.34,143.54);
\definecolor{drawColor}{RGB}{0,0,0}

\path[draw=drawColor,line width= 0.6pt,line join=round] (340.09, 27.60) --
	(342.45, 27.84) --
	(344.82, 28.13) --
	(347.18, 28.51) --
	(349.55, 28.97) --
	(351.91, 29.55) --
	(354.28, 30.25) --
	(356.64, 31.09) --
	(359.01, 32.11) --
	(361.37, 33.31) --
	(363.74, 34.72) --
	(366.10, 36.35) --
	(368.47, 38.24) --
	(370.83, 40.38) --
	(373.20, 42.80) --
	(375.56, 45.50) --
	(377.93, 48.48) --
	(380.29, 51.74) --
	(382.66, 55.28) --
	(385.02, 59.07) --
	(387.39, 63.09) --
	(389.75, 67.33) --
	(392.12, 71.73) --
	(394.48, 76.27) --
	(396.85, 80.90) --
	(399.22, 85.57) --
	(401.58, 90.25) --
	(403.95, 94.88) --
	(406.31, 99.41) --
	(408.68,103.82) --
	(411.04,108.05) --
	(413.41,112.07) --
	(415.77,115.86) --
	(418.14,119.40) --
	(420.50,122.66) --
	(422.87,125.64) --
	(425.23,128.34) --
	(427.60,130.76) --
	(429.96,132.90) --
	(432.33,134.79) --
	(434.69,136.42) --
	(437.06,137.83) --
	(439.42,139.04) --
	(441.79,140.05) --
	(444.15,140.90) --
	(446.52,141.60) --
	(448.88,142.17) --
	(451.25,142.64) --
	(453.61,143.01) --
	(455.98,143.31) --
	(458.34,143.54);
\end{scope}
\begin{scope}
\path[clip] (  0.00,  0.00) rectangle (469.75,216.81);
\definecolor{drawColor}{RGB}{0,0,0}

\node[text=drawColor,anchor=base west,inner sep=0pt, outer sep=0pt, scale=  1.1] at (334.18,162.26) {\bfseries $H=0.9$};
\end{scope}
\end{tikzpicture}

%% file: plots/median_normal_m=200_and_m=1000.tex
\begin{tikzpicture}[x=1pt,y=1pt]
\definecolor{fillColor}{RGB}{255,255,255}
\path[use as bounding box,fill=fillColor,fill opacity=0.00] (0,0) rectangle (469.75,433.62);
\begin{scope}
\path[clip] (  0.00,  0.00) rectangle (469.75,433.62);
\definecolor{fillColor}{RGB}{255,255,255}

\path[fill=fillColor] (167.74,408.17) rectangle (302.02,433.62);
\end{scope}
\begin{scope}
\path[clip] (  0.00,  0.00) rectangle (469.75,433.62);
\definecolor{drawColor}{RGB}{0,0,0}

\end{scope}
\begin{scope}
\path[clip] (  0.00,  0.00) rectangle (469.75,433.62);
\definecolor{fillColor}{gray}{0.95}

\path[fill=fillColor] (215.09,413.67) rectangle (229.54,428.12);
\end{scope}
\begin{scope}
\path[clip] (  0.00,  0.00) rectangle (469.75,433.62);
\definecolor{fillColor}{RGB}{146,129,188}

\path[fill=fillColor] (215.80,414.38) rectangle (228.83,427.41);
\end{scope}
\begin{scope}
\path[clip] (  0.00,  0.00) rectangle (469.75,433.62);
\definecolor{fillColor}{gray}{0.95}

\path[fill=fillColor] (265.27,413.67) rectangle (279.73,428.12);
\end{scope}
\begin{scope}
\path[clip] (  0.00,  0.00) rectangle (469.75,433.62);
\definecolor{fillColor}{RGB}{230,159,0}

\path[fill=fillColor] (265.99,414.38) rectangle (279.02,427.41);
\end{scope}
\begin{scope}
\path[clip] (  0.00,  0.00) rectangle (469.75,433.62);
\definecolor{drawColor}{RGB}{0,0,0}

\node[text=drawColor,anchor=base west,inner sep=0pt, outer sep=0pt, scale=  1.10] at (170,417.86) {method:};
\node[text=drawColor,anchor=base west,inner sep=0pt, outer sep=0pt, scale=  0.88] at (235.04,417.86) {asymp};
\end{scope}
\begin{scope}
\path[clip] (  0.00,  0.00) rectangle (469.75,433.62);
\definecolor{drawColor}{RGB}{0,0,0}

\node[text=drawColor,anchor=base west,inner sep=0pt, outer sep=0pt, scale=  0.88] at (285.23,417.86) {HOA};
\end{scope}
\begin{scope}
\path[clip] (  0.00,204.08) rectangle (234.88,408.17);
\definecolor{drawColor}{RGB}{255,255,255}
\definecolor{fillColor}{RGB}{255,255,255}

\path[draw=drawColor,line width= 0.6pt,line join=round,line cap=round,fill=fillColor] (  0.00,204.08) rectangle (234.88,408.17);
\end{scope}
\begin{scope}
\path[clip] ( 38.85,235.34) rectangle (229.38,387.16);
\definecolor{fillColor}{gray}{0.92}

\path[fill=fillColor] ( 38.85,235.34) rectangle (229.38,387.16);
\definecolor{drawColor}{RGB}{255,255,255}

\path[draw=drawColor,line width= 0.3pt,line join=round] ( 38.85,260.33) --
	(229.38,260.33);

\path[draw=drawColor,line width= 0.3pt,line join=round] ( 38.85,296.52) --
	(229.38,296.52);

\path[draw=drawColor,line width= 0.3pt,line join=round] ( 38.85,332.71) --
	(229.38,332.71);

\path[draw=drawColor,line width= 0.3pt,line join=round] ( 38.85,368.90) --
	(229.38,368.90);

\path[draw=drawColor,line width= 0.6pt,line join=round] ( 38.85,242.24) --
	(229.38,242.24);

\path[draw=drawColor,line width= 0.6pt,line join=round] ( 38.85,278.43) --
	(229.38,278.43);

\path[draw=drawColor,line width= 0.6pt,line join=round] ( 38.85,314.61) --
	(229.38,314.61);

\path[draw=drawColor,line width= 0.6pt,line join=round] ( 38.85,350.80) --
	(229.38,350.80);

\path[draw=drawColor,line width= 0.6pt,line join=round] ( 38.85,386.99) --
	(229.38,386.99);

\path[draw=drawColor,line width= 0.6pt,line join=round] ( 51.27,235.34) --
	( 51.27,387.16);

\path[draw=drawColor,line width= 0.6pt,line join=round] ( 71.98,235.34) --
	( 71.98,387.16);

\path[draw=drawColor,line width= 0.6pt,line join=round] ( 92.69,235.34) --
	( 92.69,387.16);

\path[draw=drawColor,line width= 0.6pt,line join=round] (113.40,235.34) --
	(113.40,387.16);

\path[draw=drawColor,line width= 0.6pt,line join=round] (134.11,235.34) --
	(134.11,387.16);

\path[draw=drawColor,line width= 0.6pt,line join=round] (154.82,235.34) --
	(154.82,387.16);

\path[draw=drawColor,line width= 0.6pt,line join=round] (175.53,235.34) --
	(175.53,387.16);

\path[draw=drawColor,line width= 0.6pt,line join=round] (196.24,235.34) --
	(196.24,387.16);

\path[draw=drawColor,line width= 0.6pt,line join=round] (216.95,235.34) --
	(216.95,387.16);
\definecolor{fillColor}{RGB}{146,129,188}

\path[fill=fillColor] ( 46.09,242.24) rectangle ( 51.27,375.19);

\path[fill=fillColor] ( 66.80,242.24) rectangle ( 71.98,376.64);

\path[fill=fillColor] ( 87.51,242.24) rectangle ( 92.69,377.00);

\path[fill=fillColor] (108.22,242.24) rectangle (113.40,378.45);

\path[fill=fillColor] (128.93,242.24) rectangle (134.11,379.03);

\path[fill=fillColor] (149.64,242.24) rectangle (154.82,380.26);

\path[fill=fillColor] (170.35,242.24) rectangle (175.53,380.12);

\path[fill=fillColor] (191.06,242.24) rectangle (196.24,379.83);

\path[fill=fillColor] (211.77,242.24) rectangle (216.95,379.17);
\definecolor{fillColor}{RGB}{230,159,0}

\path[fill=fillColor] ( 51.27,242.24) rectangle ( 56.45,379.03);

\path[fill=fillColor] ( 71.98,242.24) rectangle ( 77.16,380.04);

\path[fill=fillColor] ( 92.69,242.24) rectangle ( 97.87,378.81);

\path[fill=fillColor] (113.40,242.24) rectangle (118.58,379.61);

\path[fill=fillColor] (134.11,242.24) rectangle (139.29,380.19);

\path[fill=fillColor] (154.82,242.24) rectangle (160.00,380.26);

\path[fill=fillColor] (175.53,242.24) rectangle (180.71,370.92);

\path[fill=fillColor] (196.24,242.24) rectangle (201.42,334.81);

\path[fill=fillColor] (216.95,242.24) rectangle (222.13,278.79);
\definecolor{drawColor}{RGB}{169,169,169}

\path[draw=drawColor,line width= 0.6pt,dash pattern=on 4pt off 4pt ,line join=round] ( 38.85,379.75) -- (229.38,379.75);
\end{scope}
\begin{scope}
\path[clip] (  0.00,  0.00) rectangle (469.75,433.62);
\definecolor{drawColor}{gray}{0.30}

\node[text=drawColor,anchor=base east,inner sep=0pt, outer sep=0pt, scale=  0.88] at ( 33.90,239.21) {0.00};

\node[text=drawColor,anchor=base east,inner sep=0pt, outer sep=0pt, scale=  0.88] at ( 33.90,275.40) {0.25};

\node[text=drawColor,anchor=base east,inner sep=0pt, outer sep=0pt, scale=  0.88] at ( 33.90,311.58) {0.50};

\node[text=drawColor,anchor=base east,inner sep=0pt, outer sep=0pt, scale=  0.88] at ( 33.90,347.77) {0.75};

\node[text=drawColor,anchor=base east,inner sep=0pt, outer sep=0pt, scale=  0.88] at ( 33.90,383.96) {1.00};
\end{scope}
\begin{scope}
\path[clip] (  0.00,  0.00) rectangle (469.75,433.62);
\definecolor{drawColor}{gray}{0.20}

\path[draw=drawColor,line width= 0.6pt,line join=round] ( 36.10,242.24) --
	( 38.85,242.24);

\path[draw=drawColor,line width= 0.6pt,line join=round] ( 36.10,278.43) --
	( 38.85,278.43);

\path[draw=drawColor,line width= 0.6pt,line join=round] ( 36.10,314.61) --
	( 38.85,314.61);

\path[draw=drawColor,line width= 0.6pt,line join=round] ( 36.10,350.80) --
	( 38.85,350.80);

\path[draw=drawColor,line width= 0.6pt,line join=round] ( 36.10,386.99) --
	( 38.85,386.99);
\end{scope}
\begin{scope}
\path[clip] (  0.00,  0.00) rectangle (469.75,433.62);
\definecolor{drawColor}{gray}{0.20}

\path[draw=drawColor,line width= 0.6pt,line join=round] ( 51.27,232.59) --
	( 51.27,235.34);

\path[draw=drawColor,line width= 0.6pt,line join=round] ( 71.98,232.59) --
	( 71.98,235.34);

\path[draw=drawColor,line width= 0.6pt,line join=round] ( 92.69,232.59) --
	( 92.69,235.34);

\path[draw=drawColor,line width= 0.6pt,line join=round] (113.40,232.59) --
	(113.40,235.34);

\path[draw=drawColor,line width= 0.6pt,line join=round] (134.11,232.59) --
	(134.11,235.34);

\path[draw=drawColor,line width= 0.6pt,line join=round] (154.82,232.59) --
	(154.82,235.34);

\path[draw=drawColor,line width= 0.6pt,line join=round] (175.53,232.59) --
	(175.53,235.34);

\path[draw=drawColor,line width= 0.6pt,line join=round] (196.24,232.59) --
	(196.24,235.34);

\path[draw=drawColor,line width= 0.6pt,line join=round] (216.95,232.59) --
	(216.95,235.34);
\end{scope}
\begin{scope}
\path[clip] (  0.00,  0.00) rectangle (469.75,433.62);
\definecolor{drawColor}{gray}{0.30}

\node[text=drawColor,anchor=base,inner sep=0pt, outer sep=0pt, scale=  0.88] at ( 51.27,224.33) {0.55};

\node[text=drawColor,anchor=base,inner sep=0pt, outer sep=0pt, scale=  0.88] at ( 71.98,224.33) {0.6};

\node[text=drawColor,anchor=base,inner sep=0pt, outer sep=0pt, scale=  0.88] at ( 92.69,224.33) {0.65};

\node[text=drawColor,anchor=base,inner sep=0pt, outer sep=0pt, scale=  0.88] at (113.40,224.33) {0.7};

\node[text=drawColor,anchor=base,inner sep=0pt, outer sep=0pt, scale=  0.88] at (134.11,224.33) {0.75};

\node[text=drawColor,anchor=base,inner sep=0pt, outer sep=0pt, scale=  0.88] at (154.82,224.33) {0.8};

\node[text=drawColor,anchor=base,inner sep=0pt, outer sep=0pt, scale=  0.88] at (175.53,224.33) {0.85};

\node[text=drawColor,anchor=base,inner sep=0pt, outer sep=0pt, scale=  0.88] at (196.24,224.33) {0.9};

\node[text=drawColor,anchor=base,inner sep=0pt, outer sep=0pt, scale=  0.88] at (216.95,224.33) {0.95};
\end{scope}
\begin{scope}
\path[clip] (  0.00,  0.00) rectangle (469.75,433.62);
\definecolor{drawColor}{RGB}{0,0,0}

\node[text=drawColor,anchor=base,inner sep=0pt, outer sep=0pt, scale=  1.10] at (134.11,212.01) {$H$};
\end{scope}
\begin{scope}
\path[clip] (  0.00,  0.00) rectangle (469.75,433.62);
\definecolor{drawColor}{RGB}{0,0,0}

\node[text=drawColor,rotate= 90.00,anchor=base,inner sep=0pt, outer sep=0pt, scale=  1.10] at ( 13.08,311.25) {rate};
\end{scope}
\begin{scope}
\path[clip] (  0.00,  0.00) rectangle (469.75,433.62);
\definecolor{drawColor}{RGB}{0,0,0}

\node[text=drawColor,anchor=base west,inner sep=0pt, outer sep=0pt, scale=  1.10] at ( 38.85,395.09) {Coverage rate};
\end{scope}
\begin{scope}
\path[clip] (234.88,204.08) rectangle (469.75,408.17);
\definecolor{drawColor}{RGB}{255,255,255}
\definecolor{fillColor}{RGB}{255,255,255}

\path[draw=drawColor,line width= 0.6pt,line join=round,line cap=round,fill=fillColor] (234.88,204.08) rectangle (469.76,408.17);
\end{scope}
\begin{scope}
\path[clip] (262.48,235.34) rectangle (464.25,387.16);
\definecolor{fillColor}{gray}{0.92}

\path[fill=fillColor] (262.48,235.34) rectangle (464.25,387.16);
\definecolor{drawColor}{RGB}{255,255,255}

\path[draw=drawColor,line width= 0.3pt,line join=round] (262.48,265.18) --
	(464.25,265.18);

\path[draw=drawColor,line width= 0.3pt,line join=round] (262.48,311.07) --
	(464.25,311.07);

\path[draw=drawColor,line width= 0.3pt,line join=round] (262.48,356.96) --
	(464.25,356.96);

\path[draw=drawColor,line width= 0.6pt,line join=round] (262.48,242.24) --
	(464.25,242.24);

\path[draw=drawColor,line width= 0.6pt,line join=round] (262.48,288.13) --
	(464.25,288.13);

\path[draw=drawColor,line width= 0.6pt,line join=round] (262.48,334.02) --
	(464.25,334.02);

\path[draw=drawColor,line width= 0.6pt,line join=round] (262.48,379.91) --
	(464.25,379.91);

\path[draw=drawColor,line width= 0.6pt,line join=round] (275.64,235.34) --
	(275.64,387.16);

\path[draw=drawColor,line width= 0.6pt,line join=round] (297.57,235.34) --
	(297.57,387.16);

\path[draw=drawColor,line width= 0.6pt,line join=round] (319.50,235.34) --
	(319.50,387.16);

\path[draw=drawColor,line width= 0.6pt,line join=round] (341.44,235.34) --
	(341.44,387.16);

\path[draw=drawColor,line width= 0.6pt,line join=round] (363.37,235.34) --
	(363.37,387.16);

\path[draw=drawColor,line width= 0.6pt,line join=round] (385.30,235.34) --
	(385.30,387.16);

\path[draw=drawColor,line width= 0.6pt,line join=round] (407.23,235.34) --
	(407.23,387.16);

\path[draw=drawColor,line width= 0.6pt,line join=round] (429.16,235.34) --
	(429.16,387.16);

\path[draw=drawColor,line width= 0.6pt,line join=round] (451.10,235.34) --
	(451.10,387.16);
\definecolor{fillColor}{RGB}{146,129,188}

\path[fill=fillColor] (270.16,242.24) rectangle (275.64,258.82);

\path[fill=fillColor] (292.09,242.24) rectangle (297.57,263.84);

\path[fill=fillColor] (314.02,242.24) rectangle (319.50,270.40);

\path[fill=fillColor] (335.95,242.24) rectangle (341.44,278.94);

\path[fill=fillColor] (357.89,242.24) rectangle (363.37,290.07);

\path[fill=fillColor] (379.82,242.24) rectangle (385.30,304.58);

\path[fill=fillColor] (401.75,242.24) rectangle (407.23,323.49);

\path[fill=fillColor] (423.68,242.24) rectangle (429.16,348.14);

\path[fill=fillColor] (445.61,242.24) rectangle (451.10,380.26);
\definecolor{fillColor}{RGB}{230,159,0}

\path[fill=fillColor] (275.64,242.24) rectangle (281.12,251.82);

\path[fill=fillColor] (297.57,242.24) rectangle (303.06,251.85);

\path[fill=fillColor] (319.50,242.24) rectangle (324.99,251.84);

\path[fill=fillColor] (341.44,242.24) rectangle (346.92,251.83);

\path[fill=fillColor] (363.37,242.24) rectangle (368.85,251.84);

\path[fill=fillColor] (385.30,242.24) rectangle (390.78,251.84);

\path[fill=fillColor] (407.23,242.24) rectangle (412.72,250.30);

\path[fill=fillColor] (429.16,242.24) rectangle (434.65,249.50);

\path[fill=fillColor] (451.10,242.24) rectangle (456.58,248.52);
\end{scope}
\begin{scope}
\path[clip] (  0.00,  0.00) rectangle (469.75,433.62);
\definecolor{drawColor}{gray}{0.30}

\node[text=drawColor,anchor=base east,inner sep=0pt, outer sep=0pt, scale=  0.88] at (257.53,239.21) {0};

\node[text=drawColor,anchor=base east,inner sep=0pt, outer sep=0pt, scale=  0.88] at (257.53,285.10) {1};

\node[text=drawColor,anchor=base east,inner sep=0pt, outer sep=0pt, scale=  0.88] at (257.53,330.99) {2};

\node[text=drawColor,anchor=base east,inner sep=0pt, outer sep=0pt, scale=  0.88] at (257.53,376.88) {3};
\end{scope}
\begin{scope}
\path[clip] (  0.00,  0.00) rectangle (469.75,433.62);
\definecolor{drawColor}{gray}{0.20}

\path[draw=drawColor,line width= 0.6pt,line join=round] (259.73,242.24) --
	(262.48,242.24);

\path[draw=drawColor,line width= 0.6pt,line join=round] (259.73,288.13) --
	(262.48,288.13);

\path[draw=drawColor,line width= 0.6pt,line join=round] (259.73,334.02) --
	(262.48,334.02);

\path[draw=drawColor,line width= 0.6pt,line join=round] (259.73,379.91) --
	(262.48,379.91);
\end{scope}
\begin{scope}
\path[clip] (  0.00,  0.00) rectangle (469.75,433.62);
\definecolor{drawColor}{gray}{0.20}

\path[draw=drawColor,line width= 0.6pt,line join=round] (275.64,232.59) --
	(275.64,235.34);

\path[draw=drawColor,line width= 0.6pt,line join=round] (297.57,232.59) --
	(297.57,235.34);

\path[draw=drawColor,line width= 0.6pt,line join=round] (319.50,232.59) --
	(319.50,235.34);

\path[draw=drawColor,line width= 0.6pt,line join=round] (341.44,232.59) --
	(341.44,235.34);

\path[draw=drawColor,line width= 0.6pt,line join=round] (363.37,232.59) --
	(363.37,235.34);

\path[draw=drawColor,line width= 0.6pt,line join=round] (385.30,232.59) --
	(385.30,235.34);

\path[draw=drawColor,line width= 0.6pt,line join=round] (407.23,232.59) --
	(407.23,235.34);

\path[draw=drawColor,line width= 0.6pt,line join=round] (429.16,232.59) --
	(429.16,235.34);

\path[draw=drawColor,line width= 0.6pt,line join=round] (451.10,232.59) --
	(451.10,235.34);
\end{scope}
\begin{scope}
\path[clip] (  0.00,  0.00) rectangle (469.75,433.62);
\definecolor{drawColor}{gray}{0.30}

\node[text=drawColor,anchor=base,inner sep=0pt, outer sep=0pt, scale=  0.88] at (275.64,224.33) {0.55};

\node[text=drawColor,anchor=base,inner sep=0pt, outer sep=0pt, scale=  0.88] at (297.57,224.33) {0.6};

\node[text=drawColor,anchor=base,inner sep=0pt, outer sep=0pt, scale=  0.88] at (319.50,224.33) {0.65};

\node[text=drawColor,anchor=base,inner sep=0pt, outer sep=0pt, scale=  0.88] at (341.44,224.33) {0.7};

\node[text=drawColor,anchor=base,inner sep=0pt, outer sep=0pt, scale=  0.88] at (363.37,224.33) {0.75};

\node[text=drawColor,anchor=base,inner sep=0pt, outer sep=0pt, scale=  0.88] at (385.30,224.33) {0.8};

\node[text=drawColor,anchor=base,inner sep=0pt, outer sep=0pt, scale=  0.88] at (407.23,224.33) {0.85};

\node[text=drawColor,anchor=base,inner sep=0pt, outer sep=0pt, scale=  0.88] at (429.16,224.33) {0.9};

\node[text=drawColor,anchor=base,inner sep=0pt, outer sep=0pt, scale=  0.88] at (451.10,224.33) {0.95};
\end{scope}
\begin{scope}
\path[clip] (  0.00,  0.00) rectangle (469.75,433.62);
\definecolor{drawColor}{RGB}{0,0,0}

\node[text=drawColor,anchor=base,inner sep=0pt, outer sep=0pt, scale=  1.10] at (363.37,212.01) {$H$};
\end{scope}
\begin{scope}
\path[clip] (  0.00,  0.00) rectangle (469.75,433.62);
\definecolor{drawColor}{RGB}{0,0,0}

\node[text=drawColor,rotate= 90.00,anchor=base,inner sep=0pt, outer sep=0pt, scale=  1.10] at (247.95,311.25) {length};
\end{scope}
\begin{scope}
\path[clip] (  0.00,  0.00) rectangle (469.75,433.62);
\definecolor{drawColor}{RGB}{0,0,0}

\node[text=drawColor,anchor=base west,inner sep=0pt, outer sep=0pt, scale=  1.10] at (262.48,395.09) {Interval length};
\end{scope}
\begin{scope}
\path[clip] (  0.00,  0.00) rectangle (234.88,204.08);
\definecolor{drawColor}{RGB}{255,255,255}
\definecolor{fillColor}{RGB}{255,255,255}

\path[draw=drawColor,line width= 0.6pt,line join=round,line cap=round,fill=fillColor] (  0.00,  0.00) rectangle (234.88,204.08);
\end{scope}
\begin{scope}
\path[clip] ( 38.85, 31.25) rectangle (229.38,183.08);
\definecolor{fillColor}{gray}{0.92}

\path[fill=fillColor] ( 38.85, 31.25) rectangle (229.38,183.08);
\definecolor{drawColor}{RGB}{255,255,255}

\path[draw=drawColor,line width= 0.3pt,line join=round] ( 38.85, 56.27) --
	(229.38, 56.27);

\path[draw=drawColor,line width= 0.3pt,line join=round] ( 38.85, 92.49) --
	(229.38, 92.49);

\path[draw=drawColor,line width= 0.3pt,line join=round] ( 38.85,128.72) --
	(229.38,128.72);

\path[draw=drawColor,line width= 0.3pt,line join=round] ( 38.85,164.95) --
	(229.38,164.95);

\path[draw=drawColor,line width= 0.6pt,line join=round] ( 38.85, 38.15) --
	(229.38, 38.15);

\path[draw=drawColor,line width= 0.6pt,line join=round] ( 38.85, 74.38) --
	(229.38, 74.38);

\path[draw=drawColor,line width= 0.6pt,line join=round] ( 38.85,110.61) --
	(229.38,110.61);

\path[draw=drawColor,line width= 0.6pt,line join=round] ( 38.85,146.83) --
	(229.38,146.83);

\path[draw=drawColor,line width= 0.6pt,line join=round] ( 38.85,183.06) --
	(229.38,183.06);

\path[draw=drawColor,line width= 0.6pt,line join=round] ( 51.27, 31.25) --
	( 51.27,183.08);

\path[draw=drawColor,line width= 0.6pt,line join=round] ( 71.98, 31.25) --
	( 71.98,183.08);

\path[draw=drawColor,line width= 0.6pt,line join=round] ( 92.69, 31.25) --
	( 92.69,183.08);

\path[draw=drawColor,line width= 0.6pt,line join=round] (113.40, 31.25) --
	(113.40,183.08);

\path[draw=drawColor,line width= 0.6pt,line join=round] (134.11, 31.25) --
	(134.11,183.08);

\path[draw=drawColor,line width= 0.6pt,line join=round] (154.82, 31.25) --
	(154.82,183.08);

\path[draw=drawColor,line width= 0.6pt,line join=round] (175.53, 31.25) --
	(175.53,183.08);

\path[draw=drawColor,line width= 0.6pt,line join=round] (196.24, 31.25) --
	(196.24,183.08);

\path[draw=drawColor,line width= 0.6pt,line join=round] (216.95, 31.25) --
	(216.95,183.08);
\definecolor{fillColor}{RGB}{146,129,188}

\path[fill=fillColor] ( 46.09, 38.15) rectangle ( 51.27,169.58);

\path[fill=fillColor] ( 66.80, 38.15) rectangle ( 71.98,174.87);

\path[fill=fillColor] ( 87.51, 38.15) rectangle ( 92.69,174.29);

\path[fill=fillColor] (108.22, 38.15) rectangle (113.40,176.18);

\path[fill=fillColor] (128.93, 38.15) rectangle (134.11,175.67);

\path[fill=fillColor] (149.64, 38.15) rectangle (154.82,175.74);

\path[fill=fillColor] (170.35, 38.15) rectangle (175.53,176.10);

\path[fill=fillColor] (191.06, 38.15) rectangle (196.24,174.95);

\path[fill=fillColor] (211.77, 38.15) rectangle (216.95,175.60);
\definecolor{fillColor}{RGB}{230,159,0}

\path[fill=fillColor] ( 51.27, 38.15) rectangle ( 56.45,175.89);

\path[fill=fillColor] ( 71.98, 38.15) rectangle ( 77.16,173.86);

\path[fill=fillColor] ( 92.69, 38.15) rectangle ( 97.87,175.89);

\path[fill=fillColor] (113.40, 38.15) rectangle (118.58,175.60);

\path[fill=fillColor] (134.11, 38.15) rectangle (139.29,175.89);

\path[fill=fillColor] (154.82, 38.15) rectangle (160.00,174.51);

\path[fill=fillColor] (175.53, 38.15) rectangle (180.71,161.03);

\path[fill=fillColor] (196.24, 38.15) rectangle (201.42,116.40);

\path[fill=fillColor] (216.95, 38.15) rectangle (222.13, 58.01);
\definecolor{drawColor}{RGB}{169,169,169}

\path[draw=drawColor,line width= 0.6pt,dash pattern=on 4pt off 4pt ,line join=round] ( 38.85,175.81) -- (229.38,175.81);
\end{scope}
\begin{scope}
\path[clip] (  0.00,  0.00) rectangle (469.75,433.62);
\definecolor{drawColor}{gray}{0.30}

\node[text=drawColor,anchor=base east,inner sep=0pt, outer sep=0pt, scale=  0.88] at ( 33.90, 35.12) {0.00};

\node[text=drawColor,anchor=base east,inner sep=0pt, outer sep=0pt, scale=  0.88] at ( 33.90, 71.35) {0.25};

\node[text=drawColor,anchor=base east,inner sep=0pt, outer sep=0pt, scale=  0.88] at ( 33.90,107.58) {0.50};

\node[text=drawColor,anchor=base east,inner sep=0pt, outer sep=0pt, scale=  0.88] at ( 33.90,143.80) {0.75};

\node[text=drawColor,anchor=base east,inner sep=0pt, outer sep=0pt, scale=  0.88] at ( 33.90,180.03) {1.00};
\end{scope}
\begin{scope}
\path[clip] (  0.00,  0.00) rectangle (469.75,433.62);
\definecolor{drawColor}{gray}{0.20}

\path[draw=drawColor,line width= 0.6pt,line join=round] ( 36.10, 38.15) --
	( 38.85, 38.15);

\path[draw=drawColor,line width= 0.6pt,line join=round] ( 36.10, 74.38) --
	( 38.85, 74.38);

\path[draw=drawColor,line width= 0.6pt,line join=round] ( 36.10,110.61) --
	( 38.85,110.61);

\path[draw=drawColor,line width= 0.6pt,line join=round] ( 36.10,146.83) --
	( 38.85,146.83);

\path[draw=drawColor,line width= 0.6pt,line join=round] ( 36.10,183.06) --
	( 38.85,183.06);
\end{scope}
\begin{scope}
\path[clip] (  0.00,  0.00) rectangle (469.75,433.62);
\definecolor{drawColor}{gray}{0.20}

\path[draw=drawColor,line width= 0.6pt,line join=round] ( 51.27, 28.50) --
	( 51.27, 31.25);

\path[draw=drawColor,line width= 0.6pt,line join=round] ( 71.98, 28.50) --
	( 71.98, 31.25);

\path[draw=drawColor,line width= 0.6pt,line join=round] ( 92.69, 28.50) --
	( 92.69, 31.25);

\path[draw=drawColor,line width= 0.6pt,line join=round] (113.40, 28.50) --
	(113.40, 31.25);

\path[draw=drawColor,line width= 0.6pt,line join=round] (134.11, 28.50) --
	(134.11, 31.25);

\path[draw=drawColor,line width= 0.6pt,line join=round] (154.82, 28.50) --
	(154.82, 31.25);

\path[draw=drawColor,line width= 0.6pt,line join=round] (175.53, 28.50) --
	(175.53, 31.25);

\path[draw=drawColor,line width= 0.6pt,line join=round] (196.24, 28.50) --
	(196.24, 31.25);

\path[draw=drawColor,line width= 0.6pt,line join=round] (216.95, 28.50) --
	(216.95, 31.25);
\end{scope}
\begin{scope}
\path[clip] (  0.00,  0.00) rectangle (469.75,433.62);
\definecolor{drawColor}{gray}{0.30}

\node[text=drawColor,anchor=base,inner sep=0pt, outer sep=0pt, scale=  0.88] at ( 51.27, 20.24) {0.55};

\node[text=drawColor,anchor=base,inner sep=0pt, outer sep=0pt, scale=  0.88] at ( 71.98, 20.24) {0.6};

\node[text=drawColor,anchor=base,inner sep=0pt, outer sep=0pt, scale=  0.88] at ( 92.69, 20.24) {0.65};

\node[text=drawColor,anchor=base,inner sep=0pt, outer sep=0pt, scale=  0.88] at (113.40, 20.24) {0.7};

\node[text=drawColor,anchor=base,inner sep=0pt, outer sep=0pt, scale=  0.88] at (134.11, 20.24) {0.75};

\node[text=drawColor,anchor=base,inner sep=0pt, outer sep=0pt, scale=  0.88] at (154.82, 20.24) {0.8};

\node[text=drawColor,anchor=base,inner sep=0pt, outer sep=0pt, scale=  0.88] at (175.53, 20.24) {0.85};

\node[text=drawColor,anchor=base,inner sep=0pt, outer sep=0pt, scale=  0.88] at (196.24, 20.24) {0.9};

\node[text=drawColor,anchor=base,inner sep=0pt, outer sep=0pt, scale=  0.88] at (216.95, 20.24) {0.95};
\end{scope}
\begin{scope}
\path[clip] (  0.00,  0.00) rectangle (469.75,433.62);
\definecolor{drawColor}{RGB}{0,0,0}

\node[text=drawColor,anchor=base,inner sep=0pt, outer sep=0pt, scale=  1.10] at (134.11,  7.93) {$H$};
\end{scope}
\begin{scope}
\path[clip] (  0.00,  0.00) rectangle (469.75,433.62);
\definecolor{drawColor}{RGB}{0,0,0}

\node[text=drawColor,rotate= 90.00,anchor=base,inner sep=0pt, outer sep=0pt, scale=  1.10] at ( 13.08,107.17) {rate};
\end{scope}
\begin{scope}
\path[clip] (  0.00,  0.00) rectangle (469.75,433.62);
\definecolor{drawColor}{RGB}{0,0,0}

\node[text=drawColor,anchor=base west,inner sep=0pt, outer sep=0pt, scale=  1.10] at ( 38.85,191.01) {Coverage rate};
\end{scope}
\begin{scope}
\path[clip] (234.88,  0.00) rectangle (469.75,204.08);
\definecolor{drawColor}{RGB}{255,255,255}
\definecolor{fillColor}{RGB}{255,255,255}

\path[draw=drawColor,line width= 0.6pt,line join=round,line cap=round,fill=fillColor] (234.88,  0.00) rectangle (469.76,204.08);
\end{scope}
\begin{scope}
\path[clip] (262.48, 31.25) rectangle (464.25,183.08);
\definecolor{fillColor}{gray}{0.92}

\path[fill=fillColor] (262.48, 31.25) rectangle (464.25,183.08);
\definecolor{drawColor}{RGB}{255,255,255}

\path[draw=drawColor,line width= 0.3pt,line join=round] (262.48, 63.02) --
	(464.25, 63.02);

\path[draw=drawColor,line width= 0.3pt,line join=round] (262.48,112.76) --
	(464.25,112.76);

\path[draw=drawColor,line width= 0.3pt,line join=round] (262.48,162.49) --
	(464.25,162.49);

\path[draw=drawColor,line width= 0.6pt,line join=round] (262.48, 38.15) --
	(464.25, 38.15);

\path[draw=drawColor,line width= 0.6pt,line join=round] (262.48, 87.89) --
	(464.25, 87.89);

\path[draw=drawColor,line width= 0.6pt,line join=round] (262.48,137.63) --
	(464.25,137.63);

\path[draw=drawColor,line width= 0.6pt,line join=round] (275.64, 31.25) --
	(275.64,183.08);

\path[draw=drawColor,line width= 0.6pt,line join=round] (297.57, 31.25) --
	(297.57,183.08);

\path[draw=drawColor,line width= 0.6pt,line join=round] (319.50, 31.25) --
	(319.50,183.08);

\path[draw=drawColor,line width= 0.6pt,line join=round] (341.44, 31.25) --
	(341.44,183.08);

\path[draw=drawColor,line width= 0.6pt,line join=round] (363.37, 31.25) --
	(363.37,183.08);

\path[draw=drawColor,line width= 0.6pt,line join=round] (385.30, 31.25) --
	(385.30,183.08);

\path[draw=drawColor,line width= 0.6pt,line join=round] (407.23, 31.25) --
	(407.23,183.08);

\path[draw=drawColor,line width= 0.6pt,line join=round] (429.16, 31.25) --
	(429.16,183.08);

\path[draw=drawColor,line width= 0.6pt,line join=round] (451.10, 31.25) --
	(451.10,183.08);
\definecolor{fillColor}{RGB}{146,129,188}

\path[fill=fillColor] (270.16, 38.15) rectangle (275.64, 46.86);

\path[fill=fillColor] (292.09, 38.15) rectangle (297.57, 50.46);

\path[fill=fillColor] (314.02, 38.15) rectangle (319.50, 55.53);

\path[fill=fillColor] (335.95, 38.15) rectangle (341.44, 62.70);

\path[fill=fillColor] (357.89, 38.15) rectangle (363.37, 72.82);

\path[fill=fillColor] (379.82, 38.15) rectangle (385.30, 87.13);

\path[fill=fillColor] (401.75, 38.15) rectangle (407.23,107.33);

\path[fill=fillColor] (423.68, 38.15) rectangle (429.16,135.87);

\path[fill=fillColor] (445.61, 38.15) rectangle (451.10,176.18);
\definecolor{fillColor}{RGB}{230,159,0}

\path[fill=fillColor] (275.64, 38.15) rectangle (281.12, 42.81);

\path[fill=fillColor] (297.57, 38.15) rectangle (303.06, 42.81);

\path[fill=fillColor] (319.50, 38.15) rectangle (324.99, 42.81);

\path[fill=fillColor] (341.44, 38.15) rectangle (346.92, 42.81);

\path[fill=fillColor] (363.37, 38.15) rectangle (368.85, 42.81);

\path[fill=fillColor] (385.30, 38.15) rectangle (390.78, 42.81);

\path[fill=fillColor] (407.23, 38.15) rectangle (412.72, 42.07);

\path[fill=fillColor] (429.16, 38.15) rectangle (434.65, 41.69);

\path[fill=fillColor] (451.10, 38.15) rectangle (456.58, 41.21);
\end{scope}
\begin{scope}
\path[clip] (  0.00,  0.00) rectangle (469.75,433.62);
\definecolor{drawColor}{gray}{0.30}

\node[text=drawColor,anchor=base east,inner sep=0pt, outer sep=0pt, scale=  0.88] at (257.53, 35.12) {0};

\node[text=drawColor,anchor=base east,inner sep=0pt, outer sep=0pt, scale=  0.88] at (257.53, 84.86) {1};

\node[text=drawColor,anchor=base east,inner sep=0pt, outer sep=0pt, scale=  0.88] at (257.53,134.60) {2};
\end{scope}
\begin{scope}
\path[clip] (  0.00,  0.00) rectangle (469.75,433.62);
\definecolor{drawColor}{gray}{0.20}

\path[draw=drawColor,line width= 0.6pt,line join=round] (259.73, 38.15) --
	(262.48, 38.15);

\path[draw=drawColor,line width= 0.6pt,line join=round] (259.73, 87.89) --
	(262.48, 87.89);

\path[draw=drawColor,line width= 0.6pt,line join=round] (259.73,137.63) --
	(262.48,137.63);
\end{scope}
\begin{scope}
\path[clip] (  0.00,  0.00) rectangle (469.75,433.62);
\definecolor{drawColor}{gray}{0.20}

\path[draw=drawColor,line width= 0.6pt,line join=round] (275.64, 28.50) --
	(275.64, 31.25);

\path[draw=drawColor,line width= 0.6pt,line join=round] (297.57, 28.50) --
	(297.57, 31.25);

\path[draw=drawColor,line width= 0.6pt,line join=round] (319.50, 28.50) --
	(319.50, 31.25);

\path[draw=drawColor,line width= 0.6pt,line join=round] (341.44, 28.50) --
	(341.44, 31.25);

\path[draw=drawColor,line width= 0.6pt,line join=round] (363.37, 28.50) --
	(363.37, 31.25);

\path[draw=drawColor,line width= 0.6pt,line join=round] (385.30, 28.50) --
	(385.30, 31.25);

\path[draw=drawColor,line width= 0.6pt,line join=round] (407.23, 28.50) --
	(407.23, 31.25);

\path[draw=drawColor,line width= 0.6pt,line join=round] (429.16, 28.50) --
	(429.16, 31.25);

\path[draw=drawColor,line width= 0.6pt,line join=round] (451.10, 28.50) --
	(451.10, 31.25);
\end{scope}
\begin{scope}
\path[clip] (  0.00,  0.00) rectangle (469.75,433.62);
\definecolor{drawColor}{gray}{0.30}

\node[text=drawColor,anchor=base,inner sep=0pt, outer sep=0pt, scale=  0.88] at (275.64, 20.24) {0.55};

\node[text=drawColor,anchor=base,inner sep=0pt, outer sep=0pt, scale=  0.88] at (297.57, 20.24) {0.6};

\node[text=drawColor,anchor=base,inner sep=0pt, outer sep=0pt, scale=  0.88] at (319.50, 20.24) {0.65};

\node[text=drawColor,anchor=base,inner sep=0pt, outer sep=0pt, scale=  0.88] at (341.44, 20.24) {0.7};

\node[text=drawColor,anchor=base,inner sep=0pt, outer sep=0pt, scale=  0.88] at (363.37, 20.24) {0.75};

\node[text=drawColor,anchor=base,inner sep=0pt, outer sep=0pt, scale=  0.88] at (385.30, 20.24) {0.8};

\node[text=drawColor,anchor=base,inner sep=0pt, outer sep=0pt, scale=  0.88] at (407.23, 20.24) {0.85};

\node[text=drawColor,anchor=base,inner sep=0pt, outer sep=0pt, scale=  0.88] at (429.16, 20.24) {0.9};

\node[text=drawColor,anchor=base,inner sep=0pt, outer sep=0pt, scale=  0.88] at (451.10, 20.24) {0.95};
\end{scope}
\begin{scope}
\path[clip] (  0.00,  0.00) rectangle (469.75,433.62);
\definecolor{drawColor}{RGB}{0,0,0}

\node[text=drawColor,anchor=base,inner sep=0pt, outer sep=0pt, scale=  1.10] at (363.37,  7.93) {$H$};
\end{scope}
\begin{scope}
\path[clip] (  0.00,  0.00) rectangle (469.75,433.62);
\definecolor{drawColor}{RGB}{0,0,0}

\node[text=drawColor,rotate= 90.00,anchor=base,inner sep=0pt, outer sep=0pt, scale=  1.10] at (247.95,107.17) {length};
\end{scope}
\begin{scope}
\path[clip] (  0.00,  0.00) rectangle (469.75,433.62);
\definecolor{drawColor}{RGB}{0,0,0}

\node[text=drawColor,anchor=base west,inner sep=0pt, outer sep=0pt, scale=  1.10] at (262.48,191.01) {Interval length};
\end{scope}
\end{tikzpicture}

%% file: plots/median_normal_Hestimated_m=200_and_m=1000.tex
\begin{tikzpicture}[x=1pt,y=1pt]
\definecolor{fillColor}{RGB}{255,255,255}
\path[use as bounding box,fill=fillColor,fill opacity=0.00] (0,0) rectangle (469.75,433.62);
\begin{scope}
\path[clip] (  0.00,  0.00) rectangle (469.75,433.62);
\definecolor{fillColor}{RGB}{255,255,255}

\path[fill=fillColor] (167.74,408.17) rectangle (302.02,433.62);
\end{scope}
\begin{scope}
\path[clip] (  0.00,  0.00) rectangle (469.75,433.62);
\definecolor{drawColor}{RGB}{0,0,0}
\end{scope}
\begin{scope}
\path[clip] (  0.00,  0.00) rectangle (469.75,433.62);
\definecolor{fillColor}{gray}{0.95}

\path[fill=fillColor] (215.09,413.67) rectangle (229.54,428.12);
\end{scope}
\begin{scope}
\path[clip] (  0.00,  0.00) rectangle (469.75,433.62);
\definecolor{fillColor}{RGB}{146,129,188}

\path[fill=fillColor] (215.80,414.38) rectangle (228.83,427.41);
\end{scope}
\begin{scope}
\path[clip] (  0.00,  0.00) rectangle (469.75,433.62);
\definecolor{fillColor}{gray}{0.95}

\path[fill=fillColor] (265.27,413.67) rectangle (279.73,428.12);
\end{scope}
\begin{scope}
\path[clip] (  0.00,  0.00) rectangle (469.75,433.62);
\definecolor{fillColor}{RGB}{230,159,0}

\path[fill=fillColor] (265.99,414.38) rectangle (279.02,427.41);
\end{scope}
\begin{scope}
\path[clip] (  0.00,  0.00) rectangle (469.75,433.62);
\definecolor{drawColor}{RGB}{0,0,0}

\node[text=drawColor,anchor=base west,inner sep=0pt, outer sep=0pt, scale=  1.10] at (170, 417.86) {method:};
\node[text=drawColor,anchor=base west,inner sep=0pt, outer sep=0pt, scale=  0.88] at (235.04,417.86) {asymp};
\end{scope}
\begin{scope}
\path[clip] (  0.00,  0.00) rectangle (469.75,433.62);
\definecolor{drawColor}{RGB}{0,0,0}

\node[text=drawColor,anchor=base west,inner sep=0pt, outer sep=0pt, scale=  0.88] at (285.23,417.86) {HOA};
\end{scope}
\begin{scope}
\path[clip] (  0.00,204.08) rectangle (234.88,408.17);
\definecolor{drawColor}{RGB}{255,255,255}
\definecolor{fillColor}{RGB}{255,255,255}

\path[draw=drawColor,line width= 0.6pt,line join=round,line cap=round,fill=fillColor] (  0.00,204.08) rectangle (234.88,408.17);
\end{scope}
\begin{scope}
\path[clip] ( 38.85,235.34) rectangle (229.38,387.16);
\definecolor{fillColor}{gray}{0.92}

\path[fill=fillColor] ( 38.85,235.34) rectangle (229.38,387.16);
\definecolor{drawColor}{RGB}{255,255,255}

\path[draw=drawColor,line width= 0.3pt,line join=round] ( 38.85,259.91) --
	(229.38,259.91);

\path[draw=drawColor,line width= 0.3pt,line join=round] ( 38.85,295.27) --
	(229.38,295.27);

\path[draw=drawColor,line width= 0.3pt,line join=round] ( 38.85,330.62) --
	(229.38,330.62);

\path[draw=drawColor,line width= 0.3pt,line join=round] ( 38.85,365.98) --
	(229.38,365.98);

\path[draw=drawColor,line width= 0.6pt,line join=round] ( 38.85,242.24) --
	(229.38,242.24);

\path[draw=drawColor,line width= 0.6pt,line join=round] ( 38.85,277.59) --
	(229.38,277.59);

\path[draw=drawColor,line width= 0.6pt,line join=round] ( 38.85,312.95) --
	(229.38,312.95);

\path[draw=drawColor,line width= 0.6pt,line join=round] ( 38.85,348.30) --
	(229.38,348.30);

\path[draw=drawColor,line width= 0.6pt,line join=round] ( 38.85,383.65) --
	(229.38,383.65);

\path[draw=drawColor,line width= 0.6pt,line join=round] ( 51.27,235.34) --
	( 51.27,387.16);

\path[draw=drawColor,line width= 0.6pt,line join=round] ( 71.98,235.34) --
	( 71.98,387.16);

\path[draw=drawColor,line width= 0.6pt,line join=round] ( 92.69,235.34) --
	( 92.69,387.16);

\path[draw=drawColor,line width= 0.6pt,line join=round] (113.40,235.34) --
	(113.40,387.16);

\path[draw=drawColor,line width= 0.6pt,line join=round] (134.11,235.34) --
	(134.11,387.16);

\path[draw=drawColor,line width= 0.6pt,line join=round] (154.82,235.34) --
	(154.82,387.16);

\path[draw=drawColor,line width= 0.6pt,line join=round] (175.53,235.34) --
	(175.53,387.16);

\path[draw=drawColor,line width= 0.6pt,line join=round] (196.24,235.34) --
	(196.24,387.16);

\path[draw=drawColor,line width= 0.6pt,line join=round] (216.95,235.34) --
	(216.95,387.16);
\definecolor{fillColor}{RGB}{146,129,188}

\path[fill=fillColor] ( 46.09,242.24) rectangle ( 51.27,370.78);

\path[fill=fillColor] ( 66.80,242.24) rectangle ( 71.98,365.34);

\path[fill=fillColor] ( 87.51,242.24) rectangle ( 92.69,361.31);

\path[fill=fillColor] (108.22,242.24) rectangle (113.40,355.02);

\path[fill=fillColor] (128.93,242.24) rectangle (134.11,342.01);

\path[fill=fillColor] (149.64,242.24) rectangle (154.82,335.36);

\path[fill=fillColor] (170.35,242.24) rectangle (175.53,325.04);

\path[fill=fillColor] (191.06,242.24) rectangle (196.24,313.86);

\path[fill=fillColor] (211.77,242.24) rectangle (216.95,302.69);
\definecolor{fillColor}{RGB}{230,159,0}

\path[fill=fillColor] ( 51.27,242.24) rectangle ( 56.45,375.38);

\path[fill=fillColor] ( 71.98,242.24) rectangle ( 77.16,375.88);

\path[fill=fillColor] ( 92.69,242.24) rectangle ( 97.87,375.88);

\path[fill=fillColor] (113.40,242.24) rectangle (118.58,376.51);

\path[fill=fillColor] (134.11,242.24) rectangle (139.29,376.44);

\path[fill=fillColor] (154.82,242.24) rectangle (160.00,376.65);

\path[fill=fillColor] (175.53,242.24) rectangle (180.71,377.71);

\path[fill=fillColor] (196.24,242.24) rectangle (201.42,377.79);

\path[fill=fillColor] (216.95,242.24) rectangle (222.13,380.26);
\definecolor{drawColor}{RGB}{169,169,169}

\path[draw=drawColor,line width= 0.6pt,dash pattern=on 4pt off 4pt ,line join=round] ( 38.85,376.58) -- (229.38,376.58);
\end{scope}
\begin{scope}
\path[clip] (  0.00,  0.00) rectangle (469.75,433.62);
\definecolor{drawColor}{gray}{0.30}

\node[text=drawColor,anchor=base east,inner sep=0pt, outer sep=0pt, scale=  0.88] at ( 33.90,239.21) {0.00};

\node[text=drawColor,anchor=base east,inner sep=0pt, outer sep=0pt, scale=  0.88] at ( 33.90,274.56) {0.25};

\node[text=drawColor,anchor=base east,inner sep=0pt, outer sep=0pt, scale=  0.88] at ( 33.90,309.92) {0.50};

\node[text=drawColor,anchor=base east,inner sep=0pt, outer sep=0pt, scale=  0.88] at ( 33.90,345.27) {0.75};

\node[text=drawColor,anchor=base east,inner sep=0pt, outer sep=0pt, scale=  0.88] at ( 33.90,380.62) {1.00};
\end{scope}
\begin{scope}
\path[clip] (  0.00,  0.00) rectangle (469.75,433.62);
\definecolor{drawColor}{gray}{0.20}

\path[draw=drawColor,line width= 0.6pt,line join=round] ( 36.10,242.24) --
	( 38.85,242.24);

\path[draw=drawColor,line width= 0.6pt,line join=round] ( 36.10,277.59) --
	( 38.85,277.59);

\path[draw=drawColor,line width= 0.6pt,line join=round] ( 36.10,312.95) --
	( 38.85,312.95);

\path[draw=drawColor,line width= 0.6pt,line join=round] ( 36.10,348.30) --
	( 38.85,348.30);

\path[draw=drawColor,line width= 0.6pt,line join=round] ( 36.10,383.65) --
	( 38.85,383.65);
\end{scope}
\begin{scope}
\path[clip] (  0.00,  0.00) rectangle (469.75,433.62);
\definecolor{drawColor}{gray}{0.20}

\path[draw=drawColor,line width= 0.6pt,line join=round] ( 51.27,232.59) --
	( 51.27,235.34);

\path[draw=drawColor,line width= 0.6pt,line join=round] ( 71.98,232.59) --
	( 71.98,235.34);

\path[draw=drawColor,line width= 0.6pt,line join=round] ( 92.69,232.59) --
	( 92.69,235.34);

\path[draw=drawColor,line width= 0.6pt,line join=round] (113.40,232.59) --
	(113.40,235.34);

\path[draw=drawColor,line width= 0.6pt,line join=round] (134.11,232.59) --
	(134.11,235.34);

\path[draw=drawColor,line width= 0.6pt,line join=round] (154.82,232.59) --
	(154.82,235.34);

\path[draw=drawColor,line width= 0.6pt,line join=round] (175.53,232.59) --
	(175.53,235.34);

\path[draw=drawColor,line width= 0.6pt,line join=round] (196.24,232.59) --
	(196.24,235.34);

\path[draw=drawColor,line width= 0.6pt,line join=round] (216.95,232.59) --
	(216.95,235.34);
\end{scope}
\begin{scope}
\path[clip] (  0.00,  0.00) rectangle (469.75,433.62);
\definecolor{drawColor}{gray}{0.30}

\node[text=drawColor,anchor=base,inner sep=0pt, outer sep=0pt, scale=  0.88] at ( 51.27,224.33) {0.55};

\node[text=drawColor,anchor=base,inner sep=0pt, outer sep=0pt, scale=  0.88] at ( 71.98,224.33) {0.6};

\node[text=drawColor,anchor=base,inner sep=0pt, outer sep=0pt, scale=  0.88] at ( 92.69,224.33) {0.65};

\node[text=drawColor,anchor=base,inner sep=0pt, outer sep=0pt, scale=  0.88] at (113.40,224.33) {0.7};

\node[text=drawColor,anchor=base,inner sep=0pt, outer sep=0pt, scale=  0.88] at (134.11,224.33) {0.75};

\node[text=drawColor,anchor=base,inner sep=0pt, outer sep=0pt, scale=  0.88] at (154.82,224.33) {0.8};

\node[text=drawColor,anchor=base,inner sep=0pt, outer sep=0pt, scale=  0.88] at (175.53,224.33) {0.85};

\node[text=drawColor,anchor=base,inner sep=0pt, outer sep=0pt, scale=  0.88] at (196.24,224.33) {0.9};

\node[text=drawColor,anchor=base,inner sep=0pt, outer sep=0pt, scale=  0.88] at (216.95,224.33) {0.95};
\end{scope}
\begin{scope}
\path[clip] (  0.00,  0.00) rectangle (469.75,433.62);
\definecolor{drawColor}{RGB}{0,0,0}

\node[text=drawColor,anchor=base,inner sep=0pt, outer sep=0pt, scale=  1.10] at (134.11,212.01) {$H$};
\end{scope}
\begin{scope}
\path[clip] (  0.00,  0.00) rectangle (469.75,433.62);
\definecolor{drawColor}{RGB}{0,0,0}

\node[text=drawColor,rotate= 90.00,anchor=base,inner sep=0pt, outer sep=0pt, scale=  1.10] at ( 13.08,311.25) {rate};
\end{scope}
\begin{scope}
\path[clip] (  0.00,  0.00) rectangle (469.75,433.62);
\definecolor{drawColor}{RGB}{0,0,0}

\node[text=drawColor,anchor=base west,inner sep=0pt, outer sep=0pt, scale=  1.10] at ( 38.85,395.09) {Coverage rate};
\end{scope}
\begin{scope}
\path[clip] (234.88,204.08) rectangle (469.75,408.17);
\definecolor{drawColor}{RGB}{255,255,255}
\definecolor{fillColor}{RGB}{255,255,255}

\path[draw=drawColor,line width= 0.6pt,line join=round,line cap=round,fill=fillColor] (234.88,204.08) rectangle (469.76,408.17);
\end{scope}
\begin{scope}
\path[clip] (273.72,235.34) rectangle (464.25,387.16);
\definecolor{fillColor}{gray}{0.92}

\path[fill=fillColor] (273.72,235.34) rectangle (464.26,387.16);
\definecolor{drawColor}{RGB}{255,255,255}

\path[draw=drawColor,line width= 0.3pt,line join=round] (273.72,261.22) --
	(464.25,261.22);

\path[draw=drawColor,line width= 0.3pt,line join=round] (273.72,299.20) --
	(464.25,299.20);

\path[draw=drawColor,line width= 0.3pt,line join=round] (273.72,337.17) --
	(464.25,337.17);

\path[draw=drawColor,line width= 0.3pt,line join=round] (273.72,375.15) --
	(464.25,375.15);

\path[draw=drawColor,line width= 0.6pt,line join=round] (273.72,242.24) --
	(464.25,242.24);

\path[draw=drawColor,line width= 0.6pt,line join=round] (273.72,280.21) --
	(464.25,280.21);

\path[draw=drawColor,line width= 0.6pt,line join=round] (273.72,318.19) --
	(464.25,318.19);

\path[draw=drawColor,line width= 0.6pt,line join=round] (273.72,356.16) --
	(464.25,356.16);

\path[draw=drawColor,line width= 0.6pt,line join=round] (286.15,235.34) --
	(286.15,387.16);

\path[draw=drawColor,line width= 0.6pt,line join=round] (306.86,235.34) --
	(306.86,387.16);

\path[draw=drawColor,line width= 0.6pt,line join=round] (327.57,235.34) --
	(327.57,387.16);

\path[draw=drawColor,line width= 0.6pt,line join=round] (348.28,235.34) --
	(348.28,387.16);

\path[draw=drawColor,line width= 0.6pt,line join=round] (368.99,235.34) --
	(368.99,387.16);

\path[draw=drawColor,line width= 0.6pt,line join=round] (389.70,235.34) --
	(389.70,387.16);

\path[draw=drawColor,line width= 0.6pt,line join=round] (410.41,235.34) --
	(410.41,387.16);

\path[draw=drawColor,line width= 0.6pt,line join=round] (431.12,235.34) --
	(431.12,387.16);

\path[draw=drawColor,line width= 0.6pt,line join=round] (451.83,235.34) --
	(451.83,387.16);
\definecolor{fillColor}{RGB}{146,129,188}

\path[fill=fillColor] (280.97,242.24) rectangle (286.15,299.44);

\path[fill=fillColor] (301.68,242.24) rectangle (306.86,306.37);

\path[fill=fillColor] (322.39,242.24) rectangle (327.57,315.66);

\path[fill=fillColor] (343.10,242.24) rectangle (348.28,325.07);

\path[fill=fillColor] (363.81,242.24) rectangle (368.99,335.34);

\path[fill=fillColor] (384.52,242.24) rectangle (389.70,346.30);

\path[fill=fillColor] (405.23,242.24) rectangle (410.41,357.46);

\path[fill=fillColor] (425.94,242.24) rectangle (431.12,370.13);

\path[fill=fillColor] (446.65,242.24) rectangle (451.83,380.26);
\definecolor{fillColor}{RGB}{230,159,0}

\path[fill=fillColor] (286.15,242.24) rectangle (291.33,273.94);

\path[fill=fillColor] (306.86,242.24) rectangle (312.04,274.02);

\path[fill=fillColor] (327.57,242.24) rectangle (332.75,274.00);

\path[fill=fillColor] (348.28,242.24) rectangle (353.46,274.04);

\path[fill=fillColor] (368.99,242.24) rectangle (374.17,273.95);

\path[fill=fillColor] (389.70,242.24) rectangle (394.88,274.05);

\path[fill=fillColor] (410.41,242.24) rectangle (415.59,273.98);

\path[fill=fillColor] (431.12,242.24) rectangle (436.30,274.06);

\path[fill=fillColor] (451.83,242.24) rectangle (457.01,273.67);
\end{scope}
\begin{scope}
\path[clip] (  0.00,  0.00) rectangle (469.75,433.62);
\definecolor{drawColor}{gray}{0.30}

\node[text=drawColor,anchor=base east,inner sep=0pt, outer sep=0pt, scale=  0.88] at (268.77,239.21) {0.00};

\node[text=drawColor,anchor=base east,inner sep=0pt, outer sep=0pt, scale=  0.88] at (268.77,277.18) {0.25};

\node[text=drawColor,anchor=base east,inner sep=0pt, outer sep=0pt, scale=  0.88] at (268.77,315.16) {0.50};

\node[text=drawColor,anchor=base east,inner sep=0pt, outer sep=0pt, scale=  0.88] at (268.77,353.13) {0.75};
\end{scope}
\begin{scope}
\path[clip] (  0.00,  0.00) rectangle (469.75,433.62);
\definecolor{drawColor}{gray}{0.20}

\path[draw=drawColor,line width= 0.6pt,line join=round] (270.97,242.24) --
	(273.72,242.24);

\path[draw=drawColor,line width= 0.6pt,line join=round] (270.97,280.21) --
	(273.72,280.21);

\path[draw=drawColor,line width= 0.6pt,line join=round] (270.97,318.19) --
	(273.72,318.19);

\path[draw=drawColor,line width= 0.6pt,line join=round] (270.97,356.16) --
	(273.72,356.16);
\end{scope}
\begin{scope}
\path[clip] (  0.00,  0.00) rectangle (469.75,433.62);
\definecolor{drawColor}{gray}{0.20}

\path[draw=drawColor,line width= 0.6pt,line join=round] (286.15,232.59) --
	(286.15,235.34);

\path[draw=drawColor,line width= 0.6pt,line join=round] (306.86,232.59) --
	(306.86,235.34);

\path[draw=drawColor,line width= 0.6pt,line join=round] (327.57,232.59) --
	(327.57,235.34);

\path[draw=drawColor,line width= 0.6pt,line join=round] (348.28,232.59) --
	(348.28,235.34);

\path[draw=drawColor,line width= 0.6pt,line join=round] (368.99,232.59) --
	(368.99,235.34);

\path[draw=drawColor,line width= 0.6pt,line join=round] (389.70,232.59) --
	(389.70,235.34);

\path[draw=drawColor,line width= 0.6pt,line join=round] (410.41,232.59) --
	(410.41,235.34);

\path[draw=drawColor,line width= 0.6pt,line join=round] (431.12,232.59) --
	(431.12,235.34);

\path[draw=drawColor,line width= 0.6pt,line join=round] (451.83,232.59) --
	(451.83,235.34);
\end{scope}
\begin{scope}
\path[clip] (  0.00,  0.00) rectangle (469.75,433.62);
\definecolor{drawColor}{gray}{0.30}

\node[text=drawColor,anchor=base,inner sep=0pt, outer sep=0pt, scale=  0.88] at (286.15,224.33) {0.55};

\node[text=drawColor,anchor=base,inner sep=0pt, outer sep=0pt, scale=  0.88] at (306.86,224.33) {0.6};

\node[text=drawColor,anchor=base,inner sep=0pt, outer sep=0pt, scale=  0.88] at (327.57,224.33) {0.65};

\node[text=drawColor,anchor=base,inner sep=0pt, outer sep=0pt, scale=  0.88] at (348.28,224.33) {0.7};

\node[text=drawColor,anchor=base,inner sep=0pt, outer sep=0pt, scale=  0.88] at (368.99,224.33) {0.75};

\node[text=drawColor,anchor=base,inner sep=0pt, outer sep=0pt, scale=  0.88] at (389.70,224.33) {0.8};

\node[text=drawColor,anchor=base,inner sep=0pt, outer sep=0pt, scale=  0.88] at (410.41,224.33) {0.85};

\node[text=drawColor,anchor=base,inner sep=0pt, outer sep=0pt, scale=  0.88] at (431.12,224.33) {0.9};

\node[text=drawColor,anchor=base,inner sep=0pt, outer sep=0pt, scale=  0.88] at (451.83,224.33) {0.95};
\end{scope}
\begin{scope}
\path[clip] (  0.00,  0.00) rectangle (469.75,433.62);
\definecolor{drawColor}{RGB}{0,0,0}

\node[text=drawColor,anchor=base,inner sep=0pt, outer sep=0pt, scale=  1.10] at (368.99,212.01) {$H$};
\end{scope}
\begin{scope}
\path[clip] (  0.00,  0.00) rectangle (469.75,433.62);
\definecolor{drawColor}{RGB}{0,0,0}

\node[text=drawColor,rotate= 90.00,anchor=base,inner sep=0pt, outer sep=0pt, scale=  1.10] at (247.95,311.25) {length};
\end{scope}
\begin{scope}
\path[clip] (  0.00,  0.00) rectangle (469.75,433.62);
\definecolor{drawColor}{RGB}{0,0,0}

\node[text=drawColor,anchor=base west,inner sep=0pt, outer sep=0pt, scale=  1.10] at (273.72,395.09) {Interval length};
\end{scope}
\begin{scope}
\path[clip] (  0.00,  0.00) rectangle (234.88,204.08);
\definecolor{drawColor}{RGB}{255,255,255}
\definecolor{fillColor}{RGB}{255,255,255}

\path[draw=drawColor,line width= 0.6pt,line join=round,line cap=round,fill=fillColor] (  0.00,  0.00) rectangle (234.88,204.08);
\end{scope}
\begin{scope}
\path[clip] ( 38.85, 31.25) rectangle (229.38,183.08);
\definecolor{fillColor}{gray}{0.92}

\path[fill=fillColor] ( 38.85, 31.25) rectangle (229.38,183.08);
\definecolor{drawColor}{RGB}{255,255,255}

\path[draw=drawColor,line width= 0.3pt,line join=round] ( 38.85, 56.18) --
	(229.38, 56.18);

\path[draw=drawColor,line width= 0.3pt,line join=round] ( 38.85, 92.24) --
	(229.38, 92.24);

\path[draw=drawColor,line width= 0.3pt,line join=round] ( 38.85,128.29) --
	(229.38,128.29);

\path[draw=drawColor,line width= 0.3pt,line join=round] ( 38.85,164.35) --
	(229.38,164.35);

\path[draw=drawColor,line width= 0.6pt,line join=round] ( 38.85, 38.15) --
	(229.38, 38.15);

\path[draw=drawColor,line width= 0.6pt,line join=round] ( 38.85, 74.21) --
	(229.38, 74.21);

\path[draw=drawColor,line width= 0.6pt,line join=round] ( 38.85,110.27) --
	(229.38,110.27);

\path[draw=drawColor,line width= 0.6pt,line join=round] ( 38.85,146.32) --
	(229.38,146.32);

\path[draw=drawColor,line width= 0.6pt,line join=round] ( 38.85,182.38) --
	(229.38,182.38);

\path[draw=drawColor,line width= 0.6pt,line join=round] ( 51.27, 31.25) --
	( 51.27,183.08);

\path[draw=drawColor,line width= 0.6pt,line join=round] ( 71.98, 31.25) --
	( 71.98,183.08);

\path[draw=drawColor,line width= 0.6pt,line join=round] ( 92.69, 31.25) --
	( 92.69,183.08);

\path[draw=drawColor,line width= 0.6pt,line join=round] (113.40, 31.25) --
	(113.40,183.08);

\path[draw=drawColor,line width= 0.6pt,line join=round] (134.11, 31.25) --
	(134.11,183.08);

\path[draw=drawColor,line width= 0.6pt,line join=round] (154.82, 31.25) --
	(154.82,183.08);

\path[draw=drawColor,line width= 0.6pt,line join=round] (175.53, 31.25) --
	(175.53,183.08);

\path[draw=drawColor,line width= 0.6pt,line join=round] (196.24, 31.25) --
	(196.24,183.08);

\path[draw=drawColor,line width= 0.6pt,line join=round] (216.95, 31.25) --
	(216.95,183.08);
\definecolor{fillColor}{RGB}{146,129,188}

\path[fill=fillColor] ( 46.09, 38.15) rectangle ( 51.27,170.34);

\path[fill=fillColor] ( 66.80, 38.15) rectangle ( 71.98,165.65);

\path[fill=fillColor] ( 87.51, 38.15) rectangle ( 92.69,158.94);

\path[fill=fillColor] (108.22, 38.15) rectangle (113.40,151.73);

\path[fill=fillColor] (128.93, 38.15) rectangle (134.11,141.20);

\path[fill=fillColor] (149.64, 38.15) rectangle (154.82,130.96);

\path[fill=fillColor] (170.35, 38.15) rectangle (175.53,122.81);

\path[fill=fillColor] (191.06, 38.15) rectangle (196.24,106.52);

\path[fill=fillColor] (211.77, 38.15) rectangle (216.95, 95.70);
\definecolor{fillColor}{RGB}{230,159,0}

\path[fill=fillColor] ( 51.27, 38.15) rectangle ( 56.45,175.31);

\path[fill=fillColor] ( 71.98, 38.15) rectangle ( 77.16,174.95);

\path[fill=fillColor] ( 92.69, 38.15) rectangle ( 97.87,176.18);

\path[fill=fillColor] (113.40, 38.15) rectangle (118.58,175.38);

\path[fill=fillColor] (134.11, 38.15) rectangle (139.29,174.88);

\path[fill=fillColor] (154.82, 38.15) rectangle (160.00,173.87);

\path[fill=fillColor] (175.53, 38.15) rectangle (180.71,172.35);

\path[fill=fillColor] (196.24, 38.15) rectangle (201.42,171.92);

\path[fill=fillColor] (216.95, 38.15) rectangle (222.13,171.78);
\definecolor{drawColor}{RGB}{169,169,169}

\path[draw=drawColor,line width= 0.6pt,dash pattern=on 4pt off 4pt ,line join=round] ( 38.85,175.17) -- (229.38,175.17);
\end{scope}
\begin{scope}
\path[clip] (  0.00,  0.00) rectangle (469.75,433.62);
\definecolor{drawColor}{gray}{0.30}

\node[text=drawColor,anchor=base east,inner sep=0pt, outer sep=0pt, scale=  0.88] at ( 33.90, 35.12) {0.00};

\node[text=drawColor,anchor=base east,inner sep=0pt, outer sep=0pt, scale=  0.88] at ( 33.90, 71.18) {0.25};

\node[text=drawColor,anchor=base east,inner sep=0pt, outer sep=0pt, scale=  0.88] at ( 33.90,107.24) {0.50};

\node[text=drawColor,anchor=base east,inner sep=0pt, outer sep=0pt, scale=  0.88] at ( 33.90,143.29) {0.75};

\node[text=drawColor,anchor=base east,inner sep=0pt, outer sep=0pt, scale=  0.88] at ( 33.90,179.35) {1.00};
\end{scope}
\begin{scope}
\path[clip] (  0.00,  0.00) rectangle (469.75,433.62);
\definecolor{drawColor}{gray}{0.20}

\path[draw=drawColor,line width= 0.6pt,line join=round] ( 36.10, 38.15) --
	( 38.85, 38.15);

\path[draw=drawColor,line width= 0.6pt,line join=round] ( 36.10, 74.21) --
	( 38.85, 74.21);

\path[draw=drawColor,line width= 0.6pt,line join=round] ( 36.10,110.27) --
	( 38.85,110.27);

\path[draw=drawColor,line width= 0.6pt,line join=round] ( 36.10,146.32) --
	( 38.85,146.32);

\path[draw=drawColor,line width= 0.6pt,line join=round] ( 36.10,182.38) --
	( 38.85,182.38);
\end{scope}
\begin{scope}
\path[clip] (  0.00,  0.00) rectangle (469.75,433.62);
\definecolor{drawColor}{gray}{0.20}

\path[draw=drawColor,line width= 0.6pt,line join=round] ( 51.27, 28.50) --
	( 51.27, 31.25);

\path[draw=drawColor,line width= 0.6pt,line join=round] ( 71.98, 28.50) --
	( 71.98, 31.25);

\path[draw=drawColor,line width= 0.6pt,line join=round] ( 92.69, 28.50) --
	( 92.69, 31.25);

\path[draw=drawColor,line width= 0.6pt,line join=round] (113.40, 28.50) --
	(113.40, 31.25);

\path[draw=drawColor,line width= 0.6pt,line join=round] (134.11, 28.50) --
	(134.11, 31.25);

\path[draw=drawColor,line width= 0.6pt,line join=round] (154.82, 28.50) --
	(154.82, 31.25);

\path[draw=drawColor,line width= 0.6pt,line join=round] (175.53, 28.50) --
	(175.53, 31.25);

\path[draw=drawColor,line width= 0.6pt,line join=round] (196.24, 28.50) --
	(196.24, 31.25);

\path[draw=drawColor,line width= 0.6pt,line join=round] (216.95, 28.50) --
	(216.95, 31.25);
\end{scope}
\begin{scope}
\path[clip] (  0.00,  0.00) rectangle (469.75,433.62);
\definecolor{drawColor}{gray}{0.30}

\node[text=drawColor,anchor=base,inner sep=0pt, outer sep=0pt, scale=  0.88] at ( 51.27, 20.24) {0.55};

\node[text=drawColor,anchor=base,inner sep=0pt, outer sep=0pt, scale=  0.88] at ( 71.98, 20.24) {0.6};

\node[text=drawColor,anchor=base,inner sep=0pt, outer sep=0pt, scale=  0.88] at ( 92.69, 20.24) {0.65};

\node[text=drawColor,anchor=base,inner sep=0pt, outer sep=0pt, scale=  0.88] at (113.40, 20.24) {0.7};

\node[text=drawColor,anchor=base,inner sep=0pt, outer sep=0pt, scale=  0.88] at (134.11, 20.24) {0.75};

\node[text=drawColor,anchor=base,inner sep=0pt, outer sep=0pt, scale=  0.88] at (154.82, 20.24) {0.8};

\node[text=drawColor,anchor=base,inner sep=0pt, outer sep=0pt, scale=  0.88] at (175.53, 20.24) {0.85};

\node[text=drawColor,anchor=base,inner sep=0pt, outer sep=0pt, scale=  0.88] at (196.24, 20.24) {0.9};

\node[text=drawColor,anchor=base,inner sep=0pt, outer sep=0pt, scale=  0.88] at (216.95, 20.24) {0.95};
\end{scope}
\begin{scope}
\path[clip] (  0.00,  0.00) rectangle (469.75,433.62);
\definecolor{drawColor}{RGB}{0,0,0}

\node[text=drawColor,anchor=base,inner sep=0pt, outer sep=0pt, scale=  1.10] at (134.11,  7.93) {$H$};
\end{scope}
\begin{scope}
\path[clip] (  0.00,  0.00) rectangle (469.75,433.62);
\definecolor{drawColor}{RGB}{0,0,0}

\node[text=drawColor,rotate= 90.00,anchor=base,inner sep=0pt, outer sep=0pt, scale=  1.10] at ( 13.08,107.17) {rate};
\end{scope}
\begin{scope}
\path[clip] (  0.00,  0.00) rectangle (469.75,433.62);
\definecolor{drawColor}{RGB}{0,0,0}

\node[text=drawColor,anchor=base west,inner sep=0pt, outer sep=0pt, scale=  1.10] at ( 38.85,191.01) {Coverage rate};
\end{scope}
\begin{scope}
\path[clip] (234.88,  0.00) rectangle (469.75,204.08);
\definecolor{drawColor}{RGB}{255,255,255}
\definecolor{fillColor}{RGB}{255,255,255}

\path[draw=drawColor,line width= 0.6pt,line join=round,line cap=round,fill=fillColor] (234.88,  0.00) rectangle (469.76,204.08);
\end{scope}
\begin{scope}
\path[clip] (269.32, 31.25) rectangle (464.25,183.08);
\definecolor{fillColor}{gray}{0.92}

\path[fill=fillColor] (269.32, 31.25) rectangle (464.25,183.08);
\definecolor{drawColor}{RGB}{255,255,255}

\path[draw=drawColor,line width= 0.3pt,line join=round] (269.32, 56.20) --
	(464.25, 56.20);

\path[draw=drawColor,line width= 0.3pt,line join=round] (269.32, 92.28) --
	(464.25, 92.28);

\path[draw=drawColor,line width= 0.3pt,line join=round] (269.32,128.37) --
	(464.25,128.37);

\path[draw=drawColor,line width= 0.3pt,line join=round] (269.32,164.45) --
	(464.25,164.45);

\path[draw=drawColor,line width= 0.6pt,line join=round] (269.32, 38.15) --
	(464.25, 38.15);

\path[draw=drawColor,line width= 0.6pt,line join=round] (269.32, 74.24) --
	(464.25, 74.24);

\path[draw=drawColor,line width= 0.6pt,line join=round] (269.32,110.32) --
	(464.25,110.32);

\path[draw=drawColor,line width= 0.6pt,line join=round] (269.32,146.41) --
	(464.25,146.41);

\path[draw=drawColor,line width= 0.6pt,line join=round] (269.32,182.49) --
	(464.25,182.49);

\path[draw=drawColor,line width= 0.6pt,line join=round] (282.04, 31.25) --
	(282.04,183.08);

\path[draw=drawColor,line width= 0.6pt,line join=round] (303.23, 31.25) --
	(303.23,183.08);

\path[draw=drawColor,line width= 0.6pt,line join=round] (324.41, 31.25) --
	(324.41,183.08);

\path[draw=drawColor,line width= 0.6pt,line join=round] (345.60, 31.25) --
	(345.60,183.08);

\path[draw=drawColor,line width= 0.6pt,line join=round] (366.79, 31.25) --
	(366.79,183.08);

\path[draw=drawColor,line width= 0.6pt,line join=round] (387.98, 31.25) --
	(387.98,183.08);

\path[draw=drawColor,line width= 0.6pt,line join=round] (409.17, 31.25) --
	(409.17,183.08);

\path[draw=drawColor,line width= 0.6pt,line join=round] (430.35, 31.25) --
	(430.35,183.08);

\path[draw=drawColor,line width= 0.6pt,line join=round] (451.54, 31.25) --
	(451.54,183.08);
\definecolor{fillColor}{RGB}{146,129,188}

\path[fill=fillColor] (276.74, 38.15) rectangle (282.04, 71.82);

\path[fill=fillColor] (297.93, 38.15) rectangle (303.23, 79.46);

\path[fill=fillColor] (319.12, 38.15) rectangle (324.41, 87.95);

\path[fill=fillColor] (340.30, 38.15) rectangle (345.60, 98.82);

\path[fill=fillColor] (361.49, 38.15) rectangle (366.79,110.91);

\path[fill=fillColor] (382.68, 38.15) rectangle (387.98,124.72);

\path[fill=fillColor] (403.87, 38.15) rectangle (409.17,142.47);

\path[fill=fillColor] (425.06, 38.15) rectangle (430.35,159.19);

\path[fill=fillColor] (446.25, 38.15) rectangle (451.54,176.18);
\definecolor{fillColor}{RGB}{230,159,0}

\path[fill=fillColor] (282.04, 38.15) rectangle (287.33, 55.05);

\path[fill=fillColor] (303.23, 38.15) rectangle (308.52, 55.05);

\path[fill=fillColor] (324.41, 38.15) rectangle (329.71, 55.05);

\path[fill=fillColor] (345.60, 38.15) rectangle (350.90, 55.05);

\path[fill=fillColor] (366.79, 38.15) rectangle (372.09, 55.05);

\path[fill=fillColor] (387.98, 38.15) rectangle (393.27, 55.06);

\path[fill=fillColor] (409.17, 38.15) rectangle (414.46, 55.04);

\path[fill=fillColor] (430.35, 38.15) rectangle (435.65, 55.05);

\path[fill=fillColor] (451.54, 38.15) rectangle (456.84, 54.96);
\end{scope}
\begin{scope}
\path[clip] (  0.00,  0.00) rectangle (469.75,433.62);
\definecolor{drawColor}{gray}{0.30}

\node[text=drawColor,anchor=base east,inner sep=0pt, outer sep=0pt, scale=  0.88] at (264.37, 35.12) {0.0};

\node[text=drawColor,anchor=base east,inner sep=0pt, outer sep=0pt, scale=  0.88] at (264.37, 71.21) {0.2};

\node[text=drawColor,anchor=base east,inner sep=0pt, outer sep=0pt, scale=  0.88] at (264.37,107.29) {0.4};

\node[text=drawColor,anchor=base east,inner sep=0pt, outer sep=0pt, scale=  0.88] at (264.37,143.38) {0.6};

\node[text=drawColor,anchor=base east,inner sep=0pt, outer sep=0pt, scale=  0.88] at (264.37,179.46) {0.8};
\end{scope}
\begin{scope}
\path[clip] (  0.00,  0.00) rectangle (469.75,433.62);
\definecolor{drawColor}{gray}{0.20}

\path[draw=drawColor,line width= 0.6pt,line join=round] (266.57, 38.15) --
	(269.32, 38.15);

\path[draw=drawColor,line width= 0.6pt,line join=round] (266.57, 74.24) --
	(269.32, 74.24);

\path[draw=drawColor,line width= 0.6pt,line join=round] (266.57,110.32) --
	(269.32,110.32);

\path[draw=drawColor,line width= 0.6pt,line join=round] (266.57,146.41) --
	(269.32,146.41);

\path[draw=drawColor,line width= 0.6pt,line join=round] (266.57,182.49) --
	(269.32,182.49);
\end{scope}
\begin{scope}
\path[clip] (  0.00,  0.00) rectangle (469.75,433.62);
\definecolor{drawColor}{gray}{0.20}

\path[draw=drawColor,line width= 0.6pt,line join=round] (282.04, 28.50) --
	(282.04, 31.25);

\path[draw=drawColor,line width= 0.6pt,line join=round] (303.23, 28.50) --
	(303.23, 31.25);

\path[draw=drawColor,line width= 0.6pt,line join=round] (324.41, 28.50) --
	(324.41, 31.25);

\path[draw=drawColor,line width= 0.6pt,line join=round] (345.60, 28.50) --
	(345.60, 31.25);

\path[draw=drawColor,line width= 0.6pt,line join=round] (366.79, 28.50) --
	(366.79, 31.25);

\path[draw=drawColor,line width= 0.6pt,line join=round] (387.98, 28.50) --
	(387.98, 31.25);

\path[draw=drawColor,line width= 0.6pt,line join=round] (409.17, 28.50) --
	(409.17, 31.25);

\path[draw=drawColor,line width= 0.6pt,line join=round] (430.35, 28.50) --
	(430.35, 31.25);

\path[draw=drawColor,line width= 0.6pt,line join=round] (451.54, 28.50) --
	(451.54, 31.25);
\end{scope}
\begin{scope}
\path[clip] (  0.00,  0.00) rectangle (469.75,433.62);
\definecolor{drawColor}{gray}{0.30}

\node[text=drawColor,anchor=base,inner sep=0pt, outer sep=0pt, scale=  0.88] at (282.04, 20.24) {0.55};

\node[text=drawColor,anchor=base,inner sep=0pt, outer sep=0pt, scale=  0.88] at (303.23, 20.24) {0.6};

\node[text=drawColor,anchor=base,inner sep=0pt, outer sep=0pt, scale=  0.88] at (324.41, 20.24) {0.65};

\node[text=drawColor,anchor=base,inner sep=0pt, outer sep=0pt, scale=  0.88] at (345.60, 20.24) {0.7};

\node[text=drawColor,anchor=base,inner sep=0pt, outer sep=0pt, scale=  0.88] at (366.79, 20.24) {0.75};

\node[text=drawColor,anchor=base,inner sep=0pt, outer sep=0pt, scale=  0.88] at (387.98, 20.24) {0.8};

\node[text=drawColor,anchor=base,inner sep=0pt, outer sep=0pt, scale=  0.88] at (409.17, 20.24) {0.85};

\node[text=drawColor,anchor=base,inner sep=0pt, outer sep=0pt, scale=  0.88] at (430.35, 20.24) {0.9};

\node[text=drawColor,anchor=base,inner sep=0pt, outer sep=0pt, scale=  0.88] at (451.54, 20.24) {0.95};
\end{scope}
\begin{scope}
\path[clip] (  0.00,  0.00) rectangle (469.75,433.62);
\definecolor{drawColor}{RGB}{0,0,0}

\node[text=drawColor,anchor=base,inner sep=0pt, outer sep=0pt, scale=  1.10] at (366.79,  7.93) {$H$};
\end{scope}
\begin{scope}
\path[clip] (  0.00,  0.00) rectangle (469.75,433.62);
\definecolor{drawColor}{RGB}{0,0,0}

\node[text=drawColor,rotate= 90.00,anchor=base,inner sep=0pt, outer sep=0pt, scale=  1.10] at (247.95,107.17) {length};
\end{scope}
\begin{scope}
\path[clip] (  0.00,  0.00) rectangle (469.75,433.62);
\definecolor{drawColor}{RGB}{0,0,0}

\node[text=drawColor,anchor=base west,inner sep=0pt, outer sep=0pt, scale=  1.10] at (269.32,191.01) {Interval length};
\end{scope}
\end{tikzpicture}